\documentclass[11pt, a4paper,leqno]{amsart}
\usepackage[OT1]{fontenc}

\usepackage{amsmath,amsthm,amscd,amssymb,amsfonts, amsbsy}
\usepackage{latexsym}
\usepackage{txfonts}
\usepackage{exscale}
\usepackage{enumerate}
\usepackage{enumitem}
\usepackage{moreenum}
 \usepackage{float}
\usepackage{mathrsfs}
\usepackage{currvita}
\usepackage[hidelinks,pagebackref,hypertexnames=false]{hyperref}

\usepackage[a4paper, total={6.5in, 10in}]{geometry}

\usepackage[active]{srcltx}

\calclayout
\allowdisplaybreaks

\theoremstyle{plain}
\newtheorem{theorem}[equation]{Theorem}
\newtheorem{lemma}[equation]{Lemma}
\newtheorem{corollary}[equation]{Corollary}
\newtheorem{proposition}[equation]{Proposition}

\newtheorem*{conjecture*}{Conjecture}

\theoremstyle{definition}
\newtheorem{definition}[equation]{Definition}

\theoremstyle{remark}
\newtheorem{remark}[equation]{Remark}

\numberwithin{equation}{section}

\newcommand{\RR}{\mathbb{R}}

\def\Xint#1{\mathchoice
{\XXint\displaystyle\textstyle{#1}}%
{\XXint\textstyle\scriptstyle{#1}}%
{\XXint\scriptstyle\scriptscriptstyle{#1}}%
{\XXint\scriptscriptstyle%
\scriptscriptstyle{#1}}%
\!\int}
\def\XXint#1#2#3{{\setbox0=\hbox{$#1{#2#3}{%
\int}$ }
\vcenter{\hbox{$#2#3$ }}\kern-.6\wd0}}
\def\barint{\,\Xint -} 
\def\bariint{\barint_{} \kern-.4em \barint}
\def\bariiint{\bariint_{} \kern-.4em \barint}
\renewcommand{\iint}{\int_{}\kern-.34em \int} 
\renewcommand{\iiint}{\iint_{}\kern-.34em \int} 

\newcommand{\la}{\lambda}

\newcommand{\dist}{\operatorname{dist}}

\newcommand{\I}{\mathcal{I}}
\newcommand{\J}{\mathcal{J}}

\renewcommand{\d}{\, \mathrm{d}}

\DeclareMathOperator{\supp}{supp}

\DeclareMathOperator{\diam}{diam}

\begin{document}
\allowdisplaybreaks

\title[Riesz Transform Bounds Imply Parabolic Uniform Rectifiability of Graphs]{Riesz Transform Bounds Imply Parabolic Uniform Rectifiability of Graphs}

\author[John Hoffman]{John Hoffman}
\address{Department of Mathematics, University of the Basque Country (EHU), 48940 Leioa, Spain}
\email{johnloche.hoffman@ehu.eus}
\thanks{{\fontencoding{OT1}\fontfamily{txr}\bfseries Funding.} The author was supported in part by grant EUR2025-165042, funded by MICIU/AEI/10.13039/501100011033, and by grant PID2024-157724NB-I00, funded by MICIU/AEI/10.13039/501100011033 and by ERDF/EU}



\date{October 23, 2026}

\begin{abstract}
We prove that if a $(k+2)$-Ahlfors--David regular measure $\mu$ in
parabolic space $\mathbb P^n$ has support equal to a parabolic Lipschitz graph with
sufficiently small Lipschitz constant, and the parabolic Riesz transform is bounded on $L^2(\mu)$, then $\mu$ is parabolic uniformly
rectifiable and the graph is regular parabolic Lipschitz. The result holds
for every integer $0\leq k\leq n-1$.
\end{abstract}

\maketitle
\tableofcontents

\section{Introduction}

Given a measure $\mu$ on $\mathbb{R}^n$, the relationship between the $L^2(\mu)$ boundedness of singular integral operators and the regularity properties of $\mu$ has been a longstanding subject of study. In 1991, David and Semmes (\cite{DS1}) showed that for measures $\mu$ satisfying the Ahlfors-David regularity condition (i.e., surface balls have measure that scales appropriately with their radii), $L^2(\mu)$ boundedness of a large class of singular integral operators implies uniform rectifiability.  Loosely speaking, an Ahlfors-David regular measure $\mu$ is uniformly rectifiable if it exhibits approximate Lipschitz regularity at every scale. In integer dimensions, Jaye, Nazarov, and Tolsa \cite{JNT} extended this characterization to general non-atomic measures. Their characterization replaces uniform rectifiability with an upper growth bound and a square-function estimate in which the squared $\beta$-numbers are weighted by the square of the density.

David and Semmes conjectured that, for every integer $1\leq k\leq n-1$, a $k$-ADR measure $\mu$ on $\mathbb R^n$ is uniformly rectifiable if the $k$-dimensional Riesz transform is bounded on $L^2(\mu)$ \cite{DS2}. For $1$-ADR measures in $\mathbb C$, this was answered affirmatively by Mattila, Melnikov, and Verdera through their characterization of uniform rectifiability in terms of the Cauchy transform \cite{MMV}. For $(n-1)$-ADR measures in $\mathbb R^n$, Nazarov, Tolsa, and Volberg proved the corresponding characterization using the Riesz transforms \cite{NTV}. A solution of the David--Semmes conjecture in the remaining higher-codimension cases has recently been claimed in \cite{OpenAI2026Riesz}. Its blow-up argument appears to be heavily influenced by the cylindrical blow-up procedure of Jaye, Nazarov, and Tolsa \cite[Section~11]{JNT}, although that work is not cited in \cite{OpenAI2026Riesz}.

For each integer $d\geq0$, we write
$\mathbb P^d=\mathbb R^d_{sp}\times\mathbb R$ for parabolic space, equipped with the metric
\[
d_p\big((x,t),(y,s)\big)=|x-y|+|t-s|^{1/2}.
\]
In particular, $\mathbb P^0$ is the time line with distance $|t-s|^{1/2}$.

  For a $(k+2)$-ADR measure $\mu$ on $\mathbb P^n$, $0\leq k\leq n-1$, set
\[
\beta_{2,\mu}((X,t),r)^2
:=\inf_L\frac{1}{r^{k+2}}\int_{C_r(X,t)}
\left(\frac{d_p((Y,s),L)}{r}\right)^2\,\d\mu(Y,s),
\]
where $C_r(X,t)=\{(Y,s):|Y-X|<r,\ |s-t|<r^2\}$ and the infimum is taken over vertical planes $L=P\times\mathbb R$, with $P\subset\mathbb R^n$ an affine $k$-plane. We say that $\mu$ is \textit{parabolic uniformly rectifiable} if there is a constant $C$ such that the $\beta_2$-number square function satisfies the Carleson measure estimate
\[
\int_0^R\int_{C_R(X,t)}\beta_{2,\mu}((Y,s),r)^2
\,\d\mu(Y,s)\,\frac{\d r}{r}
\leq CR^{k+2}
\]
for every $(X,t)\in\supp\mu$ and $0<R<\operatorname{diam}_p(\supp\mu)$.

In this paper we consider whether boundedness of the parabolic Riesz transform implies regularity of parabolic Lipschitz graphs. For a $(k+2)$-ADR measure, the transform is given formally by convolution against the vector-valued kernel
\[
K_k^+(X,t):=
\begin{cases}
\displaystyle \frac{X}{t^{(k+3)/2}}\exp\left(-\frac{|X|^2}{4t}\right),&t>0,\\[4pt]
0,&t\leq0,
\end{cases}
\qquad (X,t)\in\mathbb P^n\setminus\{(0,0)\}.
\]
When $k=n-1$, this is a constant multiple of the spatial gradient of the heat kernel. We call a graph \textit{parabolic Lipschitz} if its defining function is Lip$(1,1/2)$, that is, Lipschitz with respect to the parabolic metric.  Unlike Euclidean Lipschitz graphs, parabolic Lipschitz graphs need not be uniformly rectifiable; see \cite{KW80,LS88} and \cite[pp.~383--384]{HLN}. Thus, even for graphs, there is an additional regularity condition to recover from singular integral bounds.

Hofmann characterized regular parabolic Lipschitz functions in terms of the $L^2$ boundedness of the parabolic commutator \cite{H2}. For parabolic Lipschitz graphs, this regularity is equivalent to parabolic uniform rectifiability by the work of Hofmann, Lewis, and Nystr\"om \cite{HLN2}. More generally, the present author and Jaye proved that the $L^2(\mu)$ boundedness of all parabolic singular integral operators, together with the higher spatial dimension condition (HSDC), implies parabolic uniform rectifiability \cite{HJ}. In the converse direction, Bortz, the present author, Hofmann, Luna-Garc\'ia, and Nystr\"om showed that in codimension one, parabolic uniform rectifiability implies the $L^2(\mu)$ boundedness of all sufficiently nice parabolic singular integral operators \cite{BHHLN3}.  
In \cite{HJ} the authors provide examples of $(k+2)$-ADR measures for which all admissible parabolic singular integral operators are bounded on $L^2$, but which fail to be parabolic uniformly rectifiable. The HSDC is a sharp extra condition one can impose in addition to boundedness of singular integrals on a $(k+2)$-ADR measure to obtain parabolic uniform rectifiability.

Roughly speaking, measures that fail the HSDC have too many scales at which they are too concentrated in time.  In codimension one, the prototypical HSDC-bad measures are $\mathcal L^n\times\nu$, where $\nu$ has unbounded support and is $1/2$-ADR in the usual Euclidean metric on the time line \cite[Theorem~7.1(A)]{HJ}. The HSDC requires a Carleson packing bound on the locations and scales at which the measure is close to such product measures; see \cite[Section~3.1]{HJ}.  Thus, in \cite{HJ}, the authors conjectured the following:

\begin{conjecture*}[{\cite[Conjecture~1.3]{HJ}}]
Let $\mu$ be an $(n+1)$-ADR measure on $\mathbb P^n$ satisfying the $\varepsilon$-HSDC for some $\varepsilon>0$. If the parabolic Riesz transform is bounded on $L^2(\mu)$, then $\mu$ is parabolic uniformly rectifiable.
\end{conjecture*}

In this paper, we make a first step toward resolving this conjecture.  Here we assume bounds for the parabolic Riesz transforms on a graph with small Lipschitz constant, in which case the HSDC is automatically satisfied. We prove the following theorem.

\begin{theorem}\label{mainthm}
Let $0\leq k\leq n-1$, and suppose $A:\mathbb P^k\to\mathbb R^{n-k}$ is a Lip-$(1,1/2)$ function with $\|A\|_{\mathrm{lip}}\ll1$. Set $\Gamma=\{(x,A(x,t),t):(x,t)\in\mathbb P^k\} \subseteq \mathbb P^n$, and let $\mu$ be a $(k+2)$-ADR measure such that $\operatorname{supp}\mu=\Gamma$. If the parabolic Riesz transforms in each of the $n-k$ transversal directions are bounded from $L^2(\mu)$ to $L^2(\mu)$, then $\mu$ is a $(k+2)$-parabolic uniformly rectifiable measure and $A$ is a regular Lip-$(1,1/2)$ function.
\end{theorem}

A more precise statement is given in Theorem~\ref{mainthm2}. Notice that only the $n-k$ transversal components of the transform are needed. The theorem holds in every integer codimension, including the case $k=0$, where the graph has the form $\{(A(t),t):t\in\mathbb R\}$.

The connection between Riesz transforms and rectifiability has also played an important role in free boundary problems for harmonic measure. Estimates for the Green function can be used to obtain bounds for the Riesz transform, allowing one to apply the theorem of Nazarov, Tolsa, and Volberg and its refinements to deduce rectifiability. This approach has been used to prove rectifiability from absolute continuity assumptions on harmonic measure, including in two-phase problems involving the harmonic measures of two disjoint domains; see Azzam, Hofmann, Martell, Mayboroda, Mourgoglou, Tolsa, and Volberg \cite{AHMMMTV} and Azzam, Mourgoglou, and Tolsa \cite{AMT20}.

We first prove the following estimate for graphs of compactly supported functions. This theorem is a parabolic analogue of \cite[Theorem~1.3]{To2}.

\begin{theorem}\label{localdpara-intro}
Let $0\leq k\leq n-1$, and suppose $A:\mathbb P^k\to\mathbb R^{n-k}$ is Lip-$(1,1/2)$ with $\|A\|_{\mathrm{lip}}\ll1$ and is supported in a cylinder $C'_R$ of radius $R$. Let $\Gamma$ be the graph of $A$, and let $\mu$ be a $(k+2)$-ADR measure such that $\operatorname{supp}\mu=\Gamma$. If the parabolic Riesz transforms in each of the $n-k$ transversal directions are bounded from $L^2(\mu)$ to $L^2(\mu)$, then
\[
\|D_{par}A\|_{L^2(\mathbb P^k)}^2\lesssim R^{k+2},
\]
where $D_{par}A$ is defined component-wise by
\[
D_{par}A_m(x,t)
:=
\big(\widehat A_m(\xi,\tau)\|(\xi,\tau)\|\big)^{\check{}}(x,t),
\qquad m=1,\dots,n-k.
\]
\end{theorem}

The proof of the theorem above uses ideas from \cite{To1} and \cite{To2}. In \cite{To2}, the author shows that compactly supported Lipschitz functions with small Lipschitz constant have the property that the $L^2$ norm of their gradient is controlled by the Riesz transform bounds of the graph. This result of \cite{To2}, in turn, relies on tools (relating to the so-called ``$\alpha$-numbers'') developed in \cite{To1}. The novel step in the proof of Theorem~\ref{localdpara-intro} lies in showing that the smoothed graphs satisfy Riesz transform bounds that are uniform in the smoothing parameter; see Subsection~\ref{smoothing}. For Theorem~\ref{mainthm}, the corresponding novel step is to show that the localized graphs retain Riesz transform bounds; see Subsection~\ref{localization}.

To deduce Theorem~\ref{mainthm}, we fix a surface ball of $\Gamma$ and construct a graph that agrees with $\Gamma$ on this ball and whose defining function, after subtracting a constant, is supported in a fixed larger cylinder in $\mathbb P^k$. Proposition~\ref{localsios} gives bounds for the time-symmetric Riesz transforms in the transversal directions on the graph of the localized function. After the reduction in Subsection~\ref{symreduction}, the proof of Theorem~\ref{localdpara-intro} uses only these bounds. Applying this proof and Lemma~\ref{vectorDorronsoro}, we obtain the Carleson estimate for the $\beta_2$ square function over the chosen surface ball.  The two graphs agree on the original surface ball, so the same estimate holds there for $\Gamma$. Since the ball was arbitrary, $\Gamma$ is parabolic uniformly rectifiable.

The proof of Theorem~\ref{localdpara-intro} also requires an estimate for interactions between different scales of the transform. We prove this estimate in Proposition~\ref{offdiagprop}, using the $\alpha$-numbers. These numbers measure, at a given position and scale, how well a measure is approximated by a constant multiple of flat measure on a vertical plane. For $1\leq k\leq n-1$, we prove the estimate for general $(k+2)$-ADR measures. We note that we only obtain the results of Section~\ref{last} for kernels that are even in time.  This is fine for our purposes because, in Section~\ref{graphregsection}, we show that boundedness of the parabolic Riesz transform implies boundedness of a time-symmetrized Riesz transform with kernel even in time, and we apply the estimates of Section~\ref{last} to this time-symmetrized transform. Proposition~\ref{offdiagprop} also gives the following bound for the smooth truncations of $T\mu$.

\begin{theorem}\label{alphatosio}
Suppose that $\mu$ is a $(k+2)$-ADR measure on $\mathbb P^n$, $k \in \{1, \dots, n-1\}$, and let $T$ be a $(k+2)$-dimensional parabolic Calder\'on-Zygmund operator whose kernel $K$ is smooth away from the origin, is odd in the spatial variables and even in time, and satisfies
$$
|\nabla^j\partial_t^q K(X,t)|\leq\frac{C_{j,q}}{\|(X,t)\|^{k+2+j+2q}}
$$
for all integers $j,q\geq0$ and all $(X,t)\in\mathbb P^n\setminus\{0\}$. For the finite smooth truncations
$T_{a,m}\mu:=\sum_{j=a}^{m}T_{(j)}\mu$ defined in
Subsection~\ref{dyadicgriddecomp}, we have, uniformly in integers $a\leq m$,
\begin{align}\label{smoothtruncbound}
\|T_{a,m}\mu\|^2_{L^2(\mu)}
\lesssim
\sum_{Q\in\mathbb D}\alpha(Q)^2\mu(Q).
\end{align}
If the right-hand side is finite, then for each $m\in\mathbb Z$ the limit
\[
T_m\mu:=\lim_{a\to-\infty}T_{a,m}\mu
\]
exists in $L^2(\mu)$, and \eqref{smoothtruncbound} also holds with
$T_{a,m}\mu$ replaced by $T_m\mu$.
\end{theorem}

The result of Section~\ref{last} is used as an input in the proof of the main result of Section~\ref{graphregsection}. We defer its proof, which is long and technical, to better highlight the geometric and harmonic analytic intuition behind the rest of the proof of Theorem~\ref{mainthm}.

It is not difficult to extend the codimension-one result of \cite[Theorem~5]{BHHLN3} to regular Lip-$(1,1/2)$ graphs in arbitrary codimension, including graphs over $\mathbb P^0$. We give the details in Appendix~\ref{bhhlnappendix}. This result is needed for the smoothing and localization arguments in Subsections~\ref{smoothing} and~\ref{localization}.

\begin{theorem}\label{highcodimgeneralSIO}
Let $0\leq k\leq n-1$. There is an integer $N=N(n,k)$ such that the
following holds. Suppose that $F:\mathbb P^k\to\mathbb R^{n-k}$ is
regular Lip-$(1,1/2)$ and that
$K\in C^N(\mathbb P^n\setminus\{0\};\mathbb R)$ satisfies
\[
K(-X,t)=-K(X,t),
\qquad
|\nabla_X^j\partial_t^qK(X,t)|
\leq C_{j,q}\|(X,t)\|^{-k-2-j-2q}
\]
for all integers $j,q\geq0$ with $j+q\leq N$ and all $(X,t)\neq0$.
Define
\[
T_\ell^F f(x,t):=
\int_{\mathbb P^k}
K(x-y,F(x,t)-F(y,s),t-s)\Psi_\ell(x-y,t-s)f(y,s)\,\d y\,\d s,
\]
where $\Psi_\ell$ is the smooth cutoff used in \eqref{smoothriesz}.
Then
\[
\sup_{\ell\in\mathbb Z}\|T_\ell^F f\|_{L^2(\mathbb P^k)}
\leq C\|f\|_{L^2(\mathbb P^k)},
\]
where $C$ depends only on $n$, $k$, the regular Lip-$(1,1/2)$ constants
of $F$, and the constants $C_{j,q}$. The kernel need not be homogeneous
or even in time.
\end{theorem}

\noindent{\fontencoding{OT1}\fontfamily{txr}\bfseries Acknowledgments.} The author thanks Simon Bortz and Steve Hofmann for helpful conversations.

\noindent{\fontencoding{OT1}\fontfamily{txr}\bfseries AI assistance.} All mathematical ideas in this paper are due to the author. AI tools were used under the author's direction to assist with drafting and revising the manuscript, improving the exposition, and checking notation and references. The author takes full responsibility for the mathematical content and the final version of the paper.

\section{\texorpdfstring{$L^2$ Boundedness of The Riesz Transforms Implies Regularity of Lip$(1,1/2)$ Graphs}{L2 Boundedness of The Riesz Transforms Implies Regularity of Lip(1,1/2) Graphs}}\label{graphregsection}

\subsection{Notation, Setting}\label{graphnotation}
Our ambient space is $\mathbb P^n = \mathbb{R}^n_{sp} \times \mathbb{R}$.  In the proofs preceding Subsection~\ref{timeonlycase}, we fix an integer $k$, $1 \leq k \leq n-1$, and identify
$\mathbb P^k$ with the plane $\mathbb{R}^{k}_{sp} \times \{0_{n-k}\} \times \mathbb{R}$.  At times we will denote points of $\mathbb P^n$ as ordered pairs $(X,t)$,
where $X$ is spatial and $t$ is the time variable. Throughout, $\nabla$ denotes the gradient with respect to the full spatial variable $X$ only; subscripts such as $\nabla_x$ are used when differentiation is restricted to the horizontal spatial variables.  At other times, we will make the transversal variable visible and write $(x,z,t)$ instead of $(X,t)$, where $x\in \mathbb R^k$ is the horizontal spatial variable and $z\in \mathbb R^{n-k}$ is the transversal variable, so that $X=(x,z)\in \mathbb R^k\times \mathbb R^{n-k}$.
When convenient, we also write $z=(z_1,\dots,z_{n-k})$ and $X=(x_1,\dots,x_k,z_1,\dots,z_{n-k})$.
We use $\mathbb R^{k+1}$ to denote the Fourier frequency space, with variables $(\xi,\tau)$.
We will let $C_r(X,t)$ denote the cylinder
$$C_r(X,t) = \{(Y,s) \in \mathbb P^n: |X-Y|<r, \, |t-s|<r^2 \}.$$
Here $|X|$ denotes the standard Euclidean length.  We will also consider parabolic cylinders defined on the horizontal plane $\mathbb P^k$:
$C'_r(x,t) = \{(y,s)=(y,0_{n-k},s) \in \mathbb P^k=\mathbb R^k\times \{0_{n-k}\}\times \mathbb R: |x-y|<r,\ |t-s|<r^2\}.$ 
We also write
\[
Q'_r(x,t)
:=
\{(y,s)\in\mathbb P^k:
|x_j-y_j|<r \text{ for }1\leq j\leq k,\ |t-s|<r^2\}
\]
for the centered parabolic cube of spatial radius $r$ and temporal radius $r^2$.
For $\lambda>0$, we define its parabolic dilate by
\[
\lambda Q'_r(x,t):=Q'_{\lambda r}(x,t).
\]

Finally, given $(X,t),(Y,s) \in \mathbb P^n$, we use the parabolic norm and parabolic distance
$$\|(X,t)\| = |X| + |t|^{1/2},$$
$$d_p\big((X,t),(Y,s)\big) =\dist\big((X,t),(Y,s)\big)= \|(X-Y,t-s)\|.$$
Throughout Section \ref{graphregsection}, we work with a fixed function $A=(A_1,\dots,A_{n-k}): \mathbb P^k \rightarrow \mathbb R^{n-k}$, and let its graph be $\Gamma = \Gamma(A) = \{(x,A(x,t),t)\}_{(x,t) \in \mathbb P^k}\subseteq \mathbb P^n$.  We will always
assume that $A$ is Lip-$(1,1/2)$, by which
we mean that there is a constant $C$ such that for every pair $(x,t),(y,s) \in \mathbb P^k$ there holds
$$|A(x,t)-A(y,s)| \leq C\|(x,t)-(y,s)\|.$$
The smallest constant $C$ such that the above inequality holds uniformly over all pairs will be denoted by $\|A\|_{lip}$, and is called
``the Lip-$(1,1/2)$ constant of $A$."  We will also refer to this as the ``Lip-$(1,1/2)$ constant for $\Gamma$."  We will also apply these definitions to functions other than $A$.  We fix a (small) parameter $\epsilon>0$, and always assume that $\|A\|_{lip} \leq \epsilon \ll1$.\\

We work with the standard parabolic dyadic grid on
$\mathbb P^k=\mathbb R^k\times\mathbb R$. For $j\in\mathbb Z$,
$m\in\mathbb Z^k$, and $q\in\mathbb Z$, set
\[
Q_{j,m,q}
:=
2^{-j}\bigl(m+[0,1)^k\bigr)
\times
2^{-2j}\bigl(q+[0,1)\bigr),
\]
and define
\[
\mathbb D_j
:=
\{Q_{j,m,q}:m\in\mathbb Z^k,\ q\in\mathbb Z\},
\qquad
\mathbb D
:=
\bigcup_{j\in\mathbb Z}\mathbb D_j.
\]
Thus $\mathbb D_j$ is a partition of $\mathbb P^k$. For
$Q=Q_{j,m,q}\in\mathbb D_j$, set $\ell(Q):=2^{-j}$. The spatial side
length of $Q$ is $\ell(Q)$ and its temporal side length is $\ell(Q)^2$.
Every cube has a unique parent in $\mathbb D_{j-1}$ and
$2^{k+2}$ children in $\mathbb D_{j+1}$.

The center of $Q=Q_{j,m,q}$ is
\[
(z_Q,\tau_Q)
=
\left(
2^{-j}\left(m+\frac12(1,\ldots,1)\right),
2^{-2j}\left(q+\frac12\right)
\right).
\]
Given $\lambda\geq1$, we define
\begin{align}\label{dilates}
\lambda Q
:=
\left\{
(x,t)\in\mathbb P^k:
|x-z_Q|_\infty\leq\frac{\lambda\ell(Q)}2,
\quad
|t-\tau_Q|\leq\frac{\lambda^2\ell(Q)^2}{2}
\right\}.
\end{align}
Given $\Omega\subseteq\mathbb P^k$, we also write
\[
\mathbb D(\Omega):=\{Q\in\mathbb D:Q\subseteq\Omega\}.
\]
\begingroup
For $Q\in\mathbb D$, we write $|Q|$ for the Lebesgue measure of $Q$ in
$\mathbb P^k$. Thus
\[
|Q|=\ell(Q)^{k+2}.
\]
\endgroup

Recall that
$d_p((x,t),(y,s))=|x-y|+|t-s|^{1/2}$, where the spatial norm is Euclidean.
Consequently,
\[
\sup_{(x,t)\in Q}
d_p\bigl((x,t),(z_Q,\tau_Q)\bigr)
=
\left(\frac{\sqrt{k}}2+\frac1{\sqrt2}\right)\ell(Q),
\qquad
\diam_p(Q)
=
(\sqrt{k}+1)\ell(Q).
\]
Throughout this section, $\diam(Q)$ denotes this parabolic diameter.

Given $Q\in\mathbb D$, we define the ambient parabolic cube
\[
C_Q
:=
\left\{
(X,t)\in\mathbb P^n:
\left|X-\bigl(z_Q,A(z_Q,\tau_Q)\bigr)\right|_\infty
\leq 2\ell(Q),\quad
|t-\tau_Q|\leq 8\ell(Q)^2
\right\}.
\]
Thus $C_Q$ has spatial side length $4\ell(Q)$ and temporal side
length $16\ell(Q)^2$. For $\lambda>0$, we write
\[
\lambda C_Q
:=
\left\{
(X,t)\in\mathbb P^n:
\left|X-\bigl(z_Q,A(z_Q,\tau_Q)\bigr)\right|_\infty
\leq 2\lambda\ell(Q),\quad
|t-\tau_Q|\leq 8\lambda^2\ell(Q)^2
\right\}.
\]
We use the letter $C$ for the ambient cubes, reserving $Q$ for the cubes of $\mathbb D$.

Let $\pi:\mathbb P^n\to\mathbb P^k$ denote the projection $\pi(x,z,t):=(x,t)$.
We set $C_Q^\Gamma:=C_Q\cap\Gamma$,
$(\lambda C_Q)^\Gamma:=(\lambda C_Q)\cap\Gamma$, and
$Q^\Gamma:=\pi^{-1}(Q)\cap\Gamma$. Since
$\|A\|_{lip}\leq\epsilon\ll1$, we have
\[
Q^\Gamma\subseteq(\tfrac12 C_Q)^\Gamma\subseteq C_Q^\Gamma,
\qquad
\pi(C_Q)=4Q\subseteq5Q.
\]

We let $\d x\,\d t$ denote the standard Lebesgue measure on $\mathbb P^k$. Given a Borel set \(\Sigma\subset\mathbb P^n\), we define its \textit{parabolic slice measure} \(\sigma_\Sigma\) by
\[
\sigma_\Sigma(E)
:=
\int_{\mathbb R}
\mathcal H^k\big(\{X\in\mathbb R^n:(X,t)\in E\cap\Sigma\}\big)
\,\d\mathcal H^1(t)
\]
for every Borel set \(E\subset\mathbb P^n\). This measure has parabolic homogeneous dimension \(k+2\). We also write \(\mathcal H_p^\eta\) for the \(\eta\)-dimensional Hausdorff measure obtained from the usual Hausdorff construction using the parabolic metric \(d_p\). See \cite[Remark~2.8 and Appendix~B]{BHHLN5} for a discussion of the relationship between \(\sigma_\Sigma\) and the standard parabolic Hausdorff measure \(\mathcal H_p^{k+2}|_\Sigma\).

Let
\[
\Phi(x,t):=(x,A(x,t),t),
\qquad
\mu_0:=\Phi_\#(dx\,dt).
\]
Since \(\Phi\) is parabolic bi-Lipschitz, \(\mu_0\) is a \((k+2)\)-ADR measure with support \(\Gamma\). \\

For the proofs of Theorems~\ref{mainthm2} and~\ref{localdpara}, it is enough to take \(\mu=\mu_0\). Indeed, if \(\mu\) is any \((k+2)\)-ADR measure with \(\supp(\mu)=\Gamma\), then
\[
\d\mu=w\,\d\mu_0,
\qquad
0<c\leq w\leq C<\infty
\quad \mu_0\text{-a.e.}
\]
Writing \(\mathcal R_{\nu,\varepsilon}\) for a truncation of the Riesz transform formed using the measure \(\nu\), we have
\[
\mathcal R_{\mu_0,\varepsilon}f
=
\mathcal R_{\mu,\varepsilon}\left(\frac{f}{w}\right),
\qquad
\mathcal R_{\mu,\varepsilon}f
=
\mathcal R_{\mu_0,\varepsilon}(fw).
\]
Consequently,
\[
\|\mathcal R_{\mu_0,\varepsilon}f\|_{L^2(\mu_0)}
\lesssim
\left\|\mathcal R_{\mu,\varepsilon}\left(\frac{f}{w}\right)\right\|_{L^2(\mu)}
\lesssim
\left\|\frac{f}{w}\right\|_{L^2(\mu)}
\lesssim
\|f\|_{L^2(\mu_0)}.
\]
The same calculation with the roles of \(\mu\) and \(\mu_0\) reversed gives the converse implication. Since the two identities hold for every truncation parameter, the same argument also applies to the corresponding maximal truncations.

Accordingly, throughout the proofs of Theorems~\ref{mainthm2} and~\ref{localdpara}, we write \(\mu=\mu_0\). Thus, for every integrable function \(F\) on \(\Gamma\),
\[
\int_\Gamma F(X,t)\,\d\mu(X,t)
=
\int_{\mathbb P^k}
F(x,A(x,t),t)\,\d x\,\d t.
\]
\begingroup
In particular, for every $Q\in\mathbb D(\mathbb P^k)$,
\[
\mu(Q^\Gamma)=|Q|.
\]
\endgroup
The parabolic slice measure satisfies
\[
\d\sigma_\Gamma=j_A\,\d\mu_0,
\qquad
j_A(x,t)
:=
\sqrt{\det\!\left(I_k+(\nabla_xA(x,t))^T\nabla_xA(x,t)\right)}.
\]
Only the spatial gradient of \(A\) enters \(j_A\), since the slice measure is defined by taking \(k\)-dimensional Hausdorff measure on each fixed-time slice.

\begin{definition}\label{verticalplanes}
Let $\tilde P \subset \mathbb R^n$ be a plane of dimension $k$.  We call $P = \tilde P \times \mathbb R \subset \mathbb P^n$ a \textit{$k$-dimensional vertical plane}.
\end{definition}

For a vertical affine plane $L=(a+V)\times\mathbb R$, where
$V\subset\mathbb R^n$ is a linear $k$-plane, we write
$L^\perp:=V^\perp\times\{0\}$ and let $\pi_{L^\perp}$ denote the linear
Euclidean orthogonal projection onto this normal space.

\begin{definition}\label{alphadefs}

For locally finite Borel measures $\mu$ and $\nu$ on $\mathbb P^n$,
we define their ``parabolic distance" with respect to
$Q\in\mathbb D(\mathbb P^k)$ by
\begin{align}\label{measdist}
\dist_{C_Q}(\mu,\nu)
:=
\sup \left\{
\left|\int f\,\d\mu-\int f\,\d\nu\right|:
\|f\|_{lip}\leq1,\ 
\supp(f)\subseteq C_Q
\right\}.
\end{align}

We then set
\begin{align}\label{alphadef}
\alpha(Q) = \frac{1}{\ell(Q)^{k+3}}\inf_{\substack{c \geq 0 \\ L}} \dist_{C_Q}(\mu, c\sigma_L),
\end{align}
where $L$ ranges over all vertical planes.  Given $Q\in\mathbb D(\mathbb P^k)$, choose $c_Q\geq0$ and a vertical plane $L_Q$ realizing the infimum above. We adopt the notation
$\mathcal L_Q:=c_Q\sigma_{L_Q}$. 
\end{definition}

\begin{definition}\label{betadef}
Given a cube $Q \in \mathbb D$, we set
\begin{align*}
\beta_1^\Gamma(Q)=\beta_1(Q) &:= \inf_{L} \frac{1}{\mu(Q^\Gamma)}\int_{Q^\Gamma} \frac{d_p\big((X,t),L\big)}{\ell(Q)}  d\mu(X,t)\\
\beta_2^\Gamma(Q)=\beta_2(Q) &:= \inf_{L} \left(\frac{1}{\mu(Q^\Gamma)}\int_{Q^\Gamma} \left(\frac{d_p\big((X,t),L\big)}{\ell(Q)} \right)^2 d\mu(X,t)\right)^{1/2}\\
\beta^\Gamma_{\infty}(Q)=\beta_\infty(Q) &:= \inf_{L} \Big(\sup_{(X,t) \in \pi^{-1}(Q) \cap L} \frac{d_p\big((X,t),\Gamma\big)}{\ell(Q)} 
+ \sup_{(X,t) \in Q^\Gamma} \frac{d_p\big((X,t),L\big)}{\ell(Q)}\Big) 
\end{align*}
In each definition, we take the infimum over $k$-dimensional vertical planes $L$.
The quantities $\beta_j^\Gamma(C')$, for $j=1,2,\infty$, are defined analogously for parabolic cylinders $C'\subseteq\mathbb P^k$.  For example, we have, for $C'=C'_r$ a cylinder of radius $r$
$$\beta^\Gamma_1(C'):= \inf_{L} \frac{1}{\mu\big(\pi^{-1}(C')\cap
\Gamma\big)}\int_{\pi^{-1}(C')\cap
\Gamma} \frac{d_p\big((X,t),L\big)}{r} d\mu(X,t).$$
The same definitions apply to dilates of cubes.  \begingroup
For any function $F$ on $\mathbb P^k$ with values in a Euclidean
space, we write
\[
\Phi_F(x,t):=(x,F(x,t),t),
\qquad
\Gamma_F:=\Phi_F(\mathbb P^k),
\]
and set
\[
\mu_{\Gamma_F}:=(\Phi_F)_\#(\d x\,\d t).
\]
Whenever $\beta_j^{\Gamma_F}$ appears, it denotes the corresponding
$\beta_j$-number computed with respect to $\mu_{\Gamma_F}$. In
particular, $\mu_{\Gamma_F}(Q^{\Gamma_F})=|Q|$. For $F=A$, this
agrees with the normalization $\mu=\mu_0$ above. If $F$ is spatially
Lipschitz, then $\sigma_{\Gamma_F}$ is comparable to
$\mu_{\Gamma_F}$.
\endgroup

\end{definition}

For each transversal component $m\in\{1,\dots,n-k\}$ and for
$t\neq0$, we define
\begin{align}\label{commkernel}
J_{comm}^m(X,t) := \frac{z_m}{|t|^{(k+1)/2+1}} \exp\left(-\frac{|x|^2}{4|t|}\right),
\end{align}
and
\begin{align}\label{rieszkernel}
J^m(X,t) := \frac{z_m}{|t|^{(k+1)/2+1}}\exp\left(-\frac{|x|^2+|z|^2}{4|t|}\right),\\
J^{+,m}(X,t) := \frac{z_m}{|t|^{(k+1)/2+1}}\exp\left(-\frac{|x|^2+|z|^2}{4|t|}\right)\chi_{\{t>0\}} ,\label{rieszkernel+} \\
J^{-,m}(X,t) := \frac{z_m}{|t|^{(k+1)/2+1}}\exp\left(-\frac{|x|^2+|z|^2}{4|t|}\right)\chi_{\{t<0\}} .\label{rieszkernel-}
\end{align}
We extend $J^m$ and $J^{\pm,m}$ to $t=0$ away from the origin by setting
\[
J^m(X,0)=J^{\pm,m}(X,0)=0
\qquad\text{whenever }X\neq0.
\]
The singular set of $J_{comm}^m$ is instead the set
$(x,t)=(0,0)$, independently of $z$. We therefore regard $J_{comm}^m$ as a
kernel on
\[
\{(X,t)\in\mathbb P^n:(x,t)\neq(0,0)\}
\]
and set $J_{comm}^m(X,0)=0$ when $x\neq0$. Finally, all
truncated kernels introduced below are set equal to zero when $(x,t)=(0,0)$.

Given $\Gamma = \{(x,A(x,t),t)\}$, these kernels also have the representations
\begin{align}\label{commkernelgraph}
K_{comm}^m(x,t,y,s) := \frac{A_m(x,t) - A_m(y,s)}{|t-s|^{(k+1)/2+1}} \exp\left(-\frac{|x-y|^2}{4|t-s|}\right),
\end{align}
\begin{align}\label{rieszkernelgraph}
K^m(x,t,y,s):= \frac{A_m(x,t) - A_m(y,s)}{|t-s|^{(k+1)/2+1}}\exp\left(-\frac{|x-y|^2+|A(x,t) - A(y,s)|^2}{4|t-s|}\right),\\
K^{+,m}(x,t,y,s):= \frac{A_m(x,t) - A_m(y,s)}{|t-s|^{(k+1)/2+1}}\exp\left(-\frac{|x-y|^2+|A(x,t) - A(y,s)|^2}{4|t-s|}\right)\chi_{\{t>s\}}, \label{rieszkernelgraph+}\\
K^{-,m}(x,t,y,s):= \frac{A_m(x,t) - A_m(y,s)}{|t-s|^{(k+1)/2+1}}\exp\left(-\frac{|x-y|^2+|A(x,t) - A(y,s)|^2}{4|t-s|}\right)\chi_{\{s>t\}}. \label{rieszkernelgraph-}
\end{align}

A priori, we do not know whether the principal value operators associated with these kernels exist, so we need to define appropriate smooth truncations.  Let
\[
\rho(x,t):=\left(|x|^4+t^2\right)^{1/4}.
\]
Then $\rho$ is smooth away from the origin, parabolically homogeneous, and comparable to the parabolic norm $\|(x,t)\|$. Let $\psi_0:[0,\infty)\to[0,1]$ be a smooth non-increasing function such that
\begin{equation}\label{psi0}
\chi_{[0,1/2]}\leq \psi_0\leq \chi_{[0,2]}.
\end{equation}
For $j\in\mathbb Z$, set
\[
\psi_j(x,t):=\psi_0(2^j\rho(x,t)),
\]
and further define $\widetilde\varphi_j(x,t):=\psi_{j+3}(x,t)-\psi_{j+4}(x,t)$.  Since $\psi_0$ is non-increasing, $\widetilde\varphi_j\geq 0$. Moreover, if $\widetilde\varphi_j(x,t)\neq0$, then
\[
2^{j+3}\rho(x,t)\leq 2
\qquad \text{and} \qquad
2^{j+4}\rho(x,t)>1/2.
\]
Indeed, if $2^{j+3}\rho(x,t)>2$, then both terms in the difference vanish, while if $2^{j+4}\rho(x,t)\leq1/2$, then both terms are equal to $1$. Therefore
\[
2^{-j-5}<\rho(x,t)\leq 2^{-j-2}.
\]
Using $\rho(x,t)\leq\|(x,t)\|$ and $2^{-3/4}\|(x,t)\|\leq\rho(x,t)$, we obtain
\[
2^{-j-5}<\|(x,t)\|\leq 2^{-j-1}.
\]
Thus $\widetilde\varphi_j$ is supported in $A(0,2^{-j-5},2^{-j-1})$.  Here, 
\[
A((x,t),M_1,M_2):=\{(y,s):M_1\leq\|(x,t)-(y,s)\|\leq M_2\}.
\]
Moreover,
\[
\sum_{j\in\mathbb Z}\widetilde\varphi_j(x,t)>c
\qquad \text{for all }(x,t)\in\mathbb P^k\setminus\{0\}.
\]
Additionally,
\[
|\widetilde\varphi_j|
\lesssim \chi_{A(0,2^{-j-5},2^{-j-1})},\qquad
|\nabla_x\widetilde\varphi_j|
\lesssim 2^j\chi_{A(0,2^{-j-5},2^{-j-1})},\qquad
|\partial_t\widetilde\varphi_j|
\lesssim 2^{2j}\chi_{A(0,2^{-j-5},2^{-j-1})}.
\]
We set
\[
\varphi_j(x,t)=\frac{\widetilde\varphi_j(x,t)}{\sum_\ell \widetilde\varphi_\ell(x,t)}.
\]
Clearly, we have
\[
\sum_{j\in\mathbb Z}\varphi_j(x,t)=1
\qquad \text{for all }(x,t)\in\mathbb P^k\setminus\{0\}.
\]
Also, we still have the inequalities
\[
|\varphi_j|
\lesssim \chi_{A(0,2^{-j-5},2^{-j-1})},\qquad
|\nabla_x\varphi_j|
\lesssim 2^j\chi_{A(0,2^{-j-5},2^{-j-1})},\qquad
|\partial_t\varphi_j|
\lesssim 2^{2j}\chi_{A(0,2^{-j-5},2^{-j-1})}.
\] 

For each transversal component $m\in\{1,\dots,n-k\}$ we define the following smooth annular truncations:
\begin{align}\label{smoothannularriesz}
\mathcal{R}^{m}_{(\ell)}f(x,t):=\int_{\mathbb P^k} K^m(x,t,y,s)\varphi_\ell\big((x,t)-(y,s)\big)f(y,s)\d y\d s ,\\
\mathcal{R}^{+,m}_{(\ell)}f(x,t):=\int_{\mathbb P^k} K^{+,m}(x,t,y,s)\varphi_\ell\big((x,t)-(y,s)\big)f(y,s)\d y\d s,\\
\mathcal{R}^{-,m}_{(\ell)}f(x,t):=\int_{\mathbb P^k} K^{-,m}(x,t,y,s)\varphi_\ell\big((x,t)-(y,s)\big)f(y,s)\d y\d s, \\
\label{smoothannularcomm}
\mathcal{C}^{m}_{(\ell)}f(x,t):= \int_{\mathbb P^k} K_{comm}^m(x,t,y,s)\varphi_\ell\big((x,t)-(y,s)\big)f(y,s)\d y\d s.
\end{align}
We also define the truncations away from the singularity as follows. Set

\[
\Psi_\ell(x,t)=\sum_{j\leq\ell}\varphi_j(x,t).
\]

Then define
\begin{align}
\mathcal{R}^{m}_{\ell}f(x,t):=\int_{\mathbb P^k} K^{m}(x,t,y,s)\Psi_\ell\big((x,t)-(y,s)\big)f(y,s)\d y\d s,\label{smoothriesz}\\
\mathcal{R}^{+,m}_{\ell}f(x,t):=\int_{\mathbb P^k} K^{+,m}(x,t,y,s)\Psi_\ell\big((x,t)-(y,s)\big)f(y,s)\d y\d s, \label{smoothriesz+}\\
\mathcal{R}^{-,m}_{\ell}f(x,t):=\int_{\mathbb P^k} K^{-,m}(x,t,y,s)\Psi_\ell\big((x,t)-(y,s)\big)f(y,s)\d y\d s, \label{smoothriesz-}\\
\mathcal{C}^{m}_{\ell}f(x,t):= \int_{\mathbb P^k} K_{comm}^m(x,t,y,s)\Psi_\ell\big((x,t)-(y,s)\big)f(y,s)\d y \d s.\label{smoothcomm}
\end{align}
We write $K^m_{(j)}$, $K^{\pm,m}_{(j)}$, and $K^m_{comm,(j)}$ for the kernels obtained by multiplying $K^m$, $K^{\pm,m}$, and $K^m_{comm}$, respectively, by $\varphi_j((x,t)-(y,s))$. Replacing $\varphi_j$ by $\Psi_j$ gives $K^m_j$, $K^{\pm,m}_j$, and $K^m_{comm,j}$.  
The parameterized operators above are exactly the pullbacks of the corresponding operators on \(\Gamma\) formed using \(\mu=\mu_0\). Indeed, let \(F\) be a function on \(\Gamma\), set \(f=F\circ\Phi\), and write \((X,t)=\Phi(x,t)\) and \((Y,s)=\Phi(y,s)\). Then
\[
\mathcal R^m_{(\ell)}f(x,t)
=
\int_\Gamma
J^m(X-Y,t-s)
\varphi_\ell\big((x,t)-(y,s)\big)
F(Y,s)\,\d\mu(Y,s).
\]
The same identity holds for \(\mathcal R^m_\ell\), \(\mathcal C^m_{(\ell)}\), and \(\mathcal C^m_\ell\).

We define
\[
J^m_{(\ell)}(X-Y,t-s)
:=
J^m(X-Y,t-s)
\varphi_\ell\big((x,t)-(y,s)\big).
\]
We also define \[J^m_{comm,(\ell)}(X-Y,t-s):=J^m_{comm}(X-Y,t-s)\varphi_\ell\big((x,t)-(y,s)\big).\] Because \(\|A\|_{lip}\) is small, \(\|(x,t)-(y,s)\|\) is comparable to the parabolic distance between \((X,t)\) and \((Y,s)\) on \(\Gamma\). Thus, this is equivalent to a truncation involving all spatial variables. We define
\[
J^m_\ell(X-Y,t-s)
:=
J^m(X-Y,t-s)
\Psi_\ell\big((x,t)-(y,s)\big)
\]
and similarly define \[J^m_{comm,\ell}(X-Y,t-s):=J^m_{comm}(X-Y,t-s)\Psi_\ell\big((x,t)-(y,s)\big).\]

Unless a component $m$ has already been fixed, omitting the superscript $m$
from a kernel or singular integral denotes the corresponding vector-valued
tuple of all $n-k$ components.

Throughout this section, unless otherwise specified, 
\[
\|f\|_2:=\|f\|_{L^2(\mathbb P^k)},
\qquad
\langle f,g\rangle:=\int_{\mathbb P^k} f(x,t)g(x,t)\,\d x\,\d t.
\]
For an operator $T$ bounded on $L^2(\mathbb P^k)$, we write
\[
\|T\|_{2\to2}:=\|T\|_{L^2(\mathbb P^k)\to L^2(\mathbb P^k)}.
\]
When $\mathcal R^m$, $\mathcal R^{+,m}$, or $\mathcal R^{-,m}$ appears without a truncation index in an operator norm, we use the convention
\[
\|\mathcal R^{\pm,m}\|
:=
\sup_{\ell\in\mathbb Z}\|\mathcal R^{\pm,m}_\ell\|_{2\to2},
\]
and similarly for $\mathcal R^m$. By the standard Cotlar inequality for these smooth truncations, this convention is quantitatively equivalent to using the $L^2$ operator norm of the corresponding maximal truncation
\[
f\mapsto \sup_{\ell\in\mathbb Z}|\mathcal R^{\pm,m}_\ell f|
\qquad\text{or}\qquad
f\mapsto \sup_{\ell\in\mathbb Z}|\mathcal R^m_\ell f|.
\]

In the sequel, when the superscript $m$ is omitted in an operator norm, we use the conventions
\[
\|\mathcal R\|
:=
\sum_{m=1}^{n-k}\|\mathcal R^m\|,
\qquad
\|\mathcal R^\pm\|
:=
\sum_{m=1}^{n-k}\|\mathcal R^{\pm,m}\|.
\]

\begingroup
More generally, if $F:\mathbb P^k\to\mathbb R^{n-k}$, set
\[
K^{F,m}(x,t,y,s)
:=
J^m\bigl(x-y,F(x,t)-F(y,s),t-s\bigr)
\]
and
\[
\mathcal R^{F,m}_\ell f(x,t)
:=
\int_{\mathbb P^k}
K^{F,m}(x,t,y,s)
\Psi_\ell\bigl((x,t)-(y,s)\bigr)
f(y,s)\,\d y\,\d s.
\]
The kernels $K^{\pm,F,m}$ and $K_{comm}^{F,m}$, and their corresponding
operators, are defined analogously using $J^{\pm,m}$ and $J_{comm}^m$.
We use the conventions
\[
\|\mathcal R^{F,m}\|
:=
\sup_{\ell\in\mathbb Z}
\|\mathcal R^{F,m}_\ell\|_{2\to2},
\qquad
\|\mathcal R^F\|
:=
\sum_{m=1}^{n-k}\|\mathcal R^{F,m}\|.
\]
When $F=A$, we omit the superscript $F$, in agreement with the notation above.
\endgroup

\begin{definition}\label{halftime}

For a compactly supported function
$f:\mathbb P^k\to\mathbb R^{n-k}$, we first define the half-time
derivative componentwise by the Fourier multiplier
\[
D_t^{1/2}f(x,t)
:=
\bigl(|\tau|^{1/2}\widehat f(x,\tau)\bigr)^{\check{}}(x,t),
\]
where the Fourier transform and inverse Fourier transform are taken in the
time variable.

We also have the principal-value representation
\[
D_t^{1/2}f(x,t)
:=
c\,\operatorname{p.v.}
\int_{\mathbb R}
\frac{f(x,s)-f(x,t)}{|s-t|^{3/2}}
\,\d s,
\]
where $c$ is chosen so that this singular-integral definition agrees with the
Fourier multiplier above whenever both are defined. For a noncompactly
supported Lip-$(1,1/2)$ function, the principal-value formula is understood
distributionally on mean-zero test functions. More precisely, if
$\varphi\in C_c^\infty(\mathbb P^k)$ and
\[
\int_{\mathbb P^k}\varphi(x,t)\,\d x\,\d t=0,
\]
then

\[
\left\langle D_t^{1/2}f,\varphi\right\rangle
:=
\int_{\mathbb R}
\left(
\int_{\mathbb R^k}
f(x,t)D_t^{1/2}\varphi(x,t)\,\d x
\right)\d t,
\]
where the spatial integral is taken first. To see that this is well defined,
write $f(x,t)=[f(x,t)-f(0,t)]+f(0,t)$ and set
$h(t):=\int_{\mathbb R^k}\varphi(x,t)\,\d x$.
The first contribution is jointly absolutely integrable: spatial
Lipschitz continuity bounds $f(x,t)-f(0,t)$ uniformly on the compact
spatial support of $\varphi$, and
$D_t^{1/2}\varphi(x,t)=O(|t|^{-3/2})$ uniformly in $x$ as $|t|\to\infty$.
For the second contribution, $\int_{\mathbb R}h(t)\,\d t=0$ gives
$D_t^{1/2}h(t)=O(|t|^{-5/2})$ as $|t|\to\infty$.
Since $f(0,t)=O(1+|t|^{1/2})$, its time integral is also absolutely
convergent. Thus the displayed iterated integral converges absolutely
in $t$ after the spatial integration, componentwise for vector-valued $f$.

Thus $D_t^{1/2}f$ is defined modulo constants.

\end{definition}

\begin{definition}\label{reglip}
A Lip-$(1,1/2)$ function $f:\mathbb P^k\to\mathbb R^{n-k}$ is \textit{regular} Lip-$(1,1/2)$ if each component of $D_t^{1/2}f$ belongs to $BMO(\mathbb P^k)$, where $BMO(\mathbb P^k)$ is defined with respect to parabolic cubes.
\end{definition}

The implication from regularity to the $\beta_2$-Carleson estimate is due
to Hofmann \cite[p.~249, estimate~(35)]{H}; we use the formulation in
\cite[Theorem~3.8]{Orp21}. The reverse implication is due to Hofmann,
Lewis, and Nystr{\"o}m \cite[Section~2]{HLN2}; see also
\cite[Lemma~5.40]{BHHLN5}. The cited results are scalar-valued.

Let $F=(F_1,\dots,F_{n-k}):\mathbb P^k\to\mathbb R^{n-k}$ be a Lip-$(1,1/2)$ function.
The scalar results imply the vector-valued results. Indeed, for every
parabolic cylinder $C'$,
\[
\beta_2^{\Gamma_{F_m}}(C')
\leq
\beta_2^{\Gamma_F}(C'),
\qquad m=1,\dots,n-k,
\]
while
\[
\beta_2^{\Gamma_F}(C')^2
\lesssim
\sum_{m=1}^{n-k}
\beta_2^{\Gamma_{F_m}}(C')^2.
\]
Here the implicit constant depends only on $k$ and $\|F\|_{lip}$.
The first inequality follows by projecting an approximating plane onto the
$(x,z_m,t)$-coordinates and, if necessary, enlarging its image to a
$k$-dimensional vertical plane.
For the second inequality, write $C'=C'_r$ and define 
\[
b_F(C')^2:=\inf_a\fint_{C'}\frac{|F(x,t)-a(x)|^2}{r^2}\,\d x\,\d t,
\]
where $a$ ranges over affine maps independent of $t$. Then
\[
\beta_2^{\Gamma_F}(C')^2\leq b_F(C')^2
=\sum_{m=1}^{n-k}b_{F_m}(C')^2.
\]

For a scalar function $f$ and $k\geq1$, fix $\delta>0$ sufficiently small,
depending only on $k$, and let $L$ be a best approximating vertical plane
for $\Gamma_f$ over $C'$. If $\beta_2^{\Gamma_f}(C')<\delta$, then $L$
is the graph of an affine function $a$ whose slope is bounded in terms
of $k$ and $\|f\|_{lip}$. Hence $|f(x,t)-a(x)|$ is controlled by the
distance from $(x,f(x,t),t)$ to $L$, and integration gives
\[
b_f(C')^2\lesssim_{k,\|f\|_{lip}}\beta_2^{\Gamma_f}(C')^2.
\]
If $\beta_2^{\Gamma_f}(C')\geq\delta$, let $(x_0,t_0)$ be the center of
$C'$ and take $a=f(x_0,t_0)$. The Lip-$(1,1/2)$ condition gives
$|f(x,t)-a|\lesssim\|f\|_{lip}r$ on $C'$, so
\[
b_f(C')^2\lesssim\|f\|_{lip}^2
\leq\delta^{-2}\|f\|_{lip}^2\beta_2^{\Gamma_f}(C')^2.
\]
Since $\delta$ depends only on $k$, this gives the same estimate.

For $k=0$, the spatially affine approximants are constants, so these estimates are even easier to prove in this case.
We apply this scalar comparison to each component to obtain the second inequality. 
Applying Hofmann's estimate \cite[p.~249, estimate~(35)]{H}, in the formulation
of \cite[Theorem~3.8]{Orp21}, to each component of $F$ shows that regular
Lip-$(1,1/2)$ functions satisfy the $\beta_2$-Carleson estimate in the following
theorem. Conversely, this $\beta_2$-Carleson estimate implies that each $F_m$
is regular Lip-$(1,1/2)$ by Hofmann--Lewis--Nystr{\"o}m
\cite[Section~2]{HLN2}; see also \cite[Lemma~5.40]{BHHLN5}.

\begin{theorem}[\cite{H,Orp21,HLN2,BHHLN5}]\label{regularbetaequiv}
Let $F:\mathbb P^k\to\mathbb R^{n-k}$ be Lip-$(1,1/2)$. Then $F$
is regular Lip-$(1,1/2)$ if and only if there is a constant $C$ such that,
for every $(z,\tau)\in\mathbb P^k$ and $R>0$,
\[
\int_0^R\int_{C'_R(z,\tau)}
\beta_2^{\Gamma_F}(C'_r(x,t))^2
\frac{\d x\,\d t\,\d r}{r}
\leq
CR^{k+2}.
\]

More precisely, set
\[
M_F:=
\left(
\sum_{m=1}^{n-k}
\|D_t^{1/2}F_m\|_{BMO(\mathbb P^k)}^2
\right)^{1/2}.
\]
If $M_F<\infty$, then the Carleson estimate above holds with
\[
C\leq C_1(n,k,\|F\|_{lip})\bigl(1+M_F^2\bigr).
\]
Conversely, if the Carleson estimate holds with constant $C$, then
\[
M_F\leq C_2(n,k,\|F\|_{lip})\,C^{1/2}.
\]

\end{theorem}

\begin{remark}

The proof in \cite[Section~2]{HLN2} is written under the assumption that
$\|F\|_{lip}$ is small, but very minor modifications establish the result for any Lipschitz constant.

\end{remark}

\subsection{Main Results}

We now give precise statements of Theorems \ref{localdpara-intro} and \ref{mainthm}. We begin with Theorem \ref{mainthm}.

\begin{theorem}\label{mainthm2}
For each integer $0\leq k\leq n-1$, there is a parameter $0<\epsilon\ll 1$ such that the following holds. Suppose $A:\mathbb P^k\to\mathbb R^{n-k}$ is a Lip-$(1,1/2)$ function with $\|A\|_{lip}\leq \epsilon$. If
\[
\big\|\sup_j |\mathcal R_j^{+,m} f|\big\|_2
\lesssim
\|f\|_2,
\qquad m=1,\dots,n-k,
\]
then
\[
\|D_t^{1/2} A\|_{BMO(\mathbb P^k)}
\lesssim
1+
\|\mathcal R^+\|,
\]
where
\[
\|D_t^{1/2}A\|_{BMO(\mathbb P^k)}
:=
\sum_{m=1}^{n-k}
\|D_t^{1/2}A_m\|_{BMO(\mathbb P^k)}.
\]
\end{theorem}

We next restate Theorem \ref{localdpara-intro}. For $k\geq1$, we will prove it first and use it to prove Theorem \ref{mainthm2}.

\begin{theorem}\label{localdpara}
For each integer $0\leq k\leq n-1$, there is a parameter $0<\epsilon\ll 1$ such that the following holds. Suppose that $A:\mathbb P^k\to\mathbb R^{n-k}$ is Lip-$(1,1/2)$ with $\|A\|_{lip}\leq \epsilon$, and supported in a cylinder $C'_R$ of radius $R$. If
\[
\big\|\sup_j |\mathcal R_j^{+,m} f|\big\|_2
\lesssim
\|f\|_2,
\qquad m=1,\dots,n-k,
\]
then
\[
\|D_{par} A\|_2^2 =
\sum_{m=1}^{n-k}\|D_{par}A_m\|_2^2
\lesssim
\left(
1+
\|\mathcal R^+\|^2
\right)R^{k+2},
\]
where $D_{par}A$ is defined componentwise by
\[
D_{par}A_m(x,t)
:=
\big(\widehat A_m(\xi,\tau)\|(\xi,\tau)\|\big)^{\check{}}(x,t).
\]

We use the convention
\[
\|D_{par}A\|_2^2
:=
\sum_{m=1}^{n-k}\|D_{par}A_m\|_2^2.
\]

\end{theorem}
\subsection{Reduction to the ``Time-Symmetric'' Riesz Transform}\label{symreduction}

We first show that one-sided singular integral bounds with respect to time imply time-symmetric singular integral bounds. Fix a transversal component $m\in\{1,\dots,n-k\}$ and assume the bound
\[
\big\|\sup_{\ell\in\mathbb Z} |\mathcal R_\ell^{+,m} f|\big\|_2
\lesssim
\|f\|_2.
\]

Since $\Psi_\ell$ is even,
\[
(\mathcal R_\ell^{+,m})^*
=
-\mathcal R_\ell^{-,m}.
\]
Duality therefore gives
\[
\|\mathcal R^{-,m}\|
=
\|\mathcal R^{+,m}\|.
\]
Applying Cotlar's inequality for the smooth truncations, we obtain
\[
\big\|\sup_{\ell\in\mathbb Z}
|\mathcal R_\ell^{-,m}f|\big\|_2
\lesssim
\|f\|_2.
\]
Since
\[
\mathcal R_\ell^m
=
\mathcal R_\ell^{+,m}
+
\mathcal R_\ell^{-,m},
\]
it follows that
\[
\big\|\sup_{\ell\in\mathbb Z}
|\mathcal R_\ell^m f|\big\|_2
\lesssim
\|f\|_2.
\]

 \subsection{Smoothing}\label{smoothing}

To prove Theorem~\ref{localdpara}, we mollify $A$, prove the estimate for the mollified functions, and pass to the limit. We need to show that the Riesz transform bounds for the smoothed graphs depend only on those of the original graph and are independent of the smoothing parameter.

We use the following consequence of Theorem~\ref{highcodimgeneralSIO}.
\begin{theorem}\label{highcodimBHHLN}
Let $0\leq k\leq n-1$. Suppose that $F:\mathbb P^k\to\mathbb R^{n-k}$ is regular
Lip-$(1,1/2)$, and let
$\Gamma_F:=\{(x,F(x,t),t):(x,t)\in\mathbb P^k\}$.
Then $\|\mathcal R^F\|<\infty$. The bound depends
only on $n$, $k$, and the regular Lip-$(1,1/2)$ constants of $F$.
\end{theorem}

The deduction, including the case $k=0$, is given in Appendix~\ref{bhhlnappendix}.

Let $\eta\in C_c^\infty(\mathbb P^k)$ be a non-negative function satisfying
\[
\supp \eta \subset C'_1(0,0),
\qquad
\int_{\mathbb P^k}\eta=1.
\]
For $\la>0$, define
\[
    \eta_\la(x,t)
    =
    \frac{1}{\la^{k+2}}
    \eta\left(\frac{x}{\la},\frac{t}{\la^2}\right).
\]
Then
\[
    \eta_\la\in C_c^\infty(C'_\lambda(0,0)),
    \qquad
    \eta_\la\geq 0,
    \qquad
    \int_{\mathbb P^k}\eta_\la(x,t)\,\d x\,\d t=1.
\]
Define $B_\la:\mathbb P^k\to \RR^{n-k}$ componentwise by
\[
    B_\la(x,t)=A*\eta_\la(x,t)+2C_A\la {\bf 1},
    \qquad {\bf 1}=(1,\ldots,1)\in\RR^{n-k},
\]
where $C_A := 1 + \|A\|_{lip}$. In particular, for each $m=1,\ldots,n-k$, using that $\|A\|_{lip} \leq \epsilon \ll 1$,
\[
    C_A\la
    \leq
    B_{\la,m}(x,t)-A_m(x,t)
    \leq
    3C_A\la.
\]
Consequently,
\[
    |B_\la(x,t)-A(x,t)|\approx \la,
\]
with implicit constants depending only on $n-k$ and $C_A$.

We use the following kernels:
\[
K(x,t,y,s)
=
\frac{A(x,t)-A(y,s)}{|t-s|^{(k+1)/2+1}}
\exp\left(
    -\frac{|x-y|^2+|A(x,t)-A(y,s)|^2}{4|t-s|}
\right),
\]
\[
K^{B_\la,A}(x,t,y,s)
=
\frac{B_\la(x,t)-A(y,s)}{|t-s|^{(k+1)/2+1}}
\exp\left(
    -\frac{|x-y|^2+|B_\la(x,t)-A(y,s)|^2}{4|t-s|}
\right),
\]
\[
K^{A,B_\la}(x,t,y,s)
=
\frac{A(x,t)-B_\la(y,s)}{|t-s|^{(k+1)/2+1}}
\exp\left(
    -\frac{|x-y|^2+|A(x,t)-B_\la(y,s)|^2}{4|t-s|}
\right),
\]
\[
K^{B_\la}(x,t,y,s)
=
\frac{B_\la(x,t)-B_\la(y,s)}{|t-s|^{(k+1)/2+1}}
\exp\left(
    -\frac{|x-y|^2+|B_\la(x,t)-B_\la(y,s)|^2}{4|t-s|}
\right).
\]
\begingroup
We first verify the size and smoothness estimates for these kernels.

\begin{definition}\label{standardparabolicCZ}
A kernel
\[
U:\bigl(\mathbb P^k\times\mathbb P^k\bigr)
\setminus\{(u,u):u\in\mathbb P^k\}
\longrightarrow\mathbb R^{n-k}
\]
is a standard parabolic Calder\'on--Zygmund kernel if
\[
|U(x,t,y,s)|
\leq
\frac{C}{\|(x,t)-(y,s)\|^{k+2}},
\]
and if
\[
|U(x,t,y,s)-U(x',t',y,s)|
\leq
C\frac{\|(x,t)-(x',t')\|}
{\|(x,t)-(y,s)\|^{k+3}}
\]
whenever
\[
\|(x,t)-(x',t')\|
\leq
\frac12\|(x,t)-(y,s)\|.
\]
We require the analogous estimate in the second pair of variables:
\[
|U(x,t,y,s)-U(x,t,y',s')|
\leq
C\frac{\|(y,s)-(y',s')\|}
{\|(x,t)-(y,s)\|^{k+3}}
\]
whenever
\[
\|(y,s)-(y',s')\|
\leq
\frac12\|(x,t)-(y,s)\|.
\]
The least admissible constant $C$ will be called the standard parabolic
Calder\'on--Zygmund constant of $U$.
\end{definition}

\begin{lemma}\label{smoothingCZkernels}
Let $F,G:\mathbb P^k\to\mathbb R^{n-k}$ be
Lip-$(1,1/2)$ functions, and define
\[
K^{F,G,m}(x,t,y,s)
:=
J^m\bigl(x-y,F(x,t)-G(y,s),t-s\bigr).
\]
Then $K^{F,G}=(K^{F,G,m})_{m=1}^{n-k}$ is a standard parabolic
Calder\'on--Zygmund kernel. Its Calder\'on--Zygmund constant depends only
on $n$, $k$, $\|F\|_{lip}$, and $\|G\|_{lip}$.
\end{lemma}

\begin{proof}
For $P=(X,\tau)\neq0$, differentiation of the formula for $J^m$
gives
\[
|J^m(P)|\lesssim\|P\|^{-(k+2)},\qquad
|\nabla_XJ^m(P)|\lesssim\|P\|^{-(k+3)},\qquad
|\partial_\tau J^m(P)|\lesssim\|P\|^{-(k+4)}.
\]
It follows, by applying
the mean value theorem first in $X$ and then in $\tau$, that
\begin{equation}\label{ambientJsmoothness}
|J^m(P)-J^m(P')|
\lesssim
\frac{\|P-P'\|}{\|P\|^{k+3}}
\end{equation}
whenever $\|P-P'\|\leq\frac12\|P\|$.

Write $u=(x,t)$, $v=(y,s)$, and
\[
\rho:=\|u-v\|,\qquad
D:=\|\Phi_F(u)-\Phi_G(v)\|.
\]
Since $D\geq\rho$, the size estimate for $J^m$ gives
\[
|K^{F,G,m}(u,v)|\lesssim\rho^{-(k+2)}.
\]
Now let $u'=(x',t')$ and
$\delta:=\|u-u'\|\leq\rho/2$. If
\begin{equation}\label{liftclose}
\|\Phi_F(u)-\Phi_F(u')\|\leq\frac12D,
\end{equation}
then \eqref{ambientJsmoothness} and
\[
\|\Phi_F(u)-\Phi_F(u')\|
\leq
(1+\|F\|_{lip})\delta
\]
give
\[
|K^{F,G,m}(u,v)-K^{F,G,m}(u',v)|
\lesssim
\frac{(1+\|F\|_{lip})\delta}{\rho^{k+3}}.
\]
If \eqref{liftclose} fails, then
$\delta\gtrsim_F D\geq\rho$, while
$\|u'-v\|\geq\rho/2$. The size estimate applied at $(u,v)$ and
$(u',v)$ therefore gives the same bound. This proves the smoothness
estimate in the first variable. The proof in the second variable is
identical, with $G$ in place of $F$.
\end{proof}

Because $\|B_\la\|_{lip}\lesssim\|A\|_{lip}$ uniformly in $\la$, the lemma shows that
\[
K=K^{A,A},\qquad K^{B_\la,A},\qquad K^{A,B_\la},
\qquad\text{and}\qquad K^{B_\la}=K^{B_\la,B_\la}
\]
have standard parabolic Calder\'on--Zygmund constants bounded uniformly
in $\la$. We will also apply the lemma to the kernels $K^{B_{\la,i}}$ introduced
below, after establishing uniform Lip-$(1,1/2)$ bounds for $B_{\la,i}$.
\endgroup

For $\ell\in\mathbb Z$, we define the corresponding truncations away from the singularity by
\[
\begin{aligned}
K_\ell^{B_\la,A}(x,t,y,s)
&:=K^{B_\la,A}(x,t,y,s)
\Psi_\ell\bigl((x,t)-(y,s)\bigr),\\
K_\ell^{A,B_\la}(x,t,y,s)
&:=K^{A,B_\la}(x,t,y,s)
\Psi_\ell\bigl((x,t)-(y,s)\bigr),\\
K_\ell^{B_\la}(x,t,y,s)
&:=K^{B_\la}(x,t,y,s)
\Psi_\ell\bigl((x,t)-(y,s)\bigr).
\end{aligned}
\]

The following proposition is the main result of this subsection.
\begin{proposition}\label{Blabounded}
The following inequality holds uniformly over $\ell\in\mathbb{Z}$ and $\la>0$:
\begin{equation*}
    \left\|
    \int_{\mathbb P^k}
        K^{B_\la}_\ell(\cdot,\cdot,y,s)\,f(y,s)\,\d y\,\d s
    \right\|_2^2
    \lesssim
    \bigl(1+\|\mathcal{R}\|^2\bigr)\|f\|_2^2.
\end{equation*}
\end{proposition}

We will prove Proposition~\ref{Blabounded} after establishing a few lemmas.

\begin{lemma}\label{vectorvaluedsmoothingestimate}
We have, for each $\ell \in \mathbb{Z}$,
\begin{align}
    |K_\ell^{B_\la,A}(x,t,y,s)-K_\ell(x,t,y,s)| 
        &\lesssim \frac{\la}{\|(x,t)-(y,s)\|^{k+3}}, 
        \label{farBAestimate}\\
    |K_\ell^{B_\la,A}(x,t,y,s)| 
        &\lesssim \frac{1}{\la^{k+2}}, \qquad 
        \text{if }\, \|(x,t)-(y,s)\|<\la, 
        \label{nearBAestimate}\\
    |K_\ell^{A,B_\la}(x,t,y,s) - K_\ell(x,t,y,s)| 
        &\lesssim \frac{\la}{\|(x,t)-(y,s)\|^{k+3}}, 
        \label{farABestimate}\\
    |K_\ell^{A,B_\la}(x,t,y,s)| 
        &\lesssim \frac{1}{\la^{k+2}}, \qquad 
        \text{if }\, \|(x,t)-(y,s)\|<\la,
        \label{nearABestimate}\\
    |K_\ell^{B_\la}(x,t,y,s) - K_\ell^{A,B_\la}(x,t,y,s)| 
        &\lesssim \frac{\la}{\|(x,t)-(y,s)\|^{k+3}}.
        \label{farBlaABestimate}
\end{align}
The implicit constants are independent of $\la$, but depend on $\|A\|_{lip}$.
\end{lemma}

\begin{proof}
Notice that $|\Psi_\ell|\leq 1$, and hence, for each $\ell$,
\[
    |K_\ell^{B_\la,A}(x,t,y,s)-K_\ell(x,t,y,s)| 
    \leq |K^{B_\la,A}(x,t,y,s)-K(x,t,y,s)|,
\]
and
\[
    |K_\ell^{A,B_\la}(x,t,y,s)-K_\ell(x,t,y,s)| 
    \leq |K^{A,B_\la}(x,t,y,s)-K(x,t,y,s)|.
\]
Thus it suffices to prove the five estimates for the
untruncated kernels. We prove \eqref{farBAestimate} and 
\eqref{nearBAestimate}. Since
$K_\ell^{A,B_\la}(x,t,y,s)=-K_\ell^{B_\la,A}(y,s,x,t)$ and
$K_\ell(x,t,y,s)=-K_\ell(y,s,x,t)$, estimates
\eqref{farABestimate} and \eqref{nearABestimate} follow directly from
\eqref{farBAestimate} and \eqref{nearBAestimate}, respectively, by interchanging
$(x,t)$ and $(y,s)$.

Set
\[
    h=|x-y|,
    \qquad
    \tau=t-s,
    \qquad
    \rho=\|(x,t)-(y,s)\|=h+|\tau|^{1/2}.
\]
If $\tau=0$, then $x\neq y$ off the diagonal, and all the kernels under consideration vanish by our definition of $J^m(X,0)$. Thus we may assume that $\tau\neq0$.

We first prove \eqref{farBAestimate}. Write
\[
    \alpha=B_\la(x,t)-A(y,s),
    \qquad
    \beta=A(x,t)-A(y,s).
\]
Then
\[
|\alpha-\beta|=|B_\la(x,t)-A(x,t)|\lesssim C_A\la.
\]
For fixed $h$ and $\tau$, define
\[
    F(z)
    =
    \frac{z}{|\tau|^{(k+1)/2+1}}
    \exp\left(-\frac{h^2+|z|^2}{4|\tau|}\right),
    \qquad z\in\RR^{n-k}.
\]
Then
\[
K^{B_\la,A}(x,t,y,s)-K(x,t,y,s)=F(\alpha)-F(\beta).
\]
The gradient satisfies
\[
    |\nabla F(z)|
    \lesssim
    \frac{1}{|\tau|^{(k+1)/2+1}}
    \left(1+\frac{|z|^2}{|\tau|}\right)
    \exp\left(-\frac{h^2+|z|^2}{4|\tau|}\right).
\]
Using
\[
\left(1+\frac{|z|^2}{|\tau|}\right)
e^{-|z|^2/(4|\tau|)}
\lesssim 1,
\]
we get
\[
    |\nabla F(z)|
    \lesssim
    \frac{1}{|\tau|^{(k+1)/2+1}}
    \exp\left(-\frac{h^2}{4|\tau|}\right),
\]
\begingroup
Since
\[
    \frac{1}{|\tau|^{(k+1)/2+1}}
    \exp\left(-\frac{h^2}{4|\tau|}\right)
    \lesssim
    \frac{1}{\rho^{k+3}},
\]
the mean value theorem yields
\[
    |K^{B_\la,A}(x,t,y,s)-K(x,t,y,s)|
    \leq
    |\alpha-\beta|
    \sup_{0\leq \theta\leq 1}
    |\nabla F(\beta+\theta(\alpha-\beta))|
    \lesssim
    \frac{C_A\la}{\rho^{k+3}}.
\]
\endgroup

We now prove \eqref{nearBAestimate}. Assume $\rho<\la$. For each component $m=1,\dots,n-k$, we have
\[
B_{\la,m}(x,t)-A_m(y,s)
=
2C_A\la
+
\big(A_m*\eta_\la(x,t)-A_m(x,t)\big)
+
\big(A_m(x,t)-A_m(y,s)\big).
\]

Hence, 
\begin{align*}
B_{\la,m}(x,t)-A_m(y,s)
&=
B_{\la,m}(x,t)-A_m(x,t)
+
A_m(x,t)-A_m(y,s)\\
&\geq
C_A\la-\|A\|_{lip}\rho
\geq
\la,
\end{align*}
since $\rho<\la$ and $C_A=1+\|A\|_{lip}$. Similarly,
\[
|B_{\la,m}(x,t)-A_m(y,s)|
\leq
3C_A\la+\|A\|_{lip}\rho
\lesssim
C_A\la.
\]

Consequently, $|B_\la(x,t)-A(y,s)|\approx\la$, and therefore
\[
    |K^{B_\la,A}(x,t,y,s)|
    \lesssim
    \frac{\la}{|\tau|^{(k+1)/2+1}}
    \exp\!\left(-c\frac{\la^2}{|\tau|}\right).
\]
Since
\[
\sup_{u>0}
u^{-(k+1)/2-1}e^{-c\la^2/u}
\lesssim
\la^{-(k+3)},
\]
we obtain
\[
    |K^{B_\la,A}(x,t,y,s)|
    \lesssim
    \la^{-k-2}.
\]

To prove \eqref{farBlaABestimate}, we again use the mean value theorem. For fixed $h=|x-y|$ and $\tau=t-s$, define
\[
    G(z)
    =
    \frac{z-B_\la(y,s)}{|\tau|^{(k+1)/2+1}}
    \exp\left(
        -\frac{h^2+|z-B_\la(y,s)|^2}{4|\tau|}
    \right),
    \qquad z\in\RR^{n-k}.
\]
Then
\[
K^{B_\la}(x,t,y,s)-K^{A,B_\la}(x,t,y,s)
=
G(B_\la(x,t))-G(A(x,t)).
\]
Differentiating $G$ gives
\[
|\nabla G(z)|
\lesssim
\frac{1}{|\tau|^{(k+1)/2+1}}
\exp\left(-\frac{h^2}{4|\tau|}\right)
\lesssim
\frac{1}{\rho^{k+3}}.
\]
Since
\[
|B_\la(x,t)-A(x,t)|\lesssim C_A\la,
\]
the mean value theorem gives
\[
|K^{B_\la}(x,t,y,s)-K^{A,B_\la}(x,t,y,s)|
\lesssim
\frac{C_A\la}{\rho^{k+3}}.
\]
Multiplying by the truncation factor $\Psi_\ell$, with $|\Psi_\ell|\leq 1$, gives \eqref{farBlaABestimate}.
\end{proof}
\begin{lemma}\label{BlaAbounded}
The following inequality holds uniformly over $\ell\in\mathbb{Z}$ 
and $\la>0$:
\begin{equation*}
    \left\|
    \int_{\mathbb P^k}
        K^{B_\la,A}_\ell(\cdot,\cdot,y,s)\,f(y,s)\,\d y\,\d s
    \right\|_2^2
    \lesssim
    \bigl(1+\|\mathcal{R}\|^2\bigr)\|f\|_2^2.
\end{equation*}
\end{lemma}

\begin{proof}
We split the expression according to whether $\|(x,t)-(y,s)\|\geq\la$ or 
$\|(x,t)-(y,s)\|<\la$:
\begin{align*}
    &\int_{\mathbb P^k}
    \left|\int_{\mathbb P^k}
        K^{B_\la,A}_\ell(x,t,y,s)\,f(y,s)\,\d y\,\d s
    \right|^2\d x\,\d t \\
    &\quad\lesssim
    \int_{\mathbb P^k}
    \left|\int_{\|(x,t)-(y,s)\|\geq\la}
        \bigl(K^{B_\la,A}_\ell(x,t,y,s)-K_\ell(x,t,y,s)\bigr)f(y,s)\,\d y\,\d s
    \right|^2\d x\,\d t \\
    &\qquad+
    \int_{\mathbb P^k}
    \left|\int_{\|(x,t)-(y,s)\|\geq\la}
        K_\ell(x,t,y,s)\,f(y,s)\,\d y\,\d s
    \right|^2\d x\,\d t \\
    &\qquad+
    \int_{\mathbb P^k}
    \left|\int_{\|(x,t)-(y,s)\|<\la}
        K^{B_\la,A}_\ell(x,t,y,s)\,f(y,s)\,\d y\,\d s
    \right|^2\d x\,\d t
    =: I_1+I_2+I_3.
\end{align*}
We bound term $I_1$ by $C\|f\|_2^2$ using
\eqref{farBAestimate} and Schur's test, since
\[
\sup_{(x,t)}
\int_{\|(x,t)-(y,s)\|\geq\lambda}
\frac{\lambda}{\|(x,t)-(y,s)\|^{k+3}}\,\d y\,\d s
\lesssim 1,
\]
and the analogous estimate holds with the roles of $(x,t)$ and $(y,s)$ reversed.

For $I_2$, the standard comparison between sharp and smooth truncations gives
\[
\left|
\int_{\|(x,t)-(y,s)\|\geq\la}
K_\ell(x,t,y,s)f(y,s)\,\d y\,\d s
\right|
\lesssim
\sup_{j\in\mathbb Z}|\mathcal R_j f(x,t)|
+
\mathcal M f(x,t).
\]
Consequently,
\[
I_2
\lesssim
\bigl(1+\|\mathcal R\|^2\bigr)\|f\|_2^2.
\]

We bound term $I_3$ by a constant times $\|f\|_2^2$ using
\eqref{nearBAestimate} and the Hardy--Littlewood maximal
inequality.
\end{proof}

\begin{lemma}\label{ABlabounded}
The following inequality holds uniformly over $\ell\in\mathbb{Z}$ 
and $\la>0$:
\begin{equation*}
    \left\|
    \int_{\mathbb P^k}
        K^{A,B_\la}_\ell(\cdot,\cdot,y,s)\,f(y,s)\,\d y\,\d s
    \right\|_2^2
    \lesssim
    \bigl(1+\|\mathcal{R}\|^2\bigr)\|f\|_2^2.
\end{equation*}
\end{lemma}

\begin{proof}
\begingroup
For $m\in\{1,\ldots,n-k\}$, let
$T_\ell^{A,B_\la,m}$ and $T_\ell^{B_\la,A,m}$ denote the scalar
operators associated with the $m$-th components of $K_\ell^{A,B_\la}$ and
$K_\ell^{B_\la,A}$, respectively. Since
\[
K_\ell^{A,B_\la,m}(x,t,y,s)
=
-K_\ell^{B_\la,A,m}(y,s,x,t),
\]
where
\[
\Psi_\ell((x,t)-(y,s))
=
\Psi_\ell((y,s)-(x,t)),
\]
we have
\[
(T_\ell^{A,B_\la,m})^*
=
-T_\ell^{B_\la,A,m}.
\]
Lemma~\ref{BlaAbounded} bounds the vector-valued operator associated with
$K_\ell^{B_\la,A}$, and hence bounds each $T_\ell^{B_\la,A,m}$ on
$L^2$. Duality therefore gives
\[
\|T_\ell^{A,B_\la,m}f\|_2^2
\lesssim
(1+\|\mathcal R\|^2)\|f\|_2^2
\]
uniformly in $m$, $\ell$, and $\la$. Summing over
$m=1,\ldots,n-k$ proves the stated vector-valued estimate.
\endgroup
\end{proof}

Fix $\la>0$, and let $q\in\mathbb{Z}$ be such that 
\[
2^{-q-1}\leq\la<2^{-q}.
\]
We localize $B_\la$ to the cubes of $\mathbb{D}_q(\mathbb P^k)$. Fix 
$I_i\in\mathbb{D}_q(\mathbb P^k)$.

Let $\psi$ be a smooth bump equal to $1$ on $Q'_{10}(0,0)$, 
supported in $Q'_{20}(0,0)$, with $0\leq\psi\leq 1$ and 
\[
\|\nabla_x^j\psi\|_\infty+\|\partial_t^j\psi\|_\infty\leq C_j
\qquad\text{for all }j\in\mathbb{N}.
\]
With $(x_i,t_i)$ the center of $I_i$, set
\[
    \psi_i(x,t)
    :=
    \psi\left(\frac{x-x_i}{\ell(I_i)},\frac{t-t_i}{\ell(I_i)^2}\right).
\]
Then $\psi_i\equiv 1$ on $10I_i$, $\supp\psi_i\subset 40I_i$, and for $j \in \mathbb N$
\[
\|\nabla_x^j\psi_i\|_\infty\leq C_j\ell(I_i)^{-j},
\qquad
\|\partial_t^j\psi_i\|_\infty\leq C_j\ell(I_i)^{-2j}.
\]
Here $C_j$ depends on $j$ and the fixed function $\psi$, but not on $i$ or $\lambda$.
Define the localization
\[
    B_{\la,i}(x,t)
    :=
    \psi_i(x,t)\bigl(B_\la(x,t)-B_\la(x_i,t_i)\bigr).
\]

Let $\beta_{\la,i}$ denote the $L^2$-$\beta$ numbers 
adapted to the graph of $B_{\la,i}$:
\[
    \beta_{\la,i}((x,t),r)
    :=
    \inf_{a}
    \left(
        \barint_{C'_r(x,t)}
        \left(\frac{|B_{\la,i}(y,s)-a(y)|}{r}\right)^2
        \d y\,\d s
    \right)^{1/2},
\]
where the infimum is over affine functions 
$a:\RR^k\to\RR^{n-k}$, viewed as functions independent of time. We have the estimates
\begin{align}
    \beta_{\la,i}((x,t),r)^2
    &\lesssim \|A\|_{lip}^2\frac{r^2}{\ell(I_i)^2},
    \qquad r\lesssim\ell(I_i),
    \label{beta1smooth}\\
    \beta_{\la,i}((x,t),r)^2
    &\lesssim \|A\|_{lip}^2\frac{\ell(I_i)^2}{r^2},
    \qquad r\gtrsim\ell(I_i).
    \label{beta2smooth}
\end{align}

For \eqref{beta2smooth}, take $a\equiv 0$. Since $0\leq\psi_i\leq1$,
\[
|B_{\la,i}(x,t)|
\leq
|B_\la(x,t)-B_\la(x_i,t_i)|.
\]
We have
\[
B_\la(x,t)-B_\la(x_i,t_i)
=
A*\eta_\la(x,t)-A*\eta_\la(x_i,t_i).
\]
Moreover, $A*\eta_\la$ has the same Lip-$(1,1/2)$ bound as $A$. Hence, on $\supp\psi_i\subset 40I_i$,
\[
|B_{\la,i}(x,t)|
\lesssim
\|A\|_{lip}\,\|(x,t)-(x_i,t_i)\|
\lesssim
\|A\|_{lip}\,\ell(I_i).
\]
Since $B_{\la,i}$ is supported in $40I_i$, this gives
\[
\beta_{\la,i}((x,t),r)
\lesssim
\|A\|_{lip}\frac{\ell(I_i)}{r},
\qquad r\gtrsim\ell(I_i),
\]
which proves \eqref{beta2smooth}.

For \eqref{beta1smooth}, fix $(x,t)\in\mathbb P^k$ and take
\[
a(y)=B_{\la,i}(x,t)+\nabla_x B_{\la,i}(x,t)\cdot(y-x).
\]
Then, for $(y,s)\in C'_r(x,t)$, write
\[
B_{\la,i}(y,s)-a(y)
=
\bigl[B_{\la,i}(y,s)-B_{\la,i}(y,t)\bigr]
+
\bigl[
B_{\la,i}(y,t)-B_{\la,i}(x,t)
-\nabla_x B_{\la,i}(x,t)\cdot(y-x)
\bigr].
\]
The first term is controlled by the time derivative:
\[
\bigl|B_{\la,i}(y,s)-B_{\la,i}(y,t)\bigr|
\leq
\|\partial_t B_{\la,i}\|_\infty |s-t|.
\]
The second term is controlled by Taylor's theorem in the spatial variables:
\[
\bigl|
B_{\la,i}(y,t)-B_{\la,i}(x,t)
-\nabla_x B_{\la,i}(x,t)\cdot(y-x)
\bigr|
\lesssim
\|\nabla_x^2 B_{\la,i}\|_\infty |y-x|^2.
\]
Therefore,
\[
|B_{\la,i}(y,s)-a(y)|
\lesssim
\|\nabla_x^2 B_{\la,i}\|_\infty |y-x|^2
+
\|\partial_t B_{\la,i}\|_\infty |s-t|.
\]
Since $\ell(I_i)\approx \lambda$,
\[
\|\nabla_x^2 B_{\la,i}\|_\infty
\lesssim
\frac{\|A\|_{lip}}{\ell(I_i)},
\qquad
\|\partial_t B_{\la,i}\|_\infty
\lesssim
\frac{\|A\|_{lip}}{\ell(I_i)}.
\]
If $(y,s)\in C'_r(x,t)$ and $r\lesssim \ell(I_i)$, then
\[
|y-x|\lesssim r,
\qquad
|s-t|\lesssim r^2.
\]
Hence
\[
|B_{\la,i}(y,s)-a(y)|
\lesssim
\|A\|_{lip}
\left(
\frac{r^2}{\ell(I_i)}
+
\frac{r^2}{\ell(I_i)}
\right)
\lesssim
\|A\|_{lip}\frac{r^2}{\ell(I_i)}.
\]
Dividing by $r$ and averaging over $C'_r(x,t)$ gives
\[
\beta_{\la,i}((x,t),r)
\lesssim
\|A\|_{lip}\frac{r}{\ell(I_i)}.
\]
This proves \eqref{beta1smooth}.

We integrate in $r$ from $0$ to $\min\{R,\ell(I_i)\}$ using
\eqref{beta1smooth}. If $R>\ell(I_i)$, we use \eqref{beta2smooth} on the
remaining interval from $\ell(I_i)$ to $R$. These estimates give
\begin{align}\label{bicarleson}
    \int_0^R\int_{C'_R(z,\tau)}
    \beta_{\la,i}((x,t),r)^2
    \frac{\d x\,\d t\,\d r}{r}
    \lesssim
    \|A\|_{lip}^2 R^{k+2}
    \lesssim R^{k+2}.
\end{align}

Using the estimates of the derivatives of $\psi_i$ and the Lipschitz bound for the mollified function, we have
\[
\|B_{\la,i}\|_{lip}
\lesssim
\|A\|_{lip}.
\]
Thus, after decreasing $\epsilon$ if necessary, the Lip-$(1,1/2)$ constants of the $B_{\lambda,i}$
are uniformly small. For $m=1,\dots,n-k$, let $\beta_{\la,i,m}$ be defined as
$\beta_{\la,i}$, with $B_{\la,i}$ replaced by its $m$-th component and
with scalar affine functions in the infimum. If
$\Gamma_{\la,i,m}$ denotes the graph of $(B_{\la,i})_m$, then
\[
\beta_2^{\Gamma_{\la,i,m}}(C'_r(x,t))
\leq
\beta_{\la,i,m}((x,t),r)
\leq
\beta_{\la,i}((x,t),r).
\]
The first inequality follows by taking the graph of a scalar affine function as
the approximating vertical plane. The second follows by taking the $m$-th component of an
affine function $a$ in the definition of $\beta_{\la,i}$. Hence \eqref{bicarleson}  holds for each $\beta_{\lambda,i,m}$, uniformly in $m$, $i$, and $\la$.
Applying Theorem~\ref{regularbetaequiv} to each component, we conclude that
$D_t^{1/2}(B_{\la,i})_m\in BMO(\mathbb P^k)$ for every $m$, with
uniform bounds, and that $B_{\la,i}$ is regular Lip-$(1,1/2)$.

\begin{proof}[Proof of Proposition \ref{Blabounded}]
We split the expression as follows:
\begin{align*}
    &\int_{\mathbb P^k}
    \left|\int_{\mathbb P^k}
        K^{B_\la}_\ell(x,t,y,s)\,f(y,s)\,\d y\,\d s
    \right|^2\d x\,\d t \\
    &\lesssim
    \underbrace{
    \int_{\mathbb P^k}
    \left|\int_{\|(x,t)-(y,s)\|\geq\la}
        K^{B_\la}_\ell(x,t,y,s)\,f(y,s)\,\d y\,\d s
    \right|^2\d x\,\d t}_{I}
    \\
    &\quad+
    \underbrace{
    \int_{\mathbb P^k}
    \left|\int_{\|(x,t)-(y,s)\|<\la}
        K^{B_\la}_\ell(x,t,y,s)\,f(y,s)\,\d y\,\d s
    \right|^2\d x\,\d t}_{II}.
\end{align*}

We first estimate $II$.
Suppose $(x,t) \in I_i$.  Since $\|(x,t)-(y,s)\| < \lambda < \ell(I_i)$ we have $(y,s)\in 10I_i$. Since
\[
K^{B_\la}_\ell(x,t,y,s)=K^{B_{\la,i}}_\ell(x,t,y,s)
\]
for such pairs, we have
\begin{align*}
    II
    &=
    \sum_{I_i\in\mathbb{D}_q}
    \int_{I_i}
    \left|
        \int_{\|(x,t)-(y,s)\|<\la}
        K^{B_{\la,i}}_\ell(x,t,y,s)\,
        \mathbf{1}_{10I_i}(y,s)f(y,s)
        \,\d y\,\d s
    \right|^2\d x\,\d t.
\end{align*}

Since each $B_{\la,i}$ is regular Lip-$(1,1/2)$, with regularity constants
bounded uniformly in $i$ and $\la$, Theorem~\ref{highcodimBHHLN} gives a
uniform $L^2(\mathbb P^k)$ bound for the smooth truncations associated
with $K^{B_{\la,i}}$. Lemma~\ref{smoothingCZkernels} gives uniform
standard parabolic Calder\'on--Zygmund constants for these kernels.
 The standard Cotlar inequality therefore
gives
\[
\left\|\sup_{j\in\mathbb Z}|\mathcal R_j^i g|\right\|_2
\lesssim
\|g\|_2,
\]
where $\mathcal R_j^i$ denotes the smooth truncation associated with
$K^{B_{\la,i}}$. Moreover, comparison with a smooth truncation at scale
$\la$ gives
\[
\left|
\int_{\|(x,t)-(y,s)\|<\la}
K_\ell^{B_{\la,i}}(x,t,y,s)g(y,s)\,\d y\,\d s
\right|
\lesssim
\sup_{j\in\mathbb Z}|\mathcal R_j^i g(x,t)|
+
\mathcal M g(x,t).
\]
Taking
$g=\mathbf 1_{10I_i}f$, integrating over $I_i$, and using the bounded overlap
of the cubes $\{10I_i\}$, we obtain
\[
II
\lesssim
\sum_{I_i\in\mathbb D_q}
\int_{10I_i}|f(y,s)|^2\,\d y\,\d s
\lesssim
\|f\|_2^2.
\]

To estimate $I$, we write
\begin{align*}
    I
    &\lesssim
    \int_{\mathbb P^k}
    \left|
        \int_{\|(x,t)-(y,s)\|\geq\la}
        \bigl(K^{B_\la}_\ell(x,t,y,s)-K^{A,B_\la}_\ell(x,t,y,s)\bigr)
        f(y,s)\,\d y\,\d s
    \right|^2\d x\,\d t \\
    &\qquad+
    \int_{\mathbb P^k}
    \left|
        \int_{\|(x,t)-(y,s)\|\geq\la}
        K^{A,B_\la}_\ell(x,t,y,s)\,f(y,s)\,\d y\,\d s
    \right|^2\d x\,\d t \\
    &=: I_1 + I_2.
\end{align*}

By Lemma~\ref{ABlabounded}, estimate \eqref{nearABestimate}, and the
Hardy--Littlewood maximal inequality,
\begin{align*}
I_2
&=
\left\|
\int_{\|(x,t)-(y,s)\|\geq\la}
K_\ell^{A,B_\la}(x,t,y,s)f(y,s)\,\d y\,\d s
\right\|_2^2\\
&=
\left\|
\int_{\mathbb P^k}
K_\ell^{A,B_\la}(x,t,y,s)f(y,s)\,\d y\,\d s
-
\int_{\|(x,t)-(y,s)\|<\la}
K_\ell^{A,B_\la}(x,t,y,s)f(y,s)\,\d y\,\d s
\right\|_2^2\\
&\lesssim
\left\|
\int_{\mathbb P^k}
K_\ell^{A,B_\la}(x,t,y,s)f(y,s)\,\d y\,\d s
\right\|_2^2
+
\left\|
\int_{\|(x,t)-(y,s)\|<\la}
K_\ell^{A,B_\la}(x,t,y,s)f(y,s)\,\d y\,\d s
\right\|_2^2\\
&\lesssim
\bigl(1+\|\mathcal R\|^2\bigr)\|f\|_2^2
+
\|\mathcal M f\|_2^2\\
&\lesssim
\bigl(1+\|\mathcal R\|^2\bigr)\|f\|_2^2.
\end{align*}

For $I_1$, on the region $\|(x,t)-(y,s)\|\geq\la$, we use
\eqref{farBlaABestimate}. Schur's test gives
\[
I_1\lesssim\|f\|_2^2.
\]
This proves Proposition \ref{Blabounded}.
\end{proof}

\subsection{The Proof of Theorem \ref{localdpara}}

In this subsection we prove Theorem~\ref{localdpara}.

\begingroup
\begin{lemma}
\label{vectorDorronsoro}
Let $F:\mathbb P^k\to\mathbb R^{n-k}$ be a 
Lip-$(1,1/2)$ function such that $D_{par}F\in L^2(\mathbb P^k)$.
Then
\begin{equation}\label{globaldorronsoro}
\sum_{Q\in\mathbb D}
\beta_2^{\Gamma_F}(5Q)^2|Q|
\lesssim
\|D_{par}F\|_2^2.
\end{equation}
\end{lemma}

\begin{proof}
For a scalar function $f$, set
\[
\beta_2(f,Q)
:=
\inf_a
\left(
\frac1{|Q|}
\int_Q
\left(
\frac{|f(x,t)-a(x)|}{\ell(Q)}
\right)^2
\,\d x\,\d t
\right)^{1/2},
\]
where the infimum is taken over affine functions
$a:\mathbb R^k\to\mathbb R$. In \cite[p.~249, (35)]{H}, it is shown that
\begin{equation}\label{scalardorronsoro}
\sum_{Q\in\mathbb D}\beta_2(f,5Q)^2|Q|
\lesssim
\|D_{par}f\|_2^2.
\end{equation}

Write $F=(F_1,\dots,F_{n-k})$. For each $Q\in\mathbb D$ and
$m\in\{1,\dots,n-k\}$, let $a_{Q,m}$ be an affine minimizer for
$\beta_2(F_m,5Q)$, and set
\[
a_Q=(a_{Q,1},\dots,a_{Q,n-k}).
\]
The set
\[
P_Q
:=
\{(x,a_Q(x),t):x\in\mathbb R^k,\ t\in\mathbb R\}
\]
is a vertical plane. For $(x,t)\in5Q$,
\[
d_p\bigl((x,F(x,t),t),P_Q\bigr)
\leq
|F(x,t)-a_Q(x)|.
\]
Since the measure on $\Gamma_F$ is the pushforward of
$\d x\,\d t$, the definition of the graph $\beta_2$-number gives
\[
\beta_2^{\Gamma_F}(5Q)^2
\leq
\sum_{m=1}^{n-k}\beta_2(F_m,5Q)^2.
\]
Summing over $Q$ and applying \eqref{scalardorronsoro} to each component,
we obtain
\[
\sum_{Q\in\mathbb D}
\beta_2^{\Gamma_F}(5Q)^2|Q|
\lesssim
\sum_{m=1}^{n-k}\|D_{par}F_m\|_2^2
=
\|D_{par}F\|_2^2.
\]
\end{proof}
\endgroup

\begin{lemma}\label{czestim}
Set $\delta=2^{-j}$. For $(X,t)=(x,z,t)\in \mathbb R^k\times\mathbb R^{n-k}\times\mathbb R$, we have
\begin{align}
|J_{(j)}(X,t)-J_{comm,(j)}(X,t)|\label{czestim:size}
&\lesssim
\frac{|z|^3}{\delta^{k+5}}\chi_{A(0,\delta/32,\delta/2)}(x,t),\\
|\partial_{x_i}(J_{(j)}(X,t)-J_{comm,(j)}(X,t))|\label{czestim:space}
&\lesssim
\frac{|z|^3}{\delta^{k+6}}\chi_{A(0,\delta/32,\delta/2)}(x,t),
\qquad 1\leq i\leq k,\\
|\partial_{z_r}(J_{(j)}(X,t)-J_{comm,(j)}(X,t))|\label{czestim:transverse}
&\lesssim
\frac{|z|^2}{\delta^{k+5}}\chi_{A(0,\delta/32,\delta/2)}(x,t),
\qquad 1\leq r\leq n-k,\\
|\partial_t(J_{(j)}(X,t)-J_{comm,(j)}(X,t))|\label{czestim:time}
&\lesssim
\frac{|z|}{\delta^{k+5}}\chi_{A(0,\delta/32,\delta/2)}(x,t).
\end{align}
Here $A(0,M_1,M_2)$ is the annulus
\[
A(0,M_1,M_2):=\{(x,t):M_1\leq \|(x,t)\|\leq M_2\}.
\]
\end{lemma}

\begin{proof}
We will repeatedly use the bounds
\begin{align}\label{MV}
\left|e^{-\frac{|x|^2}{4|t|}}-e^{-\frac{|x|^2+|z|^2}{4|t|}}\right|
&=
e^{-\frac{|x|^2}{4|t|}}
\left|1-e^{-\frac{|z|^2}{4|t|}}\right| \notag\\
&\lesssim
\frac{|z|^2}{|t|}
e^{-\frac{|x|^2}{4|t|}},
\end{align}

and, for every $l>0$ and $c>0$,
\begin{align}\label{heatkerbd}
|t|^{-\frac{k+2+l}{2}}e^{-c\frac{|x|^2}{|t|}}
\lesssim_{l,c}
\|(x,t)\|^{-(k+2+l)}.
\end{align}

We will also use the estimate (which holds for every $\alpha\geq0$)
\begin{align}\label{gaussianabsorption}
u^\alpha e^{-u/4}\lesssim_\alpha e^{-u/8},
\qquad u\geq0.
\end{align}

Fix $m\in\{1,\dots,n-k\}$.  We will prove each estimate for the $m$\textsuperscript{th} component.  For \eqref{czestim:size}, \eqref{MV} gives
\begin{align*}
|J^m_{(j)}(X,t)-J^m_{comm,(j)}(X,t)|
&\leq
\frac{|z_m|}{|t|^{(k+1)/2+1}}
\left|
e^{-\frac{|x|^2+|z|^2}{4|t|}}
-
e^{-\frac{|x|^2}{4|t|}}
\right|
|\varphi_j(x,t)|\\
&\lesssim
\frac{|z|^3}{|t|^{(k+1)/2+2}}
e^{-\frac{|x|^2}{4|t|}}
\chi_{A(0,\delta/32,\delta/2)}(x,t)\\
&\lesssim
\frac{|z|^3}{\delta^{k+5}}
\chi_{A(0,\delta/32,\delta/2)}(x,t),
\end{align*}
where the last step uses \eqref{heatkerbd} with $l=3$.

To prove \eqref{czestim:space}, fix $i\in\{1,\dots,k\}$. Differentiating in $x_i$ yields one term from the exponential factor and one from $\varphi_j$:
\begin{align*}
|\partial_{x_i}(J^m_{(j)}-J^m_{comm,(j)})(X,t)|
&\lesssim
\frac{|x_i||z_m|}{|t|^{(k+1)/2+2}}
\left|
e^{-\frac{|x|^2+|z|^2}{4|t|}}
-
e^{-\frac{|x|^2}{4|t|}}
\right|
|\varphi_j(x,t)|\\
&\qquad
+
\frac{|z_m|}{|t|^{(k+1)/2+1}}
\left|
e^{-\frac{|x|^2+|z|^2}{4|t|}}
-
e^{-\frac{|x|^2}{4|t|}}
\right|
|\nabla_x\varphi_j(x,t)|.
\end{align*}

By \eqref{MV}, the term in which the derivative falls on the exponential
is bounded by
\[
\frac{|x_i||z|^3}{|t|^{(k+7)/2}}
e^{-\frac{|x|^2}{4|t|}}
\chi_{A(0,\delta/32,\delta/2)}(x,t),
\]
whereas the other term is bounded by
\[
\frac{|z|^3}{\delta|t|^{(k+5)/2}}
e^{-\frac{|x|^2}{4|t|}}
\chi_{A(0,\delta/32,\delta/2)}(x,t).
\]

For the first expression, \eqref{gaussianabsorption} with
\(u=|x|^2/|t|\) and \(\alpha=1/2\), followed by
\eqref{heatkerbd} with \(l=4\) and Gaussian constant \(1/8\), gives
\[
\frac{|x_i||z|^3}{|t|^{(k+7)/2}}
e^{-\frac{|x|^2}{4|t|}}
\lesssim
\frac{|z|^3}{|t|^{(k+6)/2}}
e^{-\frac{|x|^2}{8|t|}}
\lesssim
\frac{|z|^3}{\|(x,t)\|^{k+6}}.
\]

Applying \eqref{heatkerbd} with
$l=3$ to the second expression gives

\[
|\partial_{x_i}(J^m_{(j)}-J^m_{comm,(j)})(X,t)|
\lesssim
\frac{|z|^3}{\delta^{k+6}}
\chi_{A(0,\delta/32,\delta/2)}(x,t).
\]

For $r\in\{1,\dots,n-k\}$, a direct computation gives
\[
\begin{aligned}
\partial_{z_r}(J^m-J^m_{comm})
&=
\mathbf 1_{\{m=r\}}
|t|^{-(k+1)/2-1}
e^{-\frac{|x|^2}{4|t|}}
\left(e^{-\frac{|z|^2}{4|t|}}-1\right)\\
&\quad-
\frac{z_m z_r}{2|t|^{(k+1)/2+2}}
e^{-\frac{|x|^2+|z|^2}{4|t|}}.
\end{aligned}
\]
Hence, after multiplying by $\varphi_j$,
\begin{align*}
|\partial_{z_r}(J^m_{(j)}-J^m_{comm,(j)})(X,t)|
&\lesssim
\frac{|z|^2}{|t|^{(k+1)/2+2}}
e^{-\frac{|x|^2}{4|t|}}
\chi_{A(0,\delta/32,\delta/2)}(x,t)\\
&\lesssim
\frac{|z|^2}{\delta^{k+5}}
\chi_{A(0,\delta/32,\delta/2)}(x,t),
\end{align*}
where we used \eqref{heatkerbd} with $l=3$. This proves \eqref{czestim:transverse}.

Finally, we prove \eqref{czestim:time}. Since
\[
J^m_{(j)}=J^m\varphi_j \qquad\text{and}\qquad J^m_{comm,(j)}=J^m_{comm}\varphi_j,
\]
the product rule gives
\[
\partial_t\bigl(J^m_{(j)}-J^m_{comm,(j)}\bigr)=\bigl(\partial_tJ^m-\partial_tJ^m_{comm}\bigr)\varphi_j+\bigl(J^m-J^m_{comm}\bigr)\partial_t\varphi_j.
\]

We first estimate the term in which the derivative falls on the kernels.
For $t\neq0$, set
\[
u=\frac{|x|^2}{|t|} \qquad\text{and}\qquad v=\frac{|x|^2+|z|^2}{|t|}.
\]
We have
\[
\begin{aligned}
|\partial_tJ^m_{comm}(X,t)|&\lesssim \frac{|z|}{|t|^{(k+5)/2}}(1+u)e^{-u/4}\lesssim \frac{|z|}{\|(x,t)\|^{k+5}},\\
|\partial_tJ^m(X,t)|&\lesssim \frac{|z|}{|t|^{(k+5)/2}}(1+v)e^{-v/4}\lesssim \frac{|z|}{\|(x,t)\|^{k+5}}.
\end{aligned}
\]

The estimate on $\partial_t J^m_{comm}$ follows by using \eqref{gaussianabsorption} with $\alpha=0,1$ to obtain
$(1+u)e^{-u/4}\lesssim e^{-u/8}$, and then applying \eqref{heatkerbd} with
$l=3$ and $c=1/8$.  Similar reasoning yields the bound on $\partial_t J^m$. Hence, using
$|\varphi_j|\lesssim\chi_{A(0,\delta/32,\delta/2)}$ and
$\|(x,t)\|\geq\delta/32$ on $\supp\varphi_j$,
\[
\left|\bigl(\partial_tJ^m-\partial_tJ^m_{comm}\bigr)(X,t)\varphi_j(x,t)\right|
\lesssim \frac{|z|}{\delta^{k+5}}\chi_{A(0,\delta/32,\delta/2)}(x,t).
\]

It remains to estimate the term in which the derivative falls on $\varphi_j$.
By \eqref{heatkerbd} with $l=1$ and $c=1/4$,
\[
|J^m(X,t)|+|J^m_{comm}(X,t)|\lesssim \frac{|z|}{\|(x,t)\|^{k+3}}.
\]
Therefore, using
$|\partial_t\varphi_j|\lesssim
\delta^{-2}\chi_{A(0,\delta/32,\delta/2)}$ and again the lower bound
$\|(x,t)\|\geq\delta/32$ on the support of $\varphi_j$,
\[
\left|\bigl(J^m-J^m_{comm}\bigr)(X,t)\partial_t\varphi_j(x,t)\right|
\lesssim \frac{|z|}{\delta^{k+3}}\,\delta^{-2}\chi_{A(0,\delta/32,\delta/2)}(x,t)
=\frac{|z|}{\delta^{k+5}}\chi_{A(0,\delta/32,\delta/2)}(x,t).
\]
This proves all four estimates.
\end{proof}

\begin{lemma}\label{aqsmall}
$\alpha(Q)\lesssim \epsilon$ for every $Q\in\mathbb D(\mathbb P^k)$.
\end{lemma}

\begin{proof}

Let \(P\) be the translate of \(\mathbb P^k\) passing through
\(\big(z_Q,A(z_Q,\tau_Q),\tau_Q\big)\). We estimate \(\alpha(Q)\) using $\sigma_P$. If \(\|f\|_{lip}\leq1\) and \(\supp(f)\subseteq C_Q\), then the identity \(\mu=\Phi_\#(dx\,dt)\) gives
\begin{align*}
\ell(Q)^{-k-3}
\left|\int f\,\d\mu-\int f\,\d\sigma_P\right|
&=
\ell(Q)^{-k-3}
\left|
\int
\left[
f(y,A(y,s),s)
-
f(y,A(z_Q,\tau_Q),s)
\right]
\,\d y\,\d s
\right|\\
&\lesssim
\ell(Q)^{-k-3}
\int_{4Q}
|A(y,s)-A(z_Q,\tau_Q)|
\,\d y\,\d s.
\end{align*}

We restrict to $4Q=\pi(C_Q)$ because, if
$(y,s)\notin\pi(C_Q)$, both $(y,A(y,s),s)$ and
$(y,A(z_Q,\tau_Q),s)$ lie outside $C_Q$, and thus $f$ vanishes for such $(y,s)$.
For $(y,s)\in4Q$,
\[
|A(y,s)-A(z_Q,\tau_Q)|
\lesssim
\epsilon\ell(Q).
\]
Since $|4Q|\lesssim\ell(Q)^{k+2}$, taking the supremum over all such $f$ gives
$\ell(Q)^{-k-3}\dist_{C_Q}(\mu,\sigma_P)\lesssim\epsilon$.
The definition of $\alpha(Q)$ then gives $\alpha(Q)\lesssim\epsilon$.

\end{proof}

\begin{lemma}\label{lqmeetscq}
We can take $\epsilon$ sufficiently small so that, for each
$Q\in\mathbb D(\mathbb P^k)$,
\[
L_Q\cap\tfrac12 C_Q\neq\emptyset.
\]
\end{lemma}

\begin{proof}
We first choose the minimizing pair so that $L_Q\cap C_Q\neq\emptyset$.
Indeed, if a minimizing plane $L_Q$ is disjoint from $C_Q$, then
\[
\dist_{C_Q}(\mu,c_Q\sigma_{L_Q})=\dist_{C_Q}(\mu,0).
\]
Since $c=0$ is allowed in the definition of $\alpha(Q)$, we may instead
take $c_Q=0$ and choose $L_Q$ to be any vertical plane meeting $C_Q$.

Choose a smooth function $\varphi$ that equals $1$ on
$\tfrac14 C_Q$, vanishes outside $\tfrac12 C_Q$, and satisfies
\[
\|\nabla\varphi\|_\infty\lesssim\frac1{\ell(Q)},
\qquad
\|\partial_t\varphi\|_\infty\lesssim\frac1{\ell(Q)^2}.
\]
Since $L_Q\cap C_Q\neq\emptyset$ and
$\supp\varphi\subset\tfrac12 C_Q$,
\[
\dist\big((X,t),L_Q\big)\lesssim\ell(Q)
\qquad\text{on }\supp\varphi.
\]
The function
\[
f(X,t)=\varphi(X,t)\dist\big((X,t),L_Q\big)
\]
has $\|f\|_{lip}\lesssim1$, is supported in $C_Q$, and vanishes on
$L_Q$. Hence
\[
\int f\,\d\mu
\lesssim
\alpha(Q)\ell(Q)^{k+3}.
\]
Since
$\mu(\Gamma\cap\tfrac14 C_Q)\approx\ell(Q)^{k+2}$, there is a
point $(X,t)\in\Gamma\cap\tfrac14 C_Q$ such that
\[
\dist\big((X,t),L_Q\big)
\lesssim
\alpha(Q)\ell(Q)
\lesssim
\epsilon\ell(Q),
\]
where the last inequality follows from Lemma \ref{aqsmall}. If
$\epsilon$ is sufficiently small, this gives a point of $L_Q$ in
$\tfrac12 C_Q$.
\end{proof}

\begin{lemma}\label{b1abd}
For each $Q\in\mathbb D(\mathbb P^k)$, we have
\[
\int_{Q^\Gamma}
\frac{\dist\big((X,t),L_Q\big)}{\ell(Q)}
\,\d\mu(X,t)
\lesssim
\alpha(Q)|Q|.
\]
\end{lemma}

\begin{proof}
By Lemma \ref{lqmeetscq}, $L_Q\cap\tfrac12 C_Q\neq\emptyset$.
Consequently,
\[
\dist\big((Y,s),L_Q\big)\lesssim\ell(Q)
\qquad\text{for every }(Y,s)\in C_Q.
\]
Let $\zeta:\mathbb P^n\to\mathbb R$ be smooth, equal to $1$ on
$\tfrac12 C_Q$, equal to $0$ outside of $C_Q$, and such that
\[
\|\nabla \zeta\|_\infty\lesssim \frac{1}{\ell(Q)},
\qquad
\|\partial_t\zeta\|_\infty\lesssim \frac{1}{\ell(Q)^2}.
\]
Set
\[
f(Y,s)=\dist\big((Y,s),L_Q\big)\zeta(Y,s).
\]
Then, $\|f\|_{lip}\lesssim1$. Since $f$ vanishes on $L_Q$, we have
\begin{multline*}
\int_{Q^\Gamma}\dist\big((Y,s),L_Q\big)\,\d\mu(Y,s)
\leq
\int f(Y,s)\,\d\mu(Y,s)\\
=
\left|
\int f(Y,s)\,\d\mu(Y,s)
-
\int f(Y,s)\,\d\mathcal L_Q(Y,s)
\right|
\lesssim
\alpha(Q)\ell(Q)^{k+3}.
\end{multline*}
Dividing by $\ell(Q)$ and using
$|Q|\approx\ell(Q)^{k+2}$ proves the lemma.
\end{proof}

\begin{lemma}\label{cqbd}
For each $Q\in\mathbb D(\mathbb P^k)$, $c_Q\approx 1$ if $\epsilon\ll 1$.
\end{lemma}

\begin{proof}

Let $\psi$ be a smooth function such that
\[
0\leq\psi\leq1,
\qquad
\psi\equiv1\text{ on }\tfrac34 C_Q,
\qquad
\supp\psi\subset C_Q,
\]
and
\[
\|\psi\|_{lip}\lesssim\frac1{\ell(Q)}.
\]

Then
\[
\left|
\int \psi\,\d\mu
-
c_Q\int \psi\,\d\sigma_{L_Q}
\right|
\lesssim
\alpha(Q)|Q|.
\]
Hence
\[
\int \psi\,\d\mu
-
C\alpha(Q)|Q|
\leq
c_Q\int \psi\,\d\sigma_{L_Q}
\leq
\int \psi\,\d\mu
+
C\alpha(Q)|Q|.
\]
By Lemma \ref{lqmeetscq},
\[
\int \psi\,\d\sigma_{L_Q}\approx |Q|.
\]
We also have $\int \psi\,\d\mu\approx \ell(Q)^{k+2}$, and by Lemma \ref{aqsmall},
$\alpha(Q)\lesssim \epsilon$. Thus the lemma holds for $\epsilon\ll 1$.
\end{proof}

For the remainder of this section we take $\epsilon$ small enough that $c_Q\approx1$ for every $Q$, as in Lemma~\ref{cqbd}.

\begin{lemma}\label{pqlq}
Fix $Q\in\mathbb D(\mathbb P^k)$. Let $P_Q$ be the
$k$-dimensional vertical plane minimizing $\beta_2(5Q)$. There holds
\[
\dist_H(P_Q\cap C_Q,L_Q\cap C_Q)
\lesssim
\big(\alpha(Q)+\beta_2(5Q)\big)\ell(Q).
\]
\end{lemma}

\begin{proof}
Let $\zeta:\mathbb P^n\to\mathbb R$ be smooth, equal to $1$ on
$\tfrac34 C_Q$, equal to $0$ outside of $C_Q$, and such that
\[
\|\nabla \zeta\|_\infty\lesssim \frac{1}{\ell(Q)},
\qquad
\|\partial_t\zeta\|_\infty\lesssim \frac{1}{\ell(Q)^2}.
\]
Set
\[
f(Y,s)=\dist((Y,s),P_Q)\zeta(Y,s)
\]
and notice that $\|f\|_{lip}\lesssim 1$. So, we have
\begin{multline*}
\left|
\int f(Y,s)\,\d\mathcal L_Q(Y,s)
\right|
\leq
\left|
\int f(Y,s)\big(\d\mathcal L_Q(Y,s)-\d\mu(Y,s)\big)
\right|
+
\left|
\int f(Y,s)\,\d\mu(Y,s)
\right|\\
\lesssim
\big(\alpha(Q)+\beta_2(5Q)\big)\ell(Q)^{k+3}.
\end{multline*}
The desired estimate follows because $L_Q$ and $P_Q$ both intersect
$\tfrac12 C_Q$ and $c_Q\approx1$. Thus there is a subset of $\tfrac34 C_Q\cap L_Q$ with $\sigma_{L_Q}$-measure comparable to $\ell(Q)^{k+2}$ on which
\[
f(\cdot)\approx \dist_H(P_Q\cap C_Q,L_Q\cap C_Q).
\]
\end{proof}

\begin{lemma}\label{commapprox}
Fix $Q\in\mathbb D(\mathbb P^k)$ with $\ell(Q)=2^{-j}$.  If $\epsilon$ is taken sufficiently small, then
\[
\int_{Q^\Gamma} |\mathcal R_{(j)}1(X,t)-\mathcal C_{(j)}1(X,t)|^2\,\d\mu(X,t)
\lesssim
\epsilon^2\big[\alpha(Q)^2+\beta_2(5Q)^2\big]|Q|.
\]
\end{lemma}

\begin{proof}
Consider the kernel $D_{(j)}=J_{(j)}-J_{comm,(j)}$ and let $T_{(j)}$ be the associated operator.

Set
\[
H_Q
:=
\bigl\{(y,A(z_Q,\tau_Q),s):y\in\mathbb R^k,\ s\in\mathbb R\bigr\}.
\]
The Lipschitz bound for $A$ gives
\[
d_p\bigl((y,A(y,s),s),H_Q\bigr)
\lesssim
\epsilon\ell(Q),
\qquad (y,s)\in10Q.
\]
In particular, $H_Q$ is an admissible competitor for $\beta_2(5Q)$.
Since $P_Q$ realizes the infimum,
\[
\fint_{(5Q)^\Gamma}
d_p\bigl((Y,s),P_Q\bigr)^2\,\d\mu(Y,s)
\approx
\beta_2(5Q)^2\ell(Q)^2
\lesssim
\epsilon^2\ell(Q)^2.
\]
For $(y,s)\in5Q$, the triangle inequality gives
\[
d_p\bigl((y,A(z_Q,\tau_Q),s),P_Q\bigr)
\leq
|A(y,s)-A(z_Q,\tau_Q)|
+
d_p\bigl((y,A(y,s),s),P_Q\bigr).
\]
Consequently,
\[
\fint_{5Q}
d_p\bigl((y,A(z_Q,\tau_Q),s),P_Q\bigr)^2\,\d y\,\d s
\lesssim
\epsilon^2\ell(Q)^2.
\]
Thus,
\[
\sup_{(Y,s)\in H_Q\cap\pi^{-1}(4Q)}
d_p((Y,s),P_Q)
\lesssim \epsilon\ell(Q).
\]
So, $\angle(P_Q,H_Q) \lesssim \epsilon$.
Combining these estimates with the Lipschitz bound for $A$, we obtain
\[
\Gamma\cap C_Q
\subset
\bigl\{(Y,s):d_p((Y,s),P_Q)\leq C\epsilon\ell(Q)\bigr\},
\qquad
P_Q\cap C_Q
\subset
\bigl\{(Y,s):d_p((Y,s),\Gamma)\leq C\epsilon\ell(Q)\bigr\}.
\]

Fix $(X,t)=(x,x^\perp,t)\in Q^\Gamma$, and let $(X',t)$ be its
orthogonal projection onto $P_Q$. Then
\[
\|(X',t)-(X,t)\|
\lesssim
\epsilon\ell(Q).
\]
The support of $\varphi_j$ and the estimates above imply
\[
\supp\!\left(
D_{(j)}((X',t)-\cdot)\big|_\Gamma
\right)
\cup
\supp\!\left(
D_{(j)}((X',t)-\cdot)\big|_{P_Q}
\right)
\subset C_Q.
\]

We first estimate $T_{(j)}1(X',t)$. Since $P_Q$ is vertical, we may choose a smooth function $\chi_Q$, independent of the time variable, such that
\[
\chi_Q=1
\quad\text{when }d_p((Y,s),P_Q)\leq C\epsilon\ell(Q),
\qquad
\supp\chi_Q
\subset
\{(Y,s):d_p((Y,s),P_Q)\leq 2C\epsilon\ell(Q)\},
\]
and
\[
|\nabla\chi_Q|
\lesssim
\frac{1}{\epsilon\ell(Q)}.
\]
Set
\[
f(Y,s)
=
D_{(j)}\big((X',t)-(Y,s)\big),
\qquad
\widetilde f(Y,s)
=
\chi_Q(Y,s)f(Y,s).
\]
The small slope of $P_Q$ and the support of $f$ imply that
\[
|x'^\perp-y^\perp|
\lesssim
\epsilon\ell(Q)
\]
whenever $\chi_Q(Y,s)f(Y,s)\neq0$. Lemma~\ref{czestim} therefore gives
\[
|f(Y,s)|
\lesssim
\frac{\epsilon^3}{\ell(Q)^{k+2}},
\qquad
|\nabla f(Y,s)|
\lesssim
\frac{\epsilon^2}{\ell(Q)^{k+3}},
\qquad
|\partial_s f(Y,s)|
\lesssim
\frac{\epsilon}{\ell(Q)^{k+4}}
\]
on $\supp\chi_Q\cap\supp f$. Consequently,
\[
\|\widetilde f\|_\infty
\lesssim
\frac{\epsilon^3}{\ell(Q)^{k+2}},
\qquad
|\nabla\widetilde f|
\leq
|\nabla\chi_Q|\,|f|+|\nabla f|
\lesssim
\frac{\epsilon^2}{\ell(Q)^{k+3}},
\qquad
|\partial_s\widetilde f|
\lesssim
\frac{\epsilon}{\ell(Q)^{k+4}}.
\]
For $|s-s'|\leq \ell(Q)^2$, the bound for
$\partial_s\widetilde f$ gives
\[
|\widetilde f(Y,s)-\widetilde f(Y,s')|
\lesssim
\frac{\epsilon}{\ell(Q)^{k+3}}|s-s'|^{1/2}.
\]
For $|s-s'| > \ell(Q)^2$, we have the same bound using the $L^\infty$ estimate.  Together with the spatial gradient bound, this yields
\[
\|\widetilde f\|_{lip}
\lesssim
\frac{\epsilon}{\ell(Q)^{k+3}}.
\]
Moreover, $\supp\widetilde f\subset C_Q$, and $\widetilde f=f$ on
$\Gamma\cap\supp f$ and on $P_Q\cap\supp f$.

Since $J_{(j)}$ and $J_{comm,(j)}$ are odd in the spatial variables, reflection in the spatial part of $P_Q$ about $X'$ gives
\[
\int D_{(j)}\big((X',t)-(Y,s)\big)\,\d\sigma_{P_Q}(Y,s)=0.
\]
Since $\widetilde f=f$ on $\Gamma\cap\supp f$ and on $P_Q\cap\supp f$,
\begin{multline*}
\left|
\int D_{(j)}\big((X',t)-(Y,s)\big)\,\d\mu(Y,s)
\right|\\
=
\bigg|
\int D_{(j)}\big((X',t)-(Y,s)\big)\,\d\mu(Y,s)
-
c_Q\int D_{(j)}\big((X',t)-(Y,s)\big)\,\d\sigma_{P_Q}(Y,s)
\bigg|\\
\leq
\left|
\int \widetilde f(Y,s)\,\d\mu(Y,s)
-
c_Q\int \widetilde f(Y,s)\,\d\sigma_{L_Q}(Y,s)
\right|\\
+
c_Q
\left|
\int \widetilde f(Y,s)\,\d\sigma_{L_Q}(Y,s)
-
\int \widetilde f(Y,s)\,\d\sigma_{P_Q}(Y,s)
\right|\\
\lesssim
\epsilon\big(\alpha(Q)+\beta_2(5Q)\big).
\end{multline*}
Here $c_Q$ is the constant associated to $L_Q$ minimizing $\alpha(Q)$. In the last inequality we used Lemmas
\ref{cqbd} and \ref{pqlq}.

Therefore
\[
|T_{(j)}1(X',t)|
\lesssim
\epsilon\big(\alpha(Q)+\beta_2(5Q)\big).
\]
To compare the values at $(X,t)$ and $(X',t)$, we apply the mean value
theorem along $[X',X]$, the spatial segment joining $X'$ and $X$:
\[
|T_{(j)}1(X,t)-T_{(j)}1(X',t)|
\leq
\dist((X,t),P_Q)
\sup_{\xi\in[X',X]}|\nabla T_{(j)}1(\xi,t)|.
\]
If $\xi\in[X',X]$ and
$D_{(j)}((\xi,t)-(Y,s))\neq0$ for $(Y,s)\in\Gamma$, then the
estimates above give
\[
|\xi^\perp-y^\perp|
\lesssim
\epsilon\ell(Q).
\]
Lemma~\ref{czestim} and $\mu(C_Q)\lesssim\ell(Q)^{k+2}$ therefore give
\[
|\nabla T_{(j)}1(\xi,t)|
\leq
\int
|\nabla D_{(j)}((\xi,t)-(Y,s))|\,\d\mu(Y,s)
\lesssim
\frac{\epsilon^2}{\ell(Q)}.
\]
Hence
\[
|T_{(j)}1(X,t)-T_{(j)}1(X',t)|
\lesssim
\epsilon^2
\frac{\dist((X,t),P_Q)}{\ell(Q)}.
\]
We then have
\begin{align*}
|T_{(j)}1(X,t)|^2
&\lesssim
|T_{(j)}1(X',t)|^2
+
|T_{(j)}1(X,t)-T_{(j)}1(X',t)|^2\\
&\lesssim
\epsilon^2
\left(
\alpha(Q)^2
+
\left(\frac{\dist((X,t),P_Q)}{\ell(Q)}\right)^2
+
\beta_2(5Q)^2
\right).
\end{align*}
Since $Q^\Gamma\subset(5Q)^\Gamma$ and $P_Q$ minimizes
$\beta_2(5Q)$,
\[
\int_{Q^\Gamma}
\left(
\frac{\dist((X,t),P_Q)}{\ell(Q)}
\right)^2
\,\d\mu(X,t)
\lesssim
\beta_2(5Q)^2|Q|.
\]
Integrating the preceding pointwise estimate proves the lemma.
\end{proof}

\begin{lemma}\label{bigbound}
For all $Q\in\mathbb D(\mathbb P^k)$ with $\ell(Q)=2^{-j}$, there holds
\[
\int_{Q^\Gamma} |\mathcal R_{(j)}1|^2\,\d\mu
\lesssim
\big[\alpha(Q)^2+\beta_2(5Q)^2\big]|Q|.
\]
\end{lemma}

\begin{proof}
We obtain this estimate by using Lemma \ref{b1abd} in the argument of
Lemma \ref{commapprox}. See also the proof of Lemma \ref{TQintbd} in Section \ref{last}.
\end{proof}

\begin{proposition}\label{errorest}
For $j \in \mathbb{Z}$, set
\begin{align*}
\beta_{2,j}(\Gamma)^2 = \sum_{Q \in \mathbb{D}_j} \beta_2(5Q)^2|Q|,\\
\alpha_j(\Gamma)^2 = \sum_{Q \in \mathbb{D}_j} \alpha(Q)^2 |Q|.
\end{align*}
We have the estimate
\begin{align*}
|\langle\mathcal R_{(j)}1, \mathcal R_{(\ell)}1 \rangle-\langle \mathcal C_{(j)} 1, \mathcal C_{(\ell)} 1\rangle| \lesssim \epsilon(\beta_{2,j}(\Gamma) + \alpha_j(\Gamma))(\beta_{2,\ell}(\Gamma) + \alpha_\ell(\Gamma)).
\end{align*}
\end{proposition}
\begin{proof}
First we use the triangle inequality
\begin{align*}
\left|
\left\langle\mathcal R_{(j)}1,\mathcal R_{(\ell)}1\right\rangle
-
\left\langle\mathcal C_{(j)}1,\mathcal C_{(\ell)}1\right\rangle
\right|
\leq
\left|
\left\langle \mathcal R_{(j)}1,\mathcal R_{(\ell)}1 - \mathcal C_{(\ell)}1\right\rangle
\right|
+
\left|
\left\langle \mathcal R_{(j)}1-\mathcal C_{(j)}1,
\mathcal C_{(\ell)}1\right\rangle
\right|.
\end{align*}
We now apply H\"older's inequality, Lemma \ref{commapprox}, and Lemma \ref{bigbound}. More precisely,
Lemma \ref{commapprox} gives
$\|\mathcal R_{(j)}1-\mathcal C_{(j)}1\|_2
\lesssim \epsilon\big(\beta_{2,j}(\Gamma)+\alpha_j(\Gamma)\big)$,
and Lemma \ref{bigbound} then gives the desired bound.
\end{proof}
\begin{proposition}\label{fourierest}
Assume that $A\in C_c^\infty(\mathbb P^k;\mathbb R^{n-k})$.
For each $1\leq m\leq n-k$ and each $j,\ell\in\mathbb Z$, we have
\[
\left\langle \mathcal C^m_{(j)}1,\mathcal C^m_{(\ell)}1\right\rangle
\geq 0.
\]
Moreover,
\[
\sum_{j\in\mathbb Z}
\|\mathcal C^m_{(j)}1\|_2^2
\gtrsim
\int_{\mathbb R^{k+1}}
|\widehat{A_m}(\xi,\tau)|^2
\|(\xi,\tau)\|^2
\,\d\xi\,\d\tau.
\]
Equivalently,
\[
\sum_{j\in\mathbb Z}
\|\mathcal C^m_{(j)}1\|_2^2
\gtrsim
\|D_{par}A_m\|_2^2.
\]
\end{proposition}

\begin{proof}
Set
\[
\eta_j(x,t)
:=
\frac{\varphi_j(x,t)}{|t|^{(k+1)/2+1}}
\exp\left(-\frac{|x|^2}{4|t|}\right).
\]
Then
\[
\mathcal C^m_{(j)}1(x,t)
=
\int_{\mathbb P^k}
\big(A_m(x,t)-A_m(y,s)\big)
\eta_j(x-y,t-s)\,\d y\,\d s.
\]
Equivalently, after the change of variables $(p,u)=(x-y,t-s)$,
\[
\mathcal C^m_{(j)}1(x,t)
=
\int_{\mathbb P^k}
\big(A_m(x,t)-A_m(x-p,t-u)\big)
\eta_j(p,u)\,\d p\,\d u.
\]

Let $\delta_j:=2^{-j}$. By the definition of $\varphi_j$,\[\varphi_j(x,t)=\varphi_0\left(\frac{x}{\delta_j},\frac{t}{\delta_j^2}\right).\]Define\[\eta(x,t):=\begin{cases}\dfrac{\varphi_0(x,t)}{|t|^{(k+1)/2+1}}\exp\left(-\dfrac{|x|^2}{4|t|}\right), & t\ne 0,\\0, & t=0.\end{cases}\] Then\[\eta_j(x,t)=\delta_j^{-(k+3)}\eta\left(\frac{x}{\delta_j},\frac{t}{\delta_j^2}\right).\]Consequently,\[\widehat{\eta_j}(\xi,\tau)=\delta_j^{-1}\widehat{\eta}(\delta_j\xi,\delta_j^2\tau).\]Taking the Fourier transform in $(x,t)$ gives
\[
\widehat{\mathcal C^m_{(j)}1}(\xi,\tau)
=
\widehat{A_m}(\xi,\tau)
\left(
\widehat{\eta_j}(0,0)-\widehat{\eta_j}(\xi,\tau)
\right).
\]
Thus
\[
\widehat{\mathcal C^m_{(j)}1}(\xi,\tau)
=
\delta_j^{-1}\widehat{A_m}(\xi,\tau)
\left(
\widehat{\eta}(0,0)
-
\widehat{\eta}(\delta_j\xi,\delta_j^2\tau)
\right).
\]

Define
\[
g(\xi,\tau)
:=
\widehat{\eta}(0,0)-\widehat{\eta}(\xi,\tau).
\]
Then

\begin{equation}\label{commutatorfourieridentity}
\widehat{\mathcal C^m_{(j)}1}(\xi,\tau)
=
\delta_j^{-1}\widehat{A_m}(\xi,\tau)
g(\delta_j\xi,\delta_j^2\tau).
\end{equation}

Since $\eta$ is even in both the spatial and time variables, $\widehat\eta$ is real-valued, and hence $g$ is real-valued. Moreover,
\[
g(\xi,\tau)
=
\int_{\mathbb P^k}
\bigl(1-\cos(2\pi(x\cdot \xi+t\tau))\bigr)\eta(x,t)\,\d x\,\d t.
\]
The integrand is non-negative, so $g(\xi,\tau)\geq 0$ for all $(\xi,\tau)$. Therefore, by Plancherel and \eqref{commutatorfourieridentity},
\begin{align*}
\left\langle \mathcal C^m_{(j)}1,\mathcal C^m_{(\ell)}1\right\rangle
&=
\int_{\mathbb R^{k+1}}
|\widehat{A_m}(\xi,\tau)|^2
\delta_j^{-1}\delta_\ell^{-1}
g(\delta_j\xi,\delta_j^2\tau)
g(\delta_\ell\xi,\delta_\ell^2\tau)
\,\d\xi\,\d\tau\\
&\geq 0.
\end{align*}

We also have
\[
\|\mathcal C^m_{(j)}1\|_2^2
=
\int_{\mathbb R^{k+1}}
|\widehat{A_m}(\xi,\tau)|^2
\delta_j^{-2}
\left|
g(\delta_j\xi,\delta_j^2\tau)
\right|^2
\,\d\xi\,\d\tau.
\]
Summing in $j$, we obtain
\[
\sum_{j\in\mathbb Z}
\|\mathcal C^m_{(j)}1\|_2^2
=
\int_{\mathbb R^{k+1}}
|\widehat{A_m}(\xi,\tau)|^2
M(\xi,\tau)
\,\d\xi\,\d\tau,
\]
where
\[
M(\xi,\tau)
:=
\sum_{j\in\mathbb Z}
\delta_j^{-2}
\left|
g(\delta_j\xi,\delta_j^2\tau)
\right|^2.
\]

We claim that
\[
M(\xi,\tau)\gtrsim \|(\xi,\tau)\|^2
\qquad
\text{for every }(\xi,\tau)\neq(0,0).
\]

We need to show that $g$ is strictly positive away from the origin. If $g(\xi,\tau)=0$, then, since the integrand defining $g$ is non-negative,
\[
1-\cos(2\pi(x\cdot \xi+t\tau))=0
\]
for almost every $(x,t)$ in the nonempty open set $\{\eta>0\}$. Hence $x\cdot\xi+t\tau\in\mathbb Z$ for almost every $(x,t)\in\{\eta>0\}$. If $(\xi,\tau)\neq(0,0)$, however, the set where $x\cdot\xi+t\tau\in\mathbb Z$ is a countable union of hyperplanes and therefore has Lebesgue measure zero, a contradiction. Thus $(\xi,\tau)=(0,0)$.  It follows that $g(\xi,\tau)>0$ whenever $(\xi,\tau)\neq(0,0)$.

Now fix $(\xi,\tau)\neq(0,0)$ and choose $j_0\in\mathbb Z$ so that $1\leq\|(\delta_{j_0}\xi,\delta_{j_0}^2\tau)\|\leq2$. Since $g$ is continuous and strictly positive on the annulus $\{(\zeta,\sigma):1\leq\|(\zeta,\sigma)\|\leq2\}$, we have $|g(\delta_{j_0}\xi,\delta_{j_0}^2\tau)|\gtrsim1$; moreover, $\delta_{j_0}^{-1}\approx\|(\xi,\tau)\|$. Therefore\[M(\xi,\tau)\geq\delta_{j_0}^{-2}|g(\delta_{j_0}\xi,\delta_{j_0}^2\tau)|^2\gtrsim\delta_{j_0}^{-2}\approx\|(\xi,\tau)\|^2.\]Consequently,\[\sum_{j\in\mathbb Z}\|\mathcal C^m_{(j)}1\|_2^2\gtrsim\int_{\mathbb R^{k+1}}|\widehat{A_m}(\xi,\tau)|^2\|(\xi,\tau)\|^2\,\d\xi\,\d\tau.\]This proves the proposition.
\end{proof}

\begin{definition}\label{DMD}
Given $g:\mathbb P^k\to\mathbb R$ and $I\in\mathbb D(\mathbb P^k)$, define
\[
\Delta_I g
:=
\sum_{\substack{P\subseteq I,\ P\in\mathbb D(\mathbb P^k)\\ \ell(P)=2^{-1}\ell(I)}}
\mathbf 1_P\,\langle g\rangle_P
-
\mathbf 1_I\,\langle g\rangle_I.
\]
Here
\[
\langle g\rangle_E
:=
\frac{1}{|E|}\int_E g(y,s)\,\d y\,\d s
\]
for every parabolic dyadic cube $E$.
\end{definition}

We next bound the $\alpha$-number sum by the corresponding $\beta$-number sum for a Lip-$(1,1/2)$ graph.

\begin{proposition}\label{alphabetagraph}
We have
\[
\sum_{Q\in\mathbb D(\mathbb P^k)}\alpha(Q)^2|Q|
\lesssim
\sum_{Q\in\mathbb D(\mathbb P^k)}\beta_1(5Q)^2|Q|.
\]
\end{proposition}

\begin{proof}

Fix \(Q\in\mathbb D(\mathbb P^k)\), and let \(L\) be a vertical plane minimizing \(\beta_1(5Q)\). Because \(\|A\|_{lip}\leq\epsilon\), \(L\) has uniformly bounded slope. So, we can write
\[
L=\{(x,a(x),t):(x,t)\in\mathbb P^k\},
\]
where \(a:\mathbb R^k\to\mathbb R^{n-k}\) is affine. Set
\[
j_a:=\sqrt{\det\!\left(I_k+(\nabla a)^T\nabla a\right)},
\qquad
c:=j_a^{-1}.
\]
Let \(f\) be supported in \(C_Q\) and have Lipschitz constant at most \(1\). The projection of \(C_Q\) onto \(\mathbb P^k\) is contained in \(5Q\). Since \(\mu=\Phi_\#(dx\,dt)\), while \(d\sigma_L=j_a\,dx\,dt\) under the parametrization \((x,t)\mapsto(x,a(x),t)\), we have
\begin{align*}
&\left|
\int f(X,t)\,\d\mu(X,t)
-
c\int f(X,t)\,\d\sigma_L(X,t)
\right|\\
&\quad=
\left|
\int_{5Q}
\big[
f(x,A(x,t),t)-f(x,a(x),t)
\big]
\,\d x\,\d t
\right|\\
&\quad\lesssim
\int_{5Q}
d_p\big((x,A(x,t),t),L\big)
\,\d x\,\d t\\
&\quad\lesssim
\beta_1(5Q)\ell(Q)^{k+3}.
\end{align*}
Here the first inequality uses the Lipschitz bound for \(f\) and the uniformly bounded slope of \(L\), while the second follows from the definition of \(\beta_1(5Q)\) and
\[
\mu\big(\pi^{-1}(5Q)\cap\Gamma\big)
=
|5Q|
\approx
\ell(Q)^{k+2}.
\]
Thus
\[
\alpha(Q)\lesssim\beta_1(5Q).
\]
Squaring this inequality, multiplying by \(|Q|\), and summing over \(Q\in\mathbb D(\mathbb P^k)\) proves the proposition.

\end{proof}

\begin{proposition}\label{offdiagbeta}
For each integer $N_0>0$, we have
\[
\sum_{i,j:\ |i-j|>N_0}
\big|\langle \mathcal R_{(i)}1,\mathcal R_{(j)}1\rangle
\big|
\lesssim
2^{-cN_0}
\sum_{Q\in\mathbb D(\mathbb P^k)}
\beta_2(5Q)^2|Q|.
\]
\end{proposition}

\begin{proof}
\begingroup
Let $\mathbb D_0=\mathbb D(\mathbb P^k)$. We construct a dyadic grid of $\Gamma$ as follows.  Fix $\Delta\geq1$ so that, when
$Q\in(\mathbb D_0)_q$, the lifted cube $Q^\Gamma$ is assigned to
generation $q-\Delta$. With this indexing,
\[
\widetilde\ell(Q^\Gamma)
=
2^\Delta\ell(Q),
\]
and $\{Q^\Gamma:Q\in\mathbb D_0\}$ satisfies the dyadic-grid
requirements of Section~\ref{last}. We write
$\widetilde\alpha(Q^\Gamma)$ for the $\alpha$-number defined in Section~\ref{last}
and $\widetilde C_{Q^\Gamma}$ for its testing cylinder.

Relabeling the annular pieces by the same fixed shift does not change
$|i-j|$. Proposition~\ref{offdiagprop} therefore gives
\[
\sum_{|i-j|>N_0}
\bigl|\langle\mathcal R_{(i)}1,\mathcal R_{(j)}1\rangle\bigr|
\lesssim
2^{-cN_0}
\sum_{Q\in\mathbb D_0}
\widetilde\alpha(Q^\Gamma)^2\mu(Q^\Gamma).
\]

Choose an integer $M>\Delta$, fixed independently of $Q$, as follows. For
$Q\in\mathbb D_0$, let $Q^\sharp$ be its ancestor with
\[
\ell(Q^\sharp)=2^M\ell(Q).
\]
Since
\[
\widetilde\ell(Q^\Gamma)
=
2^\Delta\ell(Q),
\]
we may fix $M$ large enough that
\[
\widetilde C_{Q^\Gamma}
\subset
C_{Q^\sharp}
\qquad
\text{for every }Q\in\mathbb D_0.
\]
Here $C_{Q^\sharp}$ is the ambient cube used in Section~\ref{graphregsection}
to define $\alpha(Q^\sharp)$. This containment and the fact that
$\ell(Q^\sharp)/\widetilde\ell(Q^\Gamma)$ is constant give
\[
\widetilde\alpha(Q^\Gamma)
\lesssim
\alpha(Q^\sharp).
\]
Moreover, each $Q^\sharp$ has a fixed number of descendants $Q$ with
$\ell(Q)=2^{-M}\ell(Q^\sharp)$, and the sum of their measures is
$|Q^\sharp|$. Since $\mu(Q^\Gamma)=|Q|$,
\[
\sum_{Q\in\mathbb D_0}
\widetilde\alpha(Q^\Gamma)^2\mu(Q^\Gamma)
\lesssim
\sum_{R\in\mathbb D_0}
\alpha(R)^2|R|.
\]
Proposition~\ref{alphabetagraph} and $\beta_1(5R)\leq\beta_2(5R)$
complete the proof.
\endgroup
\end{proof}

We now combine these estimates to prove Theorem \ref{localdpara}.
\begin{proof}[Proof of Theorem \ref{localdpara}]
For the moment, we assume that $A$ is smooth and compactly supported. Then each component
$A_m$ belongs to $C_c^\infty(\mathbb P^k)$, and hence
$A_m,\nabla_x A_m,\partial_tA_m\in L^2(\mathbb P^k)$. Since
\[
\|(\xi,\tau)\|^2\lesssim |\xi|^2+|\tau|
\qquad\text{and}\qquad
|\tau|\leq 1+\tau^2,
\]
Plancherel gives
\[
\int_{\mathbb R^{k+1}}
|\widehat A_m(\xi,\tau)|^2\|(\xi,\tau)\|^2\,\d\xi\,\d\tau
\lesssim
\|\nabla_x A_m\|_2^2+\|A_m\|_2^2+\|\partial_tA_m\|_2^2<\infty.
\]
Summing over $m=1,\dots,n-k$, we obtain
\[
\|D_{par}A\|_2^2
=
\int_{\mathbb R^{k+1}}
|\widehat A(\xi,\tau)|^2\|(\xi,\tau)\|^2\,\d \xi\,\d\tau
<\infty.
\]

Fix an integer $\ell$. We will take $\ell$ large in this proof. Write
\begin{equation}\label{initdecom}
\|\mathcal R_\ell1\|_2^2
=
\sum_{\substack{i,j:\ |i-j|\leq N_0\\ i,j\leq\ell}}
\left\langle \mathcal R_{(i)}1,\mathcal R_{(j)}1\right\rangle
+
\sum_{\substack{i,j:\ |i-j|>N_0\\ i,j\leq\ell}}
\left\langle \mathcal R_{(i)}1,\mathcal R_{(j)}1\right\rangle
=:T_1+T_2.
\end{equation}
Proposition~\ref{offdiagbeta} gives
\[
|T_2|
\lesssim
2^{-cN_0}
\sum_{Q\in\mathbb D(\mathbb P^k)}
\beta_2(5Q)^2|Q|.
\]

To obtain a lower bound for $T_1$, we write
\[
\left\langle \mathcal R_{(i)}1,\mathcal R_{(j)}1\right\rangle
=
\left\langle \mathcal C_{(i)}1,\mathcal C_{(j)}1\right\rangle
+
\left(
\left\langle \mathcal R_{(i)}1,\mathcal R_{(j)}1\right\rangle
-
\left\langle \mathcal C_{(i)}1,\mathcal C_{(j)}1\right\rangle
\right).
\]
By Proposition \ref{fourierest},
\[
\left\langle \mathcal C_{(i)}1,\mathcal C_{(j)}1\right\rangle\geq 0
\qquad\text{for all }i,j\in\mathbb Z.
\]
Therefore
\[
\sum_{\substack{i,j:\ |i-j|\leq N_0\\ i,j\leq\ell}}
\left\langle \mathcal C_{(i)}1,\mathcal C_{(j)}1\right\rangle
\geq
\sum_{j\leq\ell}
\|\mathcal C_{(j)}1\|_2^2.
\]
Thus, by Proposition \ref{fourierest},
\[
\liminf_{\ell\to\infty}
\sum_{\substack{i,j:\ |i-j|\leq N_0\\ i,j\leq\ell}}
\left\langle \mathcal C_{(i)}1,\mathcal C_{(j)}1\right\rangle
\gtrsim
\|D_{par}A\|_2^2.
\]

On the other hand, by Propositions \ref{errorest} and \ref{alphabetagraph}, we have, uniformly in $\ell$,
\[
\sum_{\substack{i,j:\ |i-j|\leq N_0\\ i,j\leq\ell}}
\left|
\left\langle \mathcal R_{(i)}1,\mathcal R_{(j)}1\right\rangle
-
\left\langle \mathcal C_{(i)}1,\mathcal C_{(j)}1\right\rangle
\right|
\lesssim
N_0\epsilon
\sum_{Q\in\mathbb D(\mathbb P^k)}
\beta_2(5Q)^2|Q|.
\]
Applying Lemma~\ref{vectorDorronsoro} with $F=A$, the sum on the
right-hand side is bounded by $N_0\epsilon\|D_{par}A\|_2^2$. Thus there are constants \(c_0,C_1>0\), independent of \(N_0\) and \(\epsilon\), such that
\[
\liminf_{\ell\to\infty} T_1
\geq c_0
\|D_{par}A\|_2^2
-
C_1N_0\epsilon
\|D_{par}A\|_2^2.
\]
Applying Lemma~\ref{vectorDorronsoro} to the estimate for $T_2$ gives a constant $C_2>0$, independent of $N_0$, such that
\[
\sup_{\ell\in\mathbb Z}|T_2|
\leq
C_2 2^{-cN_0}\|D_{par}A\|_2^2.
\]
Choose $N_0$ so large that $C_2 2^{-cN_0}\leq c_0/4$. Then, having chosen $N_0$, choose $\epsilon>0$ so that $C_1N_0\epsilon\leq c_0/4$. It follows from \eqref{initdecom} that
\[
\|D_{par}A\|_2^2
\lesssim
\limsup_{\ell\to\infty}
\|\mathcal R_\ell1\|_2^2.
\]

It remains to bound the right-hand side of the display above. For each $\ell$,
\[
\|\mathcal R_\ell1\|_2^2
\lesssim
\|\mathcal R_\ell\chi_{2C'_R}\|_2^2
+
\|\mathcal R_\ell(1-\chi_{2C'_R})\|_2^2.
\]
The first term satisfies
\[
\|\mathcal R_\ell\chi_{2C'_R}\|_2^2
\lesssim
\|\mathcal R\|^2 R^{k+2}.
\]
For the second term, we use that $A$ is supported in $C'_R$. If $(x,t)\in C'_R$, then
\[
|\mathcal R_\ell(1-\chi_{2C'_R})(x,t)|
\lesssim
R\int_{\|(x,t)-(y,s)\|\gtrsim R}
\frac{1}{\|(x,t)-(y,s)\|^{k+3}}\,\d y\,\d s
\lesssim 1.
\]
If $(x,t)\notin C'_R$, then $\mathcal R_\ell(1-\chi_{2C'_R})(x,t)=0$. Hence
\[
\|\mathcal R_\ell(1-\chi_{2C'_R})\|_2^2
\lesssim
R^{k+2}.
\]
Therefore,
\begin{equation}\label{e:dparsmootha}
\|D_{par}A\|_2^2
\lesssim
\left(1+\|\mathcal R\|^2\right)R^{k+2}.
\end{equation}

Now suppose that $A$ is not necessarily smooth. Let $\eta_\lambda$ be as in Subsection~\ref{smoothing}, and set $B_\lambda=A*\eta_\lambda$. This differs from the function denoted by $B_\lambda$ in that subsection by the constant vector $2C_A\lambda{\bf 1}$.
Adding a constant vector to a graph function leaves the associated kernel unchanged, and $D_{par}$ annihilates constants, so the estimates obtained from Proposition~\ref{Blabounded} are unchanged.
For $0<\lambda\leq R$, we have $\supp(A*\eta_\lambda)\subset C'_{R+\lambda}\subset C'_{2R}$. Hence the argument above applies to $A*\eta_\lambda$ with $C'_R$ replaced by $C'_{2R}$. Thus, by Proposition \ref{Blabounded} and \eqref{e:dparsmootha}, we have
\[
\|D_{par}B_\lambda\|_2^2
\lesssim
\left(1+\|\mathcal R\|^2\right)R^{k+2},
\]
with an implicit constant independent of $0<\lambda\leq R$.

Finally, we pass to the limit as $\lambda\to0$. We have
\[
D_{par}B_\lambda
=
D_{par}(A*\eta_\lambda).
\]
On the Fourier transform side,
\[
\widehat{A*\eta_\lambda}(\xi,\tau)
=
\widehat A(\xi,\tau)\widehat\eta(\lambda\xi,\lambda^2\tau).
\]
Since $\widehat\eta(0,0)=1$, we have
\[
\widehat{A*\eta_\lambda}(\xi,\tau)
\to
\widehat A(\xi,\tau)
\qquad\text{for every }(\xi,\tau).
\]
Thus, by Fatou's lemma,
\[
\|D_{par}A\|_2^2
=
\int_{\mathbb R^{k+1}}
|\widehat A(\xi,\tau)|^2\|(\xi,\tau)\|^2\,\d\xi\,\d\tau
\leq
\liminf_{\lambda\to0}
\int_{\mathbb R^{k+1}}
|\widehat{A*\eta_\lambda}(\xi,\tau)|^2\|(\xi,\tau)\|^2\,\d\xi\,\d\tau.
\]
Therefore,
\[
\|D_{par}A\|_2^2
\leq
\liminf_{\lambda\to0}
\|D_{par}B_\lambda\|_2^2
\lesssim
\left(1+\|\mathcal R\|^2\right)R^{k+2}.
\]
This proves the theorem.
\end{proof}

\subsection{Localization}\label{localization}

In this subsection we prove the following proposition.

\begin{proposition}\label{localsios}
\begingroup
Suppose $A:\mathbb P^k\to\mathbb R^{n-k}$ is Lip-$(1,1/2)$,
$\|A\|_{lip}\leq\epsilon\ll1$, and
$\|\mathcal R\|<\infty$. For every $(x_0,t_0)\in\mathbb P^k$ and $R>0$,
there is a function
$A_{(x_0,t_0),R}:\mathbb P^k\to\mathbb R^{n-k}$ such that
$A_{(x_0,t_0),R}=A$ on $C'_{4R}(x_0,t_0)$,
$\supp\bigl(A_{(x_0,t_0),R}-A(x_0,t_0)\bigr)
\subset C'_{101R}(x_0,t_0)$, and
\[
\|A_{(x_0,t_0),R}\|_{lip}
\lesssim
\|A\|_{lip},
\qquad
\|\mathcal R^{A_{(x_0,t_0),R}}\|
\lesssim
1+\|\mathcal R\|.
\]
\endgroup
\end{proposition}

Fix a point $(x_0,t_0)\in\mathbb P^k$ and a scale $R>0$. We will localize the
function $A$ to this point and scale. By translating and subtracting the constant
$A(x_0,t_0)$, we may assume that $(x_0,t_0)=(0,0)$ and $A(0,0)=0$.

Let $\eta:\mathbb P^k\to\mathbb R$ be equal to $1$ on
$\overline{C'_{4R}(0,0)}$, equal to $0$ outside of $C'_{25R}(0,0)$, such that
$0\leq \eta\leq 1$, and so that $\|\eta\|_{lip}\lesssim R^{-1}$. Set
\[
    \widetilde A(x,t)=A(x,t)\eta(x,t).
\]
A standard argument shows $\|\widetilde A\|_{lip}\lesssim\|A\|_{lip}$. Also, by the definition of $\widetilde A$, for every $(x,t)\in\mathbb P^k$,
\[
    |\widetilde A(x,t)-A(x,t)|
    =
    |A(x,t)(\eta(x,t)-1)|.
\]
If $(x,t)\in\overline{C'_{4R}(0,0)}$, the expression above is equal to $0$.
Otherwise, set $(x^*,t^*)$ to be a point of $\overline{C'_{4R}(0,0)}$ closest to
$(x,t)$. Then
\[
    |\widetilde A(x,t)-A(x,t)|
    =
    |A(x,t)(\eta(x,t)-\eta(x^*,t^*))|.
\]
Using the triangle inequality, we bound this by
\[
\begin{aligned}
&|(A(x,t)-A(x^*,t^*))(\eta(x,t)-\eta(x^*,t^*))|  \\
&\qquad
+
|A(x^*,t^*)(\eta(x,t)-\eta(x^*,t^*))|,
\end{aligned}
\]
which is bounded by
\[
    C\|A\|_{lip}\,d_p\big((x,t),\overline{C'_{4R}(0,0)}\big),
\]
where $C$ depends only on the fixed bound for $R\|\eta\|_{lip}$ and is independent of $R$.
Thus, setting
\[
    D(x,t):=d_p\big((x,t),\overline{C'_{4R}(0,0)}\big),
\]
we have
\begin{align}\label{graphclose}
    |\widetilde A(x,t)-A(x,t)|
    \lesssim
    \|A\|_{lip}D(x,t).
\end{align}

By Lemma~3.24 of \cite{BHHLN2}, there is a regularized distance function
\[
    H:\mathbb P^k\to[0,\infty)
\]
with the following properties:
$\|H\|_{lip}\lesssim 1$,  $H=0$ on $\overline{C'_{4R}(0,0)}$, $H>0$ on
$\mathbb P^k\setminus\overline{C'_{4R}(0,0)}$, and, on
$\mathbb P^k\setminus\overline{C'_{4R}(0,0)}$,
\begin{align*}
    c_nD(x,t)\leq H(x,t)&\leq \frac12D(x,t),\\
    \|D_t^{1/2}H\|_{BMO(\mathbb P^k)}&\lesssim 1,\\
    |\nabla_x^{|\alpha|}\partial_t^\beta H(x,t)|
    &\lesssim
    D(x,t)^{1-|\alpha|-2\beta},
    \qquad |\alpha|+2\beta\geq 1.
\end{align*}

We also need estimates for the second spatial derivatives outside
$\overline{C'_{4R}(0,0)}$. To obtain these estimates, we smooth
$\widetilde A$ as follows. 
Let $\theta\in C_c^\infty(\mathbb P^k)$ be nonnegative, with
\[
\int_{\mathbb P^k}\theta=1,
\qquad
\supp\theta
\subset
\bigl\{(x,t)\in\mathbb P^k:\|(x,t)\|\leq1\bigr\}.
\]
 For $\omega>0$, write
\[
    \theta_\omega(x,t)
    =
    \omega^{-(k+2)}
    \theta\left(\frac{x}{\omega},\frac{t}{\omega^2}\right).
\]
Define
\[
    \overline A(x,t):=\widetilde A(x,t),
    \qquad
    (x,t)\in\overline{C'_{4R}(0,0)},
\]
and
\[
    \overline A(x,t)
    :=
    \widetilde A*\theta_{H(x,t)}(x,t),
    \qquad
    (x,t)\notin\overline{C'_{4R}(0,0)}.
\]
\begingroup
We set
\[
    A_{(0,0),R}:=\overline A.
\]
By construction, $A_{(0,0),R}=A$ on $C'_{4R}(0,0)$. We also have
\[
    \supp(A_{(0,0),R})
    \subset
    \overline{C'_{100R}(0,0)}
    \subset
    C'_{101R}(0,0).
\]
Indeed, suppose that $(x,t)\notin\overline{C'_{4R}(0,0)}$ and $\overline A(x,t)\neq0$. Then there is a point $(y,s)\in\supp(\theta_{H(x,t)})$ such that
$((x,t)-(y,s))\in\supp(\widetilde A)\subset\overline{C'_{25R}(0,0)}$. Therefore,
\[
\begin{aligned}
    \|(x,t)\|
    &\leq
    \|(x,t)-(y,s)\|+\|(y,s)\| \\
    &\leq
    50R+H(x,t)
    \leq
    50R+\frac12D(x,t)
    \leq
    50R+\frac12\|(x,t)\|.
\end{aligned}
\]
Here the last inequality follows because $(0,0)\in\overline{C'_{4R}(0,0)}$. Hence $\|(x,t)\|\leq100R$.
\endgroup

When $(x,t)\in\overline{C'_{4R}(0,0)}$, we have
$\overline A(x,t)=\widetilde A(x,t)=A(x,t)$. Otherwise, using that all $(y,s)$ in the
integral below satisfy $\|(y,s)\|\lesssim H(x,t)\approx D(x,t)$, we get
\[
\begin{aligned}
    |\widetilde A(x,t)-\overline A(x,t)|
    &\leq
    \int
    |\widetilde A((x,t)-(y,s))-\widetilde A(x,t)|
    \theta_{H(x,t)}(y,s)\,\d y\,\d s  \\
    &\lesssim
    \|A\|_{lip}H(x,t)
    \lesssim
    \|A\|_{lip}D(x,t).
\end{aligned}
\]
Combining this with \eqref{graphclose}, we obtain
\begin{align}\label{Abarclose}
    |\overline A(x,t)-A(x,t)|
    \lesssim
    \|A\|_{lip}D(x,t).
\end{align}
\begingroup
Fix $(x,t)\notin\overline{C'_{4R}(0,0)}$ and put $h=H(x,t)$. For
$i\in\{1,\ldots,k\}$, the chain rule gives
\[
\partial_{x_i}\overline A(x,t)
=
(\partial_{x_i}\widetilde A*\theta_h)(x,t)
+
(\partial_{x_i}H)(x,t)(\widetilde A*\partial_h\theta_h)(x,t).
\]
Since
$\int\partial_h\theta_h=0$, parabolic scaling and the Lip-$(1,1/2)$
bound for $\widetilde A$ give
\[
|\widetilde A*\partial_h\theta_h|(x,t)
\lesssim
\|A\|_{lip}h\|\partial_h\theta_h\|_1
\lesssim
\|A\|_{lip}.
\]
Together with $|\nabla_xH|\lesssim1$, this proves
\begin{equation}\label{Abarspatialgrad}
|\nabla_x\overline A(x,t)|\lesssim\|A\|_{lip}.
\end{equation}

To estimate the second spatial derivatives, fix
$i,j\in\{1,\ldots,k\}$. We differentiate once more and move one
spatial derivative from $\widetilde A$ to $\theta_h$, obtaining
\[
\begin{aligned}
\partial_{x_j}\partial_{x_i}\overline A
={}&
(\partial_{x_i}\widetilde A*\partial_{x_j}\theta_h)
+(\partial_{x_j}H)(\partial_{x_i}\widetilde A*\partial_h\theta_h)
+(\partial_{x_i}H)(\partial_{x_j}\widetilde A*\partial_h\theta_h)\\
&+(\partial_{x_j}\partial_{x_i}H)(\widetilde A*\partial_h\theta_h)
+(\partial_{x_i}H)(\partial_{x_j}H)
 (\widetilde A*\partial_h^2\theta_h).
\end{aligned}
\]
To estimate the terms on the right-hand side, we use
\[
\|\nabla_x\widetilde A\|_\infty\lesssim\|A\|_{lip},
\qquad
|\nabla_xH(x,t)|\lesssim1,
\qquad
|\nabla_x^2H(x,t)|\lesssim D(x,t)^{-1},
\]
together with the estimates
\[
\|\nabla_x\theta_h\|_1+\|\partial_h\theta_h\|_1
\lesssim h^{-1},
\qquad
\|\partial_h^2\theta_h\|_1\lesssim h^{-2},
\]
and the mean value zero property of $\partial_h \theta_h$ and $\partial_h^2\theta_h$:
\[
\int\partial_h\theta_h=0,
\qquad
\int\partial_h^2\theta_h=0.
\]
Since $h\approx D(x,t)$, these estimates bound each term in the preceding
identity by $\|A\|_{lip}/D(x,t)$. Therefore,
\begin{align}\label{2ndgrad}
|\nabla_x^2\overline A(x,t)|
\lesssim
\frac{\|A\|_{lip}}{D(x,t)}, \quad (x,t) \notin\overline{C'_{4R}(0,0)}.
\end{align}
\endgroup

\begingroup
We estimate the time derivative in the same way. For fixed $h>0$,
we differentiate $\widetilde A*\theta_h$ and then apply the chain rule with
$h=H(x,t)$ to obtain
\[
\partial_t\overline A(x,t)
=
(\widetilde A*\partial_s\theta_h)(x,t)
+
\partial_tH(x,t)
(\widetilde A*\partial_h\theta_h)(x,t),
\qquad h=H(x,t).
\]
Using $|\partial_tH(x,t)|\lesssim D(x,t)^{-1}$ and
$H(x,t)\approx D(x,t)$, the Lip-$(1,1/2)$ bound for $\widetilde A$, and the $L^1$ estimates and vanishing integrals of $\partial_s\theta_h$ and $\partial_h\theta_h$, we conclude that
\begin{align}\label{timegrad}
|\partial_t\overline A(x,t)|
\lesssim
\frac{\|A\|_{lip}}{D(x,t)},
\qquad
(x,t)\notin\overline{C'_{4R}(0,0)}.
\end{align}
\endgroup

\begingroup
We now verify that $\overline A$ is Lip-$(1,1/2)$. Estimate \eqref{Abarspatialgrad} gives
\[
|\overline A(x,t)-\overline A(y,t)|
\lesssim
\|A\|_{lip}|x-y|.
\]
For $t,s\in\mathbb R$, set $r=|t-s|^{1/2}$. If
\[
\min\{D(x,t),D(x,s)\}\leq 2r,
\]
then the $1$-Lipschitz property of $D$ gives
$\max\{D(x,t),D(x,s)\}\leq 3r$. Hence \eqref{Abarclose} and the
Lip-$(1,1/2)$ bound for $A$ give
\[
|\overline A(x,t)-\overline A(x,s)|
\lesssim
\|A\|_{lip}r.
\]
If instead $\min\{D(x,t),D(x,s)\}>2r$, then, for every $\tau$ between
$t$ and $s$,
\[
D(x,\tau)\geq \min\{D(x,t),D(x,s)\}-r
\geq \frac12\min\{D(x,t),D(x,s)\}.
\]
The segment joining $(x,t)$ to $(x,s)$ therefore lies outside
$\overline{C'_{4R}(0,0)}$, and \eqref{timegrad} yields
\[
|\overline A(x,t)-\overline A(x,s)|
\lesssim
\|A\|_{lip}\frac{|t-s|}{\min\{D(x,t),D(x,s)\}}
\lesssim
\|A\|_{lip}r.
\]
Combining the spatial and temporal estimates, we obtain
\[
\|\overline A\|_{lip}\lesssim\|A\|_{lip}.
\]
\endgroup

Define the function $B$ by
\begin{align}\label{Bdef}
    B(x,t)=\overline A(x,t)+H(x,t){\bf 1},
    \qquad
    {\bf 1}:=(1,\ldots,1)\in\mathbb R^{n-k}.
\end{align}

Since $H=0$ on $\overline{C'_{4R}(0,0)}$, we have $B=A$ on
$\overline{C'_{4R}(0,0)}$. Moreover, by \eqref{Abarclose} and $H\approx D$ outside
$\overline{C'_{4R}(0,0)}$,
\[
    |B(x,t)-A(x,t)|
    \lesssim
    D(x,t),
\]
with constants depending only on dimension and $\|A\|_{lip}$.
\begingroup
Recall that $K=(K^1,\ldots,K^{n-k})$ is the vector-valued kernel associated to $A$,
and let
$K^{\overline A}=(K^{\overline A,1},\ldots,K^{\overline A,n-k})$ be the
corresponding kernel associated to $\overline A$, i.e.
$$K^{\overline A}(x,t,y,s) = \frac{\overline A(x,t)-\overline A(y,s)}{|t-s|^{\frac{k+1}{2}+1}}\exp\bigg(-\frac{|x-y|^2 +|\overline A(x,t)-\overline A(y,s)|^2 }{4|t-s|} \bigg).$$
We will prove
\[
\left\|
    \int_{\mathbb P^k}
    K^{\overline A}_\ell(x,t,y,s)f(y,s)\,\d y\,\d s
\right\|_2^2
\lesssim
(1+\|\mathcal R\|^2)\|f\|_2^2
\]
uniformly over $\ell\in\mathbb Z$. This is the Riesz transform bound required in Proposition~\ref{localsios}.
\endgroup

We also define the singular integrals $\mathcal R^{A,B}$, $\mathcal R^{B,A}$, and
$\mathcal R^B$, which are formally associated to the kernels
\[
K^{A,B}(x,t,y,s)
=
\frac{A(x,t)-B(y,s)}{|t-s|^{(k+1)/2+1}}
\exp\left(
-\frac{|x-y|^2+|A(x,t)-B(y,s)|^2}{4|t-s|}
\right),
\]
\[
K^{B,A}(x,t,y,s)
=
\frac{B(x,t)-A(y,s)}{|t-s|^{(k+1)/2+1}}
\exp\left(
-\frac{|x-y|^2+|B(x,t)-A(y,s)|^2}{4|t-s|}
\right),
\]
and
\[
K^B(x,t,y,s)
=
\frac{B(x,t)-B(y,s)}{|t-s|^{(k+1)/2+1}}
\exp\left(
-\frac{|x-y|^2+|B(x,t)-B(y,s)|^2}{4|t-s|}
\right).
\]
For a kernel $U$ as above, we let
\[
    U_\ell(x,t,y,s)
    :=
    U(x,t,y,s)\Psi_\ell\big((x,t)-(y,s)\big),
\]
and if $T$ denotes the operator associated to $U$, we let $T_\ell$ denote the
operator associated to $U_\ell$.

By Lemma~\ref{smoothingCZkernels}, since $A, \overline A, B$ are
Lip-$(1,1/2)$, these kernels are standard parabolic Calder\'on--Zygmund kernels. Once we have proved uniform
$L^2$ bounds for their smooth truncations, we can apply Cotlar's inequality
to the associated singular integral operators.

We record the following estimates, which will be used in the proof below.

\begin{lemma}\label{czestimates}
Let
\[
    \rho:=\|(x,t)-(y,s)\|.
\]
The following estimates hold:
\begin{align}
|K^{B,A}(x,t,y,s) - K(x,t,y,s)|
&\lesssim
\frac{D(x,t)}{\rho^{k+3}},
&&
(x,t)\notin \overline{C'_{4R}(0,0)},\quad
\rho>\frac12D(x,t),
\label{czsmoothpush1}
\\
|K^{B,A}(x,t,y,s)|
&\lesssim
\frac{1}{D(x,t)^{k+2}},
&&
(x,t)\notin \overline{C'_{4R}(0,0)},\quad
\rho\leq \frac12D(x,t),
\label{czsmoothpush2}
\\
|K^B(x,t,y,s) - K^{A,B}(x,t,y,s)|
&\lesssim
\frac{D(x,t)}{\rho^{k+3}},
&&
(x,t)\notin \overline{C'_{4R}(0,0)},\quad
\rho>\frac12D(x,t),
\label{czsmoothpush3}
\\
|K^{A,B}(x,t,y,s)|
&\lesssim
\frac{1}{\rho^{k+2}}.
\label{czsmoothpush4}
\end{align}
\end{lemma}

\begin{proof}
\begingroup
We start with \eqref{czsmoothpush1}. By the construction of $B$ and
\eqref{Abarclose},
\[
\delta:=|B(x,t)-A(x,t)|\lesssim D(x,t),
\qquad
(x,t)\notin\overline{C'_{4R}(0,0)}.
\]
If $\delta\leq\rho/2$, the standard Calder\'on--Zygmund smoothness estimate
in the transversal variable gives
\[
|K^{B,A}(x,t,y,s)-K(x,t,y,s)|
\lesssim
\frac{\delta}{\rho^{k+3}}
\lesssim
\frac{D(x,t)}{\rho^{k+3}}.
\]
If $\delta>\rho/2$, the size estimates for the two kernels give
\[
|K^{B,A}(x,t,y,s)-K(x,t,y,s)|
\lesssim
\frac{1}{\rho^{k+2}}.
\]
In this case $\delta\lesssim D(x,t)$ implies $D(x,t)\gtrsim\rho$, and hence
\[
\frac{1}{\rho^{k+2}}
\lesssim
\frac{D(x,t)}{\rho^{k+3}}.
\]
This proves \eqref{czsmoothpush1}.
\endgroup

We next prove \eqref{czsmoothpush2}. By the definition
\[
    B=\overline A+H{\bf 1},
\]
the estimate \eqref{Abarclose}, and the fact that $H(x,t)\approx D(x,t)$ outside
$\overline{C'_{4R}(0,0)}$, we have, provided $\epsilon=\|A\|_{lip}$ is sufficiently
small,
\[
    |B(x,t)-A(x,t)|\gtrsim D(x,t).
\]
If $\rho\leq D(x,t)/2$, then
\[
|A(x,t)-A(y,s)|
\leq
\|A\|_{lip}\rho
\leq
\epsilon D(x,t)/2.
\]
Hence, again taking $\epsilon$ small enough,
\[
|B(x,t)-A(y,s)|
\geq
|B(x,t)-A(x,t)|-|A(x,t)-A(y,s)|
\gtrsim
D(x,t).
\]
The standard size estimate for the Riesz kernel then gives
\[
|K^{B,A}(x,t,y,s)|
\lesssim
\frac{1}{\rho^{k+2}+|B(x,t)-A(y,s)|^{k+2}}
\lesssim
\frac{1}{D(x,t)^{k+2}},
\]
which proves \eqref{czsmoothpush2}.

\begingroup
We prove \eqref{czsmoothpush3}.  Set $\delta = |B(x,t)-A(x,t)|.$  Applying the same argument as was used to prove \eqref{czsmoothpush1} gives
\[
|K^B(x,t,y,s)-K^{A,B}(x,t,y,s)|
\lesssim
\frac{D(x,t)}{\rho^{k+3}}.
\]
\endgroup

Finally, \eqref{czsmoothpush4} follows from the standard Calder\'on--Zygmund size
estimate:
\[
|K^{A,B}(x,t,y,s)|
\lesssim
\frac{1}{\rho^{k+2}+|A(x,t)-B(y,s)|^{k+2}}
\lesssim
\frac{1}{\rho^{k+2}}.
\]
This completes the proof.
\end{proof}

\begin{lemma}\label{sepbound}
The following inequality holds uniformly over $\ell \in \mathbb Z$:
\begin{align*}
\left\|
    \int_{\mathbb P^k}
    K^{B,A}_\ell(x,t,y,s)f(y,s)\,\d y\,\d s
\right\|_2^2
\lesssim
(1+\|\mathcal R\|^2)\|f\|_2^2 .
\end{align*}
\end{lemma}

\begin{proof}
Fix $\ell\in\mathbb Z$. We estimate
\begin{multline*}
\int_{\mathbb P^k}
\left|
    \int_{\mathbb P^k}
    K^{B,A}_\ell(x,t,y,s)f(y,s)\,\d y\,\d s
\right|^2
\d x\,\d t \\
\lesssim
\int_{C'_{4R}(0,0)}
\left|
    \int_{\mathbb P^k}
    K^{B,A}_\ell(x,t,y,s)f(y,s)\,\d y\,\d s
\right|^2
\d x\,\d t \\
+
\int_{\mathbb P^k\setminus C'_{4R}(0,0)}
\left|
    \int_{\|(x,t)-(y,s)\|>\frac12D(x,t)}
    \bigl(K^{B,A}_\ell(x,t,y,s)-K_\ell(x,t,y,s)\bigr)
    f(y,s)\,\d y\,\d s
\right|^2
\d x\,\d t \\
+
\int_{\mathbb P^k\setminus C'_{4R}(0,0)}
\left|
    \int_{\|(x,t)-(y,s)\|\leq \frac12D(x,t)}
    K^{B,A}_\ell(x,t,y,s)f(y,s)\,\d y\,\d s
\right|^2
\d x\,\d t \\
+
\int_{\mathbb P^k\setminus C'_{4R}(0,0)}
\left|
    \int_{\|(x,t)-(y,s)\|\leq \frac12D(x,t)}
    K_\ell(x,t,y,s)f(y,s)\,\d y\,\d s
\right|^2
\d x\,\d t \\
+
\int_{\mathbb P^k\setminus C'_{4R}(0,0)}
\left|
    \int_{\mathbb P^k}
    K_\ell(x,t,y,s)f(y,s)\,\d y\,\d s
\right|^2
\d x\,\d t \\
\shoveleft{=: I+II+III+IV+V.}
\end{multline*}

Terms $I$ and $V$ are bounded by
\[
I+V\lesssim \|\mathcal R\|^2\|f\|_2^2.
\]
Indeed, on $C'_{4R}(0,0)$ we have $B(x,t)=A(x,t)$, and hence
$K^{B,A}_\ell(x,t,y,s)=K_\ell(x,t,y,s)$.

We next estimate $II$. By \eqref{czsmoothpush1},
\[
\left|
K^{B,A}_\ell(x,t,y,s)-K_\ell(x,t,y,s)
\right|
\lesssim
\frac{D(x,t)}{\|(x,t)-(y,s)\|^{k+3}}
\]
whenever $(x,t)\notin C'_{4R}(0,0)$ and
$\|(x,t)-(y,s)\|>\frac12D(x,t)$. Therefore,
\begin{align*}
&\left|
    \int_{\|(x,t)-(y,s)\|>\frac12D(x,t)}
    \bigl(K^{B,A}_\ell(x,t,y,s)-K_\ell(x,t,y,s)\bigr)
    f(y,s)\,\d y\,\d s
\right|  \\
&\qquad\lesssim
\sum_{q\geq 0}
\frac{D(x,t)}{(2^qD(x,t))^{k+3}}
\int_{\|(x,t)-(y,s)\|\leq 2^{q+1}D(x,t)}
|f(y,s)|\,\d y\,\d s \\
&\qquad\lesssim
\mathcal M f(x,t).
\end{align*}
Hence
\[
II\lesssim \|\mathcal M f\|_2^2\lesssim \|f\|_2^2.
\]

For $III$, using \eqref{czsmoothpush2}, we get
\[
\left|
    \int_{\|(x,t)-(y,s)\|\leq \frac12D(x,t)}
    K^{B,A}_\ell(x,t,y,s)f(y,s)\,\d y\,\d s
\right|
\lesssim
\frac{1}{D(x,t)^{k+2}}
\int_{\|(x,t)-(y,s)\|\leq \frac12D(x,t)}
|f(y,s)|\,\d y\,\d s.
\]
Thus
\[
III\lesssim \|\mathcal M f\|_2^2\lesssim \|f\|_2^2.
\]

\begingroup
We can bound the term $IV$ by\begin{multline*}\sum_{j} \bigg(\int_{2^{-j-1} \leq D(x,t) \leq 2^{-j}}  \left| \int_{2^{-j-2}\leq\|(x,t) - (y,s)\| \leq \frac{1}{2}D(x,t)} K_\ell(x,t,y,s)f(y,s) \d y \d s \right|^2 \d x \d t\\+\int_{2^{-j-1} \leq D(x,t) \leq 2^{-j}}  \left| \int_{\|(x,t) - (y,s)\| \leq 2^{-j-2}} K_\ell(x,t,y,s)f(y,s) \d y \d s \right|^2 \d x \d t\bigg).\end{multline*}The first part of the sum is bounded by$$\int \big(\mathcal Mf(x,t)\big)^2 \d x \d t \lesssim \int |f(y,s)|^2 \d y \d s.$$The standard comparison between sharp and smooth truncations, together with Cotlar's inequality, shows that the second part of the sum is bounded by a constant multiple of $(1+\|\mathcal R\|^2)\|f\|_2^2$.

\endgroup

Combining the estimates for $I$, $II$, $III$, $IV$, and $V$ gives
\[
\left\|
    \int_{\mathbb P^k}
    K^{B,A}_\ell(x,t,y,s)f(y,s)\,\d y\,\d s
\right\|_2^2
\lesssim
(1+\|\mathcal R\|^2)\|f\|_2^2,
\]
uniformly in $\ell\in\mathbb Z$.
\end{proof}
\begin{lemma}\label{dualitybound}
The following inequality holds uniformly over $\ell\in\mathbb Z$:
\begin{align*}
\left\|
    \int_{\mathbb P^k}
    K^{A,B}_\ell(x,t,y,s)f(y,s)\,\d y\,\d s
\right\|_2^2
\lesssim
(1+\|\mathcal R\|^2)\|f\|_2^2 .
\end{align*}
\end{lemma}

\begin{proof}
\begingroup
For $m\in\{1,\ldots,n-k\}$, let $T_\ell^{A,B,m}$ and
$T_\ell^{B,A,m}$ denote the scalar operators associated with the $m$-th
components of $K_\ell^{A,B}$ and $K_\ell^{B,A}$, respectively. Since $K_\ell^{A,B,m}(x,t,y,s)
=
-K_\ell^{B,A,m}(y,s,x,t)$,
we have $(T_\ell^{A,B,m})^*=-T_\ell^{B,A,m}$.

Lemma~\ref{sepbound} bounds the vector-valued operator associated with
$K_\ell^{B,A}$, and hence bounds $T_\ell^{B,A,m}$ on $L^2$ for every $m$.
Duality therefore gives
\[
\|T_\ell^{A,B,m}f\|_2^2
\lesssim
(1+\|\mathcal R\|^2)\|f\|_2^2
\]
uniformly in $m$ and $\ell$. Summing over $m=1,\ldots,n-k$ proves the estimate.
\endgroup
\end{proof}

To estimate the part of the $K^B$ integral where the two points are close, we localize $B$ to Whitney cubes. We will show that these localized functions are regular Lip-$(1,1/2)$, with constants independent of the cube. We then apply Theorem~\ref{highcodimBHHLN} to obtain bounds for the corresponding Riesz transforms, uniformly over the Whitney cubes.

Let $\{I_i\}$ be a collection of parabolic Whitney cubes for
$\mathbb P^k\setminus \overline{C'_{4R}(0,0)}$. These cubes satisfy the following
properties:
\begin{itemize}
\item each $I_i$ is closed, and the collection $\{I_i\}$ has pairwise disjoint interiors;

\item
\[
    \bigcup_i I_i
    =
    \mathbb P^k\setminus \overline{C'_{4R}(0,0)};
\]

\item
\begin{align}
\label{diamtodist}
\frac{1}{1000}
\dist\big(I_i,C'_{4R}(0,0)\big)
\leq
\diam I_i
\leq
\frac{1}{10}
\dist\big(I_i,C'_{4R}(0,0)\big).
\end{align}
\end{itemize}

Let $\psi$ be a smooth bump equal to $1$ on $Q'_2(0,0)$, supported in
$Q'_4(0,0)$, with $0\leq \psi\leq 1$, and such that
\[
    \|\nabla_x^{|\alpha|}\partial_t^\beta\psi\|_\infty\lesssim 1
    \qquad\text{for all multi-indices $\alpha$ and all $\beta\geq 0$.}
\]
Let $(x_i,t_i)$ be the center of $I_i$, and set
\[
    \psi_i(x,t)
    :=
    \psi\left(
        \frac{x-x_i}{\ell(I_i)},
        \frac{t-t_i}{\ell(I_i)^2}
    \right).
\]
Then $\psi_i$ is equal to $1$ on $2I_i$, supported in $8I_i$, and satisfies
\[
    \|\nabla_x^{|\alpha|}\partial_t^\beta\psi_i\|_\infty
    \lesssim
    \ell(I_i)^{-|\alpha|-2\beta}.
\]
We localize $B$ to $I_i$ by setting
\[
    B_i(x,t)
    :=
    \psi_i(x,t)\big(B(x,t)-B(x_i,t_i)\big).
\]

Let $\beta_i$ be the $L^2$-$\beta$ numbers adapted to the graph of $B_i$:
\[
\beta_i((x,t),r)
:=
\inf_a
\left(
    \fint_{C'_r(x,t)}
    \left(\frac{|B_i(y,s)-a(y)|}{r}\right)^2
    \d y\,\d s
\right)^{1/2},
\]
where the infimum is taken over all affine functions
$a:\mathbb R^k\to\mathbb R^{n-k}$, viewed as functions independent of time.

We claim that
\begin{align}
\label{beta1}
\beta_i((x,t),r)^2
&\lesssim
(1+\|A\|_{lip})^2\frac{r^2}{\ell(I_i)^2},
\qquad
r\lesssim \ell(I_i),
\\
\label{beta2}
\beta_i((x,t),r)^2
&\lesssim
(1+\|A\|_{lip})^2\frac{\ell(I_i)^2}{r^2},
\qquad
r\gtrsim \ell(I_i).
\end{align}

For \eqref{beta2}, take $a\equiv 0$. Since $B_i$ is supported in $8I_i$, and since
\[
    |B(x,t)-B(x_i,t_i)|
    \lesssim
    (1+\|A\|_{lip})\ell(I_i),
    \qquad
    (x,t)\in 8I_i,
\]
we have
\[
    |B_i(x,t)|
    \lesssim
    (1+\|A\|_{lip})\ell(I_i).
\]
Therefore
\[
\beta_i((x,t),r)
\lesssim
(1+\|A\|_{lip})\frac{\ell(I_i)}{r},
\qquad
r\gtrsim \ell(I_i),
\]
which proves \eqref{beta2}.

For \eqref{beta1}, fix $(x,t)\in\mathbb P^k$ and take
\[
    a(y)
    =
    B_i(x,t)+\nabla_x B_i(x,t)\cdot(y-x).
\]
For $(y,s)\in C'_r(x,t)$, write
\[
B_i(y,s)-a(y)
=
\bigl[B_i(y,s)-B_i(y,t)\bigr]
+
\bigl[
B_i(y,t)-B_i(x,t)-\nabla_xB_i(x,t)\cdot(y-x)
\bigr].
\]
The first term is controlled by
\[
    |B_i(y,s)-B_i(y,t)|
    \leq
    \|\partial_t B_i\|_\infty |s-t|,
\]
and the second term is controlled by Taylor's theorem in the spatial variables:
\[
\left|
B_i(y,t)-B_i(x,t)-\nabla_xB_i(x,t)\cdot(y-x)
\right|
\lesssim
\|\nabla_x^2B_i\|_\infty |y-x|^2.
\]
On $8I_i$, we have $D(x,t)\approx \ell(I_i)$. Using \eqref{2ndgrad}, \eqref{timegrad}, the derivative estimates for $H$,
and the estimates for $\psi_i$ and its derivatives, we obtain
\[
    \|\nabla_x^2B_i\|_\infty
    \lesssim
    (1+\|A\|_{lip})\ell(I_i)^{-1},
    \qquad
    \|\partial_tB_i\|_\infty
    \lesssim
    (1+\|A\|_{lip})\ell(I_i)^{-1}.
\]
Since $(y,s)\in C'_r(x,t)$ implies $|y-x|\lesssim r$ and $|s-t|\lesssim r^2$, we get
\[
    |B_i(y,s)-a(y)|
    \lesssim
    (1+\|A\|_{lip})\frac{r^2}{\ell(I_i)}
\]
whenever $r\lesssim \ell(I_i)$. Dividing by $r$ and averaging over $C'_r(x,t)$ gives
\[
    \beta_i((x,t),r)
    \lesssim
    (1+\|A\|_{lip})\frac{r}{\ell(I_i)},
\]
which proves \eqref{beta1}.

\begingroup
Integrating \eqref{beta1} for $r\leq\ell(I_i)$ and \eqref{beta2} for $r>\ell(I_i)$, we obtain, for every $(z,\tau)\in\mathbb P^k$ and $\varrho>0$,
\begin{align}\label{bireg}
\int_0^\varrho\int_{C'_\varrho(z,\tau)}
    \beta_i((x,t),r)^2
    \frac{\d x\,\d t\,\d r}{r}
\lesssim
(1+\|A\|_{lip})^2\varrho^{k+2}
\lesssim
\varrho^{k+2}.
\end{align}
\endgroup

\begingroup
By \eqref{bireg} and Theorem~\ref{regularbetaequiv}, each $B_i$ is regular
Lip-$(1,1/2)$, with constants bounded uniformly in $i$.
\endgroup

We now show that the cubes $\{2I_i\}$ have bounded overlap, i.e.
\[
    \sum_i \mathbf 1_{2I_i}\leq C.
\]
By \eqref{diamtodist}, if $(x,t)\in 2I_i$, then
\[
\frac{4}{5}
\dist\big(I_i,C'_{4R}(0,0)\big)
\leq
D(x,t)
\leq
\frac{6}{5}
\dist\big(I_i,C'_{4R}(0,0)\big).
\]
Consequently, if $2I_i\cap 2I_j\neq\emptyset$, then
\[
\frac{2}{3}
\dist\big(I_j,C'_{4R}(0,0)\big)
\leq
\dist\big(I_i,C'_{4R}(0,0)\big)
\leq
\frac{3}{2}
\dist\big(I_j,C'_{4R}(0,0)\big).
\]
It follows that, for fixed $I_i$, the cylinder
$C'_{1500\diam I_i}(x_i,t_i)$ contains all Whitney cubes $I_j$ such that
$2I_j\cap 2I_i\neq\emptyset$. Since the Whitney cubes have disjoint interiors and
the cubes $I_j$ for which $2I_j\cap 2I_i\neq\emptyset$ have diameters comparable to that of $I_i$, this implies the bounded overlap property.

\begin{lemma}\label{sepbound2}
The following inequality holds uniformly over $\ell\in\mathbb Z$:
\[
\left\|
    \int_{\mathbb P^k}
    K^B_\ell(x,t,y,s)f(y,s)\,\d y\,\d s
\right\|_2^2
\lesssim
(1+\|\mathcal R\|^2)\|f\|_2^2 .
\]
\end{lemma}

\begin{proof}
We split the term we seek to estimate as follows:
\begin{multline*}
\int_{\mathbb P^k}
\left|
    \int_{\mathbb P^k}
    K^B_\ell(x,t,y,s)f(y,s)\,\d y\,\d s
\right|^2
\d x\,\d t \\
\lesssim
\int_{C'_{4R}(0,0)}
\left|
    \int_{\mathbb P^k}
    K^B_\ell(x,t,y,s)f(y,s)\,\d y\,\d s
\right|^2
\d x\,\d t \\
+
\int_{\mathbb P^k\setminus C'_{4R}(0,0)}
\left|
    \int_{\|(x,t)-(y,s)\|>\frac12D(x,t)}
    \bigl(K^B_\ell(x,t,y,s)-K^{A,B}_\ell(x,t,y,s)\bigr)
    f(y,s)\,\d y\,\d s
\right|^2
\d x\,\d t \\
+
\int_{\mathbb P^k\setminus C'_{4R}(0,0)}
\left|
    \int_{\|(x,t)-(y,s)\|\leq \frac12D(x,t)}
    K^B_\ell(x,t,y,s)f(y,s)\,\d y\,\d s
\right|^2
\d x\,\d t \\
+
\int_{\mathbb P^k\setminus C'_{4R}(0,0)}
\left|
    \int_{\|(x,t)-(y,s)\|\leq \frac12D(x,t)}
    K^{A,B}_\ell(x,t,y,s)f(y,s)\,\d y\,\d s
\right|^2
\d x\,\d t \\
+
\int_{\mathbb P^k\setminus C'_{4R}(0,0)}
\left|
    \int_{\mathbb P^k}
    K^{A,B}_\ell(x,t,y,s)f(y,s)\,\d y\,\d s
\right|^2
\d x\,\d t \\
\shoveleft{=: I+II+III+IV+V.}
\end{multline*}

Term $V$ is bounded by
\[
        V
    \lesssim
    (1+\|\mathcal R\|^2)\|f\|_2^2
\]
by Lemma~\ref{dualitybound}. For term $I$, we use that
$B(x,t)=A(x,t)$ whenever $(x,t)\in C'_{4R}(0,0)$, and hence
\[
    K^B_\ell(x,t,y,s)=K^{A,B}_\ell(x,t,y,s)
\]

for such $(x,t)$. It follows from Lemma~\ref{dualitybound} that
\[
    I\lesssim (1+\|\mathcal R\|^2)\|f\|_2^2.
\]

For term $II$, we use \eqref{czsmoothpush3}. Namely,
\[
\left|
K^B_\ell(x,t,y,s)-K^{A,B}_\ell(x,t,y,s)
\right|
\lesssim
\frac{D(x,t)}{\|(x,t)-(y,s)\|^{k+3}}
\]
whenever $(x,t)\notin C'_{4R}(0,0)$ and
\[
    \|(x,t)-(y,s)\|>\frac12D(x,t).
\]
Summing over dyadic annuli centered at $(x,t)$ gives
\[
\left|
    \int_{\|(x,t)-(y,s)\|>\frac12D(x,t)}
    \bigl(K^B_\ell(x,t,y,s)-K^{A,B}_\ell(x,t,y,s)\bigr)
    f(y,s)\,\d y\,\d s
\right|
\lesssim
\mathcal M f(x,t).
\]
Therefore
\[
    II\lesssim \|\mathcal M f\|_2^2\lesssim \|f\|_2^2.
\]

\begingroup
Let $T_j^{A,B}$ denote the operator with kernel $K_j^{A,B}$.
Lemma~\ref{dualitybound} and Cotlar's inequality (applied separately to each
component) give
\[
\left\|\sup_{j\in\mathbb Z}|T_j^{A,B}f|\right\|_2
\lesssim
(1+\|\mathcal R\|)\|f\|_2.
\]
For $(x,t)\notin C'_{4R}(0,0)$, comparing the sharp and smooth truncations applied to $f$ gives
\[
\left|
\int_{\|(x,t)-(y,s)\|\leq D(x,t)/2}
K^{A,B}_\ell(x,t,y,s)f(y,s)\,\d y\,\d s
\right|
\lesssim
\sup_{j\in\mathbb Z}|T_j^{A,B}f(x,t)|+\mathcal M f(x,t).
\]
Consequently,
\[
IV\lesssim(1+\|\mathcal R\|^2)\|f\|_2^2.
\]
\endgroup

It remains to estimate term $III$. The Riesz kernel
adapted to $B_i$ is
\[
K^i(x,t,y,s)
:=
\frac{B_i(x,t)-B_i(y,s)}{|t-s|^{(k+1)/2+1}}
\exp\left(
-\frac{|x-y|^2+|B_i(x,t)-B_i(y,s)|^2}{4|t-s|}
\right).
\]
Let $\mathcal R^i$ be the singular integral associated to this kernel, and define
$K^i_N$ using the same smooth truncations as above:
\[
    K^i_N(x,t,y,s)
    :=
    K^i(x,t,y,s)\Psi_N((x,t)-(y,s)).
\]

\begingroup
By \eqref{bireg} and Theorem~\ref{regularbetaequiv}, the functions $B_i$ are
regular Lip-$(1,1/2)$, with constants bounded uniformly in $i$. Hence Theorem~\ref{highcodimBHHLN} applies to the singular integrals associated to each $K^i$, with bounds uniform in $i$.
Thus,
\[
\left\|
\sup_{N\in\mathbb Z}
\left|
\int_{\mathbb P^k}
K_N^i(x,t,y,s)g(y,s)\,\d y\,\d s
\right|
\right\|_2
\lesssim
\|g\|_2,
\]
with a constant independent of $i$.
\endgroup
Moreover, since $\psi_i\equiv1$ on
$2I_i$, if $(x,t),(y,s)\in 2I_i$, then
\[
    B_i(x,t)-B_i(y,s)=B(x,t)-B(y,s),
\]
and hence
\[
    K^i_\ell(x,t,y,s)=K^B_\ell(x,t,y,s)
\]
for such pairs.

We now return to $III$. The Whitney cubes $\{I_i\}$ cover
$\mathbb P^k\setminus \overline{C'_{4R}(0,0)}$, and
$\partial C'_{4R}(0,0)$ has Lebesgue measure zero. Thus we may write
\begin{align*}
III
&=
\sum_i
\int_{I_i}
\left|
    \int_{\|(x,t)-(y,s)\|\leq \frac12D(x,t)}
    K^B_\ell(x,t,y,s)f(y,s)\,\d y\,\d s
\right|^2
\d x\,\d t .
\end{align*}
For $(x,t)\in I_i$, we have $D(x,t)\approx \diam(I_i)$. Split the inner integral into
\[
    \ell(I_i)/4
    \leq
    \|(x,t)-(y,s)\|
    \leq
    \frac12D(x,t)
\]
and
\[
    \|(x,t)-(y,s)\|<\ell(I_i)/4.
\]
The contribution of the first region is bounded pointwise by $\mathcal M f(x,t)$,
using the size estimate for $K^B_\ell$. Thus its contribution to $III$ is bounded by
\[
    \|\mathcal M f\|_2^2
    \lesssim
    \|f\|_2^2.
\]

For the second region, note that if $(x,t)\in I_i$ and
\[
    \|(x,t)-(y,s)\|<\ell(I_i)/4,
\]
then $(y,s)\in 2I_i$. Hence
\[
    K^B_\ell(x,t,y,s)=K^i_\ell(x,t,y,s)
\]
on this region. Therefore the second contribution is bounded by
\begin{align*}
\sum_i
\int_{I_i}
\left|
    \int_{\|(x,t)-(y,s)\|<\ell(I_i)/4}
    K^i_\ell(x,t,y,s)\mathbf 1_{2I_i}(y,s)f(y,s)\,\d y\,\d s
\right|^2
\d x\,\d t .
\end{align*}
Comparing the sharp and smooth truncations applied to ${\bf1}_{2I_i}f$ and using the bounded overlap of the cubes $\{2I_i\}$, we obtain
\[
\begin{aligned}
&\sum_i
\int_{I_i}
\left|
    \int_{\|(x,t)-(y,s)\|<\ell(I_i)/4}
    K^i_\ell(x,t,y,s)\mathbf 1_{2I_i}(y,s)f(y,s)\,\d y\,\d s
\right|^2
\d x\,\d t \\
&\qquad\lesssim
\sum_i
\int_{\mathbb P^k}
\sup_N
\left|
    \int_{\mathbb P^k}
    K^i_N(x,t,y,s)\mathbf 1_{2I_i}(y,s)f(y,s)\,\d y\,\d s
\right|^2
\d x\,\d t
+
\sum_i \|\mathcal M(\mathbf 1_{2I_i}f)\|_2^2 \\
&\qquad\lesssim
\sum_i
\|\mathbf 1_{2I_i}f\|_2^2
\lesssim
\|f\|_2^2.
\end{aligned}
\]

Combining the estimates for $I,\dots,V$, we obtain
\[
\left\|
    \int_{\mathbb P^k}
    K^B_\ell(x,t,y,s)f(y,s)\,\d y\,\d s
\right\|_2^2
\lesssim
(1+\|\mathcal R\|^2)\|f\|_2^2,
\]
uniformly over $\ell\in\mathbb Z$.
\end{proof}

Lemma~\ref{sepbound2} gives, uniformly over $\ell\in\mathbb Z$,
\begin{align}\label{BL2}
\left\|
    \int_{\mathbb P^k}
    K^B_\ell(x,t,y,s)f(y,s)\,\d y\,\d s
\right\|_2^2
\lesssim
(1+\|\mathcal R\|^2)\|f\|_2^2.
\end{align}

The function $B$ does not have compact support. We now use \eqref{BL2}
to prove that the Riesz transform adapted to the graph of $\overline A$ is
bounded from $L^2$ to $L^2$. We repeat the argument for
\eqref{BL2}, with $B$ in place of $A$ and $\overline A$ in place
of $B$.

We use the following kernels, recalling the definition of $K^{\overline A}$:
\[
K^{\overline A}(x,t,y,s)
=
\frac{\overline A(x,t)-\overline A(y,s)}{|t-s|^{(k+1)/2+1}}
\exp\left(
-\frac{|x-y|^2+|\overline A(x,t)-\overline A(y,s)|^2}{4|t-s|}
\right),
\]
\[
K^{\overline A,B}(x,t,y,s)
=
\frac{\overline A(x,t)-B(y,s)}{|t-s|^{(k+1)/2+1}}
\exp\left(
-\frac{|x-y|^2+|\overline A(x,t)-B(y,s)|^2}{4|t-s|}
\right),
\]
and
\[
K^{B,\overline A}(x,t,y,s)
=
\frac{B(x,t)-\overline A(y,s)}{|t-s|^{(k+1)/2+1}}
\exp\left(
-\frac{|x-y|^2+|B(x,t)-\overline A(y,s)|^2}{4|t-s|}
\right).
\]

\begin{lemma}\label{czestimatesbarA}
Let
\[
    \rho:=\|(x,t)-(y,s)\|.
\]
The following estimates hold:
\begin{align}
|K^{\overline A,B}(x,t,y,s) - K^B(x,t,y,s)|
&\lesssim
\frac{D(x,t)}{\rho^{k+3}},
&&
(x,t)\notin \overline{C'_{4R}(0,0)},\quad
\rho>\frac12D(x,t),
\label{czsmoothpush5}
\\
|K^{\overline A,B}(x,t,y,s)|
&\lesssim
\frac{1}{D(x,t)^{k+2}},
&&
(x,t)\notin \overline{C'_{4R}(0,0)},\quad
\rho\leq \frac12D(x,t),
\label{czsmoothpush6}
\\
|K^{\overline A}(x,t,y,s) - K^{B,\overline A}(x,t,y,s)|
&\lesssim
\frac{D(x,t)}{\rho^{k+3}},
&&
(x,t)\notin \overline{C'_{4R}(0,0)},\quad
\rho>\frac12D(x,t),
\label{czsmoothpush7}
\\
|K^{B,\overline A}(x,t,y,s)|
&\lesssim
\frac{1}{\rho^{k+2}}.
\label{czsmoothpush8}
\end{align}
\end{lemma}

\begin{proof}
The proof is analogous to the proof of Lemma~\ref{czestimates}. We give the details
for completeness.

\begingroup
For \eqref{czsmoothpush5}, set
\[
\delta:=|B(x,t)-\overline A(x,t)|
=\sqrt{n-k}\,H(x,t)\lesssim D(x,t),
\qquad
(x,t)\notin \overline{C'_{4R}(0,0)}.
\]
If $\delta\leq\rho/2$, the Calder\'on--Zygmund smoothness estimate in the
first transversal variable gives
\[
|K^{\overline A,B}(x,t,y,s)-K^B(x,t,y,s)|
\lesssim
\frac{\delta}{\rho^{k+3}}
\lesssim
\frac{D(x,t)}{\rho^{k+3}}.
\]
If $\delta>\rho/2$, the size estimates for the two kernels give
\[
|K^{\overline A,B}(x,t,y,s)-K^B(x,t,y,s)|
\lesssim
\frac{1}{\rho^{k+2}}.
\]
In this case $\delta\lesssim D(x,t)$ implies $D(x,t)\gtrsim\rho$, and therefore
\[
\frac{1}{\rho^{k+2}}
\lesssim
\frac{D(x,t)}{\rho^{k+3}}.
\]
This proves \eqref{czsmoothpush5}.
\endgroup

We next prove \eqref{czsmoothpush6}. Suppose that
\[
    (x,t)\notin \overline{C'_{4R}(0,0)}
    \qquad\text{and}\qquad
    \rho\leq \frac12D(x,t).
\]
Then $(y,s)\notin \overline{C'_{4R}(0,0)}$ and
\[
    D(y,s)\approx D(x,t).
\]
Since $B=\overline A+H{\bf 1}$ and $H(y,s)\approx D(y,s)$ outside
$\overline{C'_{4R}(0,0)}$, we have
\[
    |B(y,s)-\overline A(y,s)|\gtrsim D(x,t).
\]
Moreover, using that $\overline A$ is Lip-$(1,1/2)$ with small Lipschitz constant,
\[
    |\overline A(x,t)-\overline A(y,s)|
    \leq
    \|\overline A\|_{lip}\rho
    \leq
    \|\overline A\|_{lip}\frac{D(x,t)}{2}.
\]
Taking $\|A\|_{lip}$ sufficiently small, and hence $\|\overline A\|_{lip}$ sufficiently
small, we get
\[
|\overline A(x,t)-B(y,s)|
\geq
|B(y,s)-\overline A(y,s)|
-
|\overline A(x,t)-\overline A(y,s)|
\gtrsim
D(x,t).
\]
The standard size estimate for the Riesz kernel then gives
\[
|K^{\overline A,B}(x,t,y,s)|
\lesssim
\frac{1}{\rho^{k+2}+|\overline A(x,t)-B(y,s)|^{k+2}}
\lesssim
\frac{1}{D(x,t)^{k+2}},
\]
which proves \eqref{czsmoothpush6}.

\begingroup
We now prove \eqref{czsmoothpush7}. If
$\delta=|B(x,t)-\overline A(x,t)|\leq\rho/2$, smoothness in the first
transversal variable gives
\[
|K^{\overline A}(x,t,y,s)-K^{B,\overline A}(x,t,y,s)|
\lesssim
\frac{\delta}{\rho^{k+3}}
\lesssim
\frac{D(x,t)}{\rho^{k+3}}.
\]
If $\delta>\rho/2$, the size estimates give a bound by
$\rho^{-(k+2)}$, while $\delta\lesssim D(x,t)$ implies
$D(x,t)\gtrsim\rho$. The required bound follows.
\endgroup

Finally, \eqref{czsmoothpush8} follows directly from the standard Calder\'on--Zygmund
size estimate:
\[
|K^{B,\overline A}(x,t,y,s)|
\lesssim
\frac{1}{\rho^{k+2}+|B(x,t)-\overline A(y,s)|^{k+2}}
\lesssim
\frac{1}{\rho^{k+2}}.
\]
This completes the proof.
\end{proof}

\begin{lemma}\label{sepbound3}
The following inequality holds uniformly over $\ell\in\mathbb Z$:
\[
\left\|
    \int_{\mathbb P^k}
    K^{\overline A,B}_\ell(x,t,y,s)f(y,s)\,\d y\,\d s
\right\|_2^2
\lesssim
(1+\|\mathcal R\|^2)\|f\|_2^2.
\]
\end{lemma}

\begin{proof}
Fix $\ell\in\mathbb Z$. We estimate
\begin{multline*}
\int_{\mathbb P^k}
\left|
    \int_{\mathbb P^k}
    K^{\overline A,B}_\ell(x,t,y,s)f(y,s)\,\d y\,\d s
\right|^2
\d x\,\d t \\
\lesssim
\int_{C'_{4R}(0,0)}
\left|
    \int_{\mathbb P^k}
    K^{\overline A,B}_\ell(x,t,y,s)f(y,s)\,\d y\,\d s
\right|^2
\d x\,\d t \\
+
\int_{\mathbb P^k\setminus C'_{4R}(0,0)}
\left|
    \int_{\|(x,t)-(y,s)\|>\frac12D(x,t)}
    \bigl(K^{\overline A,B}_\ell(x,t,y,s)-K^B_\ell(x,t,y,s)\bigr)
    f(y,s)\,\d y\,\d s
\right|^2
\d x\,\d t \\
+
\int_{\mathbb P^k\setminus C'_{4R}(0,0)}
\left|
    \int_{\|(x,t)-(y,s)\|\leq \frac12D(x,t)}
    K^{\overline A,B}_\ell(x,t,y,s)f(y,s)\,\d y\,\d s
\right|^2
\d x\,\d t \\
+
\int_{\mathbb P^k\setminus C'_{4R}(0,0)}
\left|
    \int_{\|(x,t)-(y,s)\|\leq \frac12D(x,t)}
    K^B_\ell(x,t,y,s)f(y,s)\,\d y\,\d s
\right|^2
\d x\,\d t \\
+
\int_{\mathbb P^k\setminus C'_{4R}(0,0)}
\left|
    \int_{\mathbb P^k}
    K^B_\ell(x,t,y,s)f(y,s)\,\d y\,\d s
\right|^2
\d x\,\d t \\
\shoveleft{=: I+II+III+IV+V.}
\end{multline*}

Term $I$ is bounded by Lemma~\ref{sepbound2}, since
$B(x,t)=\overline A(x,t)$ for $(x,t)\in C'_{4R}(0,0)$, and hence
\[
    K^{\overline A,B}_\ell(x,t,y,s)=K^B_\ell(x,t,y,s)
\]
for such $(x,t)$. Term $V$ is also bounded by Lemma~\ref{sepbound2}. Therefore
\[
    I+V\lesssim (1+\|\mathcal R\|^2)\|f\|_2^2.
\]

For term $II$, we use \eqref{czsmoothpush5}. Summing over dyadic annuli centered at
$(x,t)$ gives
\[
\left|
    \int_{\|(x,t)-(y,s)\|>\frac12D(x,t)}
    \bigl(K^{\overline A,B}_\ell(x,t,y,s)-K^B_\ell(x,t,y,s)\bigr)
    f(y,s)\,\d y\,\d s
\right|
\lesssim
\mathcal M f(x,t).
\]
Hence
\[
    II\lesssim \|\mathcal M f\|_2^2\lesssim \|f\|_2^2.
\]

For term $III$, using \eqref{czsmoothpush6}, we obtain
\[
\left|
    \int_{\|(x,t)-(y,s)\|\leq \frac12D(x,t)}
    K^{\overline A,B}_\ell(x,t,y,s)f(y,s)\,\d y\,\d s
\right|
\lesssim
\frac{1}{D(x,t)^{k+2}}
\int_{\|(x,t)-(y,s)\|\leq \frac12D(x,t)}
|f(y,s)|\,\d y\,\d s.
\]
Therefore
\[
    III\lesssim \|\mathcal M f\|_2^2\lesssim \|f\|_2^2.
\]

\begingroup
Let $T_j^B$ denote the operator with kernel $K_j^B$. Since $B$ is
Lip-$(1,1/2)$, $K^B$ is a parabolic Calder\'on--Zygmund kernel with
uniform constants. Lemma~\ref{sepbound2} and Cotlar's inequality give
\[
\left\|\sup_{j\in\mathbb Z}|T_j^Bf|\right\|_2
\lesssim
(1+\|\mathcal R\|)\|f\|_2.
\]
For each $(x,t)\notin C'_{4R}(0,0)$, comparing the sharp and smooth truncations applied to $f$ gives
\[
\left|
\int_{\|(x,t)-(y,s)\|\leq \frac12D(x,t)}
K^B_\ell(x,t,y,s)f(y,s)\,\d y\,\d s
\right|
\lesssim
\sup_{j\in\mathbb Z}|T_j^Bf(x,t)|
+
\mathcal M f(x,t).
\]
Consequently,
\[
IV
\lesssim
\left\|\sup_{j\in\mathbb Z}|T_j^Bf|\right\|_2^2
+
\|\mathcal M f\|_2^2
\lesssim
(1+\|\mathcal R\|^2)\|f\|_2^2.
\]
\endgroup

Combining the estimates for $I,\dots,V$, we conclude that
\[
\left\|
    \int_{\mathbb P^k}
    K^{\overline A,B}_\ell(x,t,y,s)f(y,s)\,\d y\,\d s
\right\|_2^2
\lesssim
(1+\|\mathcal R\|^2)\|f\|_2^2,
\]
uniformly over $\ell\in\mathbb Z$.
\end{proof}

\begin{lemma}\label{dualitybound2}
The following inequality holds uniformly over $\ell\in\mathbb Z$:
\[
\left\|
    \int_{\mathbb P^k}
    K^{B,\overline A}_\ell(x,t,y,s)f(y,s)\,\d y\,\d s
\right\|_2^2
\lesssim
(1+\|\mathcal R\|^2)\|f\|_2^2.
\]
\end{lemma}

\begin{proof}
\begingroup
For $m\in\{1,\ldots,n-k\}$, let
$T_\ell^{B,\overline A,m}$ and $T_\ell^{\overline A,B,m}$ denote the
scalar operators associated with the $m$-th components of
$K_\ell^{B,\overline A}$ and $K_\ell^{\overline A,B}$, respectively.
Since
\[
K_\ell^{B,\overline A,m}(x,t,y,s)
=
-K_\ell^{\overline A,B,m}(y,s,x,t),
\]
where
\[
\Psi_\ell((x,t)-(y,s))
=
\Psi_\ell((y,s)-(x,t)),
\]
we have
\[
(T_\ell^{B,\overline A,m})^*
=
-T_\ell^{\overline A,B,m}.
\]
Lemma~\ref{sepbound3} bounds the vector-valued operator associated with
$K_\ell^{\overline A,B}$, and hence bounds each
$T_\ell^{\overline A,B,m}$ on $L^2$. Duality therefore gives
\[
\|T_\ell^{B,\overline A,m}f\|_2^2
\lesssim
(1+\|\mathcal R\|^2)\|f\|_2^2
\]
uniformly in $m$ and $\ell$. Summing over $m=1,\ldots,n-k$ proves the lemma.
\endgroup
\end{proof}

\begin{lemma}\label{sepbound4}
The following inequality holds uniformly over $\ell\in\mathbb Z$:
\[
\left\|
    \int_{\mathbb P^k}
    K^{\overline A}_\ell(x,t,y,s)f(y,s)\,\d y\,\d s
\right\|_2^2
\lesssim
(1+\|\mathcal R\|^2)\|f\|_2^2.
\]
\end{lemma}

\begin{proof}
Fix $\ell\in\mathbb Z$. We split the term we seek to estimate as follows:
\begin{multline*}
\int_{\mathbb P^k}
\left|
    \int_{\mathbb P^k}
    K^{\overline A}_\ell(x,t,y,s)f(y,s)\,\d y\,\d s
\right|^2
\d x\,\d t \\
\lesssim
\int_{C'_{4R}(0,0)}
\left|
    \int_{\mathbb P^k}
    K^{\overline A}_\ell(x,t,y,s)f(y,s)\,\d y\,\d s
\right|^2
\d x\,\d t \\
+
\int_{\mathbb P^k\setminus C'_{4R}(0,0)}
\left|
    \int_{\|(x,t)-(y,s)\|>\frac12D(x,t)}
    \bigl(K^{\overline A}_\ell(x,t,y,s)-K^{B,\overline A}_\ell(x,t,y,s)\bigr)
    f(y,s)\,\d y\,\d s
\right|^2
\d x\,\d t \\
+
\int_{\mathbb P^k\setminus C'_{4R}(0,0)}
\left|
    \int_{\|(x,t)-(y,s)\|\leq \frac12D(x,t)}
    K^{\overline A}_\ell(x,t,y,s)f(y,s)\,\d y\,\d s
\right|^2
\d x\,\d t \\
+
\int_{\mathbb P^k\setminus C'_{4R}(0,0)}
\left|
    \int_{\|(x,t)-(y,s)\|\leq \frac12D(x,t)}
    K^{B,\overline A}_\ell(x,t,y,s)f(y,s)\,\d y\,\d s
\right|^2
\d x\,\d t \\
+
\int_{\mathbb P^k\setminus C'_{4R}(0,0)}
\left|
    \int_{\mathbb P^k}
    K^{B,\overline A}_\ell(x,t,y,s)f(y,s)\,\d y\,\d s
\right|^2
\d x\,\d t \\
\shoveleft{=: I+II+III+IV+V.}
\end{multline*}

Term $I$ is bounded by Lemma~\ref{dualitybound2}, since
$B(x,t)=\overline A(x,t)$ for $(x,t)\in C'_{4R}(0,0)$, and hence
\[
    K^{\overline A}_\ell(x,t,y,s)=K^{B,\overline A}_\ell(x,t,y,s)
\]
for such $(x,t)$. Term $V$ is bounded directly by Lemma~\ref{dualitybound2}. Thus
\[
    I+V\lesssim (1+\|\mathcal R\|^2)\|f\|_2^2.
\]

\begingroup
For term $II$, \eqref{czsmoothpush7} gives
\[
\begin{aligned}
&\left|
\int_{\|(x,t)-(y,s)\|>D(x,t)/2}
\bigl(K^{\overline A}_\ell(x,t,y,s)-K^{B,\overline A}_\ell(x,t,y,s)\bigr)
f(y,s)\,\d y\,\d s
\right|\\
&\qquad\lesssim
\sum_{q\geq0}
\frac{D(x,t)}{(2^qD(x,t))^{k+3}}
\int_{\|(x,t)-(y,s)\|\leq2^{q+1}D(x,t)}
|f(y,s)|\,\d y\,\d s
\lesssim \mathcal M f(x,t).
\end{aligned}
\]
Therefore,
\[
II\lesssim\|\mathcal M f\|_2^2\lesssim\|f\|_2^2.
\]
\endgroup

\begingroup
Let $T_j^{B,\overline A}$ denote the operator with kernel
$K_j^{B,\overline A}$. By \eqref{czsmoothpush8} and the Lip-$(1,1/2)$
bounds for $B$ and $\overline A$, the kernel $K^{B,\overline A}$ is a
parabolic Calder\'on--Zygmund kernel with uniform constants. Lemma~\ref{dualitybound2}
and Cotlar's inequality give
\[
\left\|\sup_{j\in\mathbb Z}|T_j^{B,\overline A}f|\right\|_2
\lesssim
(1+\|\mathcal R\|)\|f\|_2.
\]
For each $(x,t)\notin C'_{4R}(0,0)$, comparing the sharp and smooth truncations applied to $f$ gives
\[
\left|
\int_{\|(x,t)-(y,s)\|\leq \frac12D(x,t)}
K^{B,\overline A}_\ell(x,t,y,s)f(y,s)\,\d y\,\d s
\right|
\lesssim
\sup_{j\in\mathbb Z}|T_j^{B,\overline A}f(x,t)|
+
\mathcal M f(x,t).
\]
Consequently,
\[
IV
\lesssim
\left\|\sup_{j\in\mathbb Z}|T_j^{B,\overline A}f|\right\|_2^2
+
\|\mathcal M f\|_2^2
\lesssim
(1+\|\mathcal R\|^2)\|f\|_2^2.
\]
\endgroup

We estimate $III$ as we estimated term $III$ in the proof of
Lemma~\ref{sepbound2}. Using the same Whitney decomposition
of $\mathbb P^k\setminus \overline{C'_{4R}(0,0)}$, we define
\[
    \overline A_i(x,t)
    :=
    \psi_i(x,t)\big(\overline A(x,t)-\overline A(x_i,t_i)\big).
\]

Let $\overline\beta_i((x,t),r)$ denote the $L^2$-$\beta$ number defined above
with $B_i$ replaced by $\overline A_i$. The estimates \eqref{2ndgrad} and
\eqref{timegrad}, the Lip-$(1,1/2)$ bound for $\overline A$, and the
estimates for $\psi_i$ and its derivatives give
\begin{align*}
\overline\beta_i((x,t),r)^2
&\lesssim
\|A\|_{lip}^2\frac{r^2}{\ell(I_i)^2},
\qquad r\lesssim\ell(I_i),
\\
\overline\beta_i((x,t),r)^2
&\lesssim
\|A\|_{lip}^2\frac{\ell(I_i)^2}{r^2},
\qquad r\gtrsim\ell(I_i).
\end{align*}
Splitting the integral in $r$ at $r=\ell(I_i)$, as in \eqref{bireg}, gives
\[
\int_0^\varrho\int_{C'_\varrho(z,\tau)}
\overline\beta_i((x,t),r)^2
\frac{\d x\,\d t\,\d r}{r}
\lesssim
\|A\|_{lip}^2\varrho^{k+2}.
\]
By Theorem~\ref{regularbetaequiv}, the functions $\overline A_i$ are regular
Lip-$(1,1/2)$, with constants bounded uniformly in $i$. \begingroup
Hence, by Theorem~\ref{highcodimBHHLN},
\[
\left\|
\sup_{N\in\mathbb Z}
\left|
\int_{\mathbb P^k}
K_N^{\overline A_i}(x,t,y,s)g(y,s)\,\d y\,\d s
\right|
\right\|_2
\lesssim
\|g\|_2,
\]
with a constant independent of $i$.
\endgroup

For $(x,t)\in I_i$ and
\[
    \|(x,t)-(y,s)\|<\ell(I_i)/4,
\]
we have $(y,s)\in 2I_i$, and therefore
\[
    K^{\overline A}_\ell(x,t,y,s)
    =
    K^{\overline A_i}_\ell(x,t,y,s).
\]
The same argument used for term $III$ in Lemma~\ref{sepbound2}, together with the
bounded overlap of the cubes $\{2I_i\}$, gives
\[
    III\lesssim \|f\|_2^2.
\]

Combining the estimates for $I,\dots,V$, we obtain
\[
\left\|
    \int_{\mathbb P^k}
    K^{\overline A}_\ell(x,t,y,s)f(y,s)\,\d y\,\d s
\right\|_2^2
\lesssim
(1+\|\mathcal R\|^2)\|f\|_2^2,
\]
uniformly over $\ell\in\mathbb Z$.
\end{proof}

Since $A_{(0,0),R}=\overline A$, Lemma~\ref{sepbound4} completes the proof of Proposition~\ref{localsios}.

\subsection{Proof of Theorem~\ref{mainthm2}}

\begingroup

\begin{proof}[Proof of Theorem~\ref{mainthm2}]

Fix $Q_0\in\mathbb D$, and let $(x_0,t_0)$ be its center. Choose
$R=C_0\ell(Q_0)$, where $C_0$ is a fixed constant large enough that
$C'_{4R}(x_0,t_0)$ contains every $5Q$ with $Q\subseteq Q_0$, and let
\[
\overline A:=A_{(x_0,t_0),R}
\]
be the function supplied by Proposition~\ref{localsios}. Then
$\overline A=A$ on every such $5Q$,
\[
\supp\bigl(\overline A-A(x_0,t_0)\bigr)
\subset C'_{101R}(x_0,t_0),
\qquad
\|\mathcal R^{\overline A}\|
\lesssim
1+\|\mathcal R\|,
\]
and
\[
\|\overline A\|_{lip}\lesssim\|A\|_{lip}.
\]
We choose the constant $\epsilon$ in Theorem~\ref{mainthm2} so that $\|\overline A\|_{lip}$ is as small as required in Theorem~\ref{localdpara}.

The reduction to the time-symmetric transforms in Subsection~\ref{symreduction} gives $\|\mathcal R\|\lesssim\|\mathcal R^+\|$. In the proof of Theorem~\ref{localdpara}, after the reduction in
Subsection~\ref{symreduction}, we use only the time-symmetric bounds. We can
therefore apply the same proof to $\overline A-A(x_0,t_0)$, with $101R$ in place
of $R$. This function is
supported in $C'_{101R}(x_0,t_0)$. Subtracting the constant $A(x_0,t_0)$ leaves the kernels and the
$D_{par}$ norm unchanged. Thus, since $101R\approx\ell(Q_0)$, we obtain
\[
\|D_{par}\overline A\|_2^2
\lesssim
(1+\|\mathcal R\|^2)\ell(Q_0)^{k+2}
=
(1+\|\mathcal R\|^2)|Q_0|.
\]

Applying Lemma~\ref{vectorDorronsoro} to
$\overline A-A(x_0,t_0)$, and using that the $\beta_2$-numbers and
the $D_{par}$ norm are unchanged by the addition of a constant, we obtain
\[
\sum_{Q\in\mathbb D}
\beta_2^{\Gamma_{\overline A}}(5Q)^2|Q|
\lesssim
\|D_{par}\overline A\|_2^2.
\]
For $Q\subseteq Q_0$, the parametrizations of $\Gamma_{\overline A}$ and
$\Gamma_A$ agree over $5Q$. By
Definition~\ref{betadef}, the measures also agree there. Therefore
\[
\beta_2^{\Gamma_{\overline A}}(5Q)
=
\beta_2^{\Gamma_A}(5Q),
\qquad Q\subseteq Q_0,
\]
and we conclude that
\begin{equation}\label{dyadicbeta5}
\sum_{Q\subseteq Q_0}
\beta_2^{\Gamma_A}(5Q)^2|Q|
\lesssim
(1+\|\mathcal R\|^2)|Q_0|.
\end{equation}

It follows from \eqref{dyadicbeta5} by standard arguments that, for every
$(z,\tau)\in\mathbb P^k$ and $\rho>0$,
\begin{equation}\label{continuousbeta}
\int_0^\rho\int_{C'_\rho(z,\tau)}
\beta_2^{\Gamma_A}(C'_r(x,t))^2
\frac{\d x\,\d t\,\d r}{r}
\lesssim
(1+\|\mathcal R\|^2)\rho^{k+2}.
\end{equation}
\begingroup
Applying Theorem~\ref{regularbetaequiv} to \eqref{continuousbeta}, we obtain
\[
\|D_t^{1/2}A\|_{BMO(\mathbb P^k)}
\lesssim
1+\|\mathcal R\|
\lesssim
1+\|\mathcal R^+\|.
\]
\endgroup
Hence $A$ is regular Lip-$(1,1/2)$, which proves
Theorem~\ref{mainthm2}. In particular, $\Gamma_A$ is parabolic uniformly
rectifiable.

\end{proof}
\endgroup

\begingroup

\subsection{The case \texorpdfstring{$k=0$}{k=0}}\label{timeonlycase}

We prove Theorems~\ref{mainthm2} and~\ref{localdpara} for $k=0$ using the case $k=1$ of Theorem~\ref{mainthm2}. Here $\mathbb P^0=\mathbb R^1$ as sets, and we equip it with the parabolic distance $d_p(t,s)=|t-s|^{1/2}$. Let $A:\mathbb P^0\to\mathbb R^n$ satisfy the hypotheses of Theorem~\ref{mainthm2}, and set
$$
F(x,t)=A(t),\qquad (x,t)\in\mathbb P^1.
$$
Then $\Gamma_F=\{(x,A(t),t):x,t\in\mathbb R\}\subset\mathbb P^{n+1}$, and $\|F\|_{lip}=\|A\|_{lip}$. We take $\epsilon$ no larger than the value in Theorem~\ref{mainthm2} for $k=1$ in $\mathbb P^{n+1}$. The constants below depend only on $n$ and $\psi_0$. Recall that $\psi_0$ was defined in Subsection~\ref{graphnotation}; see \eqref{psi0}.

We first prove that the transversal Riesz transforms on $\Gamma_F$ are bounded. Fix $m\in\{1,\ldots,n\}$. The forward-time kernel on $\Gamma_A$ is
$$
K^{+,m}(t,s)=\frac{A_m(t)-A_m(s)}{(t-s)^{3/2}}\exp(-\frac{|A(t)-A(s)|^2}{4(t-s)}),\qquad t>s,
$$
and is zero for $t\leq s$. For $t>s$, the corresponding kernel on $\Gamma_F$ satisfies
$$
K^{+,F,m}(x,t,y,s)=(t-s)^{-1/2}\exp(-\frac{|x-y|^2}{4(t-s)})K^{+,m}(t,s).
$$
Let $\mathcal S_\ell^m$ be obtained from $\mathcal R_\ell^{+,F,m}$ by replacing the cutoff $\Psi_\ell(x-y,t-s)$ with $\Psi_\ell(t-s)$, where the latter is defined on $\mathbb P^0$. For $f\in C_c^\infty(\mathbb P^1)$, write $\widehat f(\xi,t)$ for the Fourier transform in $x$. Fourier transformation of the Gaussian gives
$$
\widehat{\mathcal S_\ell^m f}(\xi,t)=2\sqrt\pi\int_{-\infty}^t e^{-4\pi^2\xi^2(t-s)}K^{+,m}(t,s)\Psi_\ell(t-s)\widehat f(\xi,s)\,\d s.
$$
For $\lambda>0$, we have
$$
e^{-\lambda|t-s|}=\int_{\mathbb R}e^{2\pi i\eta(t-s)}\frac{2\lambda}{\lambda^2+4\pi^2\eta^2}\,\d\eta,\qquad \int_{\mathbb R}\frac{2\lambda}{\lambda^2+4\pi^2\eta^2}\,\d\eta=1.
$$
Multiplication of the kernel of $\mathcal R_\ell^{+,m}$ by $e^{2\pi i\eta(t-s)}$ amounts to multiplying the input by $e^{-2\pi i\eta s}$ and the output by $e^{2\pi i\eta t}$, which leaves the $L^2$ operator norm unchanged. Taking $\lambda=4\pi^2\xi^2$ therefore bounds the operator acting on $\widehat f(\xi,\cdot)$ by $2\sqrt\pi\|\mathcal R^{+,m}\|$, uniformly in $\xi\neq0$ and $\ell$. For $\xi=0$, the same bound follows directly from the definition of $\|\mathcal R^{+,m}\|$. Plancherel's theorem gives
\begin{equation}\label{timeonlyfourierbound}
\|\mathcal S_\ell^m f\|_{L^2(\mathbb P^1)}\leq2\sqrt\pi\|\mathcal R^{+,m}\|\|f\|_{L^2(\mathbb P^1)}.
\end{equation}

We next compare the two truncations. Set $r=2^{-\ell-4}$. Telescoping the sums defining $\Psi_\ell$ gives
$$
\Psi_\ell(x-y,t-s)=1-\psi_0(\rho(x-y,t-s)/r),\qquad \Psi_\ell(t-s)=1-\psi_0(|t-s|^{1/2}/r).
$$
For $t>s$, their difference vanishes unless $t-s\leq4r^2$ and $\rho(x-y,t-s)>r/2$. The Lipschitz bound for $A$ gives
$$
|K^{+,F,m}(x,t,y,s)|\leq\|A\|_{lip}(t-s)^{-3/2}\exp(-\frac{|x-y|^2}{4(t-s)}).
$$
If $t-s<r^2/16$ and the cutoffs differ, then $|x-y|>r/3$. Using the Gaussian decay when $0<t-s<r^2/16$, and the bounds $r^2/16\leq t-s\leq4r^2$ otherwise, we find that the kernel of $\mathcal R_\ell^{+,F,m}-\mathcal S_\ell^m$ is bounded in absolute value by
$$
C\|A\|_{lip}r^{-3}(1+|x-y|/r)^{-4}\chi_{\{0<t-s\leq4r^2\}}.
$$
Integrating the bound in the display above with respect to $(y,s)$, or with respect to $(x,t)$, gives at most $C\|A\|_{lip}$, independently of $\ell$. Schur's test and~\eqref{timeonlyfourierbound} give
\begin{equation}\label{timeonlyextensionbound}
\sup_{\ell\in\mathbb Z}\|\mathcal R_\ell^{+,F,m}\|_{L^2(\mathbb P^1)\to L^2(\mathbb P^1)}\leq C(\|\mathcal R^{+,m}\|+\|A\|_{lip}).
\end{equation}
The kernels on $\Gamma_F$ satisfy the parabolic Calder\'on--Zygmund estimates, so Cotlar's inequality gives the maximal-truncation bounds required in Theorem~\ref{mainthm2}. Applying its case $k=1$ and using~\eqref{timeonlyextensionbound}, we obtain
$$
\|D_t^{1/2}F\|_{BMO(\mathbb P^1)}\leq C(1+\|\mathcal R^+\|).
$$

Since $D_t^{1/2}F(x,t)=D_t^{1/2}A(t)$, its mean oscillation on a parabolic cube equals the mean oscillation of $D_t^{1/2}A$ on the corresponding time interval. Hence
$$
\|D_t^{1/2}A\|_{BMO(\mathbb P^0)}\leq C(1+\|\mathcal R^+\|),
$$
which proves Theorem~\ref{mainthm2} for $k=0$. To obtain parabolic uniform rectifiability, use $b_F$, defined before Theorem~\ref{regularbetaequiv}. Since $F$ is independent of $x$, the best affine approximant on $C'_r(x,t)$ is the constant average of $A$ on $(t-r^2,t+r^2)$. Thus
$$
\beta_2^{\Gamma_A}(C'_r(t))^2=b_F(C'_r(x,t))^2\leq C\beta_2^{\Gamma_F}(C'_r(x,t))^2,
$$
where the inequality is the comparison already proved for $k=1$. Apply the case $k=1$ of Theorem~\ref{regularbetaequiv} to $F$ and integrate over $C'_R(x_0,t_0)$. Since $\beta_2^{\Gamma_A}(C'_r(t))$ is independent of $x$, division by $2R$ gives
$$
\int_0^R\int_{|t-t_0|<R^2}\beta_2^{\Gamma_A}(C'_r(t))^2\,\d t\,\frac{\d r}{r}\leq C(1+\|\mathcal R^+\|^2)R^2.
$$
This proves the required square-function estimate for $\Gamma_A$.

Finally, suppose that $A$ satisfies the support assumption of Theorem~\ref{localdpara}. After translating time, we may assume that $\supp A\subset[-R^2,R^2]$. Then $\|A\|_\infty\leq C\|A\|_{lip}R$. Set $I=(-2R^2,2R^2)$. The integral formula in Definition~\ref{halftime} gives
\begin{equation}\label{timeonlyhalftimetail}
|D_t^{1/2}A(t)|\leq C\|A\|_{lip}R^3|t|^{-3/2},\qquad |t|>2R^2.
\end{equation}
Since $D_t^{1/2}A\in BMO$, \eqref{timeonlyhalftimetail} implies $D_t^{1/2}A\in L^1(\mathbb R)$. The Fourier identity
\[
\widehat{D_t^{1/2}A}(\tau)=|\tau|^{1/2}\widehat A(\tau)
\]
gives
\[
\int_{\mathbb R}D_t^{1/2}A(t)\,\d t=\widehat{D_t^{1/2}A}(0)=0.
\]
Hence $|\fint_I D_t^{1/2}A(t)\,\d t|\leq C\|A\|_{lip}$.
The John--Nirenberg inequality on $I$, together with~\eqref{timeonlyhalftimetail}, now yields
$$
\|D_t^{1/2}A\|_2^2\leq CR^2(\|D_t^{1/2}A\|_{BMO}^2+\|A\|_{lip}^2)\leq C(1+\|\mathcal R^+\|^2)R^2.
$$
Since $D_{par}=D_t^{1/2}$ when $k=0$, this proves Theorem~\ref{localdpara} as well.
\endgroup

\section{\texorpdfstring{$\alpha$-Numbers and Parabolic Singular Integrals}{Alpha-Numbers and Parabolic Singular Integrals}}\label{last}

\subsection{Notation and Setting}

Throughout this section we fix an integer $k\in\{1,\ldots,n-1\}$.  We consider kernels $K\in C^\infty(\mathbb P^n\setminus\{0\};\mathbb R)$ satisfying
\begin{itemize}
    \item

\begin{equation}\label{kernelestimates}
|\nabla^j\partial_t^q K(X,t)|
        \leq
        \frac{C_{j,q}}{\|(X,t)\|^{k+2+j+2q}}
\end{equation}

    for all integers $j,q\geq0$ and all $(X,t)\in\mathbb P^n\setminus\{0\}$;

    \item
    \[
        K(-X,t)=-K(X,t)
    \]
    for all $(X,t)\in\mathbb P^n\setminus\{0\}$;

    \item
    \[
        K(X,-t)=K(X,t)
    \]
    for all $(X,t)\in\mathbb P^n\setminus\{0\}$.
\end{itemize}
Set\[\|K\|_{\mathrm{CZ},4}:=\max_{\substack{j,q\geq 0\\ j+2q\leq 4}} C_{j,q}.\]All kernel-dependent implicit constants in this section may depend on $\|K\|_{\mathrm{CZ},4}$, but on no higher-order derivatives of $K$.

To prove Theorem~\ref{alphatosio} for finite-dimensional vector-valued kernels, it suffices to prove the scalar case. If $K=(K_1,\ldots,K_M)$, then each component $K_r$ satisfies the scalar kernel assumptions above. We apply the estimates to each component and sum. For integers $a\leq m$,
\[
\|T_{a,m}\mu\|_{L^2(\mu;\mathbb R^M)}^2
=\sum_{r=1}^M\|T_{a,m}^r\mu\|_{L^2(\mu)}^2.
\]
The asserted $L^2$ limits also follow componentwise. 

Given a finite signed Radon measure $\nu$ on $\mathbb P^n$ and a kernel $K$ as above, we define
\[
    T\nu(X,t)
    :=
    \int K\big((X,t)-(Y,s)\big)\,\d\nu(Y,s),
    \qquad
    (X,t)\in\mathbb P^n\setminus\supp(\nu).
\]
On $\supp(\nu)$ we use the annular truncations
$T_{(j)}\nu$ and their finite sums $T_{a,m}\nu$ introduced in
Subsection~\ref{dyadicgriddecomp}. For an ADR measure $\mu$, the one-sided
truncations $T_m\mu$ are interpreted by the limiting convention stated there.  

We say that a nonzero locally finite Borel measure $\mu$ on $\mathbb P^n$, whose support has positive parabolic diameter, is $(k+2)$-ADR if there is a constant $C_0\geq1$ such that
\[
    C_0^{-1}r^{k+2}
    \leq
    \mu(C_r(X,t))
    \leq
    C_0r^{k+2}
\]
for every $(X,t)\in\supp(\mu)$ and every $0<r\leq\operatorname{diam}(\supp(\mu))$.

\subsection{Dyadic Grids}
For the remainder of this section, we fix a $(k+2)$-ADR measure $\mu$ on $\mathbb P^n$, with ADR constant $C_0$. For a set $A\subset\mathbb P^n$ and $s>0$, let $U_s(A):=\{(X,t)\in\mathbb P^n:\dist((X,t),A)<s\}$. We write
\[
    E:=\supp(\mu).
\]
Let $0<\gamma<2^{-11}$ be a constant in the dyadic construction below. If $\operatorname{diam}(E)<\infty$, choose $j_0\in\mathbb Z$ so that
\[
2^{10}\operatorname{diam}(E)\leq2^{-j_0}<2^{11}\operatorname{diam}(E),
\]
set $\mathcal J:=\{j\in\mathbb Z:j\geq j_0\}$, and take $\mathbb D_{j_0}=\{E\}$. If $\operatorname{diam}(E)=\infty$, set $\mathcal J:=\mathbb Z$. We construct the dyadic grid only for generations $j\in\mathcal J$. We will assume that the Calder\'on--Zygmund kernels we consider satisfy \eqref{kernelestimates}.

All distances and diameters in this subsection are taken with respect to the parabolic metric. The constants implicit in $\lesssim$ and $\approx$, as well as $\gamma$, depend only on $n$, $k$, and $C_0$. Using the tools developed in \cite{Chr,DS1,DS2}, we decompose $E$ into dyadic cubes.  Specifically, for each $j\in\mathcal J$, we can construct a family $\mathbb D_j$ of Borel subsets of $E$, called the cubes of the $j$-th generation, such that
\begin{itemize}
    \item $\mathbb D_j$ is a partition of $E$, i.e.
    \[
        E=\bigcup_{Q\in\mathbb D_j}Q,
    \]
    and if $Q,Q'\in\mathbb D_j$ with $Q\neq Q'$, then
    \[
        Q\cap Q'=\emptyset;
    \]

    \item if $Q\in\mathbb D_j$ and $Q'\in\mathbb D_m$, with $m\leq j$, then either
    \[
        Q\subseteq Q'
        \qquad\text{or}\qquad
        Q\cap Q'=\emptyset;
    \]

    \item if $Q\in\mathbb D_j$, then

    \[
        \gamma 2^{-j}\leq \diam(Q)\leq 2^{-j-10}
        \qquad\text{and}\qquad
        \mu(Q)\approx 2^{-j(k+2)}.
    \]

    \item if $Q\in\mathbb D_j$, there is a point $(Z_Q,\tau_Q)\in Q$, called the center of $Q$, such that
    \[
        \dist\big((Z_Q,\tau_Q),E\setminus Q\big)
        \gtrsim
        2^{-j}.
    \]
\end{itemize}
We denote
\[
    \mathbb D=\mathbb D(\mu):=\bigcup_{j\in\mathcal J}\mathbb D_j.
\]
We regard a dyadic cube as a tagged pair $(Q,j)$ with $Q\in\mathbb D_j$; thus identical Borel subsets occurring in different generations are distinct elements of $\mathbb D$. If $j-1\in\mathcal J$ and $Q\in\mathbb D_j$, the unique cube $\widetilde Q\in\mathbb D_{j-1}$ containing $Q$ is called the parent of $Q$. We say that two cubes $Q$ and $P$ are siblings if they belong to the same generation and have the same parent. If $Q$ is from generation $j$, we write $J(Q)=j$. For $R\in\mathbb D$, we let
\[
    \mathbb D(R):=\{Q\in\mathbb D:Q\subseteq R,\ J(Q)\geq J(R)\}.
\]
Thus $\mathbb D(R)$ is the collection of dyadic descendants of $R$, including $R$ itself.
If $J(Q)=j$, we write
\[
    \ell(Q):=2^{-j}.
\]
We also define dilates of cubes.  Given $\lambda\geq1$ and $R\in\mathbb D$, we write
\[
\lambda R
:=
\left\{
(X,t)\in E:
\|(X,t)-(Z_R,\tau_R)\|
\leq
\lambda\diam(R)
\right\}.
\]

\begin{remark}
In the graph case considered earlier, the lifted cubes $\{Q^\Gamma\}$ form a dyadic grid satisfying the above requirements after relabeling the generations. With the original indexing, cubes in the $j$-th generation may fail to satisfy
\[
    \gamma 2^{-j}\leq \diam(Q^\Gamma)\leq 2^{-j-10}.
\]
We obtain this inequality by shifting the generation indices by a fixed integer.
\end{remark}

\subsection{\texorpdfstring{$\alpha$-Numbers With Respect to Cubes}{Alpha-Numbers With Respect to Cubes}}
Given a dyadic cube $Q\in\mathbb D$, we define
\[
    C_Q:=C_{3\ell(Q)}(Z_Q,\tau_Q).
\]
Given two locally finite Borel measures $\nu$ and $\sigma$ on $\mathbb P^n$, we set
\[
    \dist_{C_Q}(\nu,\sigma)
    :=
    \sup
    \left\{
        \left|\int f\,\d\nu-\int f\,\d\sigma\right|:
        \|f\|_{\operatorname{Lip}}\leq1,
        \supp(f)\subseteq C_Q
    \right\},
\]
where the Lipschitz norm is taken with respect to the parabolic distance. We define
\[
    \alpha(Q)
    :=
    \frac{1}{\ell(Q)^{k+3}}
    \inf_{c\geq0,L}
    \dist_{C_Q}\big(\mu,c\,\sigma_L\big),
\]
where $L$ ranges over all parabolic $(k+2)$-planes containing a line parallel to the $t$-axis. Equivalently, the projection of $L$ onto the spatial component of $\mathbb P^n=\mathbb R^n_{sp}\times\mathbb R$ is a $k$-dimensional affine plane. Fix a sufficiently small constant $\eta_*>0$, depending only on $n$, $k$, and $C_0$. If $\alpha(Q)\leq\eta_*$, choose $c_Q\geq0$ and $L_Q$ realizing the infimum above. If $\alpha(Q)>\eta_*$, set $c_Q=0$ and choose a vertical plane $L_Q$ through $(Z_Q,\tau_Q)$. In either case, write
\[
    \mathcal L_Q:=c_Q\,\sigma_{L_Q}.
\]
We also let $(X_Q,t_Q)$ denote the projection of $(Z_Q,\tau_Q)$ onto $L_Q$.

\begin{remark}\label{cqintlq}
We will always choose $L_Q$ so that
\[
    L_Q\cap C_Q\neq\emptyset.
\]
Indeed, if $L_Q\cap C_Q=\emptyset$, then the support of $c_Q\sigma_{L_Q}$ is disjoint from $C_Q$. Thus, for the purposes of $\dist_{C_Q}$, this choice is equivalent to taking $c_Q=0$, and we may then move $L_Q$ so that it intersects $C_Q$.
\end{remark}

We will need the following version of Lemma 3.1 of \cite{To1}.

\begin{lemma}\label{elemalphalem}
For all $Q\in\mathbb D$, we have:
\begin{itemize}
    \item[(a)] $\alpha(Q)\lesssim1$.
    \item[(b)] If $P\subset Q$ and $\ell(P)\approx\ell(Q)$, then $\alpha(P)\lesssim\alpha(Q)$.
    \item[(c)] There exists $0<\eta_0\leq\eta_*$ such that if $\alpha(Q)\leq \eta_0$, then
    \[
        L_Q\cap C_{\diam Q}(Z_Q,\tau_Q)\neq\emptyset
        \qquad\text{and}\qquad
        c_Q\approx1.
    \]
    \item[(d)] $c_Q\lesssim1$ for all $Q\in\mathbb D$.
\end{itemize}
\end{lemma}
\begin{proof}
Parts (a)--(c) follow from the proof of Lemma 3.1 in \cite{To1}, with $\eta_*$ chosen sufficiently small. That argument gives $c_Q\leq C$ whenever $\alpha(Q)\leq\eta_*$, with $C$ depending only on $n$, $k$, and $C_0$. For $\alpha(Q)>\eta_*$, our choice gives $c_Q=0$. This proves (d), independently of any later decrease of $\eta_0$.
\end{proof}

Our choices also give
\begin{equation}\label{alphaflatapproximation}
\dist_{C_Q}(\mu,\mathcal L_Q)\leq C\alpha(Q)\ell(Q)^{k+3}.
\end{equation}
For $\alpha(Q)\leq\eta_*$, this follows from the choice of $c_Q$ and $L_Q$. Otherwise, $\mathcal L_Q=0$, and every admissible test function has absolute value at most $3\ell(Q)$. The ADR upper bound therefore gives $\dist_{C_Q}(\mu,0)\leq C\ell(Q)^{k+3}$; since $\alpha(Q)>\eta_*$, this proves \eqref{alphaflatapproximation}. The constant in \eqref{alphaflatapproximation} depends only on $n$, $k$, $C_0$, and the fixed choice of $\eta_*$, and is independent of $Q$.

We define the $\beta_1$-numbers and their bilateral variant as follows. For $Q\in\mathbb D$, set
\begin{align}\label{b1def}
\beta_1(Q)
:=
\frac{1}{\ell(Q)^{k+2}}
\inf_L
\int_{C_{2\ell(Q)}(Z_Q,\tau_Q)}
\frac{\dist\big((Y,s),L\big)}{\ell(Q)}
\,\d\mu(Y,s),
\end{align}
and
\begin{align}\label{bb1def}
b\beta_1(Q)
:=
\frac{1}{\ell(Q)^{k+2}}
\inf_L
\bigg[
&\int_{C_{2\ell(Q)}(Z_Q,\tau_Q)}
\frac{\dist\big((Y,s),L\big)}{\ell(Q)}
\,\d\mu(Y,s)
\\
&+
\int_{C_{2\ell(Q)}(Z_Q,\tau_Q)}
\frac{\dist\big((X,t),\supp(\mu)\big)}{\ell(Q)}
\,\d\sigma_L(X,t)
\bigg],
\end{align}
where the infimum is taken over the same class of parabolic $(k+2)$-planes $L$ as in the definition of $\alpha(Q)$.

\begin{lemma}\label{betalphacomp}
For $Q\in\mathbb D$, we have
\[
    \beta_1(Q)\leq b\beta_1(Q)\lesssim\alpha(Q).
\]
Moreover,
\begin{equation}\label{LQbetacomp}
\begin{aligned}
\frac{1}{\ell(Q)^{k+2}}
\bigg[
&\int_{C_{2\ell(Q)}(Z_Q,\tau_Q)}
\frac{\dist\big((X,t),L_Q\big)}{\ell(Q)}
\,\d\mu(X,t)
\\
&+
\int_{C_{2\ell(Q)}(Z_Q,\tau_Q)}
\frac{\dist\big((X,t),\supp(\mu)\big)}{\ell(Q)}
\,\d\sigma_{L_Q}(X,t)
\bigg]
\lesssim
\alpha(Q).
\end{aligned}
\end{equation}
\end{lemma}

\begin{proof}
The inequality $\beta_1(Q)\leq b\beta_1(Q)$ is immediate. We prove $b\beta_1(Q)\lesssim\alpha(Q)$ and \eqref{LQbetacomp} as follows.

Let $\eta_0>0$ be the small constant from Lemma~\ref{elemalphalem}(c). If $\alpha(Q)\geq\eta_0$, then the desired estimates follow from the ADR property, and the fact that $L_Q\cap C_Q\neq\emptyset$. Indeed, the quantities appearing in the definitions of $\beta_1(Q)$, $b\beta_1(Q)$, and in \eqref{LQbetacomp} are bounded by a dimensional constant, and therefore by $C\eta_0^{-1}\alpha(Q)$.

It remains to consider the case $\alpha(Q)<\eta_0$. Then Lemma~\ref{elemalphalem}(c) implies that
\[
    c_Q\approx1.
\]
Let $\varphi$ be a smooth function such that
\[
    \chi_{C_{2\ell(Q)}(Z_Q,\tau_Q)}
    \leq
    \varphi
    \leq
    \chi_{C_Q},
    \qquad
    \|\varphi\|_{\operatorname{Lip}}\lesssim \ell(Q)^{-1}.
\]
Define
\[
    f(X,t)
    :=
    \big[
        \dist\big((X,t),L_Q\big)
        -
        \dist\big((X,t),\supp(\mu)\big)
    \big]\varphi(X,t).
\]
Since both distance functions are $1$-Lipschitz and the bracketed term is bounded by $C\ell(Q)$ on $\supp\varphi$, we have
\[
    \|f\|_{\operatorname{Lip}}\lesssim1.
\]
Hence, by the definition of $\dist_{C_Q}$,
\begin{align*}
\dist_{C_Q}(\mu,\mathcal L_Q)
&\gtrsim
\left|
    \int f\,\d\mu
    -
    c_Q\int f\,\d\sigma_{L_Q}
\right|
\\
&=
\int \varphi(X,t)\dist\big((X,t),L_Q\big)\,\d\mu(X,t)
\\
&\qquad
+
 c_Q
\int \varphi(X,t)\dist\big((X,t),\supp(\mu)\big)
\,\d\sigma_{L_Q}(X,t).
\end{align*}
Here we used that
\[
    \dist\big((X,t),\supp(\mu)\big)=0
    \quad \mu\text{-a.e.},
\]
and
\[
    \dist\big((X,t),L_Q\big)=0
    \quad \sigma_{L_Q}\text{-a.e.}
\]
Since $c_Q\approx1$ and $\varphi\equiv1$ on $C_{2\ell(Q)}(Z_Q,\tau_Q)$, we obtain
\begin{align*}
&\int_{C_{2\ell(Q)}(Z_Q,\tau_Q)}
\dist\big((X,t),L_Q\big)\,\d\mu(X,t)
\\
&\qquad
+
\int_{C_{2\ell(Q)}(Z_Q,\tau_Q)}
\dist\big((X,t),\supp(\mu)\big)
\,\d\sigma_{L_Q}(X,t)
\\
&\lesssim
\dist_{C_Q}(\mu,\mathcal L_Q)
\lesssim
\alpha(Q)\ell(Q)^{k+3}.
\end{align*}
Dividing by $\ell(Q)^{k+3}$ gives \eqref{LQbetacomp}. Since $b\beta_1(Q)$ is defined by taking the infimum over $L$, \eqref{LQbetacomp} also implies
\[
    b\beta_1(Q)\lesssim\alpha(Q).
\]
This completes the proof.
\end{proof}

We record the following lemma, which is a parabolic version of Lemma 3.4 in \cite{To1}.

\begin{lemma}\label{toalphalemma}
Let $P,Q \in \mathbb D$ be dyadic cubes such that $P \subset Q$, with
$\eta\ell(Q) \leq \ell(P) \leq \ell(Q)$ for some fixed $\eta>0$. Then
\begin{align}\label{lqlpaq}
\dist_H\big( L_P \cap C_Q, L_Q \cap C_Q \big)
\leq C(\eta)\alpha(Q)\ell(Q),
\end{align}
where $\dist_H$ denotes the parabolic Hausdorff distance. Moreover,
\[
    |c_P-c_Q|\leq C(\eta)\alpha(Q).
\]
\end{lemma}

\begin{proof}
We follow the proof of Lemma 3.4 in \cite{To1}.

Let $\eta_0>0$ be the small constant in
Lemma~\ref{elemalphalem}(c). We may assume that
\[
    \alpha(P),\alpha(Q)\leq\eta_0.
\]
Indeed, by Lemma~\ref{elemalphalem}(b), if either of these inequalities
fails, then $\alpha(Q)\gtrsim_{\eta} \eta_0$. In that case the conclusions follow
from Lemma~\ref{elemalphalem}(d), the fact that
\[
    L_P\cap C_P\neq\emptyset,
    \qquad
    L_Q\cap C_Q\neq\emptyset,
\]
and the containment $C_P\subset C_Q$.

We first prove \eqref{lqlpaq}. Let $\varphi_P$ be a smooth function such
that
\[
    \chi_{C_{2\ell(P)}(Z_P,\tau_P)}
    \leq
    \varphi_P
    \leq
    \chi_{C_P},
    \qquad
    \|\varphi_P\|_{\operatorname{Lip}}
    \lesssim\ell(P)^{-1}.
\]
Since $\ell(P)\geq\eta\ell(Q)$ and $C_P\subset C_Q$, we have
\[
    \big\|
        \varphi_P\dist(\,\cdot\,,L_Q)
    \big\|_{\operatorname{Lip}}
    \lesssim_\eta1.
\]
Therefore,
\begin{align*}
\bigg|
    \int
        \varphi_P(X,t)
        \dist\big((X,t),L_Q\big)
        \,\d\mu(X,t)
    &-
    \int
        \varphi_P(X,t)
        \dist\big((X,t),L_Q\big)
        \,\d\mathcal L_P(X,t)
\bigg|
\\
&\lesssim_\eta
\alpha(P)\ell(P)^{k+3}.
\end{align*}
Moreover, since
$\varphi_P\dist(\,\cdot\,,L_Q)$ vanishes on $L_Q$, we have (by \eqref{LQbetacomp})
\[
    \int
        \varphi_P(X,t)
        \dist\big((X,t),L_Q\big)
        \,\d\mu(X,t)
    \lesssim_\eta
    \alpha(Q)\ell(Q)^{k+3}.
\]
Taking into account that
\[
    \alpha(P)\lesssim_\eta\alpha(Q),
\]
we obtain
\[
    \int_{C_{2\ell(P)}(Z_P,\tau_P)}
        \dist\big((X,t),L_Q\big)
        \,\d\mathcal L_P(X,t)
    \lesssim_\eta
    \alpha(Q)\ell(Q)^{k+3}.
\]
By Lemma~\ref{elemalphalem}(c),
\[
    c_P\approx1
    \qquad\text{and}\qquad
    L_P\cap C_{\diam(P)}(Z_P,\tau_P)\neq\emptyset.
\]
The preceding estimate then implies
\[
    \dist_H\big(
        L_P\cap C_Q,
        L_Q\cap C_Q
    \big)
    \lesssim_\eta
    \alpha(Q)\ell(Q),
\]
which proves \eqref{lqlpaq}.

We now prove the estimate for the constants. Using the same function
$\varphi_P$, we have
\[
\left|
    \int\varphi_P\,\d\mu
    -
    c_P\int\varphi_P\,
        \d\sigma_{L_P}
\right|
\lesssim
\alpha(P)\ell(P)^{k+2}
\lesssim_\eta
\alpha(Q)\ell(Q)^{k+2}.
\]
Likewise,
\[
\left|
    \int\varphi_P\,\d\mu
    -
    c_Q\int\varphi_P\,
        \d\sigma_{L_Q}
\right|
\lesssim_\eta
\alpha(Q)\ell(Q)^{k+2}.
\]
Consequently,
\[
\left|
    c_P\int\varphi_P\,
        \d\sigma_{L_P}
    -
    c_Q\int\varphi_P\,
        \d\sigma_{L_Q}
\right|
\lesssim_\eta
\alpha(Q)\ell(Q)^{k+2}.
\]
It follows that
\begin{align*}
|c_Q-c_P|
\int\varphi_P\,
    \d\sigma_{L_P}
&\lesssim_\eta
\alpha(Q)\ell(Q)^{k+2}
\\
&\quad
+
c_Q
\left|
    \int\varphi_P\,
        \d\sigma_{L_P}
    -
    \int\varphi_P\,
        \d\sigma_{L_Q}
\right|.
\end{align*}
Using \eqref{lqlpaq}, we obtain
\[
\left|
    \int\varphi_P\,
        \d\sigma_{L_P}
    -
    \int\varphi_P\,
        \d\sigma_{L_Q}
\right|
\lesssim_\eta
\alpha(Q)\ell(Q)^{k+2}.
\]
Finally, Lemma~\ref{elemalphalem}(d) gives $c_Q\lesssim1$, while
\[
    \int\varphi_P\,
        \d\sigma_{L_P}
    \approx
    \ell(P)^{k+2}
    \approx_\eta
    \ell(Q)^{k+2}.
\]
Thus,
\[
    |c_P-c_Q|
    \lesssim_\eta
    \alpha(Q),
\]
as desired.
\end{proof}

Given a cube $Q \in \mathbb D$, we define the neighbors of $Q$ by
\begin{align}\label{neighborsdef}
N(Q)
:=
\left\{
T \in \mathbb D_{J(Q)}:
\dist(T,Q) \leq \gamma\ell(Q)
\text{ and } T \neq Q
\right\}.
\end{align}

\begin{lemma}\label{toalphalemma2} Suppose $Q\in\mathbb D$ and $T\in N(Q)$. Then \[ \dist_H\Big( L_Q\cap C_{6\ell(Q)}(Z_Q,\tau_Q), L_T\cap C_{6\ell(Q)}(Z_Q,\tau_Q) \Big) \lesssim \max\{\alpha(Q),\alpha(T)\}\ell(Q). \]
In particular, if $\max\{\alpha(Q),\alpha(T)\} \leq \eta_0$, where $\eta_0$ is as in Lemma \ref{elemalphalem} (c), then 
\[
    \dist_H\big(L_Q\cap C_Q,L_T\cap C_Q\big)
    \lesssim
    \max\{\alpha(Q),\alpha(T)\}\ell(Q).
\]
\end{lemma}

\begin{proof}
Set $r:=\ell(Q)=\ell(T)$ and $\delta:=\max\{\alpha(Q),\alpha(T)\}$.   We first consider the case $\delta > \eta_0$. Choose $W_Q\in L_Q\cap C_Q$ and $W_T\in L_T\cap C_T$.
Since $T\in N(Q)$,
\[
\begin{aligned}
d_p\big((Z_Q,\tau_Q),(Z_T,\tau_T)\big)
&\leq
\diam(Q)+\dist(Q,T)+\diam(T)
\\
&\leq
(2^{-9}+\gamma)r.
\end{aligned}
\]
Consequently, $W_Q,W_T\in C_{6r}(Z_Q,\tau_Q)$.
For any $W\in L_Q\cap C_{6r}(Z_Q,\tau_Q)$,
we have
\[
    \dist\Big(
        W,
        L_T\cap C_{6r}(Z_Q,\tau_Q)
    \Big)
    \leq d_p(W,W_T)
    \lesssim r.
\]
Similarly,
\[
    \sup_{W\in L_T\cap C_{6r}(Z_Q,\tau_Q)}
    \dist\Big(
        W,
        L_Q\cap C_{6r}(Z_Q,\tau_Q)
    \Big)
    \lesssim r.
\]
It follows that
\[
    \dist_H\Big(
        L_Q\cap C_{6r}(Z_Q,\tau_Q),
        L_T\cap C_{6r}(Z_Q,\tau_Q)
    \Big)
    \lesssim r.
\]
Since $\delta > \eta_0$, $r < \eta_0^{-1}\delta r$,
and hence
\[
    \dist_H\Big(
        L_Q\cap C_{6r}(Z_Q,\tau_Q),
        L_T\cap C_{6r}(Z_Q,\tau_Q)
    \Big)
    \lesssim
    \delta r.
\]

Suppose now that $\delta\leq\eta_0$. Since $T\in N(Q)$, the preceding center estimate gives
\[
d_p\big((Z_Q,\tau_Q),(Z_T,\tau_T)\big)\leq(2^{-9}+\gamma)r.
\]
By the choice of $\gamma$, this implies
\[
    C_{\frac52r}(Z_Q,\tau_Q)\subset C_Q\cap C_T.
\] 
Hence, by the ADR property,
\[
    \mu(C_Q\cap C_T)\gtrsim \ell(Q)^{k+2}.
\]
Using \eqref{LQbetacomp} for both $Q$ and $T$, we get
\[
\frac{1}{r^{k+2}}
\int_{C_{2r}(Z_Q,\tau_Q)\cap C_{2r}(Z_T,\tau_T)}
\frac{\dist\big((X,t),L_Q\big)}{\ell(Q)}
\,\d\mu(X,t)
\lesssim
\alpha(Q),
\]
and
\[
\frac{1}{r^{k+2}}
\int_{C_{2r}(Z_Q,\tau_Q)\cap C_{2r}(Z_T,\tau_T)}
\frac{\dist\big((X,t),L_T\big)}{\ell(Q)}
\,\d\mu(X,t)
\lesssim
\alpha(T).
\]

Choose a Lipschitz function $\varphi$ such that
\[
    0\leq\varphi\leq1,
    \qquad
    \varphi=1
    \quad\text{on }C_{2r}(Z_Q,\tau_Q),
\]
\[
    \supp\varphi
    \subset C_{\frac52r}(Z_Q,\tau_Q),
    \qquad
    \|\varphi\|_{\operatorname{Lip}}\lesssim r^{-1}.
\]

Since $\alpha(Q),\alpha(T)<\eta_0$,
Lemma~\ref{elemalphalem}(c) gives
\[
    c_Q\approx c_T\approx1,
\]
\[
    L_Q\cap C_{\diam(Q)}(Z_Q,\tau_Q)\neq\emptyset,
    \qquad
    L_T\cap C_{\diam(T)}(Z_T,\tau_T)\neq\emptyset.
\]
In particular, after taking $\gamma$ sufficiently small,
\[
    L_Q\cap C_{r/4}(Z_Q,\tau_Q)\neq\emptyset,
    \qquad
    L_T\cap C_{r/4}(Z_Q,\tau_Q)\neq\emptyset.
\]
Define
\[
    f_T(X,t)
    :=
    \varphi(X,t)\dist\big((X,t),L_T\big).
\]
The distance function to $L_T$ is $1$-Lipschitz. Moreover, since
$L_T$ meets $C_{r/4}(Z_Q,\tau_Q)$ and
\[
    \supp\varphi\subset C_{\frac52r}(Z_Q,\tau_Q),
\]
we have
\[
    \sup_{\supp\varphi}
    \dist\big((X,t),L_T\big)
    \lesssim r.
\]
Consequently,
\begin{align*}
\|f_T\|_{\operatorname{Lip}}
&\lesssim
\|\varphi\|_{\infty}
\big\|\dist(\,\cdot\,,L_T)\big\|_{\operatorname{Lip}}+
\|\varphi\|_{\operatorname{Lip}}
\sup_{\supp\varphi}
\dist(\,\cdot\,,L_T)\lesssim
1+r^{-1}r
\lesssim1.
\end{align*}
Since $f_T$ vanishes on $L_T$, the definitions of $\alpha(Q)$ and
$\alpha(T)$ give
\begin{align*}
\int f_T\,\d\mathcal L_Q
&\leq
\left|
    \int f_T\,\d\mathcal L_Q
    -
    \int f_T\,\d\mu
\right|
+
\left|
    \int f_T\,\d\mu
    -
    \int f_T\,\d\mathcal L_T
\right|
\\
&\lesssim
\alpha(Q)r^{k+3}
+
\alpha(T)r^{k+3}
\\
&\lesssim
\delta r^{k+3}.
\end{align*}
Since
\[
    \mathcal L_Q
    =
    c_Q\sigma_{L_Q},
    \qquad
    c_Q\approx1,
\]
and $\varphi=1$ on $C_{2r}(Z_Q,\tau_Q)$, it follows that
\begin{equation}\label{LQLTintegralcomparison}
\int_{L_Q\cap C_{2r}(Z_Q,\tau_Q)}
    \dist\big((X,t),L_T\big)
    \,\d\sigma_{L_Q}(X,t)
\lesssim
\delta r^{k+3}.
\end{equation}

Applying the same argument to
\[
    f_Q(X,t)
    :=
    \varphi(X,t)\dist\big((X,t),L_Q\big)
\]
gives
\begin{equation}\label{LTLQintegralcomparison}
\int_{L_T\cap C_{2r}(Z_Q,\tau_Q)}
    \dist\big((X,t),L_Q\big)
    \,\d\sigma_{L_T}(X,t)
\lesssim
\delta r^{k+3}.
\end{equation}
Therefore,
\[
    \sup_{(X,t)\in L_Q\cap C_Q}
        \dist\big((X,t),L_T\cap C_Q\big)
    \lesssim
    \delta r,
\]
and
\[
    \sup_{(X,t)\in L_T\cap C_Q}
        \dist\big((X,t),L_Q\cap C_Q\big)
    \lesssim
    \delta r.
\]
Hence
\begin{equation}\label{smallneighborcomparison}
    \dist_H\big(L_Q\cap C_Q,L_T\cap C_Q\big)
    \lesssim
    \delta r,
\end{equation}
and thus
\[
    \dist_H\Big(
        L_Q\cap C_{6r}(Z_Q,\tau_Q),
        L_T\cap C_{6r}(Z_Q,\tau_Q)
    \Big)
    \lesssim
    \delta r.
\]
\end{proof}

\begin{lemma}\label{smallanglelemma}
Suppose $Q\in\mathbb D$ and $T\in N(Q)$. There exists $\varepsilon_0>0$ such that, if
\[
    \max\{\alpha(Q),\alpha(T)\}\leq\varepsilon_0,
\]
then the angle between $L_Q$ and $L_T$ is less than $\pi/4$.
\end{lemma}

\begin{proof}
By Lemma~\ref{toalphalemma2},
\[
\dist_H\Big(L_Q\cap C_{6\ell(Q)}(Z_Q,\tau_Q),L_T\cap C_{6\ell(Q)}(Z_Q,\tau_Q)\Big)
\lesssim
\max\{\alpha(Q),\alpha(T)\}\ell(Q).
\]
 Choosing $\varepsilon_0$ sufficiently small completes the proof.
\end{proof}

By decreasing the constant $\eta_0$ from
Lemma~\ref{elemalphalem}(c) if necessary, we henceforth assume that
$\eta_0\leq\varepsilon_0$, where $\varepsilon_0$ is as in
Lemma~\ref{smallanglelemma}.

\subsection{\texorpdfstring{Decomposing $T\mu$ With Respect to the Dyadic Grid}{Decomposing T mu With Respect to the Dyadic Grid}}\label{dyadicgriddecomp}

Set
\[
    \rho(X,t):=\big(|X|^4+t^2\big)^{1/4}.
\]
Then $\rho$ is smooth away from the origin, $\rho(\lambda X,\lambda^2 t) = \lambda \rho(X,t)$, and (recalling that $\|(X,t)\| = |X| + |t|^{1/2}$)
\[
    2^{-3/4}\|(X,t)\|\leq \rho(X,t)\leq \|(X,t)\|.
\]
Let $\psi_0:[0,\infty)\to[0,1]$ be a smooth non-increasing function such that
\[
    \chi_{[0,1/2]}\leq \psi_0\leq \chi_{[0,2]}.
\]
For each $j\in\mathbb Z$, define
\[
    \psi_j(X,t):=\psi_0(2^j\rho(X,t)).
\]
Since $\psi_0$ is constant near the origin, each $\psi_j$ is smooth on all of
$\mathbb P^n$.

Fix $\omega\in\mathbb N$ sufficiently large so that $2^{-\omega}\leq\gamma/32$.
Set
\[
    \widetilde\varphi_j:=\psi_{j+\omega}-\psi_{j+1+\omega}.
\]
Then $\widetilde\varphi_j\geq0$, and
\[
    \supp \widetilde\varphi_j
    \subset
    \left\{
        (X,t):
        2^{-j-\omega-2}\leq \rho(X,t)\leq 2^{-j-\omega+1}
    \right\}.
\]
Using the comparability of $\rho$ and $\|\cdot\|$, we have
\[
    \supp \widetilde\varphi_j
    \subset
    A\bigl(0,2^{-j-\omega-2}, 2^{-j-\omega+2}\bigr).
\]

Hence if $j\in\mathcal J$, $(X,t)\in Q\in\mathbb D_j$, $(Y,s)\in E$, and
\[
    \widetilde\varphi_j\big((X,t)-(Y,s)\big)\neq0,
\]
then the dyadic cube of generation $j$ containing $(Y,s)$ is either $Q$ itself or
belongs to $N(Q)$.

We normalize the functions
$\widetilde\varphi_j$ by setting
\[
    \varphi_j(X,t)
    :=
    \frac{\widetilde\varphi_j(X,t)}
    {\sum_{\ell\in\mathbb Z}\widetilde\varphi_\ell(X,t)}.
\]
The denominator is bounded above and below by positive constants on
$\mathbb P^n\setminus\{0\}$, and hence
\[
    \sum_{j\in\mathbb Z}\varphi_j(X,t)=1,
    \qquad (X,t)\in\mathbb P^n\setminus\{0\}.
\]
The functions $\varphi_j$ satisfy
\[
    \supp \varphi_j
    \subset
    A\bigl(0,2^{-j-\omega-2},2^{-j-\omega+2}\bigr).
\]
They also satisfy the bounds
\[
    |\varphi_j|
    \lesssim
    \chi_{A(0,2^{-j-\omega-2},2^{-j-\omega+2})},
\]
\[
    |\nabla\varphi_j|
    \lesssim
    2^j\chi_{A(0,2^{-j-\omega-2},2^{-j-\omega+2})},
    \qquad
    |\partial_t\varphi_j|
    \lesssim
    2^{2j}\chi_{A(0,2^{-j-\omega-2},2^{-j-\omega+2})}.
\]

Set
\[
    \Psi_j:=\sum_{k\leq j}\varphi_k.
\]
For a kernel $K$ as above and $(X,t)\neq0$, set
\[
    K_{(j)}(X,t):=\varphi_j(X,t)K(X,t),
    \qquad
    K_j(X,t):=\Psi_j(X,t)K(X,t).
\]
We then define
\[
    T_{(j)}\mu(X,t)
    :=
    \int K_{(j)}\big((X,t)-(Y,s)\big)\,\d\mu(Y,s).
\]
For integers $a\leq m$, set
\[
T_{a,m}\mu:=\sum_{j=a}^{m}T_{(j)}\mu.
\]
Each annular integral is absolutely convergent, since its kernel is
bounded and supported in a bounded annulus away from the origin.
When $\sum_{Q\in\mathbb D}\alpha(Q)^2\mu(Q)<\infty$, we define
\[
T_m\mu:=\lim_{a\to-\infty}T_{a,m}\mu
\quad\text{in }L^2(\mu).
\]
The existence of this limit will follow from the estimates below.
If $E$ is bounded, the finite sums eventually stabilize on $E$ as
$a\to-\infty$, and their limit agrees there with the absolutely convergent
integral against $K_m$.

For integers $a\leq m$, write
\[
T_{a,m}\mu
=
\sum_{\substack{a\leq j\leq m\\ j\notin\mathcal J}}T_{(j)}\mu
+
\sum_{\substack{a\leq j\leq m\\ j\in\mathcal J}}T_{(j)}\mu,
\]
and set
\[ T_{a,m}^{\mathcal J}\mu
:=
\sum_{\substack{a\leq j\leq m\\ j\in\mathcal J}}T_{(j)}\mu.\]
When $\operatorname{diam}(E)=\infty$, we have $\mathcal J=\mathbb Z$, so the first sum is zero. Suppose now that $\operatorname{diam}(E)<\infty$, and let $Q_0=E$ denote the cube in $\mathbb D_{j_0}$. We claim that
\[
\alpha(Q_0)\gtrsim1.
\]
This is immediate if $\alpha(Q_0)\geq\eta_0$. Otherwise, Lemma~\ref{elemalphalem}(c) gives $c_{Q_0}\approx1$ and
\[
L_{Q_0}\cap C_{\operatorname{diam}(E)}(Z_{Q_0},\tau_{Q_0})\neq\emptyset.
\]
Since $\ell(Q_0)\geq2^{10}\operatorname{diam}(E)$, we may choose $W\in L_{Q_0}\cap C_{Q_0}$ such that
\[
\dist(W,E)\geq2\operatorname{diam}(E),\qquad
C_{\operatorname{diam}(E)}(W)\subset C_{Q_0}.
\]
Testing the definition of $\alpha(Q_0)$ with
\[
f(V)=\big(\operatorname{diam}(E)-\dist(V,W)\big)_+
\]
gives $\alpha(Q_0)\ell(Q_0)^{k+3}\gtrsim\operatorname{diam}(E)^{k+3}$, and hence
\begin{equation}\label{topalphamass}
\mu(E)\lesssim\alpha(Q_0)^2\mu(Q_0)
\leq
\sum_{Q\in\mathbb D}\alpha(Q)^2\mu(Q).
\end{equation}

For $j<j_0$, the support of $\varphi_j$ shows that $T_{(j)}\mu$ can be nonzero on $E$ only if
\[
2^{-j-\omega-2}\leq\operatorname{diam}(E)\leq2^{-j_0-10}.
\]
Thus there are at most $\omega$ such indices. Put $r_j=2^{-j-\omega}$. For each of these indices, the size estimate for $K$ gives
\[
|T_{(j)}\mu(X,t)|
\lesssim
r_j^{-(k+2)}
\mu\big(E\cap C_{Cr_j}(X,t)\big)
\lesssim1,
\qquad (X,t)\in E.
\]
Here the last inequality follows from ADR when $Cr_j\leq\operatorname{diam}(E)$, and from $\mu(E)\approx\operatorname{diam}(E)^{k+2}$ otherwise. Therefore, uniformly in integers $a\leq m$,
\begin{equation}\label{outsideJbound}
\left\|
\sum_{\substack{a\leq j\leq m\\ j\notin\mathcal J}}T_{(j)}\mu
\right\|_{L^2(\mu)}^2
\lesssim
\mu(E)
\lesssim
\sum_{Q\in\mathbb D}\alpha(Q)^2\mu(Q).
\end{equation}

In proving Theorem~\ref{alphatosio}, we may assume
that $\sum_{Q\in\mathbb D}\alpha(Q)^2\mu(Q)<\infty$, since otherwise
the finite-truncation estimate is automatic. Proposition~\ref{calAQ}
below then shows that every $T_{(j)}\mu$, $j\in\mathcal J$, belongs to $L^2(\mu)$;
\eqref{outsideJbound} gives the same conclusion when $j\notin\mathcal J$.
Expanding the finite sum gives
\[
\|T_{a,m}^{\mathcal J}\mu\|_{L^2(\mu)}^2
=
\sum_{\substack{a\leq j\leq m\\ j\in\mathcal J}}
\|T_{(j)}\mu\|_{L^2(\mu)}^2
+
\sum_{\substack{a\leq i,j\leq m\\ i,j\in\mathcal J\\ i\neq j}}
\big\langle T_{(i)}\mu,T_{(j)}\mu\big\rangle_{L^2(\mu)}.
\]
By \eqref{outsideJbound}, it is enough to prove, uniformly in integers $a\leq m$, that
\[
\|T_{a,m}^{\mathcal J}\mu\|_{L^2(\mu)}^2
\lesssim
\sum_{Q\in\mathbb D}\alpha(Q)^2\mu(Q).
\]
For each $Q\in\mathbb D$, set
\[
    T_Q\mu:=\chi_Q T_{(J(Q))}\mu.
\]
Since $\mathbb D_j$ partitions $E=\supp(\mu)$ for each $j\in\mathcal J$, we have
\[
    T_{(j)}\mu
    =
    \sum_{Q\in\mathbb D_j}T_Q\mu.
\]
Consequently,
\[
    T_{a,m}^{\mathcal J}\mu
    =
    \sum_{\substack{a\leq j\leq m\\ j\in\mathcal J}}T_{(j)}\mu
    =
    \sum_{\substack{a\leq j\leq m\\ j\in\mathcal J}}\sum_{Q\in\mathbb D_j}T_Q\mu.
\]
\subsection{\texorpdfstring{Estimating $T_Q\mu$ In Terms of the $\alpha$'s}{Estimating T Q mu In Terms of the Alphas}}

Given $Q \in \mathbb D$, define
\begin{align}\label{mathcalA}
\mathcal A(Q)
:=
\left(
    \frac{1}{\mu(Q)}
    \sum_{P \in \mathbb D(Q)}
    \alpha(P)^2
    \frac{\ell(P)}{\ell(Q)}
    \mu(P)
\right)^{1/2}.
\end{align}

\begin{lemma}\label{distLQalphabd}
Let $(X,t) \in \supp(\mu)$ and $Q \in \mathbb D$ be such that $(X,t) \in Q$.
Then
\begin{align}\label{distalphabound}
\dist\big((X,t),L_Q\big)
\lesssim
\sum_{\substack{P \in \mathbb D(Q)\\ (X,t)\in P}}
\alpha(P)\ell(P),
\end{align}
and
\begin{align}\label{distAbound}
\int_Q
\left(
    \frac{\dist\big((X,t),L_Q\big)}{\ell(Q)}
\right)^2
\,\d\mu(X,t)
\lesssim
\mathcal A(Q)^2\mu(Q).
\end{align}
\end{lemma}

\begin{proof}
We first show that \eqref{distalphabound} implies \eqref{distAbound}. By
Cauchy's inequality,
\begin{align*}
\left(
    \frac{\dist\big((X,t),L_Q\big)}{\ell(Q)}
\right)^2
&\lesssim
\left(
    \sum_{\substack{P \in \mathbb D(Q)\\ (X,t)\in P}}
    \alpha(P)\frac{\ell(P)}{\ell(Q)}
\right)^2
\\
&\leq
\left(
    \sum_{\substack{P \in \mathbb D(Q)\\ (X,t)\in P}}
    \alpha(P)^2\frac{\ell(P)}{\ell(Q)}
\right)
\left(
    \sum_{\substack{P \in \mathbb D(Q)\\ (X,t)\in P}}
    \frac{\ell(P)}{\ell(Q)}
\right)
\\
&\lesssim
\sum_{\substack{P \in \mathbb D(Q)\\ (X,t)\in P}}
\alpha(P)^2\frac{\ell(P)}{\ell(Q)}.
\end{align*}
Integrating over $Q$ gives
\[
\int_Q
\left(
    \frac{\dist\big((X,t),L_Q\big)}{\ell(Q)}
\right)^2
\,\d\mu(X,t)
\lesssim
\sum_{P\in\mathbb D(Q)}
\alpha(P)^2\frac{\ell(P)}{\ell(Q)}\mu(P)
=
\mathcal A(Q)^2\mu(Q),
\]
which proves \eqref{distAbound}.

It remains to prove \eqref{distalphabound}. Let
\[
    Q=Q_0\supset Q_1\supset Q_2\supset\cdots
\]
be the chain of cubes in $\mathbb D(\mu)$ containing $(X,t)$, so that
\[
    \ell(Q_m)=2^{-m}\ell(Q),
    \qquad
    (X,t)\in Q_m
\]
for every $m\geq0$.

We claim that $C_{Q_{m+1}}\subset C_{Q_m}$, for every $m \geq 0$.  Indeed, since $Q_{m+1}\subset Q_m$ and both centers belong
to $Q_m$,
\[
d_p\bigl((Z_{Q_{m+1}},\tau_{Q_{m+1}}),
        (Z_{Q_m},\tau_{Q_m})\bigr)
\leq \diam(Q_m)
\leq 2^{-10}\ell(Q_m).
\]
Thus
\[
3\ell(Q_{m+1})+2^{-10}\ell(Q_m)
=
\left(\frac32+2^{-10}\right)\ell(Q_m)
<3\ell(Q_m),
\]
which proves the claim.

Let $N\geq0$ be arbitrary. Choose $a_N\in L_{Q_N}\cap C_{Q_N}$. Since $(X,t)\in Q_N\subset C_{Q_N}$, we have
\[
    \|(X,t)-a_N\|\lesssim\ell(Q_N).
\]

For $m=N,N-1,\ldots,1$, suppose that $a_m\in L_{Q_m}\cap C_{Q_m}$
has already been chosen. Since $C_{Q_m}\subset C_{Q_{m-1}}$, we have $a_m\in L_{Q_m}\cap C_{Q_{m-1}}$.
Applying Lemma~\ref{toalphalemma} to
$Q_m\subset Q_{m-1}$ with $\eta=\frac12$, we may choose
\[
    a_{m-1}\in L_{Q_{m-1}}\cap C_{Q_{m-1}}
\]
such that
\begin{align}
\label{successiveplaneproj}
    \|a_m-a_{m-1}\|
    \lesssim
    \alpha(Q_{m-1})\ell(Q_{m-1}).
\end{align}

Since $a_0\in L_Q$, the triangle inequality and
\eqref{successiveplaneproj} give
\[
\dist(a_N,L_Q)
\leq
\sum_{m=0}^{N-1}\|a_m-a_{m+1}\|
\lesssim
\sum_{m=0}^{N-1}\alpha(Q_m)\ell(Q_m).
\]
Consequently,
\[
\dist\bigl((X,t),L_Q\bigr)
\lesssim
\ell(Q_N)
+
\sum_{m=0}^{N-1}\alpha(Q_m)\ell(Q_m).
\]
Letting $N\to\infty$, and using $\ell(Q_N)\to0$, yields
\[
\dist\bigl((X,t),L_Q\bigr)
\lesssim
\sum_{m=0}^{\infty}\alpha(Q_m)\ell(Q_m).
\]

Since the cubes $Q_m$ form a subcollection of the dyadic cubes
$P\in\mathbb D(Q)$ containing $(X,t)$, it follows that
\[
\dist\bigl((X,t),L_Q\bigr)
\lesssim
\sum_{\substack{P\in\mathbb D(Q)\\(X,t)\in P}}
\alpha(P)\ell(P),
\]
which proves \eqref{distalphabound}.
\end{proof}

\begin{lemma}\label{TQintbd}
Given $Q \in \mathbb D_m$, we have
\begin{itemize}
    \item[(a)] For any $(X,t) \in L_Q \cap C_{2\diam(Q)}(Z_Q,\tau_Q)$,
    $|T_{(m)}\mu(X,t)| \lesssim \alpha(Q)$.
    \item[(b)] $\displaystyle \int |T_Q\mu|\,\d\mu \lesssim \alpha(Q)\mu(Q)$.
    \item[(c)] $\displaystyle \int |T_Q\mu|^2\,\d\mu \lesssim \mathcal A(Q)^2\mu(Q)$.
\end{itemize}
\end{lemma}

\begin{proof}
We first prove (a). Notice that if
\[
    (X,t) \in C_{2\diam(Q)}(Z_Q,\tau_Q),
\]
then, by the choice of $\omega$ in the definition of the annular
truncations,
\[
    \supp\big(K_{(m)}((X,t)-\cdot)\big)\subseteq C_Q.
\]
Moreover,
\[
    \|K_{(m)}((X,t)-\cdot)\|_{\operatorname{Lip}}
    \lesssim
    \ell(Q)^{-(k+3)},
\]
where the Lipschitz norm is taken with respect to the parabolic distance.

Let $L$ be a parabolic $(k+2)$-plane containing a line parallel to the $t$-axis,
and let $c\geq0$. If $(X,t)\in L$, then, by the spatial oddness of $K$ and the spatial evenness of $\varphi_m$,
\[
    \int K_{(m)}\big((X,t)-(Y,s)\big)
    \,\d\sigma_L(Y,s)
    =
    0.
\]
Therefore
\begin{multline*}
\left|
    \int K_{(m)}\big((X,t)-(Y,s)\big)\,\d\mu(Y,s)
\right| \\
=
\left|
    \int K_{(m)}\big((X,t)-(Y,s)\big)\,\d\mu(Y,s)
    -
    c\int K_{(m)}\big((X,t)-(Y,s)\big)
    \,\d\sigma_L(Y,s)
\right| \\
\lesssim
\frac{1}{\ell(Q)^{k+3}}
\dist_{C_Q}\big(\mu,c\,\sigma_L\big).
\end{multline*}
Choosing $L=L_Q$ and $c=c_Q$, and using \eqref{alphaflatapproximation}, gives
\[
    |T_{(m)}\mu(X,t)|
    \lesssim
    \alpha(Q),
\]
which proves (a).

\begingroup
We now prove (b) and (c). Given $(X,t)\in Q$, let $(X',t)$ be the
orthogonal projection of $(X,t)$ onto $L_Q$.

Suppose first that there is a point $(X,t)\in Q$ for which
$(X',t)\notin C_{2\diam(Q)}(Z_Q,\tau_Q)$. Since $(X,t)\in Q$,
\begin{align*}
\dist\big((X,t),L_Q\big)
&=
d_p\big((X,t),(X',t)\big)
\\
&\geq
d_p\big((X',t),(Z_Q,\tau_Q)\big)
-
d_p\big((X,t),(Z_Q,\tau_Q)\big)
\\
&>
\diam(Q)
\gtrsim
\ell(Q).
\end{align*}
Thus there
is a constant $c_1>0$ such that $
    \dist((X,t),L_Q)\geq c_1\ell(Q).$
If $(Y,s)\in C_{\frac{c_1}{4}\ell(Q)}(X,t)$, then the $1$-Lipschitz property of the distance to $L_Q$ gives
\[
    \dist\big((Y,s),L_Q\big)\geq \frac{3c_1}{4}\ell(Q).
\]
Moreover,
\[
    C_{\frac{c_1}{4}\ell(Q)}(X,t)
    \subset C_{2\ell(Q)}(Z_Q,\tau_Q).
\]
Therefore, by ADR and \eqref{LQbetacomp},
\[
\alpha(Q)
\gtrsim
\frac{1}{\ell(Q)^{k+2}}
\int_{C_{2\ell(Q)}(Z_Q,\tau_Q)}
\frac{\dist\big((Y,s),L_Q\big)}{\ell(Q)}
\,\d\mu(Y,s)
\gtrsim 1.
\]
Since $\mathcal A(Q)\geq\alpha(Q)$ and the size estimate on $K$ gives
$|T_{(m)}\mu|\lesssim1$ on $Q$, both (b) and (c) follow.

We may therefore assume that, for every $(X,t)\in Q$,
\[
    (X',t) \in C_{2\diam(Q)}(Z_Q,\tau_Q).
\]
For such a point, the support of
$K_{(m)}((Z,t)-\cdot)$ is contained in $C_Q$ for every $Z$ on the line
segment joining $X$ to $X'$. The mean value theorem, the derivative bounds
for $K_{(m)}$, and the ADR property therefore give
\[
|T_{(m)}\mu(X,t)-T_{(m)}\mu(X',t)|
\lesssim
\frac{\dist((X,t),L_Q)}{\ell(Q)}
\]
and applying part (a) at $(X',t)$, we obtain
\[
|T_{(m)}\mu(X,t)|
\lesssim
\frac{\dist((X,t),L_Q)}{\ell(Q)}
+
|T_{(m)}\mu(X',t)|
\lesssim
\frac{\dist((X,t),L_Q)}{\ell(Q)}
+
\alpha(Q).
\]
\endgroup
Hence, for $p=1,2$,
\[
\int |T_Q\mu|^p\,\d\mu
\lesssim
\int_Q
\left(
    \frac{\dist((X,t),L_Q)}{\ell(Q)}
\right)^p
\,\d\mu(X,t)
+
\alpha(Q)^p\mu(Q).
\]

When $p=1$, \eqref{LQbetacomp} gives
\[
\int_Q
\frac{\dist((X,t),L_Q)}{\ell(Q)}
\,\d\mu(X,t)
\lesssim
\alpha(Q)\mu(Q),
\]
and therefore (b) follows. For $p=2$, (c) follows from Lemma~\ref{distLQalphabd}, together with the fact that $\alpha(Q)\leq \mathcal A(Q)$.
\end{proof}

\begin{lemma}\label{dillem}
Let $(X,t) \in E$, and let $Q \in \mathbb D$ be such that
\[
    \dist\big((X,t),Q\big)\leq \gamma\ell(Q),
\]
where $\gamma$ is the constant in the dyadic construction. Then
\begin{align}\label{ndistlqalpha}
\dist\big((X,t),L_Q\big)
\lesssim
\alpha(Q)\ell(Q)
+
\sum_{T \in N(Q)\cup \{Q\}}
\sum_{\substack{P \in \mathbb D(T)\\ (X,t) \in P}}
\alpha(P)\ell(P).
\end{align}
Consequently,
\begin{align}\label{Adilated}
    \int_{U_{\frac{\gamma}{4}\ell(Q)}(Q)}
    \left(
        \frac{\dist\big((X,t),L_Q\big)}{\ell(Q)}
    \right)^2
    \,\d\mu(X,t)
    \lesssim
    \left(
        \sum_{T\in N(Q)\cup\{Q\}}\mathcal A(T)^2
    \right)
    \mu(Q).
\end{align}
\end{lemma}

\begin{proof}
We first prove \eqref{ndistlqalpha}. If $(X,t)\in Q$, then the estimate follows
from Lemma~\ref{distLQalphabd}.  So, assume that $(X,t)\notin Q$. Let $T\in\mathbb D_{J(Q)}$ be the dyadic cube
containing $(X,t)$. Since $\dist((X,t),Q)\leq \gamma\ell(Q)$, we have
$T\in N(Q)$.

Since the neighbor relation is symmetric, $Q\in N(T)$. By
Lemma~\ref{toalphalemma2}, applied with $Q$ and $T$ interchanged, we know that
\begin{align}\label{sclaim}
\sup_{W\in L_T\cap C_{6\ell(Q)}(Z_T,\tau_T)}
\dist(W,L_Q)
\lesssim
\max\{\alpha(Q),\alpha(T)\}\ell(Q).
\end{align}

Let $(Y,s)$ be the orthogonal projection of $(X,t)$ onto the vertical
plane $L_T$. Since $(X,t)\in T$ and $L_T\cap C_T\neq\emptyset$, we have
\[
    (Y,s)\in L_T\cap C_{6\ell(Q)}(Z_T,\tau_T).
\]
Thus \eqref{sclaim} gives
\[
    \dist\big((Y,s),L_Q\big)
    \lesssim
    \max\{\alpha(Q),\alpha(T)\}\ell(Q).
\]
Consequently,
\begin{align*}
\dist\big((X,t),L_Q\big)
&\leq
\dist\big((X,t),L_T\big)
+
\dist\big((Y,s),L_Q\big)
\\
&\lesssim
\dist\big((X,t),L_T\big)
+
\alpha(Q)\ell(Q)
+
\alpha(T)\ell(T).
\end{align*}
By Lemma~\ref{distLQalphabd}, applied to $T$,
\[
\dist\big((X,t),L_T\big)
\lesssim
\sum_{\substack{P\in\mathbb D(T)\\ (X,t)\in P}}
\alpha(P)\ell(P).
\]
Since $\ell(T)=\ell(Q)$ and the term $\alpha(T)\ell(T)$ appears in the last sum,
we obtain \eqref{ndistlqalpha}.

We now prove \eqref{Adilated}. If
\[
    (X,t)\in E\cap U_{\frac{\gamma}{4}\ell(Q)}(Q),
\]
then the dyadic cube $T\in\mathbb D_{J(Q)}$ containing $(X,t)$
is either $Q$ or belongs to $N(Q)$. Therefore,
\[
    E\cap U_{\frac{\gamma}{4}\ell(Q)}(Q)
    \subset
    Q\cup\bigcup_{T\in N(Q)}T.
\]
Apply \eqref{ndistlqalpha} and Cauchy's inequality along each dyadic chain.
Integrating as in the proof of Lemma~\ref{distLQalphabd} gives
\[
\int_{U_{\frac{\gamma}{4}\ell(Q)}(Q)}
\left(
    \frac{\dist\big((X,t),L_Q\big)}{\ell(Q)}
\right)^2
\,\d\mu(X,t)
\lesssim
\alpha(Q)^2\mu(Q)
+
\sum_{T\in N(Q)\cup\{Q\}}
\mathcal A(T)^2\mu(T).
\]
Since $\alpha(Q)\leq\mathcal A(Q)$, $\mu(T)\approx\mu(Q)$ for
$T\in N(Q)\cup\{Q\}$, and the number of neighbors is uniformly bounded,
we obtain \eqref{Adilated}.
\end{proof}

\begin{lemma}\label{distLQalphabd2}
Let $R,Q\in\mathbb D$ be such that $Q\subset R$. Let $(X_Q,t_Q)$ be the
orthogonal projection of $(Z_Q,\tau_Q)$ onto $L_Q$. Then
\[
\dist\big((X_Q,t_Q),L_R\big)
\lesssim
\sum_{P:\,Q\subsetneq P\subseteq R}
\alpha(P)\ell(P).
\]
\end{lemma}

\begin{proof}
Choose the chain of dyadic ancestors of $Q$ contained in $R$,
\[
    R=:S_0\supset S_1\supset\cdots\supset S_M:=Q,
\]
where
\[
    \ell(S_{m+1})=\frac12\ell(S_m),
    \qquad 0\leq m<M.
\]

Set
\[
    a_M:=(X_Q,t_Q).
\]
Since $L_Q\cap C_Q\neq\emptyset$ and $(X_Q,t_Q)$ is the
orthogonal projection of $(Z_Q,\tau_Q)$ onto $L_Q$, we have
\[
    a_M\in L_Q\cap C_Q=L_{S_M}\cap C_{S_M}.
\]

We recursively construct points
\[
    a_m\in L_{S_m}\cap C_{S_m},
    \qquad m=M-1,\ldots,0.
\]
Suppose that
\[
    a_{m+1}\in L_{S_{m+1}}\cap C_{S_{m+1}}.
\]
Since $C_{S_{m+1}}\subset C_{S_m}$, we have
\[
    a_{m+1}\in L_{S_{m+1}}\cap C_{S_m}.
\]
Applying Lemma~\ref{toalphalemma} to
$S_{m+1}\subset S_m$, with $\eta=\frac12$, gives
\[
    \dist_H\bigl(
        L_{S_{m+1}}\cap C_{S_m},
        L_{S_m}\cap C_{S_m}
    \bigr)
    \lesssim
    \alpha(S_m)\ell(S_m).
\]
Therefore, we may choose $a_m\in L_{S_m}\cap C_{S_m}$ such that
\[
    \|a_m-a_{m+1}\|
    \lesssim
    \alpha(S_m)\ell(S_m).
\]

Since $a_0\in L_R$, the triangle inequality yields
\[
\begin{aligned}
\dist\bigl((X_Q,t_Q),L_R\bigr)
&\leq
\|(X_Q,t_Q)-a_0\|
\\
&\leq
\sum_{m=0}^{M-1}\|a_m-a_{m+1}\|
\\
&\lesssim
\sum_{m=0}^{M-1}\alpha(S_m)\ell(S_m)
\\
&\leq
\sum_{P:\,Q\subsetneq P\subseteq R}
\alpha(P)\ell(P).
\end{aligned}
\]
This proves the lemma.
\end{proof}

\begin{lemma}\label{planealpha}
Suppose $S,Q \in \mathbb D$ are cubes such that $S \subseteq Q$,
$(X,t)\in E$, and
\[
    \dist\big((X,t),S\big)\leq\diam(S).
\]
Then
\begin{align*}
    \dist\big((X,t),L_Q\big)
    \lesssim
    \dist\big((X,t),L_S\big)
    +
    \sum_{P:\,S\subsetneq P\subseteq Q}
    \alpha(P)\ell(P).
\end{align*}
\end{lemma}
\begin{proof}
If $S=Q$, there is nothing to prove. Assume $S\subsetneq Q$, and let
\[
Q=:S_0\supset S_1\supset\cdots\supset S_M:=S
\]
be the chain of dyadic cubes from $Q$ to $S$, so that $\ell(S_{m+1})=\frac12\ell(S_m)$. Since the centers of both $S_m$ and $S_{m+1}$ belong to $S_m$,
\[
d_p\big((Z_{S_{m+1}},\tau_{S_{m+1}}),(Z_{S_m},\tau_{S_m})\big)\leq\diam(S_m)\leq2^{-10}\ell(S_m),
\]
and therefore
\[
C_{S_{m+1}}\subset C_{S_m},\qquad 0\leq m<M.
\]

Let $a_M$ be the orthogonal projection of $(X,t)$ onto $L_S$. Then
\[
\|(X,t)-a_M\|=\dist\big((X,t),L_S\big).
\]
Let $b_M$ be the orthogonal projection of
$(Z_S,\tau_S)$ onto $L_S$. Since $L_S$ is vertical and
$L_S\cap C_S\neq\emptyset$, we have
\[
    b_M\in L_S\cap C_S.
\]
Thus,
\[
d_p(a_M,b_M)
\leq
d_p\big((X,t),(Z_S,\tau_S)\big)
\leq
\dist\big((X,t),S\big)+\diam(S)
\leq
2\diam(S).
\]
Since the centers of $S$ and its parent $S_{M-1}$ both belong to
$S_{M-1}$, it follows that
\begin{align*}
d_p\big(a_M,(Z_{S_{M-1}},\tau_{S_{M-1}})\big)
&\leq
d_p(a_M,b_M)
+
d_p\big(b_M,(Z_S,\tau_S)\big)
\\
&\quad+
d_p\big((Z_S,\tau_S),
       (Z_{S_{M-1}},\tau_{S_{M-1}})\big)
\\
&\leq
2\diam(S)+3\ell(S)+\diam(S_{M-1})
\\
&<
6\ell(S)
=
3\ell(S_{M-1}).
\end{align*}
Therefore,
\[
    a_M\in L_S\cap C_{S_{M-1}}.
\]

We now construct $a_m\in L_{S_m}\cap C_{S_m}$ recursively for $m=M-1,\ldots,0$. Suppose that $a_{m+1}\in L_{S_{m+1}}\cap C_{S_m}$. Lemma~\ref{toalphalemma}, applied to $S_{m+1}\subset S_m$ with $\eta=\frac12$, allows us to choose $a_m\in L_{S_m}\cap C_{S_m}$ such that
\begin{align}\label{planealpha-step}
\|a_m-a_{m+1}\|\lesssim\alpha(S_m)\ell(S_m).
\end{align}
For $m\geq1$, the inclusion $C_{S_m}\subset C_{S_{m-1}}$ allows us to continue the process.

Since $a_0\in L_Q$, the triangle inequality and \eqref{planealpha-step} give
\[
\dist\big((X,t),L_Q\big)\leq\|(X,t)-a_0\|
\leq\|(X,t)-a_M\|+\sum_{m=0}^{M-1}\|a_m-a_{m+1}\|
\lesssim\dist\big((X,t),L_S\big)+\sum_{m=0}^{M-1}\alpha(S_m)\ell(S_m).
\]
The cubes $S_0,\ldots,S_{M-1}$ are precisely the dyadic cubes $P$ satisfying $S\subsetneq P\subseteq Q$. 
Hence
\[
\dist\big((X,t),L_Q\big)
\lesssim
\dist\big((X,t),L_S\big)
+
\sum_{P:\,S\subsetneq P\subseteq Q}
\alpha(P)\ell(P),
\]
which proves the lemma.
\end{proof}

\subsection{Estimating the Diagonal Term}

\begin{proposition}\label{calAQ}
For every $R\in\mathbb D$, we have
\[
\sum_{Q\in\mathbb D(R)}
\|T_Q\mu\|_{L^2(\mu)}^2
\lesssim
\sum_{Q\in\mathbb D(R)}
\mathcal A(Q)^2\mu(Q)
\lesssim
\sum_{Q\in\mathbb D(R)}
\alpha(Q)^2\mu(Q).
\]
Moreover,
\[
\sum_{j\in\mathcal J}
\|T_{(j)}\mu\|_{L^2(\mu)}^2
=
\sum_{Q\in\mathbb D}
\|T_Q\mu\|_{L^2(\mu)}^2
\lesssim
\sum_{Q\in\mathbb D}
\alpha(Q)^2\mu(Q).
\]
\end{proposition}

\begin{proof}
By part~(c) of Lemma~\ref{TQintbd},
\begin{align*}
\sum_{Q\in\mathbb D(R)}
\|T_Q\mu\|_{L^2(\mu)}^2
&\lesssim
\sum_{Q\in\mathbb D(R)}
\mathcal A(Q)^2\mu(Q)
\\
&=
\sum_{\substack{Q\in\mathbb D\\ Q\subseteq R}}
\sum_{\substack{P\in\mathbb D\\ P\subseteq Q}}
\alpha(P)^2
\frac{\ell(P)}{\ell(Q)}
\mu(P)
\\
&=
\sum_{\substack{P\in\mathbb D\\ P\subseteq R}}
\alpha(P)^2\mu(P)
\sum_{\substack{Q\in\mathbb D\\ P\subseteq Q\subseteq R}}
\frac{\ell(P)}{\ell(Q)}.
\end{align*}
For each fixed $P\subseteq R$, there is exactly one dyadic ancestor of
$P$ at each generation between $P$ and $R$. Therefore,
\[
\sum_{\substack{Q\in\mathbb D\\ P\subseteq Q\subseteq R}}
\frac{\ell(P)}{\ell(Q)}
\leq
\sum_{r=0}^{\infty}2^{-r}
=
2.
\]
Consequently,
\[
\sum_{Q\in\mathbb D(R)}
\|T_Q\mu\|_{L^2(\mu)}^2
\lesssim
\sum_{Q\in\mathbb D(R)}
\mathcal A(Q)^2\mu(Q)
\lesssim
\sum_{Q\in\mathbb D(R)}
\alpha(Q)^2\mu(Q).
\]

We now prove the second estimate in the statement. Fix $m\in\mathcal J$. Since the cubes
in $\mathbb D_m$ form a disjoint partition of $E$, applying the 
estimate above to each $R\in\mathbb D_m$ gives
\[
\sum_{\substack{Q\in\mathbb D\\ J(Q)\geq m}}
\|T_Q\mu\|_{L^2(\mu)}^2
\lesssim
\sum_{\substack{Q\in\mathbb D\\ J(Q)\geq m}}
\alpha(Q)^2\mu(Q).
\]
Indeed, every cube $Q$ with $J(Q)\geq m$ is contained in a unique cube
$R\in\mathbb D_m$. If $\mathcal J=\mathbb Z$, we let $m\to-\infty$ and use monotone convergence. If $\mathcal J=\{j\geq j_0\}$, we take $m=j_0$. In either case, we obtain
\[
\sum_{Q\in\mathbb D}
\|T_Q\mu\|_{L^2(\mu)}^2
\lesssim
\sum_{Q\in\mathbb D}
\alpha(Q)^2\mu(Q).
\]

Finally, for every $j\in\mathcal J$,
\[
T_{(j)}\mu
=
\sum_{Q\in\mathbb D_j}T_Q\mu
\qquad
\mu\text{-a.e.}
\]
Since the cubes in $\mathbb D_j$ are disjoint and $T_Q\mu$ is supported
on $Q$, the summands have disjoint supports. Hence
\[
\|T_{(j)}\mu\|_{L^2(\mu)}^2
=
\sum_{Q\in\mathbb D_j}
\|T_Q\mu\|_{L^2(\mu)}^2.
\]
Summing over $j\in\mathcal J$ gives
\[
\sum_{j\in\mathcal J}
\|T_{(j)}\mu\|_{L^2(\mu)}^2
=
\sum_{Q\in\mathbb D}
\|T_Q\mu\|_{L^2(\mu)}^2
\lesssim
\sum_{Q\in\mathbb D}
\alpha(Q)^2\mu(Q).
\]
\end{proof}

\begin{corollary}\label{finitebanddiag}
Let $M\geq 0$ be fixed. Then
\[
\sum_{\substack{i,j\in\mathcal J\\ |i-j|\leq M}}
\big|
    \langle T_{(i)}\mu,T_{(j)}\mu\rangle
\big|
\lesssim_M
\sum_{Q\in\mathbb D}\alpha(Q)^2\mu(Q).
\]
\end{corollary}

\begin{proof}
By Cauchy--Schwarz,
\[
\big|
    \langle T_{(i)}\mu,T_{(j)}\mu\rangle
\big|
\leq
\frac12
\left(
    \|T_{(i)}\mu\|_{L^2(\mu)}^2
    +
    \|T_{(j)}\mu\|_{L^2(\mu)}^2
\right).
\]
Therefore
\[
\sum_{\substack{i,j\in\mathcal J\\ |i-j|\leq M}}
\big|
    \langle T_{(i)}\mu,T_{(j)}\mu\rangle
\big|
\lesssim_M
\sum_{j\in\mathcal J}\|T_{(j)}\mu\|_{L^2(\mu)}^2.
\]
The conclusion follows from Proposition~\ref{calAQ}.
\end{proof}

\subsection{Estimating the Off-Diagonal Term}

\begin{proposition}\label{offdiagprop}
Suppose that $\mu$ is a $(k+2)$-ADR measure on $\mathbb P^n$ and that
$T$ is a $(k+2)$-dimensional parabolic Calder\'on--Zygmund operator with
associated kernel $K$ satisfying $K(X,-t)=K(X,t)$. If $N_0$ is sufficiently large, depending only on
the dimension and the ADR constants, then
\begin{align}\label{off-diag}
\sum_{\substack{i,j\in\mathcal J\\ |i-j|>N_0}}
\left|
\left\langle T_{(i)}\mu,T_{(j)}\mu\right\rangle
\right|
\lesssim
2^{-cN_0}\sum_{Q\in\mathbb D}\alpha(Q)^2\mu(Q).
\end{align}
Here $c>0$ and the implicit constant are independent of $N_0$.
\end{proposition}

By symmetry, it suffices to consider the terms with $i>j$. Throughout
this section, $N_0$ is taken sufficiently large, depending only on the
dimension and the ADR constants.

Together with Corollary~\ref{finitebanddiag}, Proposition~\ref{offdiagprop}
gives the estimate needed to prove Theorem~\ref{alphatosio}. Indeed,
Corollary~\ref{finitebanddiag} controls the nearly diagonal terms:
\[
\sum_{\substack{i,j\in\mathcal J\\ |i-j|\leq N_0}}
\left|
\left\langle T_{(i)}\mu,T_{(j)}\mu\right\rangle
\right|
\lesssim
\sum_{Q\in\mathbb D}\alpha(Q)^2\mu(Q).
\]
Here the implicit constant is harmless, since $N_0$ is fixed sufficiently
large depending only on the dimension and the ADR constants. On the other
hand, Proposition~\ref{offdiagprop} controls the off-diagonal terms:
\[
\sum_{\substack{i,j\in\mathcal J\\ |i-j|>N_0}}
\left|
\left\langle T_{(i)}\mu,T_{(j)}\mu\right\rangle
\right|
\lesssim
\sum_{Q\in\mathbb D}\alpha(Q)^2\mu(Q).
\]
Combining these estimates yields
\[
\sum_{i,j\in\mathcal J}
\left|
\left\langle T_{(i)}\mu,T_{(j)}\mu\right\rangle
\right|
\lesssim
\sum_{Q\in\mathbb D}\alpha(Q)^2\mu(Q).
\]
Restricting the last double sum to
$a\leq i,j\leq m$ bounds $\|T_{a,m}^{\mathcal J}\mu\|_{L^2(\mu)}^2$.
Together with \eqref{outsideJbound}, this proves
\eqref{smoothtruncbound}, uniformly in $a\leq m$.

If $\sum_{Q\in\mathbb D}\alpha(Q)^2\mu(Q)<\infty$, the double series above converges absolutely. For fixed $m$ and integers $a<b\leq m$, it follows that
\[
\begin{aligned}
\|T_{a,m}^{\mathcal J}\mu-T_{b,m}^{\mathcal J}\mu\|_{L^2(\mu)}^2
&=\left\|\sum_{\substack{a\leq j<b\\j\in\mathcal J}}T_{(j)}\mu\right\|_{L^2(\mu)}^2\\
&\leq\sum_{\substack{i,j\in\mathcal J\\i,j<b}}
\big|\langle T_{(i)}\mu,T_{(j)}\mu\rangle\big|
\longrightarrow0\qquad(b\to-\infty),
\end{aligned}
\]
uniformly in $a<b$. The terms with $j\notin\mathcal J$ are absent when
$E$ is unbounded and form a finite sum on $E$ otherwise. Thus
$T_{a,m}\mu$ is Cauchy in $L^2(\mu)$ as $a\to-\infty$.
Its limit $T_m\mu$ satisfies the same bound by passage to the limit in
\eqref{smoothtruncbound}. This proves Theorem~\ref{alphatosio}.

Fix integers $i,j\in\mathcal J$ with $i>j$ and $i-j>N_0$, and set
\[
    m=\left\lfloor\frac{i+j}{2}\right\rfloor.
\]

For each \(m\in\mathcal J\), we choose a family $\bigl\{\theta_S\bigr\}_{S\in\mathbb D_m} \subset C_c^\infty(\mathbb P^n)$ with \(0\leq \theta_S\leq1\) and the following properties.  Set \[ \Omega_m^{(1)}:= \left\{(X,t):\dist\big((X,t),E\big)<\frac{\gamma}{64}2^{-m}\right\}. \] The supports have uniformly bounded overlap, and \begin{equation}\label{partitionneighborchoice} \supp(\theta_S) \subset U_{\frac{\gamma}{8}\ell(S)}(S) \subset U_{\frac{\gamma}{4}\ell(S)}(S). \end{equation} Moreover, \begin{equation}\label{partitionhigherderivatives} \big|\nabla_X^a\partial_t^q\theta_S\big| \lesssim_{a,q}\ell(S)^{-|a|-2q}, \qquad |a|+q\leq2. \end{equation} Also, \begin{equation}\label{ambientpartitionidentity} \sum_{S\in\mathbb D_m}\theta_S=1 \qquad\text{on }\Omega_m^{(1)}. \end{equation} In particular, the sum is equal to $1$ on \(E\). The construction is given in Appendix~\ref{smoothpartitionappendix}.

Since $\gamma\ell(S)\leq\diam(S)$, we have
\[
    U_{\frac{\gamma}{4}\ell(S)}(S)\subset
    \left\{
        (X,t):
        \dist\big((X,t),S\big)\leq\diam(S)/4
    \right\}.
\]

We then write
\[
\langle T_{(j)}\mu,T_{(i)}\mu\rangle
=
\sum_{S\in\mathbb D_m}
\langle\theta_ST_{(j)}\mu,T_{(i)}\mu\rangle.
\]

For each $S\in\mathbb D_m$, let $(X_S,\tau_S)$ be the orthogonal
projection of $(Z_S,\tau_S)$ onto $L_S$.

We decompose
\begin{align*}
\langle T_{(j)}\mu,T_{(i)}\mu\rangle
&=
\sum_{S\in\mathbb D_m}
\left\langle
    \theta_S\big(T_{(j)}\mu-T_{(j)}\mu(X_S,\tau_S)\big),
    T_{(i)}\mu
\right\rangle
\\
&\qquad+
\sum_{S\in\mathbb D_m}
T_{(j)}\mu(X_S,\tau_S)
\left\langle\theta_S,T_{(i)}\mu\right\rangle
\\
&=:A_{j,i}+B_{j,i}.
\end{align*}

\subsubsection{Estimating the Sum of the $A_{j,i}$'s}

\begingroup

Since $m>j$, every cube $S\in\mathbb D_m$ is contained in a unique cube
$R\in\mathbb D_j$. Therefore,
\begin{align*}
A_{j,i}
&=
\sum_{R\in\mathbb D_j}
\sum_{\substack{S\in\mathbb D_m\\ S\subset R}}
\left\langle
\theta_S
\left(
T_{(j)}\mu-T_{(j)}\mu(X_S,\tau_S)
\right),
T_{(i)}\mu
\right\rangle.
\end{align*}
Hence, by the triangle inequality,
\begin{align*}
|A_{j,i}|
&\leq
\sum_{R\in\mathbb D_j}
\left|
\sum_{\substack{S\in\mathbb D_m\\ S\subset R}}
\left\langle
\theta_S
\left(
T_{(j)}\mu-T_{(j)}\mu(X_S,\tau_S)
\right),
T_{(i)}\mu
\right\rangle
\right|.
\end{align*}

For $R\in\mathbb D$ and $i>J(R)+N_0$, set
\[
m(R,i):=
\left\lfloor\frac{J(R)+i}{2}\right\rfloor.
\]
Summing the preceding estimate over all pairs $i>j$ satisfying
$i-j>N_0$, and writing $j=J(R)$, we obtain
\begin{align*}
&\sum_{\substack{i,j\in\mathcal J\\ i-j>N_0}}
|A_{j,i}|
\\
&\qquad\leq
\sum_{R\in\mathbb D}
\sum_{i>J(R)+N_0}
\left|
\sum_{\substack{
S\in\mathbb D_{m(R,i)}\\
S\subset R
}}
\left\langle
\theta_S
\left(
T_{(J(R))}\mu
-
T_{(J(R))}\mu(X_S,\tau_S)
\right),
T_{(i)}\mu
\right\rangle
\right|.
\end{align*}

For each $R\in\mathbb D$, define
\begin{align*}
A_R
:=
\sum_{i>J(R)+N_0}
\left|
\sum_{\substack{
S\in\mathbb D_{m(R,i)}\\
S\subset R
}}
\left\langle
\theta_S
\left(
T_{(J(R))}\mu
-
T_{(J(R))}\mu(X_S,\tau_S)
\right),
T_{(i)}\mu
\right\rangle
\right|.
\end{align*}
Then
\[
\sum_{\substack{i,j\in\mathcal J\\ i-j>N_0}}
|A_{j,i}|
\leq
\sum_{R\in\mathbb D}A_R.
\]

\endgroup

We will need the following lemma.

\begin{lemma}\label{bigest}
Let $Q\in\mathbb D_j$ and suppose that
\[
    (X,t),(Y,s)\in C_{2\diam(Q)}(Z_Q,\tau_Q).
\]
Then
\begin{align*}
|T_{(j)}\mu(X,t)-T_{(j)}\mu(Y,s)|
&\lesssim
\frac{
    \alpha(Q)\ell(Q)
    +\dist\big((X,t),L_Q\big)
    +\dist\big((Y,s),L_Q\big)
}{\ell(Q)^2}
\|(X,t)-(Y,s)\|
\\
&\qquad+
\frac{
    \big\|\pi_{L_Q^\perp}\big((X,t)-(Y,s)\big)\big\|
}{\ell(Q)}.
\end{align*}
\end{lemma}

\begin{proof}
Let $u$ be a spatial unit vector parallel to $L_Q$. We first claim that
\begin{align}\label{uest}
|\nabla_uT_{(j)}\mu(W,\sigma)|
\lesssim
\frac{\alpha(Q)}{\ell(Q)}
+
\frac{\dist\big((W,\sigma),L_Q\big)}{\ell(Q)^2}
\end{align}
and
\begin{align}\label{test}
|\partial_tT_{(j)}\mu(W,\sigma)|
\lesssim
\frac{\alpha(Q)}{\ell(Q)^2}
+
\frac{\dist\big((W,\sigma),L_Q\big)}{\ell(Q)^3}
\end{align}
whenever $(W,\sigma)\in C_{2\diam(Q)}(Z_Q,\tau_Q)$.

Recall that $\supp K_{(j)}\subset C_{2^{-j-\omega+2}}(0,0)$.
Fix $\omega\in\mathbb N$ sufficiently large that
\begin{align}\label{omegabigest}
    2^{-\omega}\leq\frac{\gamma}{32}.
\end{align}
Thus, whenever $(W,\sigma)\in C_{\frac{11}{5}\diam(Q)}(Z_Q,\tau_Q)$,
the functions
\[
(Z,\rho)\longmapsto
\nabla_uK_{(j)}\big((W,\sigma)-(Z,\rho)\big)
\]
and
\[
(Z,\rho)\longmapsto
\partial_tK_{(j)}\big((W,\sigma)-(Z,\rho)\big)
\]
are supported in
\begin{align*}
C_{\frac{11}{5}\diam(Q)+2^{-j-\omega+2}}(Z_Q,\tau_Q)
&\subset
C_{\left(\frac{11}{5}2^{-10}+2^{2-\omega}\right)\ell(Q)}(Z_Q,\tau_Q)
\\
&\subset C_{3\ell(Q)}(Z_Q,\tau_Q)
=C_Q.
\end{align*}
Moreover,
\begin{align}\label{gradspacetestlip}
\left\|
    \nabla_uK_{(j)}\big((W,\sigma)-\cdot\big)
\right\|_{\operatorname{Lip}}
&\lesssim
\ell(Q)^{-(k+4)},
\\\label{gradtimetestlip}
\left\|
    \partial_tK_{(j)}\big((W,\sigma)-\cdot\big)
\right\|_{\operatorname{Lip}}
&\lesssim
\ell(Q)^{-(k+5)}.
\end{align}
By \eqref{alphaflatapproximation},
\[
    \dist_{C_Q}(\mu,\mathcal L_Q)
    \leq C\alpha(Q)\ell(Q)^{k+3}.
\]
Together with \eqref{gradspacetestlip}, this gives
\begin{align}\begin{aligned}\label{muplanegradcomparison}
&\left|
\int\nabla_uK_{(j)}\big((W,\sigma)-(Z,\rho)\big)\,\d\mu(Z,\rho)
-
\int\nabla_uK_{(j)}\big((W,\sigma)-(Z,\rho)\big)\,\d\mathcal L_Q(Z,\rho)
\right|
\\
&\hspace{7cm}\lesssim\frac{\alpha(Q)}{\ell(Q)}
\end{aligned}
\end{align}
for every $(W,\sigma)\in C_{\frac{11}{5}\diam(Q)}(Z_Q,\tau_Q)$.  For the same $(W,\sigma)$, \eqref{gradtimetestlip} implies
\begin{align}\begin{aligned}\label{muplanetimecomparison}
&\left|
\int\partial_tK_{(j)}\big((W,\sigma)-(Z,\rho)\big)\,\d\mu(Z,\rho)
-
\int\partial_tK_{(j)}\big((W,\sigma)-(Z,\rho)\big)\,\d\mathcal L_Q(Z,\rho)
\right|
\\
&\hspace{7cm}\lesssim\frac{\alpha(Q)}{\ell(Q)^2}
\end{aligned}
\end{align}

Now take $(W,\sigma)\in C_{2\diam(Q)}(Z_Q,\tau_Q)$ and let
$(W_Q,\sigma)$ be its orthogonal projection onto $L_Q$. If
\[
    (W_Q,\sigma)\notin
    C_{\frac{11}{5}\diam(Q)}(Z_Q,\tau_Q),
\]
then the triangle inequality implies
\[
    \dist\big((W,\sigma),L_Q\big)
    =|W-W_Q|
    \geq\frac15\diam(Q)
    \geq\frac{\gamma}{5}\ell(Q).
\]
Consequently, the estimate $|
\nabla T_{(j)}\mu|\lesssim\ell(Q)^{-1}$ implies \eqref{uest}, and $|\partial_t T_{(j)} \mu | \lesssim \ell(Q)^{-2}$ implies \eqref{test}.

We may therefore assume that
$(W_Q,\sigma)\in C_{\frac{11}{5}\diam(Q)}(Z_Q,\tau_Q)$.
Since $(W_Q,\sigma)\in L_Q$ and $u$ is tangent to the spatial direction
of $L_Q$, the function
\[
    s\longmapsto
    K_{(j)}
    \big((W_Q,\sigma)-((Z,\rho)+su)\big)
\]
is differentiated in a direction lying entirely in $L_Q$. Since
$K_{(j)}$ is smooth and compactly supported, integration by parts in the
$u$-direction produces no boundary terms. Hence
\[
\int
\nabla_uK_{(j)}
\big((W_Q,\sigma)-(Z,\rho)\big)
\,\d\mathcal L_Q(Z,\rho)
=0.
\]
The same argument gives 
$$\int
\partial_t K_{(j)}
\big((W_Q,\sigma)-(Z,\rho)\big)
\,\d\mathcal L_Q(Z,\rho)
=0.$$
Thus,
\begin{align*}
&\left|
\int
\nabla_uK_{(j)}\big((W,\sigma)-(Z,\rho)\big)
\,\d\mathcal L_Q(Z,\rho)
\right|
\\
&\qquad\leq
\int
\left|
\nabla_uK_{(j)}\big((W,\sigma)-(Z,\rho)\big)
-
\nabla_uK_{(j)}\big((W_Q,\sigma)-(Z,\rho)\big)
\right|
\,\d\mathcal L_Q(Z,\rho)
\\
&\qquad\lesssim
\frac{|W-W_Q|}{\ell(Q)^2}
=
\frac{\dist\big((W,\sigma),L_Q\big)}{\ell(Q)^2}.
\end{align*}
Combining this with \eqref{muplanegradcomparison} proves \eqref{uest}.  The same argument, using \eqref{gradtimetestlip}, gives \eqref{test}.

Now let $(X,t),(Y,s)\in C_{2\diam(Q)}(Z_Q,\tau_Q)$.
The spatial segment joining $(X,t)$ to $(Y,t)$ and the vertical segment joining
$(Y,t)$ to $(Y,s)$ remain in
$C_{2\diam(Q)}(Z_Q,\tau_Q)$. Along the spatial segment, decompose $X-Y$ into
its component parallel to $L_Q$ and its component in $L_Q^\perp$. Since the
distance to the affine plane $L_Q$ is convex along spatial line segments,
\eqref{uest} and the estimate
$|
\nabla T_{(j)}\mu|\lesssim\ell(Q)^{-1}$ give
\begin{align}\label{bigestspatial}
|T_{(j)}\mu(X,t)-T_{(j)}\mu(Y,t)|
&\lesssim
\frac{
    \alpha(Q)\ell(Q)
    +\dist\big((X,t),L_Q\big)
    +\dist\big((Y,t),L_Q\big)
}{\ell(Q)^2}
|X-Y|
\\
&\qquad+
\frac{
    \big\|\pi_{L_Q^\perp}\big((X,t)-(Y,t)\big)\big\|
}{\ell(Q)}.
\end{align}
Furthermore,
\[
    |t-s|^{1/2}
    \leq 4\diam(Q)
    \lesssim\ell(Q)
\]
and
\[
    |t-s|^{1/2}
    \leq\|(X,t)-(Y,s)\|.
\]
Hence
\[
    |t-s|
    \lesssim
    \ell(Q)\|(X,t)-(Y,s)\|.
\]
Using \eqref{test}, we obtain
\begin{align}\label{bigesttime}
|T_{(j)}\mu(Y,t)-T_{(j)}\mu(Y,s)|
&\leq
\sup_{\tau\in[\min\{s,t\},\max\{s,t\}]}|\partial_tT_{(j)}\mu(Y,\tau)|\,|t-s|
\\
&\lesssim
\frac{
    \alpha(Q)\ell(Q)
    +\dist\big((Y,t),L_Q\big)
    +\dist\big((Y,s),L_Q\big)
}{\ell(Q)^2}
\|(X,t)-(Y,s)\|.
\end{align}
Because $L_Q$ is vertical,
\[
    \dist\big((Y,t),L_Q\big)
    =
    \dist\big((Y,s),L_Q\big),
\]
and
\[
    \pi_{L_Q^\perp}\big((Y,t)-(Y,s)\big)=0.
\]
Combining \eqref{bigestspatial} and \eqref{bigesttime} proves the lemma.
\end{proof}

We now apply Lemma~\ref{bigest} with $Q=R$. If
$S\in\mathbb D_{m(R,i)}$ and $S\subset R$, then
\[
J(S)-J(R)
=
\left\lfloor\frac{i-J(R)}2\right\rfloor.
\]

Since $i-J(R)\geq N_0+1$, we have
\begin{align}\label{generationseparation}
J(S)-J(R)
\geq
\left\lfloor\frac{N_0+1}{2}\right\rfloor.
\end{align}
Moreover, upon taking $N_0$ sufficiently large depending only on the dimension and the ADR constants, we have
\begin{align}\label{projectionradiusvsdiamR}(4+\gamma)\ell(S) \leq (4+\gamma)\cdot 2^{- \lfloor \frac{N_0+1}{2}  \rfloor}\ell(R) \leq \gamma \ell(R) \leq \diam R
\end{align}

Let $(X,t)\in U_{\frac{\gamma}{4}\ell(S)}(S)$. There exists $(W,\sigma)\in S$ such that
\[
    \|(X,t)-(W,\sigma)\|
    \leq\frac{\gamma}{4}\ell(S).
\]

Since $(W,\sigma),(Z_R,\tau_R)\in R$, we have
\[
    \|(W,\sigma)-(Z_R,\tau_R)\|\leq\diam(R).
\]
Consequently,

\[
    \|(X,t)-(Z_R,\tau_R)\|
    \leq\frac{\gamma}{4}\ell(S)+\diam(R)
    \leq\frac54\diam(R)
    <2\diam(R).
\]

Thus
\[
(X,t)\in C_{2\diam(R)}(Z_R,\tau_R).
\]

Since $L_S\cap C_S\neq\emptyset$ and $(X_S,\tau_S)$ is the orthogonal
projection of $(Z_S,\tau_S)$ onto $L_S$,
\[
    \|(X_S,\tau_S)-(Z_S,\tau_S)\|\leq3\ell(S).
\]
Also, $(Z_S,\tau_S),(Z_R,\tau_R)\in R$, so
\[
    \|(Z_S,\tau_S)-(Z_R,\tau_R)\|\leq\diam(R).
\]
Therefore, by \eqref{projectionradiusvsdiamR},
\[
    \|(X_S,\tau_S)-(Z_R,\tau_R)\|
    \leq3\ell(S)+\diam(R)
    \leq2\diam(R).
\]
Hence
\[
(X_S,\tau_S)\in C_{2\diam(R)}(Z_R,\tau_R).
\]

For $S\subset R$, put $r(S,R):=J(S)-J(R)$.
Since $J(S)=\lfloor(J(R)+i)/2\rfloor$, the index $i$ is one of
\[
2J(S)-J(R)
\qquad\text{or}\qquad
2J(S)-J(R)+1.
\]
Define
\[
\mathcal T(S,R)
:=
|T_{(2J(S)-J(R))}\mu|
+
|T_{(2J(S)-J(R)+1)}\mu|
\]
and
\[
\mathcal S_{N_0}(R)
:=
\left\{
S\subset R:
2r(S,R)+1>N_0
\right\}.
\]

Lemma~\ref{bigest}, applied with $Q=R$, gives
\begin{align}\label{bigAR1}
A_R
\lesssim&
\sum_{S\in\mathcal S_{N_0}(R)}
\int_{U_{\frac{\gamma}{4}\ell(S)}(S)}
\mathcal T(S,R)(X,t)
\bigg[
    \alpha(R)\frac{\ell(S)}{\ell(R)}
\nonumber\\
&\qquad+
\dist\big((X,t),L_R\big)\frac{\ell(S)}{\ell(R)^2}
+
\dist\big((X_S,\tau_S),L_R\big)\frac{\ell(S)}{\ell(R)^2}
\nonumber\\
&\qquad+
\frac{
    \big|\pi_{L_R^\perp}\big((X,t)-(X_S,\tau_S)\big)\big|
}{\ell(R)}
\bigg]
\,\d\mu(X,t).
\end{align}

By Lemma~\ref{planealpha}, for $(X,t)\in U_{\frac{\gamma}{4}\ell(S)}(S)\cap E$,
\begin{align}\label{bigAR2}
\dist\big((X,t),L_R\big)
\lesssim
\dist\big((X,t),L_S\big)
+
\sum_{P:\,S\subset P\subseteq R}\alpha(P)\ell(P).
\end{align}
By Lemma~\ref{distLQalphabd2},
\begin{align}\label{XSdistR}
\dist\big((X_S,\tau_S),L_R\big)
\lesssim
\sum_{P:\,S\subset P\subseteq R}\alpha(P)\ell(P).
\end{align}

We next estimate
\[
    \big|\pi_{L_R^\perp}\big((X,t)-(X_S,\tau_S)\big)\big|.
\]
Set $V:=(X,t)-(X_S,\tau_S)$.
Telescoping along the ancestors from $S$ to
$R$ gives
\[
\big|\pi_{L_R^\perp}V\big|
\leq
\big|\pi_{L_S^\perp}V\big|
+
\sum_{P:\,S\subseteq P\subsetneq R}
\big|
(\pi_{L_P^\perp}-\pi_{L_{P'}^\perp})V
\big|,
\]
where $P'$ is the parent of $P$. Since $(X_S,\tau_S)\in L_S$,
\[
    \big|\pi_{L_S^\perp}V\big|
    =
    \dist\big((X,t),L_S\big).
\]
Lemma~\ref{toalphalemma}, applied to the pair $P\subset P'$, gives
\[
    \|\pi_{L_P^\perp}-\pi_{L_{P'}^\perp}\|
    \lesssim\alpha(P').
\]
Indeed, when $\alpha(P')$ is small this follows from the Hausdorff comparison
of the planes on $C_{P'}$, while when $\alpha(P')$ is bounded below it follows
from the trivial bound on the two projections.
Moreover, the projection differences vanish on the time direction, so only
the spatial component of $V$ matters. Since $(X,t)\in U_{\frac{\gamma}{4}\ell(S)}(S)$ and
$|(X_S,\tau_S)-(Z_S,\tau_S)|\leq3\ell(S)$, we have
\[
|X-X_S|\lesssim\ell(S).
\]
It follows that
\begin{align}\label{bigAR3}
\big|\pi_{L_R^\perp}\big((X,t)-(X_S,\tau_S)\big)\big|
\lesssim
\dist\big((X,t),L_S\big)
+
\ell(S)
\sum_{P:\,S\subset P\subseteq R}\alpha(P).
\end{align}

Combining \eqref{bigAR1}, \eqref{bigAR2}, \eqref{XSdistR}, and
\eqref{bigAR3}, using $\ell(P)\leq\ell(R)$ for $S\subset P\subseteq R$, and
then using
\[
    \frac{\ell(S)}{\ell(R)}
    \leq
    \left(\frac{\ell(S)}{\ell(R)}\right)^{1/2},
\]
we obtain
\begin{align}
\begin{aligned} 
A_R
\lesssim&
\sum_{S\in\mathcal S_{N_0}(R)}
\frac{\ell(S)}{\ell(R)}
\int_{U_{\frac{\gamma}{4}\ell(S)}(S)}
\mathcal T(S,R)
\frac{\dist\big((X,t),L_S\big)}{\ell(S)}
\,\d\mu(X,t)
\\
&+
\sum_{S\in\mathcal S_{N_0}(R)}
\left(\frac{\ell(S)}{\ell(R)}\right)^{1/2}
\int_{U_{\frac{\gamma}{4}\ell(S)}(S)}
\mathcal T(S,R)
\left[
    \sum_{P:\,S\subset P\subseteq R}\alpha(P)
\right]
\,\d\mu(X,t)
\\
\label{ARsplit}=:&\  A_R^1+A_R^2. \end{aligned}
\end{align}

We first estimate the sum of the $A_R^1$'s. By Cauchy--Schwarz,
\begin{align}
\begin{aligned}\label{A1split}
A_R^1
\lesssim&
\sum_{S\in\mathcal S_{N_0}(R)}
\frac{\ell(S)}{\ell(R)}
\int_{U_{\frac{\gamma}{4}\ell(S)}(S)}\mathcal T(S,R)^2\,\d\mu
\\
&+
\sum_{S\in\mathcal S_{N_0}(R)}
\frac{\ell(S)}{\ell(R)}
\int_{U_{\frac{\gamma}{4}\ell(S)}(S)}
\left(
    \frac{\dist\big((X,t),L_S\big)}{\ell(S)}
\right)^2
\,\d\mu.
\end{aligned}
\end{align}

We estimate the term involving $\mathcal T(S,R)$. It suffices to treat
$T_{(2J(S)-J(R))}\mu$ because the term with index
$2J(S)-J(R)+1$ is nearly identical.  Put $q=2J(S)-J(R)$. Since
\[
T_{(q)}\mu=\sum_{Q\in\mathbb D_q}T_Q\mu
\qquad \mu\text{-a.e.},
\]
and the cubes in $\mathbb D_q$ are disjoint,
\[
\int_{U_{\frac{\gamma}{4}\ell(S)}(S)}|T_{(q)}\mu|^2\,\d\mu
\leq
\sum_{\substack{Q\in\mathbb D_q\\Q\cap U_{\frac{\gamma}{4}\ell(S)}(S)\neq\emptyset}}
\|T_Q\mu\|_{L^2(\mu)}^2
\lesssim
\sum_{\substack{Q\in\mathbb D_q\\Q\cap U_{\frac{\gamma}{4}\ell(S)}(S)\neq\emptyset}}
\mathcal A(Q)^2\mu(Q),
\]
where we used Lemma~\ref{TQintbd}(c). Since $J(Q)=2J(S)-J(R)$,
\[
\frac{\ell(S)}{\ell(R)}
=
\left(\frac{\ell(Q)}{\ell(R)}\right)^{1/2}.
\]
For the term with $J(Q)=2J(S)-J(R)+1$ the above display holds with comparability instead of equality. Also, since $S\in\mathcal S_{N_0}(R)$, in either case $J(Q)-J(R)>N_0-1$.

Finally, if $Q\cap U_{\frac{\gamma}{4}\ell(S)}(S)\neq\emptyset$, then $S\subset R$, $\ell(Q)\leq\ell(S)$, and
\[
\dist(Q,R)
\leq
\diam(Q)+\frac{\gamma}{4}\ell(S)
< \frac{1}{2}\ell(S),
\]
so $Q\subset3R$. Since at each generation there are only boundedly many relevant cubes $S$, we obtain
\begin{align*}
&\sum_{R\in\mathbb D}\sum_{S\in\mathcal S_{N_0}(R)}
\frac{\ell(S)}{\ell(R)}
\int_{U_{\frac{\gamma}{4}\ell(S)}(S)}\mathcal T(S,R)^2\,\d\mu\\
&\qquad\lesssim
\sum_{Q\in\mathbb D}\mathcal A(Q)^2\mu(Q)
\sum_{\substack{R\in\mathbb D\\Q\subset3R\\J(Q)-J(R)>N_0-1}}
\left(\frac{\ell(Q)}{\ell(R)}\right)^{1/2}.
\end{align*}
For fixed $Q$, writing $r=J(Q)-J(R)$ and using bounded overlap of the cubes $3R$ at each generation,
\[
\sum_{\substack{R\in\mathbb D\\Q\subset3R\\J(Q)-J(R)>N_0-1}}
\left(\frac{\ell(Q)}{\ell(R)}\right)^{1/2}
\lesssim
\sum_{r\geq N_0}2^{-r/2}
\lesssim
2^{-N_0/2}.
\]
Therefore, by Proposition~\ref{calAQ},
\[
\sum_{R\in\mathbb D}\sum_{S\in\mathcal S_{N_0}(R)}
\frac{\ell(S)}{\ell(R)}
\int_{U_{\frac{\gamma}{4}\ell(S)}(S)}\mathcal T(S,R)^2\,\d\mu
\lesssim
2^{-N_0/2}
\sum_{Q\in\mathbb D}\alpha(Q)^2\mu(Q).
\]

We now deal with the other term in \eqref{A1split}.  By Lemma~\ref{dillem}, namely
\eqref{Adilated}, we have
\begin{align*}
&\sum_{R\in\mathbb D}
\sum_{S\in\mathcal S_{N_0}(R)}
\frac{\ell(S)}{\ell(R)}
\int_{U_{\frac{\gamma}{4}\ell(S)}(S)}
\left(
    \frac{\dist\big((X,t),L_S\big)}{\ell(S)}
\right)^2
\,\d\mu
\\
&\qquad\lesssim
\sum_{R\in\mathbb D}
\sum_{S\in\mathcal S_{N_0}(R)}
\frac{\ell(S)}{\ell(R)}
\sum_{T\in N(S)\cup\{S\}}
\mathcal A(T)^2\mu(T)
\\
&\qquad\lesssim
\sum_{T\in\mathbb D}\mathcal A(T)^2\mu(T)
\sum_{\substack{R\in\mathbb D\\
3R\supset T\\
J(T)-J(R)\geq (N_0-1)/2}}
\frac{\ell(T)}{\ell(R)}
\\
&\qquad\lesssim
2^{-N_0/2}
\sum_{T\in\mathbb D}\mathcal A(T)^2\mu(T)
\lesssim
2^{-N_0/2}
\sum_{T\in\mathbb D}\alpha(T)^2\mu(T),
\end{align*}
Consequently,
\[
    \sum_{R\in\mathbb D}A_R^1
    \lesssim
    2^{-N_0/2}
    \sum_{Q\in\mathbb D}\alpha(Q)^2\mu(Q).
\]

We now estimate the sum of the $A_R^2$'s. By Cauchy--Schwarz,
\begin{align*}
A_R^2
\lesssim&
\sum_{S\in\mathcal S_{N_0}(R)}
\left(\frac{\ell(S)}{\ell(R)}\right)^{1/2}
\int_{U_{\frac{\gamma}{4}\ell(S)}(S)}\mathcal T(S,R)^2\,\d\mu
\\
&+
\sum_{S\in\mathcal S_{N_0}(R)}
\left(\frac{\ell(S)}{\ell(R)}\right)^{1/2}
\left[
    \sum_{P:\,S\subset P\subseteq R}\alpha(P)
\right]^2
\mu(S).
\end{align*}
For the first term, we repeat the argument used for the first term in \eqref{A1split} and obtain
\begin{align}
\sum_{R \in \mathbb D}\sum_{S\in\mathcal S_{N_0}(R)}
\left(\frac{\ell(S)}{\ell(R)}\right)^{1/2}
\int_{U_{\frac{\gamma}{4}\ell(S)}(S)}\mathcal T(S,R)^2\,\d\mu
\lesssim 
2^{- N_0/4} \sum_{Q \in \mathbb D} \alpha(Q)^2 \mu(Q)
\end{align}

So, it remains to deal with the other term. Since the number of cubes $P$ such that
$S\subset P\subseteq R$ is at most $1+J(S)-J(R)$, Cauchy--Schwarz gives
\[
\left[
    \sum_{P:\,S\subset P\subseteq R}\alpha(P)
\right]^2
\lesssim
\log\left(2+\frac{\ell(R)}{\ell(S)}\right)
\sum_{P:\,S\subset P\subseteq R}\alpha(P)^2.
\]
For $0<x\leq1$, $x^{1/2}\log(2+x^{-1})\lesssim x^{1/3}.$
Moreover, if $S\in\mathcal S_{N_0}(R)$, then
\[
J(S)-J(R)\geq\frac{N_0-1}{2}.
\]
Hence
\[
\left(\frac{\ell(S)}{\ell(R)}\right)^{1/3}
=
\left(\frac{\ell(S)}{\ell(R)}\right)^{1/4}
2^{-\frac{J(S)-J(R)}{12}}
\lesssim
2^{-N_0/24}
\left(\frac{\ell(S)}{\ell(R)}\right)^{1/4}.
\] Therefore
\begin{align*}
&\sum_{R\in\mathbb D}
\sum_{S\in\mathcal S_{N_0}(R)}
\left(\frac{\ell(S)}{\ell(R)}\right)^{1/2}
\left[
    \sum_{P:\,S\subset P\subseteq R}\alpha(P)
\right]^2
\mu(S)
\\
&\qquad\lesssim
2^{-N_0/24}
\sum_{R\in\mathbb D}
\sum_{S\subset R}
\left(\frac{\ell(S)}{\ell(R)}\right)^{1/4}
\sum_{P:\,S\subset P\subseteq R}\alpha(P)^2\mu(S)
\\
&\qquad=
2^{-N_0/24}
\sum_{P\in\mathbb D}\alpha(P)^2
\left(
    \sum_{S\subset P}\ell(S)^{1/4}\mu(S)
\right)
\left(
    \sum_{R\supset P}\ell(R)^{-1/4}
\right)
\\
&\qquad\lesssim
2^{-N_0/24}
\sum_{P\in\mathbb D}\alpha(P)^2\mu(P).
\end{align*}
Indeed, summing generation by generation,
\[
\sum_{S\subset P}\ell(S)^{1/4}\mu(S)
\lesssim
\ell(P)^{1/4}\mu(P),
\]
while, since $P$ has a unique ancestor at each generation,
\[
\sum_{R\supseteq P}\ell(R)^{-1/4}
\lesssim
\ell(P)^{-1/4}.
\]
Combining the estimates for $A_R^1$ and $A_R^2$, we conclude that
\[
\sum_{R\in\mathbb D}A_R
\lesssim
2^{-N_0/24}
\sum_{Q\in\mathbb D}\alpha(Q)^2\mu(Q).
\]

\subsubsection{Further decomposing the $B_{j,i}$'s}

We retain the notation from the preceding decomposition, so that $i>j$,
$m=\lfloor(i+j)/2\rfloor$. Set $r_i:=2^{-i}$ and $r_m:=2^{-m}$. Thus $\ell(Q)=r_i$ for $Q\in\mathbb D_i$, while $\ell(S)=\ell(T)=r_m$ for $S,T\in\mathbb D_m$. Then
\[
B_{j,i}
=
\sum_{S\in\mathbb D_m}
T_{(j)}\mu(X_S,\tau_S)\,
\langle\theta_S,T_{(i)}\mu\rangle.
\]
We justify the rearrangements below by first using finitely supported coefficients. For fixed $i,j$, set $a_S:=T_{(j)}\mu(X_S,\tau_S)$. Estimate \eqref{coarsepointBselfestimate} below, weighted Cauchy--Schwarz, and the partition of each ancestor into its generation-$m$ descendants give
$$
\sum_{S\in\mathbb D_m}|a_S|^2\mu(S)\leq C\sum_{\substack{P\in\mathbb D\\j\leq J(P)\leq m}}\alpha(P)^2\mu(P)<\infty.
$$
The bounded overlap of the partition also gives
$$
\sum_{S\in\mathbb D_m}\frac{|\langle\theta_S,T_{(i)}\mu\rangle|^2}{\mu(S)}\leq C\|T_{(i)}\mu\|_{L^2(\mu)}^2<\infty.
$$
Here $C$ depends only on $n$, $k$, $C_0$, and $\|K\|_{\mathrm{CZ},4}$, and is independent of $i,j$. For $\delta>0$, set
$$
a_S^\delta:=\operatorname{sgn}(a_S)(|a_S|-\delta)_+,\qquad B_{j,i}^\delta:=\sum_{S\in\mathbb D_m}a_S^\delta\langle\theta_S,T_{(i)}\mu\rangle,
$$
where $u_+:=\max\{u,0\}$. Since the generation-$m$ cube masses have a uniform positive lower bound, only finitely many $a_S^\delta$ are nonzero. Moreover, $|a_S^\delta|\leq|a_S|$ and $|a_S^\delta-a_T^\delta|\leq|a_S-a_T|$. Dominated convergence in $\ell^2(\mu(S))$, followed by Cauchy--Schwarz, gives $B_{j,i}^\delta\to B_{j,i}$ as $\delta\downarrow0$.

In the definitions and rearrangements of $B^1_{j,i}$, $B^2_{j,i}$, and their remainders below, we interpret the series as limits of the finite sums obtained by replacing each coefficient $a_S$ by $a_S^\delta$. Write $B^{1,\delta}_{j,i}$ and $B^{2,\delta}_{j,i}$ for the corresponding finite sums. The displayed formulas retain the original coefficients; the algebraic identities are first applied to the finite sums. The bounds on $|a_S^\delta|$ and $|a_S^\delta-a_T^\delta|$ ensure that the estimates for these sums are uniform in $\delta$. All lemmas concerning $T_{(j)}\mu$ refer to the original operator. We verify the existence of the limits after the estimates.

Since $K$ is odd in the spatial variables and even in time, \begin{equation}\label{skewkernel} K_{(i)}\big((X,t)-(Y,s)\big)=-K_{(i)}\big((Y,s)-(X,t)\big). \end{equation} For a measure $\mathcal L$, we write
\[
\langle f,g\rangle_{\mathcal L}
:=
\int fg\,d\mathcal L,
\]
and set
\[
\mathcal E_{S,T}
:=
\langle\theta_S,T_{(i)}(\theta_T\mu)\rangle
-
\langle
\theta_S,T_{(i)}(\theta_T\mathcal L_S)
\rangle_{\mathcal L_S}.
\]
We define
\begin{align*}
B^1_{j,i}
&:=
\sum_{S\in\mathbb D_m}
\sum_{T\in N(S)}
\bigl(
T_{(j)}\mu(X_S,\tau_S)
-
T_{(j)}\mu(X_T,\tau_T)
\bigr)
\langle
\theta_S,T_{(i)}(\theta_T\mathcal L_S)
\rangle_{\mathcal L_S},
\\
B^2_{j,i}
&:=
\sum_{S\in\mathbb D_m}
\sum_{T\in N(S)}
\bigl(
T_{(j)}\mu(X_S,\tau_S)
-
T_{(j)}\mu(X_T,\tau_T)
\bigr)
\mathcal E_{S,T}.
\end{align*}

Suppose that
\[
\langle\theta_S,T_{(i)}(\theta_T\mu)\rangle\neq0.
\]
Then there exist $(X,t)\in\supp(\theta_S)\cap E$ and
$(Y,s)\in\supp(\theta_T)\cap E$ such that
\[
K_{(i)}\big((X,t)-(Y,s)\big)\neq0.
\]
Since $S,T\in\mathbb D_m$, we have $\ell(S)=\ell(T)=2^{-m}$, and therefore
\[
\dist(S,T)
\leq
\frac{\gamma}{4}\ell(S)
+
2^{-i-\omega+2}
+
\frac{\gamma}{4}\ell(T).
\]
Since $i>m$,
\[
2^{-i-\omega+2}
\leq
2^{1-\omega}\ell(S)
\leq
\frac{\gamma}{2}\ell(S),
\]
and hence
\[
\dist(S,T)\leq\gamma\ell(S).
\]
Thus either $T=S$ or $T\in N(S)$.

Using $\sum_{T\in\mathbb D_m}\theta_T=1$ on $E$, we obtain
\begin{equation}
\label{premv1-corrected}
B_{j,i}
=
\sum_{S\in\mathbb D_m}
\sum_{T\in N(S)\cup\{S\}}
T_{(j)}\mu(X_S,\tau_S)\,
\langle\theta_S,T_{(i)}(\theta_T\mu)\rangle.
\end{equation}
By \eqref{skewkernel}, \[ \langle f,T_{(i)}(g\mu)
\rangle_\mu=-\langle g,T_{(i)}(f\mu)
\rangle_\mu. \] Using this identity, the symmetry of the neighbor relation, and a change of indices,
\eqref{premv1-corrected} also gives
\begin{equation}
\label{premv2-corrected}
B_{j,i}
=
-\sum_{S\in\mathbb D_m}
\sum_{T\in N(S)\cup\{S\}}
T_{(j)}\mu(X_T,\tau_T)\,
\langle\theta_S,T_{(i)}(\theta_T\mu)\rangle.
\end{equation}
Averaging \eqref{premv1-corrected} and \eqref{premv2-corrected}, and using the definitions above, yields
\begin{equation}
\label{mv}
B_{j,i}
=
\frac12
\bigl(
B^1_{j,i}+B^2_{j,i}
\bigr).
\end{equation}

\subsubsection{Estimating the Sum of the $B_{j,i}^1$'s}

\begingroup  Define \[ \Omega_m^{(2)}:= \left\{(X,t)\in\mathbb P^n: \dist\big((X,t),E\big)<\frac{3\gamma}{8}r_m \right\}. \] There exists a family \[ \{\psi_T\}_{T\in\mathbb D_m} \subset C_c^\infty(\mathbb P^n) \] whose supports have uniformly bounded overlap and such that \begin{equation}\label{planeambientpartitionidentity} \psi_T=\theta_T \qquad\text{on }E, \end{equation} \begin{equation}\label{planeambientpartitionsum} \sum_{T\in\mathbb D_m}\psi_T=1 \qquad\text{on }\Omega_m^{(2)}. \end{equation} In addition, \[ \supp(\psi_T)\subset U_{\frac{\gamma}{2}r_m}(T), \qquad |\psi_T|\lesssim1, \] and \[ \big|\nabla_X^a\partial_t^q\psi_T\big| \lesssim_{a,q}r_m^{-(|a|+2q)}, \qquad |a|+q\leq2. \] The construction is given in Appendix~\ref{smoothpartitionappendix}.

\begingroup
We use the following four lemmas to estimate the sum of the $B^1_{j,i}$'s.
We prove them after completing that estimate.

For neighboring cubes $S,T$ of the same generation, set
\[
    \alpha^\#(S,T):=\max\{\alpha(S),\alpha(T)\}.
\]

For $S\in\mathbb D$, define
\[
    \alpha^\#(S)
    :=
    \max_{T\in N(S)\cup\{S\}}\alpha(T).
\]

\begin{lemma}\label{B1ambientcancellationlemma}
For every $S\in\mathbb D_m$,
\begin{equation}\label{B1firstambientcancellation}
\sum_{T\in N(S)\cup\{S\}}
\langle
\theta_S,T_{(i)}(\psi_T\mathcal L_S)
\rangle_{\mathcal L_S}
=
\langle
\theta_S,T_{(i)}\mathcal L_S
\rangle_{\mathcal L_S}
=
0.
\end{equation}
For every $T\in\mathbb D_m$,
\begin{equation}\label{B1secondambientcancellation}
\sum_{S\in N(T)\cup\{T\}}
\langle
\psi_S,T_{(i)}(\theta_T\mathcal L_T)
\rangle_{\mathcal L_T}
=
\langle
1,T_{(i)}(\theta_T\mathcal L_T)
\rangle_{\mathcal L_T}
=
0.
\end{equation}
\end{lemma}

\begin{lemma}\label{neighborprojectionmeasure}
Let $S,T\in\mathbb D_m$ with $T\in N(S)$ and suppose that
$\alpha^\#(S,T)<\eta_0$, where $\eta_0$ is the constant from Lemma \ref{elemalphalem} (c). For $U\in\{S,T\}$, write
\[
L_U=(a_U+V_U)\times\mathbb R,
\]
where $a_U\in\mathbb R^n$ and $V_U\subset\mathbb R^n$ is a linear
subspace. Let $p:L_S\to L_T$ be the orthogonal projection. 
If $\pi_{V_T}$ denotes orthogonal
projection onto $V_T$, define
\[
\widetilde p
:=
\pi_{V_T}|_{V_S}:V_S\longrightarrow V_T.
\]
Define the pullback measure $p^\#\mathcal L_T$ on $L_S$ by
\[
    p^\#\mathcal L_T(A):=\mathcal L_T(p(A)),
    \qquad A\subset L_S.
\]
Then
\[
    p^\#\mathcal L_T=a_{S,T}\mathcal L_S
\]
for a constant $a_{S,T}$ satisfying
\begin{align}\label{projectionmeasurebounds}
    |a_{S,T}-1|
    +\|\widetilde p-\iota_S\|
    \lesssim
    \alpha^\#(S,T),
\end{align}
where $\iota_S:V_S\to\mathbb R^n$ is the inclusion map and
$\widetilde p$ is also viewed as taking values in $\mathbb R^n$.
\end{lemma}

\begin{lemma}\label{neighborslem}
Let $S,T\in\mathbb D_m$ with $T\in N(S)$. Then
\begin{align}\label{neighbors}
&\left|
\langle\theta_S,T_{(i)}(\theta_T\mathcal L_T)\rangle_{\mathcal L_T}
-
\langle\theta_S,T_{(i)}(\theta_T\mathcal L_S)\rangle_{\mathcal L_S}
\right|
\nonumber\\
&\hspace{3cm}\lesssim
\frac{2^{-i}}{\ell(S)}
\alpha^\#(S,T)\mu(S)
\lesssim
2^{-(i-j)/2}\alpha^\#(S,T)\mu(S).
\end{align}
\end{lemma}

\begin{lemma}\label{coarsepointB}
Let $S,T\in\mathbb D_m$ with $T\in N(S)\cup\{S\}$, let $R\in\mathbb D_j$ be the
ancestor of $S$, and assume $m-j\geq\lfloor(N_0+1)/2\rfloor$. Then
\begin{align}\label{coarsepointBestimate}
|T_{(j)}\mu(X_T,\tau_T)|
\lesssim
\alpha^\#(S)\frac{\ell(S)}{\ell(R)}
+
\sum_{P:\,S\subseteq P\subseteq R}
\alpha(P)\frac{\ell(P)}{\ell(R)}.
\end{align} In particular, if $T=S$, then \begin{align}\label{coarsepointBselfestimate}|T_{(j)}\mu(X_S,\tau_S)|\lesssim\sum_{P:\,S\subseteq P\subseteq R}\alpha(P)\frac{\ell(P)}{\ell(R)}.\end{align}
\end{lemma}
\endgroup

We now use these lemmas to estimate the sum of the $B^1_{j,i}$'s.

Set
\[
\delta_T:=\theta_T-\psi_T.
\]
Then $\delta_T=0$ on $E$,
\[
\supp(\delta_T)
\subset
U_{\frac{\gamma}{2}r_m}(T),
\]
and
\[
\|\delta_T\|_{\operatorname{Lip}}
\lesssim r_m^{-1}.
\]
Consequently,
\[
|\delta_T(X,t)|
\lesssim
r_m^{-1}\dist\big((X,t),E\big).
\]
If $T\in N(S)\cup\{S\}$, the support of $\delta_T$ on $L_S$ is contained
in $C_{2r_m}(Z_S,\tau_S)$. Hence \eqref{LQbetacomp} gives
\begin{equation}\label{B1partitiondefectL1}
\begin{aligned}
\int_{L_S}|\delta_T|\,d\mathcal L_S
&\lesssim
r_m^{-1}
\int_{L_S\cap C_{2r_m}(Z_S,\tau_S)}
\dist\big((X,t),E\big)\,d\mathcal L_S(X,t)
\\
&\lesssim
\alpha(S)\mu(S).
\end{aligned}
\end{equation}

Interchanging $S$ and $T$ in the preceding argument and using
$\mu(T)\approx\mu(S)$ gives
\[
\int_{L_T}|\delta_S|\,d\mathcal L_T
\lesssim
\alpha(T)\mu(T)
\lesssim
\alpha(T)\mu(S).
\]

Define
\begin{align*}
\mathcal R^1_{j,i}
&:=
\sum_{S\in\mathbb D_m}
T_{(j)}\mu(X_S,\tau_S)
\sum_{T\in N(S)\cup\{S\}}
\langle
\theta_S,T_{(i)}(\delta_T\mathcal L_S)
\rangle_{\mathcal L_S},
\\
\mathcal R^2_{j,i}
&:=
\sum_{T\in\mathbb D_m}
T_{(j)}\mu(X_T,\tau_T)
\sum_{S\in N(T)\cup\{T\}}
\langle
\delta_S,T_{(i)}(\theta_T\mathcal L_T)
\rangle_{\mathcal L_T},
\end{align*}
and set
\[
\mathcal R_{j,i}:=\mathcal R^1_{j,i}-\mathcal R^2_{j,i}.
\]
Since $\theta_T=\psi_T+\delta_T$, \eqref{B1firstambientcancellation}
gives
\[
\sum_{S\in\mathbb D_m}
T_{(j)}\mu(X_S,\tau_S)
\sum_{T\in N(S)\cup\{S\}}
\langle
\theta_S,T_{(i)}(\theta_T\mathcal L_S)
\rangle_{\mathcal L_S}
=
\mathcal R^1_{j,i}.
\]
Similarly, using the symmetry of the neighbor relation and
\eqref{B1secondambientcancellation},
\[
\sum_{S\in\mathbb D_m}
\sum_{T\in N(S)\cup\{S\}}
T_{(j)}\mu(X_T,\tau_T)
\langle
\theta_S,T_{(i)}(\theta_T\mathcal L_T)
\rangle_{\mathcal L_T}
=
\mathcal R^2_{j,i}.
\]
The diagonal terms $S=T$ in these two identities are identical. Therefore,
\begin{align*}
&\sum_{S\in\mathbb D_m}\sum_{T\in N(S)}
T_{(j)}\mu(X_S,\tau_S)
\langle
\theta_S,T_{(i)}(\theta_T\mathcal L_S)
\rangle_{\mathcal L_S}
\\
&\qquad=
\sum_{S\in\mathbb D_m}\sum_{T\in N(S)}
T_{(j)}\mu(X_T,\tau_T)
\langle
\theta_S,T_{(i)}(\theta_T\mathcal L_T)
\rangle_{\mathcal L_T}
+
\mathcal R_{j,i}.
\end{align*}

Consequently,
\begin{align}\label{bjk1d}
B^1_{j,i}
&=
\sum_{S\in\mathbb D_m}\sum_{T\in N(S)}
T_{(j)}\mu(X_T,\tau_T)
\bigg[
\langle
\theta_S,T_{(i)}(\theta_T\mathcal L_T)
\rangle_{\mathcal L_T}
\nonumber\\
&\hspace{5.5cm}
-
\langle
\theta_S,T_{(i)}(\theta_T\mathcal L_S)
\rangle_{\mathcal L_S}
\bigg]
+
\mathcal R_{j,i}.
\end{align}

To estimate $\mathcal R_{j,i}$, we first bound $T_{(i)}(\theta_U\mathcal L_U)(X,t)$ for $U\in\mathbb D_m$ and $(X,t)\in L_U$. Since $K_{(i)}$ is odd in the spatial variables, $T_{(i)}\mathcal L_U(X,t)=0$, and hence
\[
T_{(i)}(\theta_U\mathcal L_U)(X,t)
=
\int_{L_U}
K_{(i)}\big((X,t)-(Y,s)\big)
\big(
\theta_U(Y,s)
-
\theta_U(X,t)
\big)
\,d\mathcal L_U(Y,s).
\]
The estimates of the derivatives of $\theta_U$ in \eqref{partitionhigherderivatives} imply
\[
\big|
\theta_U(Y,s)
-
\theta_U(X,t)
\big|
\lesssim
\frac{d_p((X,t),(Y,s))}{r_m}
\lesssim
\frac{r_i}{r_m}
\]
on the support of the kernel. Moreover,
\[
\int_{L_U}
\left|
K_{(i)}\big((X,t)-(Y,s)\big)
\right|
\,d\mathcal L_U(Y,s)
\lesssim1,
\]
because $|K_{(i)}|\lesssim r_i^{-k-2}$ and the relevant portion of the vertical
plane has measure $\lesssim r_i^{k+2}$. Therefore
\[
\|T_{(i)}(\theta_T\mathcal L_T)\|_{L^\infty(L_T)}
\lesssim
\frac{r_i}{r_m}.
\]
For every $S\in\mathbb D_m$ and $T\in N(S)\cup\{S\}$, Fubini's theorem and \eqref{skewkernel} give
\[
\left\langle
\theta_S,
T_{(i)}(\delta_T\mathcal L_S)
\right\rangle_{\mathcal L_S}
=
-
\left\langle
\delta_T,
T_{(i)}(\theta_S\mathcal L_S)
\right\rangle_{\mathcal L_S}.
\]
Consequently, \eqref{B1partitiondefectL1} and the preceding
$L^\infty(L_S)$ estimate imply
\[
\left|
\left\langle
\theta_S,
T_{(i)}(\delta_T\mathcal L_S)
\right\rangle_{\mathcal L_S}
\right|
\lesssim
\frac{r_i}{r_m}\alpha(S)\mu(S).
\]
Similarly, the corresponding term in $\mathcal R^2_{j,i}$ satisfies
\[
\left|
\left\langle
\delta_S,
T_{(i)}(\theta_T\mathcal L_T)
\right\rangle_{\mathcal L_T}
\right|
\lesssim
\frac{r_i}{r_m}\alpha(T)\mu(S).
\]
Substituting these two estimates into the definitions of
$\mathcal R^1_{j,i}$ and $\mathcal R^2_{j,i}$ gives
\begin{align}\label{B1partitionremainder}
|\mathcal R_{j,i}|
\lesssim
\frac{r_i}{r_m}
\sum_{S\in\mathbb D_m}
\sum_{T\in N(S)\cup\{S\}}
&\Big(
|T_{(j)}\mu(X_S,\tau_S)|
+
|T_{(j)}\mu(X_T,\tau_T)|
\Big)
\nonumber\\
&\qquad\qquad\times
\big(\alpha(S)+\alpha(T)\big)\mu(S).
\end{align}
Since $\frac{r_i}{r_m}\approx2^{-(i-j)/2}$ and each cube has a uniformly bounded
number of neighbors, we sum \eqref{B1partitionremainder} in the same way
as the main term in \eqref{bjk1d} below.

\endgroup

Because every cube has a uniformly bounded number of neighbors,
\begin{align}\label{alphasharpsquare}
    \sum_{S\in\mathbb D}\alpha^\#(S)^2\mu(S)
    \lesssim
    \sum_{Q\in\mathbb D}\alpha(Q)^2\mu(Q).
\end{align}

We now estimate $$\sum_{\substack{i,j\in\mathcal J\\ i-j > N_0}} |B^1_{j,i}|.$$  
Fix $i,j$ with $i-j>N_0$, and put $m=\lfloor(i+j)/2\rfloor$. For each $S\in\mathbb D_m$, let $R_j(S)\in\mathbb D_j$ be the unique ancestor of $S$. Since $m=\lfloor(i+j)/2\rfloor$, \[ i-j\in\{2(m-j),\,2(m-j)+1\}, \qquad \frac{r_i}{r_m}\leq\frac{\ell(S)}{\ell(R_j(S))}. \] Define \[ F(S,R):=\frac{\ell(S)}{\ell(R)}\alpha^\#(S)\mu(S)\left[\alpha^\#(S)\frac{\ell(S)}{\ell(R)}+\sum_{P:\,S\subseteq P\subseteq R}\alpha(P)\frac{\ell(P)}{\ell(R)}\right]. \] By \eqref{bjk1d}, Lemmas~\ref{neighborslem} and~\ref{coarsepointB}, and the fact that each cube has a uniformly bounded number of neighbors, \[ |B^1_{j,i}|\lesssim\sum_{S\in\mathbb D_m}F(S,R_j(S))+|\mathcal R_{j,i}|. \]
We have 
\[ \sum_{S\in\mathbb D_m}F(S,R_j(S))=\sum_{R\in\mathbb D_j}\sum_{\substack{S\in\mathbb D_m\\S\subset R}}F(S,R). \] For $S\subset R$, set $r(S,R):=J(S)-J(R)$. Here $j=J(R)$ and $m=J(S)$, so each pair $S\subset R$ corresponds to at most two values of $i$, and $i-j>N_0$ implies $2r(S,R)+1>N_0$. Summing the preceding fixed-$(i,j)$ estimate therefore gives

\begin{align}\label{B1mastersum} &\sum_{\substack{i,j\in\mathcal J\\i-j>N_0}} |B^1_{j,i}| 
\nonumber\\ &\quad\lesssim \sum_{R\in\mathbb D} \sum_{\substack{S\subset R\\2r(S,R)+1>N_0}} \frac{\ell(S)}{\ell(R)} \alpha^\#(S)\mu(S) \left[ \alpha^\#(S)\frac{\ell(S)}{\ell(R)} + \sum_{P:\,S\subseteq P\subseteq R} \alpha(P)\frac{\ell(P)}{\ell(R)} 
\right] + \sum_{\substack{i,j\in\mathcal J\\i-j>N_0}}|\mathcal R_{j,i}| =: \mathcal{I} + \mathcal{II} + \mathcal{III}. \end{align}

We estimate term $\mathcal{I}$ as follows:
\begin{align*}
\mathcal I \leq \sum_{R\in\mathbb D}
\sum_{\substack{S\subset R\\2r(S,R)+1>N_0}}
\left(\frac{\ell(S)}{\ell(R)}\right)^2
\alpha^\#(S)^2\mu(S)
\lesssim
2^{-N_0}
\sum_{S\in\mathbb D}\alpha^\#(S)^2\mu(S).
\end{align*}
For term $\mathcal{II}$, since $\alpha(P)\leq\alpha^\#(P)$, we have
$2\alpha^\#(S)\alpha(P)\leq\alpha^\#(S)^2+\alpha^\#(P)^2$.
We estimate the contribution containing $\alpha^\#(S)^2$ in the same way as term $\mathcal I$, now obtaining a factor of $2^{-N_0/2}$,
since
$$\sum_{P:\,S\subseteq P\subseteq R}
\frac{\ell(P)}{\ell(R)}
\lesssim1.$$
To complete the estimate of term $\mathcal {II}$, it remains to bound
\begin{align*}
\Sigma
:=
\sum_{R\in\mathbb D}
\sum_{\substack{S\subset R\\2r(S,R)+1>N_0}}
\frac{\ell(S)}{\ell(R)}\mu(S)
\sum_{P:\,S\subseteq P\subseteq R}
\alpha^\#(P)^2\frac{\ell(P)}{\ell(R)}.
\end{align*}
Using Tonelli's theorem, we can interchange the inner two sums:
\begin{align*}
\Sigma
=
\sum_{R\in\mathbb D}
\sum_{P: P \subseteq R}
\alpha^\#(P)^2
\frac{\ell(P)}{\ell(R)}
\sum_{\substack{S:\,S\subseteq P \\ 2J(S) - 2J(R) > N_0-1}}
\frac{\ell(S)}{\ell(R)} \mu(S).
\end{align*}
Hence,
\begin{align*}
\Sigma=\sum_{R\in\mathbb D}\sum_{q\geq0}\sum_{\substack{P\subseteq R\\J(P)=J(R)+q}}\alpha^\#(P)^2 2^{-q}\sum_{\substack{r\geq q\\2r+1>N_0}}2^{-r}\sum_{\substack{S\subseteq P\\J(S)=J(R)+r}}\mu(S).
\end{align*}
Using H\"older, we have
$$\sum_{\substack{r\geq q\\2r+1>N_0}}2^{-r} \lesssim 2^{-q/2} 2^{-N_0/4}$$
and thus,
$$\Sigma \lesssim 2^{-N_0/4} \sum_{R \in \mathbb D} \sum_{q \geq 0} \sum_{\substack{P : P \subset R \\ J(P) = J(R) +q}} \alpha^{\#}(P)^2 2^{-\frac{3}{2}q} \mu(P)
\lesssim 2^{-N_0/4} \sum_{R \in \mathbb D} \alpha^{\#}(R)^2\mu(R) \lesssim 2^{-N_0/4} \sum_{R \in \mathbb D} \alpha(R)^2 \mu(R).$$
It remains to estimate $\mathcal{III}$. Fix $i,j$ and put $m=\lfloor(i+j)/2\rfloor$. For each $S\in\mathbb D_m$, let $R=R_j(S)\in\mathbb D_j$ be its unique ancestor. By \eqref{B1partitionremainder}, Lemma~\ref{coarsepointB}, and the fact that each cube has a uniformly bounded number of neighbors, grouping the cubes $S$ according to their ancestor $R_j(S)$ gives
\begin{align*}|\mathcal R_{j,i}|&\lesssim\sum_{R\in\mathbb D_j}\sum_{\substack{S\in\mathbb D_m\\S\subset R}}\frac{\ell(S)}{\ell(R)}\alpha^\#(S)\mu(S)\left[\alpha^\#(S)\frac{\ell(S)}{\ell(R)}+\sum_{P:\,S\subseteq P\subseteq R}\alpha(P)\frac{\ell(P)}{\ell(R)}\right].\end{align*} 
Here we used  
$$\frac{r_i}{r_m}\leq\frac{\ell(S)}{\ell(R)}, \qquad \alpha(S)+\alpha(T)\leq2\alpha^\#(S),$$ 
and Lemma~\ref{coarsepointB} to bound both $|T_{(j)}\mu(X_S,\tau_S)|$ and $|T_{(j)}\mu(X_T,\tau_T)|$. To sum over $i,j$, note that, for $S\in\mathbb D_m$ and $R\in\mathbb D_j$ with $S\subset R$, we have $r(S,R)=m-j$ and \[ i-j\in\{2r(S,R),\,2r(S,R)+1\}. \] Thus each pair $S\subset R$ occurs for at most two values of $i$, while $i-j>N_0$ implies $2r(S,R)+1>N_0$. Summing over $i,j$ gives \[ \mathcal{III}\lesssim\mathcal I+\mathcal{II}. \] Combining this estimate with \eqref{alphasharpsquare} and
\eqref{B1mastersum}, we conclude that
\begin{align*}
\sum_{\substack{i,j\in\mathcal J\\i-j>N_0}}
|B^1_{j,i}|
\lesssim
2^{-N_0/4}
\sum_{Q\in\mathbb D}\alpha(Q)^2\mu(Q).
\end{align*}

\begingroup
We now prove the four lemmas used above.

\begin{proof}[Proof of Lemma~\ref{B1ambientcancellationlemma}]
Fix $S\in\mathbb D_m$. If
$(X,t)\in\supp(\theta_S)$ and
\[
K_{(i)}\big((X,t)-(Y,s)\big)\neq0,
\]
then $i>m$ and $2^{-\omega}\leq\gamma/32$ imply
\[
\begin{aligned}
\dist\big((Y,s),E\big)
&\leq
\dist\big((X,t),S\big)
+
\|(X,t)-(Y,s)\|
\\
&\leq
\frac{\gamma}{4}r_m
+
2^{2-\omega}r_i
\\
&\leq
\left(
\frac{\gamma}{4}
+
\frac{\gamma}{16}
\right)r_m
=
\frac{5\gamma}{16}r_m
<
\frac{3\gamma}{8}r_m.
\end{aligned}
\]
Thus $(Y,s)\in\Omega_m^{(2)}$. If also $\psi_T(Y,s)\neq0$, then
\[
\begin{aligned}
\dist(S,T)
&\leq
\dist\big((X,t),S\big)
+
\|(X,t)-(Y,s)\|
+
\dist\big((Y,s),T\big)
\\
&\leq
\left(
\frac14+\frac1{16}+\frac12
\right)\gamma r_m
=
\frac{13\gamma}{16}r_m
<
\gamma r_m.
\end{aligned}
\]
Hence $T\in N(S)\cup\{S\}$. It follows that, on the support of the
integrand in \eqref{B1firstambientcancellation},
\[
\sum_{T\in N(S)\cup\{S\}}\psi_T(Y,s)
=
\sum_{T\in\mathbb D_m}\psi_T(Y,s)
=
1
\]
by \eqref{planeambientpartitionsum}. Consequently,
\[
\sum_{T\in N(S)\cup\{S\}}
\langle
\theta_S,T_{(i)}(\psi_T\mathcal L_S)
\rangle_{\mathcal L_S}
=
\langle
\theta_S,T_{(i)}\mathcal L_S
\rangle_{\mathcal L_S}.
\]
By the spatial oddness of $K_{(i)}$, the last
expression is zero. This proves \eqref{B1firstambientcancellation}.

Now fix $T\in\mathbb D_m$. Suppose that
\[
(Y,s)\in\supp(\theta_T),
\qquad
K_{(i)}\big((X,t)-(Y,s)\big)\neq0,
\qquad
\psi_S(X,t)\neq0.
\]
The preceding estimates, with $(S,X,t)$ and $(T,Y,s)$ interchanged, give
\[
(X,t)\in\Omega_m^{(2)}
\qquad\text{and}\qquad
S\in N(T)\cup\{T\}.
\]
Therefore, on the support of the integrand in
\eqref{B1secondambientcancellation},
\[
\sum_{S\in N(T)\cup\{T\}}\psi_S(X,t)
=
\sum_{S\in\mathbb D_m}\psi_S(X,t)
=
1.
\]
This proves the first equality in \eqref{B1secondambientcancellation}. For
the second, we apply Fubini's theorem and use the spatial oddness of $K_{(i)}$ to obtain
\[
\begin{aligned}
\langle
1,T_{(i)}(\theta_T\mathcal L_T)
\rangle_{\mathcal L_T}
&=
\int_{L_T}\theta_T(Y,s)
\left[
\int_{L_T}
K_{(i)}\big((X,t)-(Y,s)\big)
\,d\mathcal L_T(X,t)
\right]
d\mathcal L_T(Y,s)
\\
&=0.
\end{aligned}
\]
This proves \eqref{B1secondambientcancellation}.
\end{proof}

\begin{proof}[Proof of Lemma~\ref{neighborprojectionmeasure}]
Lemma~\ref{smallanglelemma} shows that the angle between $L_S$ and $L_T$ is
less than $\pi/4$, so $p$ is bi-Lipschitz.
 Since $p$ is affine in space and
leaves time unchanged, its Jacobian on the vertical plane is constant. Thus
$p^\#\mathcal L_T=a_{S,T}\mathcal L_S$.

Set $r:=\ell(S)=\ell(T)$ and $\delta:=\alpha^\#(S,T)$.
By Lemma~\ref{toalphalemma2},
\begin{equation}\label{neighborplanehausdorff}
\dist_H(L_S\cap C_S,L_T\cap C_S)
\lesssim
\delta r.
\end{equation}
It follows that $\|\widetilde p-\iota_S\| \lesssim \delta.$  In particular, with $J:=|\det\widetilde p|,$ we have
\begin{align}\label{neighborprojectionjacobian}
    |J-1|\lesssim\delta.
\end{align}

We next compare the constants $c_S$ and $c_T$. The proof of
Lemma~\ref{toalphalemma2} shows that
\[
    C_{\frac52r}(Z_S,\tau_S)
    \subset
    C_S\cap C_T
\]
and that both $L_S$ and $L_T$ intersect
$C_{r/4}(Z_S,\tau_S)$. Choose a parabolic Lipschitz function $\varphi$
such that $0\leq\varphi\leq1$, $\varphi=1$ on $C_{2r}(Z_S,\tau_S)$, $\supp\varphi\subset C_{\frac52r}(Z_S,\tau_S)$, and  $\|\varphi\|_{\operatorname{Lip}}
    \lesssim r^{-1}$.
For $U\in\{S,T\}$, set
\[
    I_U
    :=
    \int_{L_U}\varphi\,d\sigma_{L_U}.
\]
Because $L_S$ and $L_T$ intersect $C_{r/4}(Z_S,\tau_S)$, we have
\begin{align}\label{neighborplaneintegralsize}
    I_S\approx r^{k+2},
    \qquad
    I_T\approx r^{k+2}.
\end{align}
Since $r\varphi$ is, up to a fixed constant, an admissible test function
in the definitions of both $\alpha(S)$ and $\alpha(T)$,
\[
    \left|
        \int\varphi\,d\mu-c_UI_U
    \right|
    \lesssim
    \alpha(U)r^{k+2},
    \qquad U\in\{S,T\}.
\]
Thus,
\begin{align}\label{neighborweightedplaneintegrals}
    |c_SI_S-c_TI_T|
    \lesssim
    \delta r^{k+2}.
\end{align}

It remains to compare $I_S$ and $I_T$. Set
\[
    \Omega
    :=
    (\supp\varphi\cap L_S)
    \cup
    p^{-1}(\supp\varphi\cap L_T).
\]
Choose $W\in L_S\cap C_{r/4}(Z_S,\tau_S)$. By
\eqref{neighborplanehausdorff} and the fact that $p(W)$ is the nearest point
of $L_T$ to $W$,
\[
    d_p(W,p(W))
    \lesssim
    \delta r.
\]
If $(X,t)\in\Omega$, then either $(X,t)$ or $p(X,t)$ belongs to
$C_{\frac52r}(Z_S,\tau_S)$. In the first case,
$d_p((X,t),W)\lesssim r$ follows directly from the choice of $W$. In the
second case we have $\|p(X,t)-(Z_S,\tau_S)\|,\|p(W)-(Z_S,\tau_S)\| \lesssim r$, so the inverse Lipschitz bound for $p$ gives
$d_p((X,t),W)\lesssim r$.

Write $W=(Y,s)$. Since $p$ is affine in the spatial variables, preserves
time, and has linear part $\widetilde p$, we have
\[
    p(X,t)-(X,t)
    =
    p(W)-W
    +
    \big((\widetilde p-\iota_S)(X-Y),0\big).
\]
Hence, using $|X-Y|\leq d_p((X,t),W)$ and
$\|\widetilde p-\iota_S\|\lesssim\delta$,
\[
    d_p\big((X,t),p(X,t)\big)
    \leq
    d_p(W,p(W))
    +
    \|\widetilde p-\iota_S\|\,
    d_p((X,t),W)
    \lesssim
    \delta r,
    \qquad (X,t)\in\Omega.
\]
Changing variables by $p:L_S\to L_T$ therefore gives
\[
    I_T
    =
    J\int_{L_S}\varphi\circ p\,d\sigma_{L_S},
\]
and hence, by the Lipschitz bound for $\varphi$ and
\eqref{neighborprojectionjacobian},
\begin{align*}
    |I_S-I_T|
    &\leq
    \int_{L_S}|\varphi-\varphi\circ p|\,
        d\sigma_{L_S}
    +
    |J-1|
    \int_{L_S}\varphi\circ p\,
        d\sigma_{L_S}
    \\
    &\lesssim
    \delta r^{k+2}.
\end{align*}
Using this estimate, \eqref{neighborweightedplaneintegrals},
$c_T\approx1$, and \eqref{neighborplaneintegralsize}, we obtain
\[
    |c_S-c_T|I_S
    \leq
    |c_SI_S-c_TI_T|
    +
    c_T|I_T-I_S|
    \lesssim
    \delta r^{k+2}.
\]
Thus
\begin{align}\label{neighborconstantcomparison}
    |c_S-c_T|
    \lesssim
    \delta.
\end{align}

Finally, for every Borel set $A\subset L_S$,
\begin{align*}
p^\#\mathcal L_T(A)
&=
c_T\sigma_{L_T}(p(A))
\\
&=
c_TJ\sigma_{L_S}(A)
=
\frac{c_T}{c_S}J\mathcal L_S(A).
\end{align*}
Consequently,
\[
    a_{S,T}
    =
    \frac{c_T}{c_S}J.
\]
Since $c_S,c_T\approx1$, \eqref{neighborprojectionjacobian} and
\eqref{neighborconstantcomparison} imply
\[
    |a_{S,T}-1|
    \lesssim
    \delta.
\]
Together with the estimate for $\widetilde p-\iota_S$, this proves
\eqref{projectionmeasurebounds}.
\end{proof}

\begin{proof}[Proof of Lemma~\ref{neighborslem}]
Write $r_i:=2^{-i}$ and $r_m=2^{-m}$. Since
$m=\lfloor(i+j)/2\rfloor$, we have
\[
    \frac{r_i}{r_m}\approx2^{-(i-j)/2}.
\]
For $U\in\{S,T\}$, let $\phi$ be a cutoff satisfying
$\|\phi\|_\infty\lesssim1$, $\|\nabla\phi\|_\infty\lesssim r_m^{-1}$, and $\|\partial_t\phi\|_\infty\lesssim r_m^{-2}$. Then
\begin{align}\label{planecutoffbound}
    \|T_{(i)}(\phi\mathcal L_U)\|_{L^\infty(L_U)}
    \lesssim
    \frac{r_i}{r_m}.
\end{align}
Indeed, for $(X,t)\in L_U$ the identity $\int_{L_U}K_{(i)}((X,t)-(Y,s))\,d\mathcal L_U(Y,s)=0$ permits us to subtract $\phi(X,t)$ inside the
integral. If $K_{(i)}((X,t)-(Y,s))\neq0$, then $d_p((X,t),(Y,s))\lesssim r_i$; since $r_i\leq r_m$, the derivative bounds for $\phi$ imply $|\phi(Y,s)-\phi(X,t)|\lesssim r_i/r_m$. The size estimate for $K_{(i)}$ gives $\int_{L_U}|K_{(i)}((X,t)-(Y,s))|\,d\mathcal L_U(Y,s)\lesssim1$, which proves \eqref{planecutoffbound}.

Suppose first that $\alpha^\#(S,T)<\eta_0$. Let $p:L_S\to L_T$, $\widetilde p:V_S\to V_T$, and $a=a_{S,T}$ be as in Lemma~\ref{neighborprojectionmeasure}.
For $U\in\{S,T\}$, set \[\Omega_U:=(\supp\theta_U\cap L_S)\cup p^{-1}(\supp\theta_U\cap L_T).\] Since $\supp\theta_U\subset U_{\gamma r_m/4}(U)$ and $T\in N(S)$, either $W$ or $p(W)$ belongs to $C_{\frac52r_m}(Z_S,\tau_S)$ whenever $W\in\Omega_U$. The two-case argument in the proof of Lemma~\ref{neighborprojectionmeasure}, using the inverse Lipschitz bound for $p$ when $p(W)\in C_{\frac52r_m}(Z_S,\tau_S)$, therefore gives \begin{align}\label{cutoffprojectiondistance} d_p\big((X,t),p(X,t)\big)\lesssim\alpha^\#(S,T)r_m,\qquad (X,t)\in\Omega_U. \end{align} By \eqref{cutoffprojectiondistance},
\eqref{projectionmeasurebounds} and
\eqref{partitionhigherderivatives}, we have
\begin{align}\label{cutoffprojectionbounds}
    \|\theta_U\circ p-\theta_U\|_{L^\infty(L_S)}
    &\lesssim \alpha^\#(S,T),
    \\
    \|\theta_U\circ p-\theta_U\|_{\operatorname{Lip}(L_S)}
    &\lesssim \frac{\alpha^\#(S,T)}{r_m},
    \qquad U\in\{S,T\}.
\end{align}
Both $\theta_U\circ p$ and $\theta_U$ vanish on $L_S\setminus\Omega_U$, so the norms in \eqref{cutoffprojectionbounds} are taken over all of $L_S$. The first estimate follows from \eqref{cutoffprojectiondistance} and $\|\nabla\theta_U\|_\infty\lesssim r_m^{-1}$. For the second estimate, set $g_U:=\theta_U\circ p-\theta_U$. The chain rule, \eqref{cutoffprojectiondistance}, $\|\widetilde p-\iota_S\|\lesssim\alpha^\#(S,T)$, and \eqref{partitionhigherderivatives} give $\|
\nabla g_U\|_\infty\lesssim\alpha^\#(S,T)/r_m$ and $\|\partial_tg_U\|_\infty\lesssim\alpha^\#(S,T)/r_m^2$ on $\Omega_U$. Here the time estimate uses that $p$ leaves the time coordinate unchanged and $\|\nabla\partial_t\theta_U\|_\infty\lesssim r_m^{-3}$. If
$|t-s|^{1/2}\leq r_m$, the time derivative estimate gives $|g_U(X,t)-g_U(X,s)|\lesssim\alpha^\#(S,T)r_m^{-1}|t-s|^{1/2}$; for larger time increments the same inequality follows from the first line of \eqref{cutoffprojectionbounds}. We estimate the spatial increments in the same way using $
\nabla g_U$.

Changing variables in the outer integral gives
\[
\langle\theta_S,T_{(i)}(\theta_T\mathcal L_T)\rangle_{\mathcal L_T}
=
a\int_{L_S}(\theta_S\circ p)
\big(T_{(i)}(\theta_T\mathcal L_T)\circ p\big)
\,\d\mathcal L_S.
\]
Set $F:=T_{(i)}(\theta_T\mathcal L_T)\circ p$, $G:=aT_{(i)}((\theta_T\circ p)\mathcal L_S)$, $H:= aT_{(i)}(\theta_T\mathcal L_S)$, and $J : = T_{(i)}(\theta_T\mathcal L_S)$.  We have
\begin{equation}\label{fourtermidentity} \begin{aligned} &\langle\theta_S,T_{(i)}(\theta_T\mathcal L_T)\rangle_{\mathcal L_T} -\langle\theta_S,T_{(i)}(\theta_T\mathcal L_S)\rangle_{\mathcal L_S} \\ ={}&a\int_{L_S}(\theta_S\circ p)(F-G)\,d\mathcal L_S \\ &+a\int_{L_S}(\theta_S\circ p)(G-H)\,d\mathcal L_S \\ &+a\int_{L_S}(\theta_S\circ p)(H-J)\,d\mathcal L_S \\ &+\int_{L_S}\big(a(\theta_S\circ p)-\theta_S\big)J\,d\mathcal L_S. \end{aligned} \end{equation}
To estimate $G-H$, set $g_T:=\theta_T\circ p-\theta_T$. Then $G-H=aT_{(i)}(g_T\mathcal L_S)$. For $W\in L_S$, the spatial oddness of $K_{(i)}$ gives $T_{(i)}\mathcal L_S(W)=0$. Consequently, \[ (G-H)(W)=a\int_{L_S}K_{(i)}(W-Z)\big(g_T(Z)-g_T(W)\big)\,d\mathcal L_S(Z). \] By \eqref{cutoffprojectionbounds}, $\|g_T\|_{\operatorname{Lip}(L_S)}\lesssim\alpha^\#(S,T)/r_m$. If $K_{(i)}(W-Z)\neq0$, then $d_p(W,Z)\lesssim r_i$, and hence \[ |g_T(Z)-g_T(W)|\lesssim\alpha^\#(S,T)\frac{r_i}{r_m}. \] Since $\int_{L_S}|K_{(i)}(W-Z)|\,d\mathcal L_S(Z)\lesssim1$ and $a\approx1$, it follows that \[ \|G-H\|_{L^\infty(L_S)}\lesssim\alpha^\#(S,T)\frac{r_i}{r_m}. \] For $H-J$, \eqref{projectionmeasurebounds} and \eqref{planecutoffbound} give $\|H-J\|_{L^\infty(L_S)}\lesssim\alpha^\#(S,T)r_i/r_m$. Finally, $\|a(\theta_S\circ p)-\theta_S\|_{L^\infty(L_S)}\lesssim\alpha^\#(S,T)$, while \eqref{planecutoffbound} gives $\|J\|_{L^\infty(L_S)}\lesssim r_i/r_m$. Since $\mathcal L_S(\supp(\theta_S\circ p)\cup\supp\theta_S)\lesssim\mu(S)$,
the last three terms are bounded by
\[
    C\alpha^\#(S,T)\frac{r_i}{r_m}\mu(S).
\]
It remains to estimate the first term.
 
For that term, define $P:V_S\times\mathbb R\to V_T\times\mathbb R$ by $P(\xi,u):=(\widetilde p\xi,u)$. Since $p$ is affine in the spatial variables and fixes time, $p(X,t)-p(Y,s)=P(X-Y,t-s)$. If $Z\in V_S\times\mathbb R$ and either $K_{(i)}(Z)\neq0$ or $K_{(i)}(PZ)\neq0$, then $d_p(Z,0)\lesssim r_i$; hence \eqref{projectionmeasurebounds} gives $|PZ-Z|\lesssim\alpha^\#(S,T)r_i$. The mean value theorem and $\|
\nabla K_{(i)}\|_\infty\lesssim r_i^{-k-3}$ give
\[
\left|
K_{(i)}(PZ)-K_{(i)}(Z)
\right|
\lesssim
\alpha^\#(S,T)r_i^{-k-2}
\]
whenever either kernel is nonzero.

We claim that \[\int_{L_S}K_{(i)}\big(P((X,t)-(Y,s))\big)\,d\mathcal L_S(Y,s) =\int_{L_S}K_{(i)}\big((X,t)-(Y,s)\big)\,d\mathcal L_S(Y,s)=0.\] By the spatial oddness of $K_{(i)}$, the second integral is zero. After the linear change of variables $P:V_S\times\mathbb R\to V_T\times\mathbb R$, the first is a constant multiple of the integral of $K_{(i)}$ over $V_T\times\mathbb R$, and is therefore also zero by the spatial oddness of $K_{(i)}$.
It follows that
\[
\int_{L_S}\Big[
K_{(i)}\big(P((X,t)-(Y,s))\big)
-
K_{(i)}\big((X,t)-(Y,s)\big)
\Big]\,d\mathcal L_S(Y,s)
=
0.
\]
This gives
\begin{align*}
(F-G)(X,t)
={}&a\int_{L_S}\Big[
K_{(i)}\big(P((X,t)-(Y,s))\big)
-
K_{(i)}\big((X,t)-(Y,s)\big)
\Big]\\
&\quad\times\Big[
(\theta_T\circ p)(Y,s)
-
(\theta_T\circ p)(X,t)
\Big]\,d\mathcal L_S(Y,s).
\end{align*}

We now estimate $(\theta_T\circ p)(Y,s)-(\theta_T\circ p)(X,t)$. Put
$Z=(X,t)-(Y,s)$. If $K_{(i)}(Z)\neq0$, then the support of $K_{(i)}$
gives $d_p(Z,0)\lesssim r_i$, and the Lipschitz bound for $P$ gives
$d_p(PZ,0)\lesssim r_i$. If $K_{(i)}(PZ)\neq0$, the latter bound follows
directly from the support of $K_{(i)}$. Hence, whenever either kernel is
nonzero,
\[
d_p\big(p(X,t),p(Y,s)\big)
=
d_p(PZ,0)
\lesssim r_i.
\]
Write $p(X,t)=(X',t)$ and $p(Y,s)=(Y',s)$. The last display implies
$|X'-Y'|\lesssim r_i$ and $|t-s|\lesssim r_i^2$. Therefore the derivative
bounds
\[
|\nabla\theta_T|\lesssim r_m^{-1},
\qquad
|\partial_t\theta_T|\lesssim r_m^{-2}
\]
give
\begin{align*}
\big|
(\theta_T\circ p)(Y,s)
-
(\theta_T\circ p)(X,t)
\big|
&\lesssim
\frac{|Y'-X'|}{r_m}
+
\frac{|s-t|}{r_m^2}\\
&\lesssim
\frac{r_i}{r_m}
+
\left(\frac{r_i}{r_m}\right)^2
\lesssim
\frac{r_i}{r_m},
\end{align*}
where the last inequality uses $i>m$, and hence $r_i\leq r_m$.
Since the set of $(Y,s)\in L_S$ on which either kernel is nonzero has $\mathcal L_S$-measure $\lesssim r_i^{k+2}$, the last two estimates give $\|F-G\|_{L^\infty(L_S)}\lesssim\alpha^\#(S,T)r_i/r_m$. Finally, $\mathcal L_S(\supp(\theta_S\circ p))\lesssim\mu(S)$, so the first term in \eqref{fourtermidentity} is $\lesssim\alpha^\#(S,T)(r_i/r_m)\mu(S)$. This proves \eqref{neighbors} in
the case when $\alpha^{\#}(S,T) < \eta_0$.

If $\alpha^\#(S,T)\geq\eta_0$, the triangle inequality and H\"older's
inequality give
\begin{align*}
&\left|
\langle\theta_S,T_{(i)}(\theta_T\mathcal L_T)\rangle_{\mathcal L_T}
-
\langle\theta_S,T_{(i)}(\theta_T\mathcal L_S)\rangle_{\mathcal L_S}
\right|\\
&\quad\lesssim
\left|
\langle\theta_S,T_{(i)}(\theta_T\mathcal L_T)\rangle_{\mathcal L_T}
\right|
+
\left|
\langle\theta_S,T_{(i)}(\theta_T\mathcal L_S)\rangle_{\mathcal L_S}
\right|\\
&\quad\lesssim
\|T_{(i)}(\theta_T\mathcal L_T)\|_{L^\infty(L_T)}
\|\theta_S\|_{L^1(\mathcal L_T)}
+
\|T_{(i)}(\theta_T\mathcal L_S)\|_{L^\infty(L_S)}
\|\theta_S\|_{L^1(\mathcal L_S)}.
\end{align*}
By \eqref{planecutoffbound}, applied with
$\phi=\theta_T$ on $L_T$ and on $L_S$, each
$L^\infty$ norm in the last line is $\lesssim r_i/r_m$. Since
$0\leq\theta_S\leq1$ and
$\mathcal L_S(\supp\theta_S)
 +\mathcal L_T(\supp\theta_S)\lesssim\mu(S)$,
the right-hand side is bounded by
\[
\frac{r_i}{r_m}
\Big[
\mathcal L_T(\supp\theta_S)
+
\mathcal L_S(\supp\theta_S)
\Big]
\lesssim
\frac{r_i}{r_m}\mu(S)
\lesssim
\eta_0^{-1}\frac{r_i}{r_m}\alpha^\#(S,T)\mu(S).
\]
\end{proof}

\begin{proof}[Proof of Lemma~\ref{coarsepointB}]
Since $T\in N(S)$ and $\ell(T)=\ell(S)$,
\begin{align*}
\|(X_T,\tau_T)-(Z_R,\tau_R)\|
&\leq
\|(X_T,\tau_T)-(Z_T,\tau_T)\|
+
\|(Z_T,\tau_T)-(Z_S,\tau_S)\|
+
\|(Z_S,\tau_S)-(Z_R,\tau_R)\|
\\
&\leq
(3+\gamma+2^{-9})\ell(S)+\diam(R)
\leq
2\diam(R).
\end{align*}
Indeed, $3+\gamma+2^{-9}<4+\gamma$, while the hypothesis on $m-j$ and \eqref{projectionradiusvsdiamR} give $(4+\gamma)\ell(S)\leq\diam(R)$. Hence
$(X_T,\tau_T)\in C_{2\diam(R)}(Z_R,\tau_R)$.

If $\alpha(R)\geq\eta_0$, 
$|T_{(j)}\mu(X_T,\tau_T)|\lesssim1\lesssim\alpha(R)$, and
\eqref{coarsepointBestimate} follows. Otherwise,
$\alpha(R)<\eta_0$, and we choose
$(Y_R,\sigma_R)\in L_R\cap C_{\diam(R)}(Z_R,\tau_R)$. By
Lemma~\ref{TQintbd}(a),
\[
    |T_{(j)}\mu(Y_R,\sigma_R)|\lesssim\alpha(R).
\]
Lemma~\ref{bigest}, applied with $Q=R$, then gives
\begin{align}\label{coarsepointintermediate}
|T_{(j)}\mu(X_T,\tau_T)|
\lesssim
\alpha(R)
+
\frac{\dist((X_T,\tau_T),L_R)}{\ell(R)}.
\end{align}

By Lemma~\ref{toalphalemma2},
\[
    \dist((X_T,\tau_T),L_S)
    \lesssim
    \alpha^\#(S)\ell(S).
\]
Let $(Y_S,\tau_T)$ be the projection of $(X_T,\tau_T)$ onto $L_S$, and set
\[
v:=(Y_S,\tau_T)-(X_S,\tau_S).
\]
Then $\|v\|\lesssim\ell(S)$.
Since $(Y_S,\tau_T)$ and $(X_S,\tau_S)$ lie in $L_S$, we have
$\pi_{L_S^\perp}v=0$. Moreover, because the planes are vertical, their perpendicular projections depend only on the spatial component of $v$, whose Euclidean norm is bounded by $\|v\|$. Consequently,
\begin{align*}
\dist((Y_S,\tau_T),L_R)
&\leq
\dist((X_S,\tau_S),L_R)
+
\big|\pi_{L_R^\perp}v\big|\\
&=
\dist((X_S,\tau_S),L_R)
+
\big|(\pi_{L_R^\perp}-\pi_{L_S^\perp})v\big|.
\end{align*}
Lemma~\ref{distLQalphabd2} gives
\[
\dist((X_S,\tau_S),L_R)
\lesssim
\sum_{P:\,S\subsetneq P\subseteq R}\alpha(P)\ell(P).
\]
Fix a parent--child pair $P\subset P'$ in the chain from $S$ to $R$.
If
$\alpha(P')<\eta_0$, then \eqref{lqlpaq}, applied with $Q=P'$ and $\eta=1/2$, gives
\[ \dist_H\big(L_P\cap C_{P'},L_{P'}\cap C_{P'}\big) \lesssim \alpha(P')\ell(P'). \] For affine planes, a Hausdorff distance at most $\delta r$ in a cylinder of radius $r$ implies that their direction spaces make an angle $O(\delta)$; equivalently, the orthogonal projections onto their normal spaces differ by $O(\delta)$ in operator norm. Taking $r=\ell(P')$ and $\delta=\alpha(P')$ gives
\begin{align}\label{parentchildprojectionbound}
\|\pi_{L_P^\perp}-\pi_{L_{P'}^\perp}\|
\lesssim
\alpha(P').
\end{align}
If $\alpha(P')\geq\eta_0$, the estimate in \eqref{parentchildprojectionbound} follows from
\[
\|\pi_{L_P^\perp}-\pi_{L_{P'}^\perp}\|
\leq 2
\leq
2\eta_0^{-1}\alpha(P').
\]
Thus, \eqref{parentchildprojectionbound} holds in either case.
Telescoping these projection operators and using
$\|v\|\lesssim\ell(S)$, we obtain
\begin{align*}
\big|(\pi_{L_R^\perp}-\pi_{L_S^\perp})v\big|
&\lesssim
\ell(S)
\sum_{P:\,S\subsetneq P\subseteq R}\alpha(P)\\
&\lesssim
\sum_{P:\,S\subsetneq P\subseteq R}\alpha(P)\ell(P).
\end{align*}
Combining the last three displays yields
\[
\dist((Y_S,\tau_T),L_R)
\lesssim
\sum_{P:\,S\subseteq P\subseteq R}\alpha(P)\ell(P).
\]
Since $(Y_S,\tau_T)$ is the projection of $(X_T,\tau_T)$ onto $L_S$, \begin{align*} \dist((X_T,\tau_T),L_R) &\leq \|(X_T,\tau_T)-(Y_S,\tau_T)\| + \dist((Y_S,\tau_T),L_R)\\ &= \dist((X_T,\tau_T),L_S) + \dist((Y_S,\tau_T),L_R). \end{align*} Consequently,
\[
\dist((X_T,\tau_T),L_R)
\lesssim
\alpha^\#(S)\ell(S)
+
\sum_{P:\,S\subseteq P\subseteq R}\alpha(P)\ell(P).
\]
Substituting this into \eqref{coarsepointintermediate}, and observing that the
term $P=R$ already contains $\alpha(R)$, proves
\eqref{coarsepointBestimate}. It remains to prove \eqref{coarsepointBselfestimate}. If $\alpha(R)\geq\eta_0$, then $|T_{(j)}\mu(X_S,\tau_S)|\lesssim1\lesssim\alpha(R)$, and $\alpha(R)$ is the term $P=R$ on the right-hand side of \eqref{coarsepointBselfestimate}. Suppose now that $\alpha(R)<\eta_0$. We have \[ \|(X_S,\tau_S)-(Z_R,\tau_R)\|\leq\|(X_S,\tau_S)-(Z_S,\tau_S)\|+\|(Z_S,\tau_S)-(Z_R,\tau_R)\|\leq3\ell(S)+\diam(R)\leq2\diam(R). \] Choose $(Y_R,\sigma_R)\in L_R\cap C_{\diam(R)}(Z_R,\tau_R)$. Lemma~\ref{TQintbd}(a) gives $|T_{(j)}\mu(Y_R,\sigma_R)|\lesssim\alpha(R)$, and Lemma~\ref{bigest}, applied with $Q=R$, gives \[ |T_{(j)}\mu(X_S,\tau_S)|\lesssim\alpha(R)+\frac{\dist((X_S,\tau_S),L_R)}{\ell(R)}. \] By Lemma~\ref{distLQalphabd2}, \[ \dist((X_S,\tau_S),L_R)\lesssim\sum_{P:\,S\subsetneq P\subseteq R}\alpha(P)\ell(P). \] Therefore \[ |T_{(j)}\mu(X_S,\tau_S)|\lesssim\alpha(R)+\sum_{P:\,S\subsetneq P\subseteq R}\alpha(P)\frac{\ell(P)}{\ell(R)}\lesssim\sum_{P:\,S\subseteq P\subseteq R}\alpha(P)\frac{\ell(P)}{\ell(R)}, \] which is \eqref{coarsepointBselfestimate}.
\end{proof}
\endgroup

\subsubsection{Estimating the Sum of the $B^2_{j,i}$'s}

Recall that
\begin{align*}
B^2_{j,i}
&=
\sum_{S\in\mathbb D_m}\sum_{T\in N(S)}
\bigl(T_{(j)}\mu(X_S,\tau_S)-T_{(j)}\mu(X_T,\tau_T)\bigr)
\mathcal E_{S,T},
\end{align*}
where $m=\lfloor(i+j)/2\rfloor$ and
\begin{align*}
\mathcal E_{S,T}
&:=
\langle\theta_S,T_{(i)}(\theta_T\mu)\rangle
-
\langle\theta_S,T_{(i)}(\theta_T\mathcal L_S)\rangle_{\mathcal L_S}.
\end{align*}

We construct a partition of unity \(\{\eta_Q\}_{Q\in\mathbb D_i}\) by applying the construction of \(\{\psi_T\}\) at generation \(i\); see Appendix~\ref{smoothpartitionappendix}. Set \[\Omega_i^{(2)}:=\left\{(X,t):\dist((X,t),E)<\frac{3\gamma}{8}r_i\right\}.\] Then \[\sum_{Q\in\mathbb D_i}\eta_Q=1\qquad\text{on }\Omega_i^{(2)}.\] Moreover, \[\supp(\eta_Q)\subset U_{\frac{\gamma}{2}r_i}(Q),\qquad |\nabla\eta_Q|\lesssim r_i^{-1},\qquad |\partial_t\eta_Q|\lesssim r_i^{-2}.\]
For $Q\in\mathbb D_i$, set $\eta_{Q,S}:=\eta_Q\theta_S$, and let
\[
\mathcal Q_i(S)
:=
\{Q\in\mathbb D_i:\supp(\eta_Q)\cap\supp(\theta_S)\neq\emptyset\}.
\]
If $Q\in\mathcal Q_i(S)$ and $S_Q\in\mathbb D_m$ is the ancestor of $Q$,
then $S_Q\in N(S)\cup\{S\}$. Moreover,
\begin{align}
\sum_{Q\in\mathcal Q_i(S)}\mu(Q)\lesssim\mu(S),
\label{finepartitionmass}
\end{align}
and each $Q\in\mathbb D_i$ belongs to $\mathcal Q_i(S)$ for only boundedly
many $S\in\mathbb D_m$.

Set \[\mathcal R^\eta_{S,T}:=\left\langle \theta_S\left(1-\sum_{Q\in\mathbb D_i}\eta_Q\right),T_{(i)}(\theta_T\mathcal L_S)\right\rangle_{\mathcal L_S}.\] Then
\begin{align}
\mathcal E_{S,T}
&=
\sum_{Q\in\mathcal Q_i(S)}
\bigl[
\langle\eta_{Q,S},T_{(i)}(\theta_T\mu)\rangle
-
\langle\eta_{Q,S},T_{(i)}(\theta_T\mathcal L_Q)\rangle_{\mathcal L_Q}
\bigr]
\nonumber\\
&\quad+
\sum_{Q\in\mathcal Q_i(S)}
\bigl[
\langle\eta_{Q,S},T_{(i)}(\theta_T\mathcal L_Q)\rangle_{\mathcal L_Q}
-
\langle\eta_{Q,S},T_{(i)}(\theta_T\mathcal L_S)\rangle_{\mathcal L_S}
\bigr]-\mathcal R^\eta_{S,T}.
\label{B2finepartition}
\end{align}

We use the following five lemmas to estimate the sum of the $B^2_{j,i}$'s.
We prove them after completing that estimate.

\begin{lemma}\label{fineplanecutofflemma}
Let $i,m\in\mathcal J$ with $i>m$, let $T\in\mathbb D_m$, and let $U\in\mathbb D$. Then
\begin{align}
\|T_{(i)}(\theta_T\mathcal L_U)\|_{L^\infty(L_U)}
\lesssim\frac{r_i}{r_m}.
\label{fineplanecutoffbound}
\end{align}
\end{lemma}

\begingroup
\begin{lemma}\label{B2etaremainderlemma}
If $S,T\in\mathbb D_m$ and $T\in N(S)$, then
\begin{align}
|\mathcal R^\eta_{S,T}|
\lesssim
\alpha(S)\mu(S)
\leq
\alpha^\#(S)\mu(S).
\label{B2etaremainderbound}
\end{align}
\end{lemma}
\endgroup

\begin{lemma}\label{finealphacomparison}
If $S,T\in\mathbb D_m$, $T\in N(S)$, and $Q\in\mathcal Q_i(S)$, then
\begin{align}
\left|
\langle\eta_{Q,S},T_{(i)}(\theta_T\mu)\rangle
-
\langle\eta_{Q,S},T_{(i)}(\theta_T\mathcal L_Q)\rangle_{\mathcal L_Q}
\right|
\lesssim
\alpha(Q)\mu(Q).
\label{finealphacomparisonestimate}
\end{align}
\end{lemma}

\begin{lemma}\label{multiscaleplanecomparison}
Let $S,T\in\mathbb D_m$ with $T\in N(S)$, and let
$Q\in\mathcal Q_i(S)$. Then
\begin{align}
&\left|
\langle\eta_{Q,S},T_{(i)}(\theta_T\mathcal L_Q)\rangle_{\mathcal L_Q}
-
\langle\eta_{Q,S},T_{(i)}(\theta_T\mathcal L_S)\rangle_{\mathcal L_S}
\right|
\lesssim
\left(\alpha^\#(S)+\sum_{P:\,Q\subseteq P\subseteq S_Q}
\alpha(P)\frac{\ell(P)}{\ell(S)}\right)\mu(Q).
\label{multiscaleplanecomparisonestimate}
\end{align}
\end{lemma}

\begin{lemma}\label{neighborcoarsedifference}
If $S,T\in\mathbb D_m$, $T\in N(S)$, $R\in\mathbb D_j$ is the ancestor of
$S$, and $m-j\geq\lfloor(N_0+1)/2\rfloor$, then
\begin{align}
\left|
T_{(j)}\mu(X_S,\tau_S)-T_{(j)}\mu(X_T,\tau_T)
\right|
\lesssim
\frac{\ell(S)}{\ell(R)}\left(\alpha^\#(S)+\sum_{P:\,S\subseteq P\subseteq R}\alpha(P)\right).
\label{neighborcoarsedifferenceestimate}
\end{align}
\end{lemma}

We first control $\mathcal E_{S,T}$.

Define
\begin{align}\mathfrak e_i(S)&:=\sum_{Q\in\mathcal Q_i(S)}\alpha(Q)\mu(Q)+\alpha^\#(S)\mu(S)\\&\quad+\sum_{Q\in\mathcal Q_i(S)}\sum_{P:\,Q\subseteq P\subseteq S_Q}\alpha(P)\frac{\ell(P)}{\ell(S)}\mu(Q).\label{fineerrordef}\end{align}
By \eqref{B2finepartition} and
Lemmas~\ref{finealphacomparison}, \ref{multiscaleplanecomparison}, and
\ref{B2etaremainderlemma}, we have
\begin{align}
|\mathcal E_{S,T}|\lesssim\mathfrak e_i(S)
\qquad\text{for every }T\in N(S).
\label{fineerrorbound}
\end{align}

\begingroup
Set $a:=m-j$. For each $S\in\mathbb D_m$, let $R_j(S)\in\mathbb D_j$
be its ancestor. Since
\[
\frac{\ell(S)}{\ell(R_j(S))}=2^{-a},
\]
the definition of $B^2_{j,i}$, \eqref{fineerrorbound},
Lemma~\ref{neighborcoarsedifference}, and the bounded number of neighbors give
\begin{align}
|B^2_{j,i}|
&\leq
\sum_{S\in\mathbb D_m}
\sum_{T\in N(S)}
\left|
T_{(j)}\mu(X_S,\tau_S)
-
T_{(j)}\mu(X_T,\tau_T)
\right|
|\mathcal E_{S,T}|
\nonumber\\
&\lesssim
2^{-a}
\sum_{R\in\mathbb D_j}
\sum_{\substack{S\in\mathbb D_m\\S\subset R}}
\mathfrak e_i(S)
\left(
\alpha^\#(S)
+
\sum_{P:\,S\subseteq P\subseteq R}\alpha(P)
\right).
\label{B2finecoarsemaster}
\end{align}
Applying Cauchy--Schwarz to the right-hand side of
\eqref{B2finecoarsemaster}, we obtain
\begin{align}
|B^2_{j,i}|
&\lesssim
2^{-a}
\left(
\sum_{S\in\mathbb D_m}
\frac{\mathfrak e_i(S)^2}{\mu(S)}
\right)^{1/2}
\nonumber\\
&\qquad\times
\left(
\sum_{R\in\mathbb D_j}
\sum_{\substack{S\in\mathbb D_m\\S\subset R}}
\mu(S)
\left(
\alpha^\#(S)
+
\sum_{P:\,S\subseteq P\subseteq R}\alpha(P)
\right)^2
\right)^{1/2}.
\label{finecoarseCS}
\end{align}
We estimate these two factors separately.
\endgroup

\begingroup
For the first sum in \eqref{fineerrordef}, Cauchy--Schwarz and
\eqref{finepartitionmass} give
\begin{align*}
\frac{1}{\mu(S)}
\left(
\sum_{Q\in\mathcal Q_i(S)}\alpha(Q)\mu(Q)
\right)^2
&\leq
\frac{1}{\mu(S)}
\left(
\sum_{Q\in\mathcal Q_i(S)}\mu(Q)
\right)
\left(
\sum_{Q\in\mathcal Q_i(S)}\alpha(Q)^2\mu(Q)
\right)
\\
&\lesssim
\sum_{Q\in\mathcal Q_i(S)}\alpha(Q)^2\mu(Q).
\end{align*}
For the last sum in \eqref{fineerrordef}, weighted Cauchy--Schwarz gives
\begin{align*}
&\left(
\sum_{Q\in\mathcal Q_i(S)}
\sum_{P:\,Q\subseteq P\subseteq S_Q}
\alpha(P)\frac{\ell(P)}{\ell(S)}\mu(Q)
\right)^2
\\
&\qquad\leq
\left(
\sum_{Q\in\mathcal Q_i(S)}
\sum_{P:\,Q\subseteq P\subseteq S_Q}
\frac{\ell(P)}{\ell(S)}\mu(Q)
\right)
\left(
\sum_{Q\in\mathcal Q_i(S)}
\sum_{P:\,Q\subseteq P\subseteq S_Q}
\alpha(P)^2\frac{\ell(P)}{\ell(S)}\mu(Q)
\right).
\end{align*}
The first factor on the right satisfies
\begin{align*}
\sum_{Q\in\mathcal Q_i(S)}
\mu(Q)
\sum_{P:\,Q\subseteq P\subseteq S_Q}
\frac{\ell(P)}{\ell(S)}
\lesssim
\sum_{Q\in\mathcal Q_i(S)}\mu(Q)
\lesssim
\mu(S).
\end{align*}
Here the first inequality is the geometric sum along the chain from $Q$ to
$S_Q$, and the second is \eqref{finepartitionmass}. It follows that
\begin{align}
\frac{\mathfrak e_i(S)^2}{\mu(S)}
&\lesssim
\alpha^\#(S)^2\mu(S)
+
\sum_{Q\in\mathcal Q_i(S)}\alpha(Q)^2\mu(Q)
\nonumber\\
&\quad+
\sum_{Q\in\mathcal Q_i(S)}
\sum_{P:\,Q\subseteq P\subseteq S_Q}
\alpha(P)^2\frac{\ell(P)}{\ell(S)}\mu(Q).
\label{fineerrorsquare}
\end{align}

We now sum \eqref{fineerrorsquare} over $S\in\mathbb D_m$. We bound the first term using \eqref{alphasharpsquare}. Since each $Q\in\mathbb D_i$
belongs to $\mathcal Q_i(S)$ for only boundedly many $S\in\mathbb D_m$,
the second term satisfies
\[
\sum_{S\in\mathbb D_m}
\sum_{Q\in\mathcal Q_i(S)}
\alpha(Q)^2\mu(Q)
\lesssim
\sum_{Q\in\mathbb D_i}\alpha(Q)^2\mu(Q).
\]
For the third term, note that if $Q\subseteq P$ and $m\leq J(P)\leq i$, then $P\subseteq S_Q$, because $S_Q$ is the generation-$m$ ancestor of $Q$. Thus, using $\ell(S)=r_m$ and interchanging the nonnegative sums gives
\begin{align*}
&\sum_{S\in\mathbb D_m}
\sum_{Q\in\mathcal Q_i(S)}
\sum_{P:\,Q\subseteq P\subseteq S_Q}
\alpha(P)^2\frac{\ell(P)}{\ell(S)}\mu(Q)
\\
&\qquad=
\sum_{\substack{P\in\mathbb D\\m\leq J(P)\leq i}}
\alpha(P)^2\frac{\ell(P)}{r_m}
\sum_{S\in\mathbb D_m}
\sum_{\substack{Q\in\mathcal Q_i(S)\\Q\subseteq P}}
\mu(Q)
\\
&\qquad\lesssim
\sum_{\substack{P\in\mathbb D\\m\leq J(P)\leq i}}
\alpha(P)^2\frac{\ell(P)}{r_m}
\sum_{\substack{Q\in\mathbb D_i\\Q\subseteq P}}
\mu(Q)
\\
&\qquad=
\sum_{\substack{P\in\mathbb D\\m\leq J(P)\leq i}}
\alpha(P)^2\frac{\ell(P)}{r_m}\mu(P)
\leq
\sum_{\substack{P\in\mathbb D\\m\leq J(P)\leq i}}
\alpha(P)^2\mu(P).
\end{align*}
In the inequality we again used that each $Q\in\mathbb D_i$ belongs to
$\mathcal Q_i(S)$ for only boundedly many $S\in\mathbb D_m$.
The equality in the last line follows because the
generation-$i$ descendants of $P$ partition $P$, and the final inequality
uses $\ell(P)\leq r_m$. Combining the three estimates yields
\begin{align}
\sum_{S\in\mathbb D_m}
\frac{\mathfrak e_i(S)^2}{\mu(S)}
\lesssim
\sum_{\substack{P\in\mathbb D\\m\leq J(P)\leq i}}
\alpha(P)^2\mu(P).
\label{fineerrorsquaresum}
\end{align}
Equation~\eqref{fineerrorsquaresum} bounds the first factor in \eqref{finecoarseCS}.
\endgroup

We now estimate the second factor in \eqref{finecoarseCS}.

We apply Cauchy--Schwarz along the chain from $S$ to $R$, use the bounded
number of neighbors, and sum over the generation-$m$ descendants of each
intermediate cube to obtain
\begin{align}
&\sum_{R\in\mathbb D_j}
\sum_{\substack{S\in\mathbb D_m\\S\subset R}}
\mu(S)\left(\alpha^\#(S)+\sum_{P:\,S\subseteq P\subseteq R}\alpha(P)\right)^2
\lesssim
(a+1)
\sum_{P:\,j\leq J(P)\leq m}\alpha(P)^2\mu(P).
\label{coarsedifferencesquaresum}
\end{align}

Combining \eqref{finecoarseCS},
\eqref{fineerrorsquaresum}, and \eqref{coarsedifferencesquaresum}, we obtain
\begin{align*}
|B^2_{j,i}|
&\lesssim
2^{-a}(a+1)^{1/2}
\left(
\sum_{P:\,m\leq J(P)\leq i}\alpha(P)^2\mu(P)
\right)^{1/2}
\left(
\sum_{P:\,j\leq J(P)\leq m}\alpha(P)^2\mu(P)
\right)^{1/2}.
\end{align*}

For $\varepsilon\in\{0,1\}$, write $i=j+2a+\varepsilon$, so that
$m=j+a$. By Cauchy--Schwarz in $j$,
\begin{align*}
&\sum_{j\in\mathcal J}
\left(
\sum_{P:\,j+a\leq J(P)\leq j+2a+\varepsilon}
\alpha(P)^2\mu(P)
\right)^{1/2}
\left(
\sum_{P:\,j\leq J(P)\leq j+a}
\alpha(P)^2\mu(P)
\right)^{1/2}
\\
&\qquad\lesssim
(a+1)
\sum_{P\in\mathbb D}\alpha(P)^2\mu(P).
\end{align*}
Therefore
\begin{align*}
&\sum_{\substack{i,j\in\mathcal J\\i-j>N_0}}
|B^2_{j,i}|
\\
&\qquad\lesssim
\sum_{\substack{a\geq0\\2a+1>N_0}}
2^{-a}(a+1)^{3/2}
\sum_{P\in\mathbb D}\alpha(P)^2\mu(P)
\\
&\qquad\lesssim
2^{-cN_0}
\sum_{P\in\mathbb D}\alpha(P)^2\mu(P).
\end{align*}

We now prove the five lemmas used above.

\begin{proof}[Proof of Lemma~\ref{fineplanecutofflemma}]
Fix $W\in L_U$. Since $L_U$ is vertical, the spatial oddness of $K_{(i)}$ gives
\[
\int_{L_U}K_{(i)}(W-Z)\,\d\mathcal L_U(Z)=0.
\]
Consequently,
\[
T_{(i)}(\theta_T\mathcal L_U)(W)
=
\int_{L_U}K_{(i)}(W-Z)
\bigl(\theta_T(Z)-\theta_T(W)\bigr)
\,\d\mathcal L_U(Z).
\]
If $K_{(i)}(W-Z)\neq0$, then $\|W-Z\|\lesssim r_i$. By Lemma~\ref{elemalphalem}(d), we have $c_U\lesssim1$, so we obtain the lemma by applying the mean value theorem to $|\theta_T(Z) - \theta_T(W)|$.
\end{proof}

\begingroup
\begin{proof}[Proof of Lemma~\ref{B2etaremainderlemma}]
The functions \(\eta_Q\) have bounded overlap, so \(1-\sum_Q\eta_Q\) is uniformly bounded; it vanishes on \(\Omega_i^{(2)}\). Since \(\supp(\theta_S)\subset C_{2r_m}(Z_S,\tau_S)\), Chebyshev's inequality and \eqref{LQbetacomp}, applied to \(S\), give \begin{align}&\mathcal L_S\left(\supp(\theta_S)\cap\left\{(X,t):\dist((X,t),E)\ge\frac{3\gamma}{8}r_i\right\}\right)\nonumber\\&\qquad\lesssim\frac1{r_i}\int_{L_S\cap C_{2r_m}(Z_S,\tau_S)}\dist((X,t),E)\,d\mathcal L_S(X,t)\lesssim\frac{r_m}{r_i}\alpha(S)\mu(S).\label{B2etaremaindersupport}\end{align}
Combining \eqref{B2etaremaindersupport} with \eqref{fineplanecutoffbound} proves \eqref{B2etaremainderbound}.
\end{proof}
\endgroup

\begin{proof}[Proof of Lemma~\ref{finealphacomparison}]
\begingroup  Since $Q\in\mathbb D_i$, we have $\ell(Q)=r_i$. Since $\supp(\eta_{Q,S})\subset\supp(\eta_Q)\subset U_{\frac{\gamma}{2}r_i}(Q)$, for each $(X,t)\in\supp(\eta_{Q,S})$ there is $(Z,\tau)\in Q$ such that $\|(X,t)-(Z,\tau)\|<\gamma r_i/2$. Also, \[ \supp(K_{(i)})\subset \bigl\{(W,u):2^{-i-\omega-2}<\|(W,u)\|<2^{-i-\omega+2}\bigr\}. \] Thus, if $K_{(i)}\bigl((X,t)-(Y,s)\bigr)\neq0$, then \[ \|(X,t)-(Y,s)\|\leq2^{2-\omega}r_i\leq\frac{\gamma}{8}r_i. \] Since $\|(Z,\tau)-(Z_Q,\tau_Q)\|\leq\diam(Q)\leq2^{-10}r_i$, it follows that \[ \|(Y,s)-(Z_Q,\tau_Q)\| \leq\left(\frac{\gamma}{2}+\frac{\gamma}{8}+2^{-10}\right)r_i <3r_i. \] Here $\gamma\leq2^{-10}$ follows from $\gamma r_i\leq\diam(Q)\leq2^{-10}r_i$. Hence, for fixed $(X,t)\in\supp(\eta_{Q,S})$, the function \[ g_{X,t}(Y,s):= K_{(i)}\bigl((X,t)-(Y,s)\bigr)\theta_T(Y,s) \] is supported in $C_{3r_i}(Z_Q,\tau_Q)=C_Q$. We have \[ \|K_{(i)}\|_\infty\lesssim r_i^{-k-2},\qquad \|\nabla K_{(i)}\|_\infty\lesssim r_i^{-k-3},\qquad \|\partial_tK_{(i)}\|_\infty\lesssim r_i^{-k-4}. \] Since $i>m$, $r_i\leq r_m$, so we also have \[ \|\theta_T\|_\infty\leq1,\qquad \|\nabla\theta_T\|_\infty\lesssim r_m^{-1}\leq r_i^{-1},\qquad \|\partial_t\theta_T\|_\infty\lesssim r_m^{-2}\leq r_i^{-2}. \] Therefore \[ \|g_{X,t}\|_\infty\lesssim r_i^{-k-2},\qquad \|\nabla g_{X,t}\|_\infty\lesssim r_i^{-k-3},\qquad \|\partial_tg_{X,t}\|_\infty\lesssim r_i^{-k-4}. \] Thus, $\|g_{X,t}\|_{\operatorname{Lip}}\lesssim r_i^{-k-3}$. Hence, \eqref{alphaflatapproximation} gives \begin{align} \left| T_{(i)}(\theta_T\mu)(X,t) - T_{(i)}(\theta_T\mathcal L_Q)(X,t) \right| \lesssim\alpha(Q). \label{finepointwisealpha} \end{align} Integrating \eqref{finepointwisealpha} against $\eta_{Q,S}\,\d\mu$ gives \[ \big|\langle \eta_{Q,S},T_{(i)}(\theta_T\mu)\rangle - \langle \eta_{Q,S},T_{(i)}(\theta_T\mathcal L_Q)\rangle \big| \lesssim\alpha(Q)\mu(Q). \] Set \[ F_Q(X,t):= \eta_{Q,S}(X,t)T_{(i)}(\theta_T\mathcal L_Q)(X,t). \] The support of $F_Q$ is contained in $C_Q$, and the kernel estimates together with the derivative bounds for $\eta_Q$ and \(\theta_T\) give \[ \|F_Q\|_{\operatorname{Lip}}\lesssim r_i^{-1}. \] Therefore \begin{align}\label{finepointwisealpha2} \left| \int F_Q\,\d\mu- \int F_Q\,\d\mathcal L_Q \right| \lesssim r_i^{-1}\alpha(Q)r_i^{k+3} \lesssim \alpha(Q)\mu(Q). \end{align} Combining \eqref{finepointwisealpha} and \eqref{finepointwisealpha2} proves the lemma. \endgroup
\end{proof}

\begin{proof}[Proof of Lemma~\ref{multiscaleplanecomparison}]
\begingroup
Recall that $\ell(Q)=r_i$ and $\ell(S)=r_m$. Set
\[
\Delta:=
\alpha^\#(S)+
\sum_{P:\,Q\subseteq P\subseteq S_Q}
\alpha(P)\frac{\ell(P)}{r_m},
\qquad
\Theta:=
\alpha^\#(S)+
\sum_{P:\,Q\subseteq P\subseteq S_Q}\alpha(P).
\]
Since $\ell(P)\geq r_i$ for every cube in the sum,
\begin{align}
\Theta\frac{r_i}{r_m}\leq\Delta.
\label{weightedanglecomparison}
\end{align}

Let $c_0>0$ be a small constant, to be chosen below. We consider two cases.

\smallskip
\noindent Suppose first that $\Theta\geq c_0$.
By \eqref{weightedanglecomparison} and \eqref{fineplanecutoffbound},
\[
\begin{aligned}
&\left|
\langle\eta_{Q,S},T_{(i)}(\theta_T\mathcal L_Q)\rangle_{\mathcal L_Q}
-
\langle\eta_{Q,S},T_{(i)}(\theta_T\mathcal L_S)\rangle_{\mathcal L_S}
\right| \\
&\qquad\lesssim
\frac{r_i}{r_m}
\left[
\mathcal L_Q(\supp\eta_{Q,S})+
\mathcal L_S(\supp\eta_{Q,S})
\right]
\lesssim
\frac{r_i}{r_m}\mu(Q)
\lesssim\Delta\mu(Q).
\end{aligned}
\]
This proves the lemma in the first case.

\smallskip
\noindent Now suppose that $\Theta<c_0$.
Let
\[
Q=P_0\subset P_1\subset\cdots\subset P_N=S_Q,
\qquad
\ell(P_{h+1})=2\ell(P_h).
\]

After decreasing $c_0$ if necessary, $\Theta<c_0$ implies that $\alpha(S)<\eta_0$ and $\alpha(P_h)<\eta_0$ for $0\leq h\leq N$, where $\eta_0$ is the constant from Lemma~\ref{elemalphalem}(c). Write
\[
L_{P_h}=(a_h+V_h)\times\mathbb R,
\qquad 0\leq h\leq N.
\]
Fix $0\leq h<N$ and set $r_{h+1}:=\ell(P_{h+1})$. Lemma~\ref{toalphalemma},
applied to $P_h\subset P_{h+1}$ with $\eta=1/2$, gives
\begin{equation}
\dist_H\bigl(
L_{P_h}\cap C_{P_{h+1}},
L_{P_{h+1}}\cap C_{P_{h+1}}
\bigr)
\lesssim \alpha(P_{h+1})r_{h+1}.
\label{parentchildplanehausdorff}
\end{equation}
Thus,
\[
\angle(L_{P_h},L_{P_{h+1}})
\lesssim \alpha(P_{h+1}).
\]

For $h=N$, \eqref{projectionmeasurebounds} gives
\[
\angle(L_{S_Q},L_S)
\lesssim
\max\{\alpha(S_Q),\alpha(S)\}
\leq \alpha^\#(S).
\]
Hence
\begin{align}
\angle(L_Q,L_S)
&\leq
\sum_{h=0}^{N-1}\angle(L_{P_h},L_{P_{h+1}})
+\angle(L_{S_Q},L_S)
\nonumber\\
&\lesssim
\sum_{h=0}^{N-1}\alpha(P_{h+1})+\alpha^\#(S)
\leq\Theta.
\label{localplaneangle}
\end{align}

Let $p:L_Q\to L_S$ be orthogonal projection. By
\eqref{localplaneangle}, $p$ is one-to-one and bi-Lipschitz.
\begingroup Set $C_0:=2^{2-\omega}$. The construction of $K_{(i)}$ gives $K_{(i)}(W)=0$ whenever $\|W\|>C_0r_i$. Since $2^{-\omega}\leq\gamma/32$, we have $2C_0\leq\gamma/4$.\endgroup Set
\[
\Omega:=
\left\{
W\in L_Q:
\dist(W,\supp\eta_{Q,S})\leq 2C_0r_i
\text{ or }
\dist(p(W),\supp\eta_{Q,S})\leq 2C_0r_i
\right\}.
\]

We will show that, for $W \in \Omega$,
\[\|W-p(W)\| \lesssim r_m\Delta.\] Suppose first that
$W$ lies within $2C_0r_i$ of $\supp\eta_{Q,S}$, and set $W_0:=W$.
Put $R_0:=2C_0+1$. Since
$\supp\eta_{Q,S}\subset U_{\frac{\gamma}{2}r_i}(Q)$ and
$\diam(Q)\leq2^{-10}r_i$,
\[
d_p\bigl(W_0,(Z_Q,\tau_Q)\bigr)
\leq
2C_0r_i+\frac{\gamma}{2}r_i+\diam(Q)
<R_0r_i.
\]
We recursively choose $W_h\in L_{P_h}$ so that
\begin{align}
d_p\bigl(W_h,(Z_{P_h},\tau_{P_h})\bigr)
\leq R_0\ell(P_h).
\label{multiscalepointlocation}
\end{align}
Indeed, if $W_h$ has been chosen, then
\[
d_p\bigl(W_h,(Z_{P_{h+1}},\tau_{P_{h+1}})\bigr)
\leq
\left(\frac{R_0}{2}+2^{-10}\right)\ell(P_{h+1}).
\]
\begingroup Since $R_0\leq1+\gamma/4$, the preceding estimate shows that $W_h\in C_{P_{h+1}}$. Applying Lemma~\ref{toalphalemma} to $P_h\subset P_{h+1}$ with $\eta=\frac12$, we may choose $W_{h+1}\in L_{P_{h+1}}$, with the same time coordinate as $W_h$, so that
\[
d_p(W_h,W_{h+1})\lesssim\alpha(P_{h+1})\ell(P_{h+1}).
\]
After decreasing $c_0$ if necessary, this proves \eqref{multiscalepointlocation} for $h+1$.\endgroup

Since $S_Q\in N(S)\cup\{S\}$,
\eqref{multiscalepointlocation} gives
\[
d_p\bigl(W_N,(Z_S,\tau_S)\bigr)
\leq(R_0+1)r_m.
\]
Applying Lemma~\ref{toalphalemma2}, we may choose
$W_{N+1}\in L_S$, with the same time coordinate as $W_N$, such that
\[
d_p(W_N,W_{N+1})\lesssim_{R_0}\alpha^\#(S)r_m.
\]
Since $p(W)$ is the orthogonal projection of $W$ onto $L_S$, the
triangle inequality gives
\[
\begin{aligned}
d_p(W,p(W))
&\leq d_p(W,W_{N+1})\\
&\leq \sum_{h=0}^{N}d_p(W_h,W_{h+1})\\
&\lesssim
\alpha^\#(S)r_m+
\sum_{P:\,Q\subseteq P\subseteq S_Q}\alpha(P)\ell(P)
=r_m\Delta.
\end{aligned}
\]
Suppose instead that $p(W)$ lies within $2C_0r_i$ of
$\supp\eta_{Q,S}$. Let $U_Q:=(X_Q,t_Q)$. Applying
Lemma~\ref{distLQalphabd2} to $Q\subset S_Q$, we may choose
$U_{S_Q}\in L_{S_Q}$, with the same time coordinate as $U_Q$, such that
\[
d_p(U_Q,U_{S_Q})
\lesssim
\sum_{P:\,Q\subsetneq P\subseteq S_Q}\alpha(P)\ell(P).
\]
Since $U_Q\in C_Q=C_{3r_i}(Z_Q,\tau_Q)$, $Q\subset S_Q$,
$r_i\leq r_m/2$, and $\diam(S_Q)\leq2^{-10}r_m$,
\[
d_p\bigl(U_Q,(Z_{S_Q},\tau_{S_Q})\bigr)
\leq 6r_i+\diam(S_Q)
<4r_m.
\]
Thus, after decreasing $c_0$ if necessary,
\[
d_p\bigl(U_{S_Q},(Z_{S_Q},\tau_{S_Q})\bigr)
\leq (4+C\Theta)r_m
<5r_m.
\]
It follows that $U_{S_Q}$ lies in
$C_{6r_m}(Z_{S_Q},\tau_{S_Q})$. Since $S_Q\in N(S)\cup\{S\}$,
Lemma~\ref{toalphalemma2} allows us to choose $U_S\in L_S$, with the same time
coordinate as $U_{S_Q}$, such that
\[
d_p(U_{S_Q},U_S)\lesssim\alpha^\#(S)r_m.
\]
Therefore
\begin{align}
d_p(U_Q,U_S)
&\lesssim
\alpha^\#(S)r_m+
\sum_{P:\,Q\subsetneq P\subseteq S_Q}\alpha(P)\ell(P)
\lesssim r_m\Delta.
\label{reverseplanedisplacement}
\end{align}
Since $p(U_Q)$ is the orthogonal projection of $U_Q$ onto
$L_S$, and $U_S$ has the same time coordinate as $U_Q$,
\begin{align}
d_p(U_Q,p(U_Q))
\leq d_p(U_Q,U_S)
\lesssim r_m\Delta.
\label{basepointprojectionbound}
\end{align}

Choose $(Y,s)\in\supp\eta_{Q,S}$ such that
$d_p(p(W),(Y,s))\leq2C_0r_i$. Since
$\supp\eta_{Q,S}\subset U_{\frac{\gamma}{2}r_i}(Q)$ and
$U_Q\in L_Q\cap C_Q$, we have
\[
d_p(p(W),U_Q)\lesssim r_i.
\]
Together with \eqref{basepointprojectionbound}, this gives
\[
d_p(p(W),p(U_Q))\lesssim r_i+r_m\Delta.
\]
The inverse map $p^{-1}:L_S\to L_Q$ has Lipschitz constant $\lesssim1$, and hence
\begin{align}
d_p(W,U_Q)
&=d_p\bigl(p^{-1}(p(W)),p^{-1}(p(U_Q))\bigr)\\
&\lesssim r_i+r_m\Delta.
\label{inverseprojectionbound}
\end{align}
Write $W=(X,t)$ and $U_Q=(X_Q,t_Q)$. Since $p$ preserves the
time coordinate, $W-p(W)$ and $U_Q-p(U_Q)$ are spatial vectors. The
spatial component of
\[
(W-p(W))-(U_Q-p(U_Q))
\]
is the orthogonal projection of $X-X_Q\in V_Q$ onto $V_S^\perp$.
Consequently, by \eqref{localplaneangle},
\[
\big|(W-p(W))-(U_Q-p(U_Q))\big|
\leq
\sin\angle(L_Q,L_S)|X-X_Q|
\lesssim
\Theta d_p(W,U_Q).
\]
It follows that
\[
\begin{aligned}
d_p(W,p(W))
&=|W-p(W)|\\
&\leq |U_Q-p(U_Q)|
+\big|(W-p(W))-(U_Q-p(U_Q))\big|\\
&\lesssim
d_p(U_Q,p(U_Q))
+\Theta d_p(W,U_Q).
\end{aligned}
\]
It follows from \eqref{basepointprojectionbound},
\eqref{inverseprojectionbound}, $\Theta r_i\leq r_m\Delta$, and
$\Theta<c_0$ that
\[
d_p(W,p(W))
\lesssim
r_m\Delta+\Theta(r_i+r_m\Delta)
\lesssim r_m\Delta.
\]
Thus
\begin{align}
d_p(W,p(W))\lesssim r_m\Delta,
\qquad W\in\Omega.
\label{localplanedisplacement}
\end{align}

For $0\leq h<N$, the estimate for the constants in
Lemma~\ref{toalphalemma}, applied to $P_h\subset P_{h+1}$, gives
\[
|c_{P_h}-c_{P_{h+1}}|
\lesssim
\alpha(P_{h+1}).
\]
If $S_Q\neq S$, then $S_Q\in N(S)$ and $\alpha^\#(S,S_Q)\leq\alpha^\#(S)\leq\Theta<c_0\leq\eta_0$. We may therefore apply Lemma~\ref{neighborprojectionmeasure}, and \eqref{neighborconstantcomparison} gives $|c_{S_Q}-c_S|\lesssim\alpha^\#(S,S_Q)\leq\alpha^\#(S)$. The same estimate is trivial when $S_Q=S$. Since $P_0=Q$ and $P_N=S_Q$, it follows that
\begin{align}
|c_Q-c_S|
&\leq
\sum_{h=0}^{N-1}|c_{P_h}-c_{P_{h+1}}|
+|c_{S_Q}-c_S|
\lesssim\Theta.
\label{multiscaleconstantcomparison}
\end{align}

Let $\widetilde p:V_Q\to V_S$ be the linear part of the spatial
orthogonal projection $p$, and let $\iota_Q:V_Q\to\mathbb R^n$ be the
inclusion map. By \eqref{localplaneangle},
\[
\|\widetilde p-\iota_Q\|\lesssim\Theta.
\]
If $J_p$ denotes the Jacobian of $\widetilde p$, then
\[
|J_p-1|\lesssim\Theta.
\]
Recall that $\mathcal L_U=c_U\sigma_{L_U}$. For every Borel set
$A\subset L_Q$,
\[
\begin{aligned}
p^\#\mathcal L_S(A)
&=\mathcal L_S(p(A))
=c_S\sigma_{L_S}(p(A))\\
&=c_SJ_p\sigma_{L_Q}(A)
=\frac{c_S}{c_Q}J_p\mathcal L_Q(A).
\end{aligned}
\]
Thus
\[
p^\#\mathcal L_S=a\mathcal L_Q,
\qquad
a:=\frac{c_S}{c_Q}J_p.
\]
Since $c_Q,c_S\approx1$, \eqref{multiscaleconstantcomparison} and the
Jacobian estimate imply
\begin{align}
|a-1|+\|\widetilde p-\iota_Q\|\lesssim\Theta.
\label{multiscaleprojectionbounds}
\end{align}

For brevity, write
\[
\eta:=\eta_{Q,S},
\qquad
\theta:=\theta_T.
\]
The derivative bounds for $\eta$ and $\theta$, together with
\eqref{localplanedisplacement}, give
\begin{align}
|\theta(p(W))-\theta(W)|
&\lesssim\Delta,
&&W\in\Omega,
\label{innercutoffcomparison}\\
|\eta(p(W))-\eta(W)|
&\lesssim
\min\left\{1,\frac{r_m}{r_i}\Delta\right\},
&&W\in L_Q.
\label{outercutoffcomparison}
\end{align}
\begingroup
If $W \in \Omega$, \eqref{outercutoffcomparison} holds by \eqref{localplanedisplacement}. If $W\in L_Q\setminus\Omega$, then, by the definition of $\Omega$,
\[
\dist(W,\supp\eta)>2C_0r_i,
\qquad
\dist(p(W),\supp\eta)>2C_0r_i.
\]
Neither $W$ nor $p(W)$ belongs to $\supp\eta$, so $\eta(W)=\eta(p(W))=0$. We have therefore proved \eqref{outercutoffcomparison} for every $W\in L_Q$. Moreover,
\[
\supp(\eta\circ p-\eta)\subset\supp\eta\cup p^{-1}(\supp\eta).
\]
Since $\supp\eta\subset U_{\frac{\gamma}{2}r_i}(Q)$ and $c_Q,c_S\approx1$,
\[
\mathcal L_Q(\supp\eta)+\mathcal L_S(\supp\eta)\lesssim r_i^{k+2}\approx\mu(Q).
\]
Since $p^\#\mathcal L_S=a\mathcal L_Q$ and $a\approx1$, we also have $\mathcal L_Q(p^{-1}(\supp\eta))\lesssim\mathcal L_S(\supp\eta)$. Therefore
\begin{align}
\mathcal L_Q\bigl(\supp\eta\cup p^{-1}(\supp\eta)\bigr)\lesssim\mu(Q).
\label{outercutoffsupport}
\end{align}
\endgroup

We now perform the change of variables. Define on $L_Q$
\[
\begin{aligned}
F&:=T_{(i)}(\theta\mathcal L_S)\circ p,
&
G&:=aT_{(i)}((\theta\circ p)\mathcal L_Q),\\
H&:=aT_{(i)}(\theta\mathcal L_Q),
&
J&:=T_{(i)}(\theta\mathcal L_Q).
\end{aligned}
\]
\begingroup Set\[\begin{aligned}I_Q&:=\langle\eta,T_{(i)}(\theta\mathcal L_Q)\rangle_{\mathcal L_Q},\\I_S&:=\langle\eta,T_{(i)}(\theta\mathcal L_S)\rangle_{\mathcal L_S}.\end{aligned}\]\endgroup
\begingroup Since $p^\#\mathcal L_S=a\mathcal L_Q$, the change of variables under $p$ gives\[I_S=a\int_{L_Q}(\eta\circ p)F\,d\mathcal L_Q,\qquad I_Q=\int_{L_Q}\eta J\,d\mathcal L_Q.\]Hence\[I_S-I_Q=a\int_{L_Q}(\eta\circ p)F\,d\mathcal L_Q-\int_{L_Q}\eta J\,d\mathcal L_Q.\]Now write\[F=(F-G)+(G-H)+(H-J)+J.\]Then\endgroup\begin{align}I_S-I_Q
={}&
a\int_{L_Q}(\eta\circ p)(F-G)\,d\mathcal L_Q
+a\int_{L_Q}(\eta\circ p)(G-H)\,d\mathcal L_Q
\nonumber\\
&+
a\int_{L_Q}(\eta\circ p)(H-J)\,d\mathcal L_Q
+\int_{L_Q}\bigl(a(\eta\circ p)-\eta\bigr)J\,d\mathcal L_Q.
\label{multiscalefourtermidentity}
\end{align}

We estimate the four terms in order. For $v\in V_Q$ and $t\in\mathbb R$, define\[\mathbf P(v,t):=(\widetilde p(v),t).\]Then, for $W,Z\in L_Q$,\[p(W)-p(Z)=\mathbf P(W-Z),\qquad\|\mathbf P(v,t)-(v,t)\|\lesssim\Theta|v|.\] If either
$K_{(i)}(U)$ or $K_{(i)}(\mathbf P U)$ is nonzero, then
$\|U\|\lesssim r_i$, and the kernel derivative estimate gives
\[
|K_{(i)}(\mathbf P U)-K_{(i)}(U)|
\lesssim\Theta r_i^{-k-2}.
\]
Cancellation on $L_Q$ and $L_S$ permits us to subtract
$\theta(p(W))$, and hence
\[
\begin{aligned}
(F-G)(W)
={}&a\int_{L_Q}
\bigl[
K_{(i)}(\mathbf P(W-Z))-K_{(i)}(W-Z)
\bigr]\\
&\qquad\times
\bigl[
\theta(p(Z))-\theta(p(W))
\bigr]\,d\mathcal L_Q(Z).
\end{aligned}
\]
Whenever the integrand is nonzero,
$d_p(p(W),p(Z))\lesssim r_i$, so the estimates of the derivatives of $\theta$ in \eqref{partitionhigherderivatives} give
$|\theta(p(Z))-\theta(p(W))|\lesssim r_i/r_m$. Therefore
\begin{align}
\|F-G\|_{L^\infty(p^{-1}(\supp\eta))}
\lesssim
\Theta\frac{r_i}{r_m}
\leq\Delta.
\label{multiscalekernelterm}
\end{align}

Next set $g:=\theta\circ p-\theta$. If
$W\in p^{-1}(\supp\eta)$ and $K_{(i)}(W-Z)\neq0$, then
$Z\in\Omega$. Thus \eqref{innercutoffcomparison} and the $L^1$ kernel
bound give
\begin{align}
|(G-H)(W)|
&\leq
a\int_{L_Q}|K_{(i)}(W-Z)|\,|g(Z)|\,d\mathcal L_Q(Z)
\lesssim\Delta.
\label{multiscaleinnerterm}
\end{align}
Moreover, $H-J=(a-1)J$, so
\eqref{fineplanecutoffbound}, \eqref{multiscaleprojectionbounds}, and
\eqref{weightedanglecomparison} give
\begin{align}
\|H-J\|_{L^\infty(L_Q)}
\lesssim
\Theta\frac{r_i}{r_m}
\leq\Delta.
\label{multiscalejacobianterm}
\end{align}
Finally,
\[
|a(\eta\circ p)-\eta|
\leq
|a-1|(\eta\circ p)+|\eta\circ p-\eta|.
\]
Using \eqref{fineplanecutoffbound},
\eqref{weightedanglecomparison}, and
\eqref{outercutoffcomparison}, we obtain
\begin{align}
\bigl\|
\bigl(a(\eta\circ p)-\eta\bigr)J
\bigr\|_{L^\infty(L_Q)}
\lesssim\Delta.
\label{multiscaleouterterm}
\end{align}
Combining
\eqref{multiscalefourtermidentity}--\eqref{multiscaleouterterm} with
\eqref{outercutoffsupport} proves
$|I_Q-I_S|\lesssim\Delta\mu(Q)$, which is
\eqref{multiscaleplanecomparisonestimate}.
\endgroup
\end{proof}

\begin{proof}[Proof of Lemma~\ref{neighborcoarsedifference}]
The hypothesis on $m-j$ and \eqref{projectionradiusvsdiamR} give $(4+\gamma)\ell(S)\leq\diam(R)$. Since $T\in N(S)$, for $U\in\{S,T\}$, \[ \|(X_U,\tau_U)-(Z_R,\tau_R)\| \leq (3+\gamma+2^{-9})\ell(S)+\diam(R) \leq 2\diam(R). \] Thus both $(X_S,\tau_S)$ and
$(X_T,\tau_T)$ lie in $C_{2\diam(R)}(Z_R,\tau_R)$, so
Lemma~\ref{bigest} applies with $Q=R$. We have
\[
\|(X_S,\tau_S)-(X_T,\tau_T)\|\lesssim\ell(S).
\]
We apply Lemmas~\ref{distLQalphabd2} and~\ref{toalphalemma2} and sum the estimates along the chain from $S$ to $R$ to obtain
\begin{align*}
\dist((X_S,\tau_S),L_R)
+
\dist((X_T,\tau_T),L_R)
&\lesssim
\alpha^\#(S)\ell(S)
+
\sum_{P:\,S\subseteq P\subseteq R}\alpha(P)\ell(P),
\\
\left|
\pi_{L_R^\perp}\bigl((X_S,\tau_S)-(X_T,\tau_T)\bigr)
\right|
&\lesssim
\ell(S)\left(\alpha^\#(S)+\sum_{P:\,S\subseteq P\subseteq R}\alpha(P)\right).
\end{align*}
Substituting these estimates into Lemma~\ref{bigest} and using
$\ell(P)\leq\ell(R)$ proves \eqref{neighborcoarsedifferenceestimate}.
\end{proof}

We now let $\delta\downarrow0$. For fixed $i,j$, \eqref{bjk1d} and \eqref{B1partitionremainder}, together with the estimates used in \eqref{B1mastersum}, give a summable bound independent of $\delta$ for the terms in the cancellation formula for $B^{1,\delta}_{j,i}$. Dominated convergence therefore gives its limit $B^1_{j,i}$. Likewise, \eqref{fineerrorbound} and \eqref{B2finecoarsemaster} give a summable bound independent of $\delta$ for $B^{2,\delta}_{j,i}$, so its limit $B^2_{j,i}$ exists. Since $B_{j,i}^\delta\to B_{j,i}$, the finite-sum identity $B_{j,i}^\delta=\tfrac12(B^{1,\delta}_{j,i}+B^{2,\delta}_{j,i})$ passes to the limit and proves \eqref{mv}. Fatou's lemma preserves the estimates after summation over $i,j$.

Combining the estimates for $B^1_{j,i}$ and $B^2_{j,i}$ with
\eqref{mv}, we conclude that
\[
\sum_{\substack{i,j\in\mathcal J\\
i-j>N_0}}
|B_{j,i}|
\lesssim
2^{-cN_0}
\sum_{P\in\mathbb D}\alpha(P)^2\mu(P).
\]

Combining the estimates for $A_{j,i}$ and $B_{j,i}$,
and using symmetry in $i$ and $j$, proves Proposition~\ref{offdiagprop}.

\appendix \section{Construction of the smooth partitions} \label{smoothpartitionappendix} We construct the two smooth families used in Section~3. \subsection{Construction of \texorpdfstring{\(\theta_S\)}{theta S}} Fix \(m\in\mathcal J\). For each \(S\in\mathbb D_m\), choose \(b_S\in C_c^\infty(\mathbb P^n)\) such that \begin{equation}\label{rawpartitioncutoffs} \begin{gathered} 0\leq b_S\leq1,\qquad b_S=1\quad\text{on } U_{\frac{\gamma}{32}\ell(S)}(S),\qquad \supp(b_S)\subset U_{\frac{\gamma}{8}\ell(S)}(S), \end{gathered} \end{equation} and \[ \big|\nabla_X^a\partial_t^q b_S\big| \lesssim_{a,q}\ell(S)^{-|a|-2q},\qquad |a|+q\leq2. \] We obtain $b_S$ by smoothing the indicator function of $U_{\gamma\ell(S)/16}(S)$. The bounded-overlap and separation properties of the generation-\(m\) dyadic cubes imply that the supports of the functions \(b_S\) have uniformly bounded overlap. Hence \[ B_m:=\sum_{T\in\mathbb D_m}b_T \] is locally finite, \(B_m\lesssim1\), and \[ \big|\nabla_X^a\partial_t^q B_m\big| \lesssim_{a,q}2^{m(|a|+2q)},\qquad |a|+q\leq2. \] If \((X,t)\in\Omega_m^{(1)}\), choose \((Y,s)\in E\) with \(\|(X,t)-(Y,s)\|<\gamma 2^{-m}/64\) and let \(S\in\mathbb D_m\) contain \((Y,s)\). Then \[ \dist\big((X,t),S\big)<\frac{\gamma}{64}2^{-m}<\frac{\gamma}{32}\ell(S), \] and therefore \(b_S(X,t)=1\). Thus \(B_m\geq1\) on \(\Omega_m^{(1)}\). Choose \(\chi_0\in C^\infty([0,\infty))\) with \(0\leq\chi_0\leq1\), \(\chi_0=0\) on \([0,1/4]\), and \(\chi_0=1\) on \([1/2,\infty)\), and set \[ H_0(s):=\begin{cases} \chi_0(s)/s,&s>0,\\[2pt] 0,&s=0. \end{cases} \] Since \(\chi_0\) vanishes near zero, \(H_0\) is smooth. Define \[ \theta_S:=b_SH_0(B_m). \] Then \(0\leq\theta_S\leq1\), \eqref{partitionneighborchoice} follows from the support of \(b_S\), and \[ \sum_{S\in\mathbb D_m}\theta_S=B_mH_0(B_m)=\chi_0(B_m)=1 \qquad\text{on }\Omega_m^{(1)}. \] This proves \eqref{ambientpartitionidentity}. The bounded overlap, the derivative estimates for \(b_S\), and the chain rule give \eqref{partitionhigherderivatives}. \subsection{Construction of \texorpdfstring{\(\psi_T\)}{psi T}} Fix \(m\in\mathcal J\), put \(r_m=2^{-m}\), and choose \(\widetilde\rho_T\in C_c^\infty(\mathbb P^n)\), \(T\in\mathbb D_m\), such that \[ 0\leq\widetilde\rho_T\leq1,\qquad \widetilde\rho_T=1\quad\text{on }U_{\frac{7\gamma}{16}r_m}(T),\qquad \supp(\widetilde\rho_T)\subset U_{\frac{\gamma}{2}r_m}(T), \] and \[ \big|\nabla_X^a\partial_t^q\widetilde\rho_T\big| \lesssim_{a,q}r_m^{-(|a|+2q)},\qquad |a|+q\leq2. \] Again, the supports have uniformly bounded overlap. Consequently \[ G_m:=\sum_{T\in\mathbb D_m}\widetilde\rho_T \] is locally finite, and the estimates of the derivatives of $\widetilde\rho_T$ give the same bounds for $G_m$. If \((X,t)\in\Omega_m^{(2)}\), choose \((Y,s)\in E\) with \(\|(X,t)-(Y,s)\|<3\gamma r_m/8\) and let \(T\in\mathbb D_m\) contain \((Y,s)\). Then \[ \dist\big((X,t),T\big)<\frac{3\gamma}{8}r_m<\frac{7\gamma}{16}r_m, \] so \(\widetilde\rho_T(X,t)=1\). Hence \(G_m\geq1\) on \(\Omega_m^{(2)}\). Choose \(\chi_1\) with the same properties as \(\chi_0\), set \(H_1(s)=\chi_1(s)/s\) for \(s>0\) and \(H_1(0)=0\), and define \[ \rho_T:=\widetilde\rho_TH_1(G_m). \] Then \[ \sum_{T\in\mathbb D_m}\rho_T=1\qquad\text{on }\Omega_m^{(2)}, \] while \(\supp(\rho_T)\subset U_{\gamma r_m/2}(T)\) and \(\big|\nabla_X^a\partial_t^q\rho_T\big|\lesssim_{a,q}r_m^{-(|a|+2q)}\) for \(|a|+q\leq2\). Finally, let \[ F_m:=\sum_{T\in\mathbb D_m}\theta_T \qquad\text{and}\qquad \psi_T:=\theta_T+\rho_T(1-F_m). \] By \eqref{ambientpartitionidentity}, \(F_m=1\) on \(E\), which gives \eqref{planeambientpartitionidentity}. On \(\Omega_m^{(2)}\), \[ \sum_{T\in\mathbb D_m}\psi_T=F_m+(1-F_m)\sum_{T\in\mathbb D_m}\rho_T=1, \] which is \eqref{planeambientpartitionsum}. The support inclusions for \(\theta_T\) and \(\rho_T\) give \(\supp(\psi_T)\subset U_{\gamma r_m/2}(T)\). They also imply uniformly bounded overlap. Finally, the estimates of the derivatives of $\theta_T$ in \eqref{partitionhigherderivatives} and bounded overlap give the corresponding estimates for $F_m$. Applying the product rule to $\psi_T=\theta_T+\rho_T(1-F_m)$ gives \[ |\psi_T|\lesssim1\qquad\text{and}\qquad \big|\nabla_X^a\partial_t^q\psi_T\big| \lesssim_{a,q}r_m^{-(|a|+2q)},\qquad |a|+q\leq2. \] 

\begingroup

\section{Singular integrals on regular parabolic Lipschitz graphs}\label{bhhlnappendix}

We prove Theorem~\ref{highcodimgeneralSIO} and then deduce Theorem~\ref{highcodimBHHLN}. We use the following codimension-one result.

\begin{theorem}[{\cite[Theorem~5]{BHHLN3}}]
\label{BHHLNcalderon}
Let $k\geq1$. There is an integer $N_*=N_*(k)$ for which the following holds. Let $G:\mathbb P^k\to\mathbb R$ be regular Lip-$(1,1/2)$, and let $K\in C^{N_*}(\mathbb P^{k+1}\setminus\{0\};\mathbb R)$ satisfy $K(-X,t)=-K(X,t)$ for $(X,t)\in\mathbb P^{k+1}\setminus\{0\}$. Suppose that
\[
\mathfrak C_{N_*}(K)
:=
\max_{|\alpha|+q\leq N_*}
\sup_{(X,t)\neq0}
\|(X,t)\|^{k+2+|\alpha|+2q}
|\partial_X^\alpha\partial_t^qK(X,t)|
<\infty,
\]
where $\alpha$ ranges over spatial multiindices and $q$ over nonnegative
integers. Define
\[
\mathcal T_\varepsilon^{G,K}f(x,t)
:=
\int_{\|(x-y,t-s)\|>\varepsilon}
K\bigl(x-y,G(x,t)-G(y,s),t-s\bigr)f(y,s)\,\d y\,\d s,
\qquad (x,t)\in\mathbb P^k.
\]
Then
\[
\sup_{\varepsilon>0}
\|\mathcal T_\varepsilon^{G,K}f\|_2
\leq C(G)\mathfrak C_{N_*}(K)\|f\|_2,
\]
where $C(G)$ depends only on $k$ and the regular Lip-$(1,1/2)$ constants
of $G$. All function norms are taken on $\mathbb P^k$.
\end{theorem}

The proof in
\cite{BHHLN3} gives a uniform bound when the constants in the size and
derivative estimates for $K$ are at most $1$. If
$M:=\mathfrak C_{N_*}(K)>0$, then $K/M$ satisfies these normalized
estimates. Applying the bound to $K/M$ and multiplying the resulting
operator by $M$ gives the factor $\mathfrak C_{N_*}(K)$ in our statement.

In \cite[Definition~3]{BHHLN3}, the regularity of $G$ is measured using the commutator norm.
By \cite[Section~2.6, equation~(2.21)]{BHMN}, this norm is quantitatively
equivalent to $\|\nabla_xG\|_\infty+\|D_t^{1/2}G\|_{BMO}$.
Thus $C(G)$ depends only on the regular Lip-$(1,1/2)$ constants
of $G$.

\subsection{General kernels for \texorpdfstring{$k\geq1$}{k greater than or equal to 1}}

\begin{proof}[Proof of Theorem~\ref{highcodimgeneralSIO} for $k\geq1$]
Let $N_*$ be the integer in Theorem~\ref{BHHLNcalderon}, and choose
$$
M=N_*+n-k+1,\qquad N=N_*+M=2N_*+n-k+1.
$$
In this proof, $C$ may depend on $n$, $k$, the regular Lip-$(1,1/2)$ constants of $F$, the constants $C_{j,q}$ with $j+q\leq N$, and the functions $\psi$ and $\chi$ chosen below and the function $\psi_0$ from~\eqref{psi0}. It is independent of $\zeta$ and $\ell$.

Choose a real, even function $\psi\in C_c^\infty(\mathbb R^{n-k})$ equal to $1$ on the unit ball and supported in the ball of radius $2$, and set $L=1+2\|F\|_{lip}$. For $(x,t)\neq0$, write $r=\rho(x,t)$ and define
\begin{equation}\label{generalfourierkernel}
H_\zeta(x,t)=\int_{\mathbb R^{n-k}}\psi(w/L)K(x,rw,t)e^{-2\pi i\zeta\cdot w}\,\d w,\qquad \zeta\in\mathbb R^{n-k}.
\end{equation}
We first estimate the derivatives of $H_\zeta$. On the support of $\psi(w/L)$, the norms $\|(x,rw,t)\|$ and $r$ are comparable, with constants depending only on $L$. The derivative estimates for $K$, the product rule, and the bounds $|\partial_x^\alpha\partial_t^q\rho|\leq C_{\alpha,q}r^{1-|\alpha|-2q}$ give
$$
|\partial_x^\alpha\partial_t^q\partial_w^\beta[\psi(w/L)K(x,rw,t)]|\leq Cr^{-k-2-|\alpha|-2q},\qquad |\alpha|+q\leq N_*,\quad |\beta|\leq M.
$$
Only derivatives of $K$ of total order at most $N_*+M=N$ occur here. The integrand is supported in $|w|\leq2L$. For $|\zeta|\geq1$, integrate by parts $M$ times in the direction $\zeta/|\zeta|$ in~\eqref{generalfourierkernel}; for $|\zeta|<1$, integrate its absolute value. This proves
\begin{equation}\label{generalfourierderivatives}
|\partial_x^\alpha\partial_t^qH_\zeta(x,t)|\leq C(1+|\zeta|)^{-M}\rho(x,t)^{-k-2-|\alpha|-2q},\qquad |\alpha|+q\leq N_*.
\end{equation}
Since $K$ is real and spatially odd, changing $w$ to $-w$ in~\eqref{generalfourierkernel} gives $H_\zeta(-x,t)=-\overline{H_\zeta(x,t)}$. Thus $\operatorname{Re}H_\zeta$ is odd in $x$ and $\operatorname{Im}H_\zeta$ is even in $x$. Fourier inversion gives
\begin{equation}\label{generalfourierinversion}
\psi(z/(Lr))K(x,z,t)=\int_{\mathbb R^{n-k}}[\cos(2\pi\zeta\cdot z/r)\operatorname{Re}H_\zeta(x,t)-\sin(2\pi\zeta\cdot z/r)\operatorname{Im}H_\zeta(x,t)]\,\d\zeta.
\end{equation}
The integral is absolutely convergent by~\eqref{generalfourierderivatives}, since $M>n-k$.

For $\zeta\neq0$, set $F_\zeta=(\zeta/|\zeta|)\cdot F$. Its regular Lip-$(1,1/2)$ constants are bounded by those of $F$, independently of $\zeta$. Choose a real, even function $\chi\in C_c^\infty(\mathbb R)$ equal to $1$ on $[-1,1]$ and supported in $[-2,2]$. With $z$ now the scalar transversal coordinate, define
\begin{equation}\label{generalauxiliarykernel}
K_\zeta(X,t)=\chi(z/(Lr))\operatorname{Re}[e^{2\pi i|\zeta|z/r}H_\zeta(x,t)].
\end{equation}
The formula applies when $(x,t)\neq0$; set $K_\zeta(0,z,0)=0$ for $z\neq0$. The factor $\chi(z/(Lr))$ makes $K_\zeta$ vanish in a neighborhood of each such point, so $K_\zeta\in C^{N_*}(\mathbb P^{k+1}\setminus\{0\})$. Since $H_\zeta(-x,t)=-\overline{H_\zeta(x,t)}$ and $\chi$ is even, $K_\zeta(-X,t)=-K_\zeta(X,t)$. On its support, $|z|\leq2Lr$, so $\|(X,t)\|$ and $r$ are comparable. Differentiating~\eqref{generalauxiliarykernel} at most $N_*$ times introduces at most $N_*$ powers of $|\zeta|$. Therefore~\eqref{generalfourierderivatives} gives
\begin{equation}\label{generalauxiliaryseminorm}
\mathfrak C_{N_*}(K_\zeta)\leq C(1+|\zeta|)^{-M+N_*}=C(1+|\zeta|)^{-(n-k)-1}.
\end{equation}
Theorem~\ref{BHHLNcalderon}, applied to $F_\zeta$ and $K_\zeta$, gives $\|\mathcal T_\varepsilon^{F_\zeta,K_\zeta}f\|_2\leq C(1+|\zeta|)^{-(n-k)-1}\|f\|_2$, uniformly in $\varepsilon>0$. Define the smooth truncations by
$$
T_{\zeta,\ell}f(x,t)=\int_{\mathbb P^k}K_\zeta(x-y,F_\zeta(x,t)-F_\zeta(y,s),t-s)\Psi_\ell(x-y,t-s)f(y,s)\,\d y\,\d s.
$$
For $\varepsilon_\ell=2^{-\ell-5}$, the sharp and smooth cutoffs differ only where $\varepsilon_\ell<\|(x-y,t-s)\|<8\varepsilon_\ell$. By~\eqref{generalauxiliaryseminorm}, the difference between the sharp and smooth truncations applied to $f$ is bounded in absolute value by
$$
C(1+|\zeta|)^{-(n-k)-1}\int_{\varepsilon_\ell<\|(x-y,t-s)\|<8\varepsilon_\ell}\frac{|f(y,s)|}{\|(x-y,t-s)\|^{k+2}}\,\d y\,\d s\leq C(1+|\zeta|)^{-(n-k)-1}\mathcal M f(x,t).
$$
Here $\mathcal M$ is the parabolic Hardy--Littlewood maximal operator. Its $L^2$ boundedness and the sharp truncation estimate give
\begin{equation}\label{generalsmoothbound}
\|T_{\zeta,\ell}f\|_2\leq C(1+|\zeta|)^{-(n-k)-1}\|f\|_2.
\end{equation}

For $(x,t)\neq(y,s)$, the Lipschitz bound and the definition of $L$ give
$$
|F_\zeta(x,t)-F_\zeta(y,s)|\leq|F(x,t)-F(y,s)|\leq L\rho(x-y,t-s).
$$
Consequently,
$$
\psi(\frac{F(x,t)-F(y,s)}{L\rho(x-y,t-s)})=1,\qquad \chi(\frac{F_\zeta(x,t)-F_\zeta(y,s)}{L\rho(x-y,t-s)})=1.
$$
Substituting~\eqref{generalfourierinversion} into the definition of $T_\ell^F$ gives
$$
T_\ell^F f=\int_{\mathbb R^{n-k}}T_{\zeta,\ell}f\,\d\zeta.
$$
For $f\in C_c^\infty(\mathbb P^k)$ and fixed $\ell$, $\Psi_\ell(x-y,t-s)$ vanishes near $(x,t)=(y,s)$, and~\eqref{generalfourierderivatives} gives an integrable bound in $\zeta$. Fubini's theorem therefore justifies the displayed identity after integrating both sides against a compactly supported smooth function. Minkowski's inequality and~\eqref{generalsmoothbound} then give
$$
\|T_\ell^Ff\|_2\leq C\|f\|_2\int_{\mathbb R^{n-k}}(1+|\zeta|)^{-(n-k)-1}\,\d\zeta\leq C\|f\|_2.
$$
Density proves the theorem for every $f\in L^2(\mathbb P^k)$, uniformly in $\ell$.
\end{proof}

\subsection{General kernels for \texorpdfstring{$k=0$}{k=0}}

\begin{proof}[Proof of Theorem~\ref{highcodimgeneralSIO} for $k=0$]
Here $\mathbb P^0=\mathbb R^1$ with distance $d_p(t,s)=|t-s|^{1/2}$. Write $A:\mathbb P^0\to\mathbb R^n$ for the function in the theorem and set $F(x,t)=A(t)$ on $\mathbb P^1$. Then $F$ has the same Lip-$(1,1/2)$ constant as $A$, and $D_t^{1/2}F(x,t)=D_t^{1/2}A(t)$. Averages over parabolic cylinders give
$$
\|D_t^{1/2}F\|_{BMO(\mathbb P^1)}=\|D_t^{1/2}A\|_{BMO(\mathbb P^0)}.
$$
Thus $F$ is regular Lip-$(1,1/2)$. We will apply the case $k=1$, in ambient space $\mathbb P^{n+1}$, with a kernel whose integral in the additional spatial variable equals $K$.

Choose a nonnegative even function $\eta\in C_c^\infty((-1,1))$ with $\int_{\mathbb R}\eta=1$. For $(z,t)\neq0$, define
\begin{equation}\label{generalextendedkernel}
\widetilde K(x,z,t)=\rho(z,t)^{-1}\eta(x/\rho(z,t))K(z,t),\qquad (x,z,t)\in\mathbb P^{n+1}.
\end{equation}
Set $\widetilde K(x,0,0)=0$ for $x\neq0$. The factor $\eta(x/\rho(z,t))$ vanishes in a neighborhood of each such point, so $\widetilde K$ is $C^N$ away from the origin. Since $\eta$ is even and $K(-z,t)=-K(z,t)$, we have $\widetilde K(-x,-z,t)=-\widetilde K(x,z,t)$. On the support of the kernel and of its derivatives, $|x|\leq\rho(z,t)$, and $\|(x,z,t)\|$ is comparable to $\rho(z,t)$. The function $\rho(z,t)^{-1}\eta(x/\rho(z,t))$, extended by zero when $(z,t)=0$ and $x\neq0$, is homogeneous of degree $-1$ under parabolic dilations and smooth away from the origin. Consequently, the product rule gives
\begin{equation}\label{generalextendedderivatives}
|\partial_{x,z}^\alpha\partial_t^q\widetilde K(x,z,t)|\leq C\|(x,z,t)\|^{-3-|\alpha|-2q},\qquad |\alpha|+q\leq N,
\end{equation}
where $C$ depends only on $n$, $N$, $\eta$, and the derivative constants for $K$ through order $N$. No derivatives of higher order occur. We may therefore take $N(n,0)=N(n+1,1)$. Applying the already proved $k=1$ case to $F$ and $\widetilde K$ gives uniform $L^2(\mathbb P^1)$ bounds with cutoff $\Psi_\ell(x-y,t-s)$.

We next replace that cutoff by $\Psi_\ell(t-s)$. Put $r=2^{-\ell-4}$, $L=\|A\|_{lip}$, and $D=(L^4+2)^{1/4}$. On the support of $\widetilde K(x-y,A(t)-A(s),t-s)$,
$$
|t-s|^{1/2}\leq\rho(x-y,t-s)\leq D|t-s|^{1/2}.
$$
Indeed, $|x-y|\leq\rho(A(t)-A(s),t-s)$ and $|A(t)-A(s)|\leq L|t-s|^{1/2}$. Since $\Psi_\ell=1-\psi_0(\rho/r)$, the two cutoffs can differ only when
$$
\frac{r}{2D}<|t-s|^{1/2}<2r.
$$
The nonnegativity and normalization of $\eta$ give
$$
\int_{\mathbb R}|\widetilde K(x-y,A(t)-A(s),t-s)|\,\d y=|K(A(t)-A(s),t-s)|\leq\frac{C_{0,0}}{|t-s|}.
$$
Integrating $C_{0,0}/|t-s|$ over $r/(2D)<|t-s|^{1/2}<2r$ shows that the integral in $(y,s)$ of the absolute difference between the two truncated kernels is at most $CC_{0,0}\log(4D)$. The same bound holds when integrating in $(x,t)$ instead of $(y,s)$. Schur's test therefore proves that the operators
$$
\widetilde T_\ell f(x,t)=\int_{\mathbb P^1}\widetilde K(x-y,A(t)-A(s),t-s)\Psi_\ell(t-s)f(y,s)\,\d y\,\d s
$$
satisfy
\begin{equation}\label{generalextendedoperatorbound}
\|\widetilde T_\ell f\|_{L^2(\mathbb P^1)}\leq C\|f\|_{L^2(\mathbb P^1)},
\end{equation}
where $C$ depends only on $n$, the regular Lip-$(1,1/2)$ constants of $A$, and the constants $C_{j,q}$ with $j+q\leq N$, and is independent of $\ell$.

Choose a real function $\chi\in C_c^\infty(\mathbb R)$ with $\|\chi\|_2=1$ and put $\chi_R(x)=R^{-1/2}\chi(x/R)$. For fixed $t\neq s$, a change of variables rewrites
$$
\int_{\mathbb R}\int_{\mathbb R}\widetilde K(x-y,A(t)-A(s),t-s)\chi_R(y)\chi_R(x)\,\d y\,\d x
$$
as
$$
\int_{\mathbb R}\widetilde K(y,A(t)-A(s),t-s)\int_{\mathbb R}\chi(x-y/R)\chi(x)\,\d x\,\d y.
$$
The inner integral tends to $1$ as $R\to\infty$ and has absolute value at most $1$. Also,~\eqref{generalextendedkernel} gives
$$
\int_{\mathbb R}\widetilde K(y,A(t)-A(s),t-s)\,\d y=K(A(t)-A(s),t-s).
$$
Dominated convergence therefore shows that the double integral tends to $K(A(t)-A(s),t-s)$, and its absolute value is bounded by $|K(A(t)-A(s),t-s)|$. For real $f,g\in C_c^\infty(\mathbb R)$ and fixed $\ell$, $\Psi_\ell(t-s)$ vanishes near $t=s$. We can thus use dominated convergence once more to obtain
$$
\langle T_\ell^A f,g\rangle=\lim_{R\to\infty}\langle\widetilde T_\ell(\chi_R(x)f(t)),\chi_R(x)g(t)\rangle.
$$
Since $\|\chi_R(x)f(t)\|_{L^2(\mathbb P^1)}=\|f\|_{L^2(\mathbb R)}$, estimate~\eqref{generalextendedoperatorbound} implies
$$
|\langle T_\ell^A f,g\rangle|\leq C\|f\|_{L^2(\mathbb R)}\|g\|_{L^2(\mathbb R)}.
$$
Density and duality prove the required $L^2$ bound, uniformly in $\ell$.
\end{proof}

To obtain Theorem~\ref{highcodimBHHLN}, we apply Theorem~\ref{highcodimgeneralSIO} to the kernels $J^m$ in~\eqref{rieszkernel}, for $0\leq k\leq n-1$. Each $J^m$ is odd in $X$ and smooth away from the origin: at $t=0$ and $X\neq0$, the Gaussian and all its derivatives tend to zero faster than any power of $|t|$. Moreover,
$$
J^m(\lambda X,\lambda^2t)=\lambda^{-k-2}J^m(X,t),\qquad \lambda>0.
$$
Differentiating this identity and bounding the derivatives on $\|(X,t)\|=1$ gives
$$
|\partial_X^\alpha\partial_t^qJ^m(X,t)|\leq C_{n,k,\alpha,q}\|(X,t)\|^{-k-2-|\alpha|-2q}
$$
for every spatial multiindex $\alpha$ and integer $q\geq0$. Thus each $J^m$ satisfies the hypotheses of Theorem~\ref{highcodimgeneralSIO}. Applying that theorem for $m=1,\ldots,n-k$ and summing the operator bounds proves Theorem~\ref{highcodimBHHLN}.

\endgroup

\bibliographystyle{alpha}
\bibliography{parabolic_UR}

\end{document}